\documentclass[11pt]{amsart}
\usepackage[left=10mm,right=10mm]{geometry}
\usepackage[english]{babel}
\usepackage[T1]{fontenc}
\usepackage[utf8]{inputenc}
\usepackage{amsthm}
\usepackage{amsmath}
\title{Limit theorems for the trajectory of the self-repelling random walk with directed edges}
\author{Laure Mar{\^e}ch\'e}
\email{laure.mareche@math.unistra.fr}
\address{Institut de Recherche Mathématique Avancée, 
UMR 7501 Université de Strasbourg et CNRS, 
7 rue René-Descartes, 67000 Strasbourg, France}
\author{Thomas Mountford}
\email{thomas.mountford@epfl.ch}
\address{Institut de Mathématiques, École Polytechnique Fédérale de Lausanne (ÉPFL), 
Station 8, CH-1015 Lausanne}
\theoremstyle{plain}
\newtheorem{theorem}{Theorem}

\newtheorem{lemma}[theorem]{Lemma}
\newtheorem{proposition}[theorem]{Proposition}
\newtheorem{observation}[theorem]{Observation}
\newtheorem{claim}[theorem]{Claim}
\theoremstyle{definition}
\newtheorem{definition}[theorem]{Definition}
\theoremstyle{remark}
\newtheorem{remark}[theorem]{Remark}
\usepackage{dsfont}
\usepackage{hyperref}

\begin{document}

\begin{abstract}
The self-repelling random walk with directed edges was introduced by Tóth and Vető in 2008 \cite{Toth_et_al2008} as a nearest-neighbor random walk on $\mathds{Z}$ that is non-Markovian: at each step, the probability to cross a directed edge depends on the number of previous crossings of this directed edge. Tóth and Vető found this walk to have a very peculiar behavior, and conjectured that, denoting the walk by $(X_m)_{m\in\mathds{N}}$, for any $t \geq 0$ the quantity $\frac{1}{\sqrt{N}}X_{\lfloor Nt \rfloor}$ converges in distribution to a non-trivial limit when $N$ tends to $+\infty$, but the process $(\frac{1}{\sqrt{N}}X_{\lfloor Nt \rfloor})_{t \geq 0}$ does \emph{not} converge in distribution. In this paper, we prove not only that $(\frac{1}{\sqrt{N}}X_{\lfloor Nt \rfloor})_{t \geq 0}$ admits no limit in distribution in the standard Skorohod topology, but more importantly that the trajectories of the random walk still satisfy another limit theorem, of a new kind. Indeed, we show that for $n$ suitably smaller than $N$ and $T_N$ in a large family of stopping times, the process $(\frac{1}{n}(X_{T_N+tn^{3/2}}-X_{T_N}))_{t \geq 0}$ admits a non-trivial limit in distribution. The proof partly relies on combinations of reflected and absorbed Brownian motions which may be interesting in their own right.
\end{abstract}

\maketitle


\section{Introduction}

The ``true'' self-avoiding random walk was introduced by Amit, Parisi and Peliti in \cite{Amit_et_al1983} in order to approximate a random self-avoiding path on $\mathds{Z}^d$, which cannot be constructed step by step in a straightforward way, by a random walk constructed step by step. In dimension 1, it is a random walk on $\mathds{Z}$ that is discrete-time, nearest-neighbor and \emph{non-Markovian} (in this paper, the term ``random walk'' will often be used for non-Markovian processes), defined so that at each time, if the process is at $i\in\mathds{Z}$, it may go to $i+1$ or $i-1$ with a transition probability depending on the time already spent by the process at sites $i+1$ and $i-1$ (the \emph{local time} at these sites). This transition probability is defined so that the process is \emph{self-repelling}: if the process spent more time at $i+1$ than at $i-1$ in the past, it will have a larger probability to go to $i-1$ than to $i+1$.

However, the non-Markovian nature of the ``true'' self-avoiding random walk makes it hard to study. This led to the introduction by Tóth in the fundamental series of papers \cite{Toth1994,Toth1995,Toth1996} of models where the probability of going to $i+1$ or $i-1$ does not depend on the local time at the sites $i+1$ and $i-1$, but instead of the local time of the non-oriented edges $\{i,i+1\}$ and $\{i,i-1\}$, that is of the number of times the process already went through these edges. These processes are easier to study because they allow the use of a \emph{Ray--Knight argument}: under some conditions, the local times on the edges form a Markov process, and its Markovian nature allows its analysis. This kind of argument was first used for simple random walks (see the original papers of Knight \cite{Knight1963} and Ray \cite{Ray1963}), then applied to random walks in random environments in \cite{Kesten_et_al_1975}. In \cite{Toth1994,Toth1995,Toth1996}, Tóth was able to extend this Ray-Knight argument to self-repelling random walks and proved that the process of their local times, once properly rescaled, converges in distribution. The limit, as well as the rescaling, depends on the exact definition of the transition probabilities, but is always a random process, either a power of a reflected Brownian motion or a gluing of squared Bessel processes (a non-Markovian random walk with a deterministic limit was studied by Tóth in \cite{Toth1997}, but it is very different as it is \emph{self-attracting} instead of self-repelling: the more an edge was crossed in the past, the more likely it is to be crossed again).

In \cite{Toth_et_al2008}, Tóth and Vető introduced a self-repelling random walk whose transition probabilities are defined through the local time on \emph{oriented} edges rather than non-oriented ones. This random walk $(X_m)_{m \in \mathds{N}}$ on $\mathds{Z}$ is defined as follows. Let $w : \mathds{Z} \mapsto (0,+\infty)$ be a non-decreasing, non-constant function. If the cardinal of a set $A$ is denoted by $|A|$, for any $m \in \mathds{N}$, $i\in \mathds{Z}$, we denote 
\begin{equation}\label{eq_def_ell}
\ell_{m,i}^\pm = |\{0 \leq k \leq m-1 \,|\, (X_k,X_{k+1})=(i,i\pm1)\}|
\end{equation}
the local time of the oriented edge $(i,i\pm1)$, and 
\begin{equation}\label{eq_def_Delta}
\Delta_{m,i} = \ell_{m,i}^- -\ell_{m,i}^+.
\end{equation}
We then set $X_0=0$, and for all $m\in \mathds{N}$,
\begin{equation}\label{eq_def_Xn}
 \mathds{P}(X_{m+1}=X_m+1) = 1-\mathds{P}(X_{m+1}=X_m-1) = \frac{w(\Delta_{m,X_m})}{w(\Delta_{m,X_m})+w(-\Delta_{m,X_m})}.
\end{equation}
On an intuitive level, it is not a priori clear why this processs should behave differently from the processes with non-oriented edges, especially the process introduced by Tóth in \cite{Toth1995}, which seems to be very similar when $w$ is exponential. However, the process of Tóth and Vető \cite{Toth_et_al2008} exhibits a sharply different behavior. Indeed, building on the Ray Knight techniques developed by Tóth in \cite{Toth1994,Toth1995,Toth1996}, Tóth and Vető proved in \cite{Toth_et_al2008} that the renormalized process of the local times of $(X_m)_{m \in \mathds{N}}$ does converge, but to a \emph{deterministic} limit forming a triangle $f(x) = (1 -|x|)_+$, instead of a random process (the fluctuations around this deterministic limit were studied by the first author in \cite{Mareche2022}). Since this model behaves differently from the self-repelling models previously studied, it is interesting to explore its behavior in more depth.

In \cite{Mountford_et_al2014}, Pimentel, Valle and the second author proved that $\frac{X_n}{\sqrt{n}}$ converges in distribution to the uniform distribution on $[-1,1]$: the random walk has a \emph{diffusive scaling}. This suggests the process $(\frac{1}{n}X_{\lfloor n^2t \rfloor})_{t \geq 0}$ should converge in distribution when $n$ tends to $+\infty$, which would be a \emph{diffusive renormalization}. However, the simulations of Tóth and Vető in \cite{Toth_et_al2008} seem to indicate that this process does not converge. This is the starting point of this work.

We prove not only that $(\frac{1}{n}X_{\lfloor n^2t \rfloor})_{t \geq 0}$ has no limit with respect to the topology of continuous real processes on $[0,+\infty)$, but the stronger result that there is no limit point in the standard Skorohod topology for càdlàg processes on $[0,+\infty)$ (see \cite{PollardJ1} or \cite{Ethier_Kurtz1986} for an introduction to this topology).
 \begin{proposition}\label{prop_continuity}
  $(\frac{1}{n}X_{\lfloor t n^2 \rfloor})_{t \in [0,+\infty)}$ admits no limit point in distribution in the standard Skorohod topology for càdlàg processes on $[0,+\infty)$ when $n$ tends to $+\infty$.
 \end{proposition}
 Proposition \ref{prop_continuity} means that there is no diffusive renormalization, but we show nonetheless that a non-trivial renormalization of the process exists. This renormalization is the first result of its kind to our knowledge: we show that there exist stopping times of order $N^2$ so that for $n \ll N$, the random walk started at these stopping times and considered on a scale $n$ admits a \emph{superdiffusive renormalization}. We define the scale $n$ as follows:
 \begin{equation}\label{eq_scale_n}
  n=\phi(N)\,\text{ where }\, \exists\, \alpha > 1,N_0 \in \mathds{N}^*,\forall\, N \geq N_0, \phi(N) \leq N^{1/\alpha}\text{ and }\lim_{N \to +\infty}\phi(N)=+\infty
 \end{equation}
 and $\phi : \mathds{N}^* \mapsto \mathds{N}^*$. For any $m \in \mathds{N}$, $i\in \mathds{Z}$, let us denote 
 \begin{equation}\label{eq_def_T_mi}
 \mathbf{T}_{m,i}^\pm = \inf\{k \geq 0 \,|\, \ell_{k,i}^\pm=m\}.
 \end{equation}
 We set $\theta > 0$, $x \in \mathds{R}$. For any $N \in \mathds{N}^*$, we denote $(Y_t^{N})_{t \in \mathds{R}^+}$ the continuous process defined by 
 \begin{equation}\label{eq_def_YN}
 Y_t^N=\frac{X_{\mathbf{T}_{\lfloor N \theta\rfloor,\lfloor Nx \rfloor}^\pm+t n^{3/2}}-X_{\mathbf{T}_{\lfloor N \theta\rfloor,\lfloor Nx \rfloor}^\pm}}{n}
 \end{equation}
 when $t n^{3/2}$ is an integer and by linear interpolation otherwise. We prove the following. 
\begin{theorem}\label{thm_main}
$(Y_t^N)_{t \in [0,+\infty)}$ converges in distribution in the topology of continuous real processes on $[0,+\infty)$ when $N$ tends to $+\infty$, to a limit different from the null function. 
\end{theorem}
Theorem \ref{thm_main} means that locally after $\mathbf{T}_{\lfloor N \theta\rfloor,\lfloor Nx \rfloor}^\pm$, the process $(X_m)_{m \in \mathds{N}}$ has a superdiffusive behavior. It thus fluctuates more quickly than diffusively, which explains why $(\frac{1}{n}X_{\lfloor t n^2 \rfloor})_{t \in [0,+\infty)}$ admits no limit in distribution. Once Theorem \ref{thm_main} is established, proving Proposition \ref{prop_continuity} is rather easy. In order to show Theorem \ref{thm_main}, we follow the approach recently introduced by Kosygina, Peterson and the second author \cite{Kosygina_et_al2022} for another kind of non-Markovian random walk, called an excited random walk with Markovian cookie stacks. They used that approach to prove the convergence of their renormalized random walk to a Brownian motion perturbed at extrema. For some $\varepsilon>0$, we consider ``mesoscopic times'' depending on $N$: $T_0=\mathbf{T}_{\lfloor N \theta\rfloor,\lfloor Nx \rfloor}^\pm$, and $T_{k+1}$ is the first moment $m$ after $T_k$ at which $|X_{m}-X_{T_{k}}| = \lfloor\varepsilon n\rfloor$. The convergence of $(Y_t^N)_{t \in [0,+\infty)}$ can be deduced from the convergence in distribution of the $\frac{1}{n}(X_{T_{k+1}}-X_{T_{k}})$ and the $\frac{1}{n^{3/2}}(T_{k+1}-T_{k})$ when $N$, and therefore $n$, tends to $+\infty$, which is obtained by using Ray-Knight arguments for the process $(X_{T_k+m})_{m\in\mathds{N}}$. 

There are important differences between the argument in \cite{Kosygina_et_al2022} and ours. In \cite{Kosygina_et_al2022}, the behavior when $X_m$ was near the extremities of the range of $(X_m)_{0 \leq m \leq T_k}$ was different from its behavior in the ``bulk'' of the range. In our work, the normalization considered keeps the process far from the extremities of the range, so we never need to take this different behavior into account.

Furthermore, the Ray-Knight arguments for the process $(X_{T_k+m})_{m\in\mathds{N}}$ give a different law for its local times than in \cite{Kosygina_et_al2022}, so they need a different treatment. Interestingly, the behavior of $(X_{T_k+m})_{m\in\mathds{N}}$ is close to that of the random walk of \cite{Toth1995}, which allows to use arguments similar to those in \cite{Toth1995}, though the processes are different enough so they do not suffice. We roughly have that $(\sum_j^i \Delta_{T_{k+1},j})_i$ is a random walk reflected on or absorbed by $(\sum_j^i \Delta_{T_{k},j})_i$ (see Definition \ref{def_reflected_BM} for the notion of reflection), which we may consider as an ``environment'', hence $(\frac{1}{\sqrt{n}}\sum_j^i \Delta_{T_{k+1},j})_i$ converges in distribution to a Brownian motion reflected by or absorbed on the limit of $(\frac{1}{\sqrt{n}}\sum_j^i \Delta_{T_{k},j})_i$. In \cite{Toth1995} the environment was absent, therefore we had to find new ideas to control the interaction between $(\sum_j^i \Delta_{T_{k+1},j})_i$ and $(\sum_j^i \Delta_{T_{k},j})_i$. Moreover, we need to study the properties of the limit processes, which lead us to study combinations of reflected and absorbed Brownian motions which we consider novel and of interest in their own right. Indeed, Brownian motions reflected on other Brownian motions have been studied before (see \cite{Burdzy_et_al2002,Soucaliuc_et_al2000,Toth_et_al2007,Warren2007}), but the results found in those papers were insufficient for our purposes. 

Finally, the limit of the random walk in \cite{Kosygina_et_al2022} was known, expected from prior results on particular cases. Here the limit is unknown, and we do not identify it beyond noting that it exists and is continuous. It is not obvious whether the limit is intimately related to the process of \cite{Toth_et_al1998,Newman_et_al2006}, and it would be useful to develop the ideas presented here to understand this limit process better. The lack of knowledge about the limit forced us to find novel arguments to prove the convergence. An attribute of our approach is that the ``coarse-graining'' with the mesoscopic times relies purely on Ray-Knight properties. This, we feel, gives it the potential to be generalized to yield limits for a much larger class of self-interacting random walks. 

The paper unfolds as follows. In Section \ref{sec_outline_proof} we give an outline of the proof. In Section \ref{sec_conclusion} we give the proofs of Theorem \ref{thm_main} and Proposition \ref{prop_continuity} conditionally on the results proven in the later sections. In Section \ref{sec_first_variables} we introduce much notation, and auxiliary random variables we will use throughout the paper. Section \ref{sec_bad_events} considers some ``bad events'' outside which the environment and some associated variables behave well, and proves that they have very small probability. In Section \ref{sec_reflected_RW}, we prove that outside of the bad events, $(\sum_j^i \Delta_{T_{k+1},j})_i$ is indeed close to a random walk reflected on the environment. Section \ref{sec_lower_bounds_Tk} is the most important in that it shows that with very high probability, the stopping times $T_k$ do not accumulate and $T_k -T_0$ is at least of order $k n ^{3/2}$. We need such a control on the $T_k$ because we do not know the limit of $(Y_t^N)_{t \in [0,+\infty)}$; it is the most novel part of the work. Section \ref{sec_limit_processes} discusses the limit process of the environment and introduces the reflected/absorbed processes which may be of interest in their own right; this section is mostly independent from the rest of the paper. Finally, in Section \ref{sec_conv_limit_processes}, we prove that the environments indeed converge to these limit processes and we use this convergence to deduce the convergence in distribution of the ``mesoscopic quantities'' $\frac{1}{n}(X_{T_{k+1}}-X_{T_{k}})$ and $\frac{1}{n^{3/2}}(T_{k+1}-T_{k})$. An appendix contains some arguments that are necessary to complete the proof but not very speciﬁc or novel, and which a reader might want to omit.

\section{Outline of the proof}\label{sec_outline_proof}

This section being an outline, most of its content will be non-rigorous. We first outline how to prove Theorem \ref{thm_main}, which is done rigorously in Section \ref{sec_conclusion_thm} modulo subsequent technical results. In order to prove the convergence in distribution of the renormalized process $(Y_t^N)_{t \in [0,+\infty)}$, we need to prove its tightness and the convergence of its finite-dimensional marginals. Let us concentrate on the finite-dimensional marginals for now. 

The proof of the convergence of the finite-dimensional marginals (Proposition \ref{prop_marginals}) is partially inspired from the method introduced by Kosygina, Peterson and the second author in \cite{Kosygina_et_al2022}: we define ``mesoscopic times'' $(T_k)_{k\in\mathds{N}}$ so that $T_0=\mathbf{T}_{\lfloor N \theta\rfloor,\lfloor Nx \rfloor}^\pm$, and $T_{k+1}$ is the first time $m$ after $T_k$ at which $|X_{m}-X_{T_{k}}| = \lfloor\varepsilon n\rfloor$. For $m\in\{T_k,...,T_{k+1}\}$ we have $|X_{m}-X_{T_{k}}| \leq \lfloor\varepsilon n\rfloor$, hence if $t \in[\frac{1}{n^{3/2}}(T_{k}-T_{0}),\frac{1}{n^{3/2}}(T_{k+1}-T_{0})]$ we have $|Y_t^N-\frac{1}{n}(X_{T_k}-X_{T_{0}})| \leq \varepsilon$. Consequently, if we can prove the convergence in distribution of the $\frac{1}{n}(X_{T_{k+1}}-X_{T_{k}})$ and the $\frac{1}{n^{3/2}}(T_{k+1}-T_{k})$ (Proposition \ref{prop_conv_ZT}), we can prove that the finite-dimensional marginals of $(Y_t^N)_{t \in [0,+\infty)}$ are close to those of a limit process depending on $\varepsilon$, which we may call $(Y_t^\varepsilon)_{t \in [0,+\infty)}$. In \cite{Kosygina_et_al2022}, the limit of $(Y_t^N)_{t \in [0,+\infty)}$ was known, and (the equivalent of) $(Y_t^\varepsilon)_{t \in [0,+\infty)}$ converges towards it when $\varepsilon$ tends to 0, so this suffices. However, here we do not know the limit of $(Y_t^N)_{t \in [0,+\infty)}$, which forces us to add another step. We notice that if the finite-dimensional marginals converge, then their limit has to be close to the finite-dimensional marginals of $(Y_t^\varepsilon)_{t \in [0,+\infty)}$ for any $\varepsilon$, so the limit is uniquely determined. Consequently, if the finite-dimensional marginals are tight, then they converge. 

However, this means we also have to prove the tightness of the finite-dimensional marginals. In order to do that, we prove that the $\frac{1}{n^{3/2}}(T_{k}-T_{0})$ are at least of order $k$ (or rather $k$ times a constant), which is the work of Section \ref{sec_lower_bounds_Tk}. Indeed, when $m\in\{T_0,T_0+1,...,T_k\}$ we have $|X_m-X_{T_0}| \leq k \lfloor\varepsilon n\rfloor$, thus $\frac{1}{n^{3/2}}(T_{k}-T_{0})$ is the smallest time at which $(Y_t^N)_{t \in [0,+\infty)}$ can reach $k\varepsilon$. Therefore, if $\frac{1}{n^{3/2}}(T_{k}-T_{0})$ tends to $+\infty$ with $k$, then the finite-dimensional marginals of $(Y_t^N)_{t \in [0,+\infty)}$ will be tight, hence they will converge. Consequently, to prove the convergence of the finite-dimensional marginals of $(Y_t^N)_{t \in [0,+\infty)}$, it is sufficient to prove that $\frac{1}{n^{3/2}}(T_{k}-T_{0})$ is of order $k$ as well as the convergence in distribution of the $\frac{1}{n}(X_{T_{k+1}}-X_{T_{k}})$, $\frac{1}{n^{3/2}}(T_{k+1}-T_{k})$. Actually, proving that also yields the tightness of the process $(Y_t^N)_{t \in [0,+\infty)}$ (Proposition \ref{prop_module}). Indeed, it is tight when $(X_m)_{m\in\mathds{N}}$ does not fluctuate too quickly, which is the same thing as the $\frac{1}{n^{3/2}}(T_{k+1}-T_{k})$ not being too small. Therefore, we have two main things to prove: the convergence in distribution of the $\frac{1}{n}(X_{T_{k+1}}-X_{T_{k}})$, $\frac{1}{n^{3/2}}(T_{k+1}-T_{k})$ (Proposition \ref{prop_conv_ZT}) and the fact that $\frac{1}{n^{3/2}}(T_{k}-T_{0})$ is of order $k$ (Proposition \ref{prop_T_K}).

\subsection{Convergence in distribution of the $\frac{1}{n}(X_{T_{k+1}}-X_{T_{k}})$, $\frac{1}{n^{3/2}}(T_{k+1}-T_{k})$} \label{sec_outline_conv}

As was done in \cite{Kosygina_et_al2022}, we prove this convergence through a study of ``mesoscopic'' local times. We let $\beta_{T_k}^-$ be the first time $m$ after $T_k$ at which $X_{m}=X_{T_{k}}-\lfloor\varepsilon n\rfloor$ (see \eqref{eq_def_Tk_pm^}), then $\beta_{T_k}^-$ will be $T_{k+1}$ if $(X_{T_k+m})_{m\in\mathds{N}}$ reaches $X_{T_{k}}-\lfloor\varepsilon n\rfloor$ before $X_{T_{k}}+\lfloor\varepsilon n\rfloor$. For any $i\in\mathds{Z}$, let $L_i^{T_k,-}$ the local time on the oriented edge $(i-1,i)$ between times $T_{k}$ and $\beta_{T_k}^-$, that is the number of times the process went from $i-1$ to $i$ between times $T_{k}$ and $\beta_{T_k}^-$ (rigorously defined in Definition \ref{def_local_times}). Then we will have $X_{T_{k+1}}=X_{T_k}-\lfloor\varepsilon n\rfloor$ if and only if there exists $i\in\{X_{T_k},...,X_{T_k}+\lfloor\varepsilon n\rfloor\}$ so that $L_i^{T_k,-}=0$, because this means that before $\beta_{T_k}^-$ i.e. before $(X_{T_k+m})_{m\in\mathds{N}}$ goes to $X_{T_{k}}-\lfloor\varepsilon n\rfloor$, it does not reach $i$ hence does not reach $X_{T_{k}}+\lfloor\varepsilon n\rfloor$. Consequently, one can know whether $X_{T_{k+1}}=X_{T_k}-\lfloor\varepsilon n\rfloor$ or $X_{T_k}+\lfloor\varepsilon n\rfloor$ by looking at the local times $L_i^{T_k,-}$. Moreover, if $X_{T_{k+1}}=X_{T_k}-\lfloor\varepsilon n\rfloor$, we have $T_{k+1}=\beta_{T_k}^-$, and at each step made between times $T_k$ and $\beta_{T_k}^-$ the random walk crosses an edge, so one can compute $T_{k+1}-T_k$ from the local times $L_i^{T_k,-}$, and if $X_{T_{k+1}}=X_{T_k}+\lfloor\varepsilon n\rfloor$, one can compute $T_{k+1}-T_k$ from local times defined in a symmetric way. In order to establish the convergence in distribution of the $\frac{1}{n}(X_{T_{k+1}}-X_{T_{k}})$, $\frac{1}{n^{3/2}}(T_{k+1}-T_{k})$, it is thus enough to understand the local times $L_i^{T_k,-}$. 

As in \cite{Kosygina_et_al2022}, we will study these local times through a Ray-Knight argument, that is by exploiting their Markov properties. However, the use of the ideas of \cite{Kosygina_et_al2022} stops here, because the dynamics of our process is different from theirs. We are able to express $L_i^{T_k,-}$ as roughly $\sum^i \zeta_j^{T_k,-,E}-\sum^i \zeta_j^{T_k,-,B}$ (Observation \ref{obs_rec_temps_local}), where the $\zeta_j^{T_k,-,E}$, defined in Definition \ref{def_local_times}, are small modifications of the $\Delta_{\beta_{T_k}^-,j}$ (see \eqref{eq_def_Delta}) and the $\zeta_j^{T_k,-,B}$, also defined in Definition \ref{def_local_times}, are small modifications of the $\Delta_{T_k,j}$. We thus express $L_i^{T_k,-}$ as the difference between the random walk $\sum^i \zeta_j^{T_k,-,E}$ and the random walk $\sum^i \zeta_j^{T_k,-,B}$. We then need to study these walks, hence the $\Delta_{\beta_{T_k}^-,i}$ and $\Delta_{T_k,i}$. Part of this study resembles what was done in \cite{Toth1995}, though there are very important differences. 

 We first notice that $(\Delta_{m,i})_m$ resembles a Markov chain. We are only interested in $i \geq X_{T_k}-\lfloor\varepsilon n\rfloor$, since $(X_m)_{m\in\mathds{N}}$ does not go below $X_{T_k}-\lfloor\varepsilon n\rfloor$ between times $T_k$ and $\beta_{T_k}^-$ so $L_i^{T_k,-}=0$ for $i < X_{T_k}-\lfloor\varepsilon n\rfloor$. We notice that given the transition probabilities \eqref{eq_def_Xn}, when $X_m=i$ the probability for $X_{m+1}$ to be $i-1$ or $i+1$, hence for $\Delta_{m+1,i}$ to be $\Delta_{m,i}+1$ or $\Delta_{m,i}-1$, depends only on $\Delta_{m,X_m}=\Delta_{m,i}$, therefore if we only keep track of the changes of $\Delta_{m,i}$, we get a Markov chain (whose transition probabilities are given in \eqref{eq_def_xi}). If we only keep track of the values of $\Delta_{m,i}$ when $\Delta_{m,i}=\Delta_{m-1,i}+1$ (respectively $\Delta_{m,i}=\Delta_{m-1,i}-1$), which means the last move of the walk at $i$ was to go to the left of $i$ (respectively to the right), we obtain another Markov chain, the $\oplus$-Markov chain (respectively the $\ominus$-Markov chain). These chains correspond respectively to the $-\eta_+$ and $\eta_-$ defined in \eqref{eq_def_tau} and \eqref{eq_def_eta}. Their equilibrium measures are called $\rho_+$ and $\rho_-$, defined in \eqref{eq_def_rho+} and \eqref{eq_def_rho-}. 

 This yields that the $\Delta_{\beta_{T_k}^-,i}$ are roughly i.i.d. with law $\rho_+$ and independent from the $\Delta_{T_k,i}$ (Proposition \ref{prop_law_zeta}). Indeed, we have $X_{\beta_{T_k}^-}=X_{T_k}-\lfloor\varepsilon n\rfloor$, so for $i \geq X_{T_k}-\lfloor\varepsilon n\rfloor$, our self-repelling random walk is at the left of $i$ at time $\beta_{T_k}^-$, hence $\Delta_{\beta_{T_k}^-,i}$ is a step of the $\oplus$-Markov chain at $i$. If $L_i^{T_k,-}$ is large, the $\oplus$-Markov chain at $i$ made many steps between times $T_k$ and $\beta_{T_k}^-$, therefore at time $\beta_{T_k}^-$ it will have forgotten the value of $\Delta_{T_k,i}$ and the law of $\Delta_{\beta_{T_k}^-,i}$ will be close to $\rho_+$. This implies that when $L_i^{T_k,-}$ is large, the $\Delta_{\beta_{T_k}^-,i}$ are roughly i.i.d. with law $\rho_+$ and independent from the $\Delta_{T_k,i}$. 

This allows to understand the behavior of $\sum^i \zeta_j^{T_k,-,E}$, which is done in Sections \ref{sec_bad_events} and \ref{sec_reflected_RW}. Indeed, since $L_i^{T_k,-}$ is roughly $\sum^i \zeta_j^{T_k,-,E}-\sum^i \zeta_j^{T_k,-,B}$, this means that when $\sum^i \zeta_j^{T_k,-,E}$ is well above $\sum^i \zeta_j^{T_k,-,B}$, then $\sum^i \zeta_j^{T_k,-,E}$ behaves like a random walk with i.i.d. increments independent from $\sum^i \zeta_j^{T_k,-,B}$. Furthermore, we have roughly $\sum^i \zeta_j^{T_k,-,E}-\sum^i \zeta_j^{T_k,-,B}=L_i^{T_k,-}$ with $L_i^{T_k,-}$ non-negative, hence $\sum^i \zeta_j^{T_k,-,E}$ remains larger than $\sum^i \zeta_j^{T_k,-,B}$ at all times. More precisely, for $i\in\{X_{T_k}-\lfloor\varepsilon n\rfloor,...,X_{T_k}\}$, the process $\sum^i \zeta_j^{T_k,-,E}$ will behave like a random walk reflected on the ``environment'' $\sum^i \zeta_j^{T_k,-,B}$ (Proposition \ref{prop_sym}). Moreover, for $i>X_{T_k}$, when $L_i^{T_k,-}=0$ then $(X_m)_{m\in\mathds{N}}$ does not reach $i$ between times $T_k$ and $\beta_{T_k}^-$ (that is when going from $X_{T_k}$ to $X_{T_k}-\lfloor\varepsilon n\rfloor$), so it will not reach any $j>i$, so $L_j^{T_k,-}=0$ for any $j>i$. This implies that as soon as $\sum^i \zeta_j^{T_k,-,E}=\sum^i \zeta_j^{T_k,-,B}$ then $\sum^j \zeta_{j'}^{T_k,-,E}=\sum^j \zeta_{j'}^{T_k,-,B}$ for any $j>i$, which means the random walk $\sum^i \zeta_j^{T_k,-,E}$ is ``absorbed'' by the environment $\sum^i \zeta_j^{T_k,-,B}$ when it hits said environment. 

We can now study the behavior of $\sum^i \zeta_j^{T_k,-,B}$ and $\sum^i \zeta_j^{T_k,-,E}$ when $N$ tends to $+\infty$, which was done in Section \ref{sec_conv_limit_processes}. If $\frac{1}{\sqrt{n}}\sum^i \zeta_j^{T_k,-,B}$ converges to some limit process, then $\frac{1}{\sqrt{n}}\sum^i \zeta_j^{T_k,-,E}$ converges to a Brownian motion that is partly reflected on the limit of $\frac{1}{\sqrt{n}}\sum^i \zeta_j^{T_k,-,B}$ and partly absorbed by this limit. We can then use the convergence of $\frac{1}{\sqrt{n}}\sum^i \zeta_j^{T_k,-,E}$ to deduce the convergence of $\frac{1}{\sqrt{n}}\sum^i \zeta_j^{T_{k+1},-,B}$. It is thus possible to prove the joint convergence of the $\frac{1}{\sqrt{n}}\sum^i \zeta_j^{T_k,-,B}$, $\frac{1}{\sqrt{n}}\sum^i \zeta_j^{T_k,-,E}$ by induction on $k$, which is Proposition \ref{prop_conv_envts}. This yields control of the $L_i^{T_k,-}$. In \cite{Toth1995}, Tóth used a similar strategy to prove the convergence of the local times process of a self-repelling random walk with undirected edges, but he had no equivalent of $\frac{1}{\sqrt{n}}\sum^i \zeta_j^{T_k,-,B}$ (his random walk is simply reflected on 0).

There are three major problems for putting this approach into practice to prove the convergence in distribution of the $\frac{1}{n}(X_{T_{k+1}}-X_{T_{k}})$, $\frac{1}{n^{3/2}}(T_{k+1}-T_{k})$. Firstly, though we know that when $\sum^i \zeta_j^{T_k,-,E}$ is well above $\sum^i \zeta_j^{T_k,-,B}$, then $\sum^i \zeta_j^{T_k,-,E}$ behaves like a random walk with i.i.d. increments independent from $\sum^i \zeta_j^{T_k,-,B}$, we do not have this sort of control when $\sum^i \zeta_j^{T_k,-,E}$ is close to $\sum^i \zeta_j^{T_k,-,B}$, so it is not that easy to prove that $\sum^i \zeta_j^{T_k,-,E}$ behaves like a random walk reflected on $\sum^i \zeta_j^{T_k,-,B}$. Our model being very different from the one studied by Tóth in \cite{Toth1995}, we had to find a novel argument, which is used in the proof of Proposition \ref{prop_TCL}. We notice that though when $L_i^{T_k,-}$ is small the $\oplus$-Markov chain at $i$ is not at equilibrium, it can be coupled with another that is at equilibrium, and can therefore be controlled. Even with this control, we need to establish rather complex inequalities (see (\ref{eq_reflected_hard})) to prove $\sum^i \zeta_j^{T_k,-,E}$ is close to a random walk reflected on $\sum^i \zeta_j^{T_k,-,B}$.

The second problem lies in the definition of the limit process of $\frac{1}{\sqrt{n}}\sum^i \zeta_j^{T_{k+1},-,B}$. Indeed, $T_{k+1}$ is $\beta_{T_k}^-$ when there exists $i\in\{X_{T_k},...,X_{T_k}+\lfloor\varepsilon n\rfloor\}$ so that $L_i^{T_k,-}=0$, i.e. $\sum^i \zeta_j^{T_k,-,E}=\sum^i \zeta_j^{T_k,-,B}$, which means $\sum^i \zeta_j^{T_k,-,E}$ is absorbed by $\sum^i \zeta_j^{T_k,-,B}$. In this case we have $\Delta_{T_{k+1},i}=\Delta_{\beta_{T_k}^-,i}$, hence the $\zeta_j^{T_{k+1},-,B}$ can be obtained from the $\Delta_{\beta_{T_k}^-,i}$ hence from the $\zeta_j^{T_{k},-,E}$. The limit of $\frac{1}{\sqrt{n}}\sum^i \zeta_j^{T_{k+1},-,B}$ is then obtained from the limit of $\frac{1}{\sqrt{n}}\sum^i \zeta_j^{T_{k},-,E}$, and this works roughly in the case where the limit of $\frac{1}{\sqrt{n}}\sum^i \zeta_j^{T_{k},-,E}$ is absorbed by the limit of $\frac{1}{\sqrt{n}}\sum^i \zeta_j^{T_{k},-,B}$. However, we also have to consider the case $T_{k+1}\neq \beta_{T_k}^-$, that is $X_{T_{k+1}}=X_{T_k}+\lfloor\varepsilon n\rfloor$. We can study it in the same way that the case $X_{T_{k+1}}=X_{T_k}-\lfloor\varepsilon n\rfloor$, defining symmetric quantities $\zeta_i^{T_{k},+,B},\zeta_i^{T_{k},+,E}$. We then get that the behavior of the $\Delta_{T_{k+1},i}$, hence the limit of $\frac{1}{\sqrt{n}}\sum^i \zeta_j^{T_{k+1},-,B}$, can be obtained from the limit of $\frac{1}{\sqrt{n}}\sum^i \zeta_j^{T_{k},+,E}$ when the latter is absorbed by the limit of $\frac{1}{\sqrt{n}}\sum^i \zeta_j^{T_{k},+,B}$. Consequently, to be able to construct the limit process of $\frac{1}{\sqrt{n}}\sum^i \zeta_j^{T_{k+1},-,B}$ (which is done in Definition \ref{def_lim_envts}), we have to show that the probability that the limit of $\frac{1}{\sqrt{n}}\sum^i \zeta_j^{T_{k},+,E}$ is absorbed by the limit of $\frac{1}{\sqrt{n}}\sum^i \zeta_j^{T_{k},+,B}$ is one minus the probability the limit of $\frac{1}{\sqrt{n}}\sum^i \zeta_j^{T_{k},-,E}$ is absorbed by the limit of $\frac{1}{\sqrt{n}}\sum^i \zeta_j^{T_{k},-,B}$. In order to do that, in Section \ref{sec_BM_results} we study the following setting: we have a Brownian motion reflected by some function called the ``barrier'' from time $-1$ to time 0 and absorbed by the barrier from time 0 to time 1, and another Brownian motion going backwards, reflected above the same barrier from time 1 to time 0 and absorbed by the barrier from time 0 to time $-1$. We prove several conditions (Propositions \ref{prop_Z}, \ref{prop_better_than_Z} and \ref{prop_Z_lower_fcts}) for the probability that the first Brownian motion actually gets absorbed to be one minus the probability that the second Brownian motion is absorbed. We believe this study to be of independent interest. 

The third problem lies in deducing rigorously the convergence of $\frac{1}{n}(X_{T_{k+1}}-X_{T_{k}})$ from the convergence of the processes $\frac{1}{\sqrt{n}}\sum^i \zeta_j^{T_k,-,B}$ and $\frac{1}{\sqrt{n}}\sum^i \zeta_j^{T_k,-,E}$. Indeed, we know that $X_{T_{k+1}}=X_{T_k}-\lfloor\varepsilon n\rfloor$ if and only if $\sum^i \zeta_j^{T_k,-,E}$ gets absorbed by $\sum^i \zeta_j^{T_k,-,B}$, but proving that the probability of this absorption converges to the probability of absorption of the limit processs requires some property of continuity of the absorption time for the limit process. In order to show such a property, in Section \ref{sec_process_above} we study the limit processes of the environments $\frac{1}{\sqrt{n}}\sum^i \zeta_j^{T_k,-,B}$, $k\in\mathds{N}$. These limit processes, constructed in Definition \ref{def_lim_envts}, may be interesting on their own: they are the sequence of processes obtained by firstly running either a Brownian motion first reflected then absorbed on another Brownian motion, conditioned to absorption, or a backwards Brownian motion with the same properties, and then iterating this procedure by reflecting and absorbing the new Brownian motion on the resulting process. We prove that the law of the limit processes thus obtained, on certain small intervals, is close in some sense either to the law of a Brownian motion or to the law of a Brownian motion reflected on a Brownian motion (Proposition \ref{prop_rec_little_ints}). These latter processes being easy to control, this allows us to deduce the required continuity property. 

\subsection{$\frac{1}{n^{3/2}}(T_{k}-T_{0})$ is of order $k$}

This relies on an entirely novel argument, laid out in Section \ref{sec_lower_bounds_Tk}. We first link $T_{k+1}-T_k$ to the behavior of the walks $\sum^i \zeta_j^{T_k,-,E}$, $\sum^i \zeta_j^{T_k,-,B}$. As we already mentioned at the beginning of Section \ref{sec_outline_conv}, if $\beta_{T_k}^-=T_k$, we can deduce $T_{k+1}-T_k$ from the $L_i^{T_k,-}$; actually we roughly have $T_{k+1}-T_k=2\sum L_i^{T_k,-}$, where the sum is on $i\in\{X_{T_k}-\lfloor\varepsilon n\rfloor,...,X_{T_k}+\lfloor\varepsilon n\rfloor\}$. We also know that $L_i^{T_k,-}$ is the difference between the random walks $\sum^i \zeta_j^{T_k,-,E}$ and $\sum^i \zeta_j^{T_k,-,B}$, and that $\sum^i \zeta_j^{T_k,-,E}$ is an i.i.d. random walk reflected on $\sum^i \zeta_j^{T_k,-,B}$ for $i\in\{X_{T_k}-\lfloor\varepsilon n\rfloor,...,X_{T_k}\}$ and absorbed by $\sum^i \zeta_j^{T_k,-,B}$ for $i\in\{X_{T_k},...,X_{T_k}+\lfloor\varepsilon n\rfloor\}$. Since we need only a lower bound on $T_{k+1}-T_k=2\sum L_i^{T_k,-}$, we can consider only the sum on $i\in\{X_{T_k}-\lfloor\varepsilon n\rfloor,...,X_{T_k}\}$, where the walk $\sum^i \zeta_j^{T_k,-,E}$ is reflected. Then since $\sum^i \zeta_j^{T_k,-,E}$ is an i.i.d. random walk reflected on $\sum^i \zeta_j^{T_k,-,B}$, it will be larger than some i.i.d. random walk which we call $\sum^i \zeta_j^{T_k,-,I}$ (the construction of the $\zeta_j^{T_k,-,I}$ can be found just before Proposition \ref{prop_def_zetaI}). We deduce $L_i^{T_k,-} = \sum^i \zeta_j^{T_k,-,E}-\sum^i \zeta_j^{T_k,-,B} \geq \sum^i \zeta_j^{T_k,-,I}-\sum^i \zeta_j^{T_k,-,B}$. 

If $\sum^i \zeta_j^{T_k,-,B}$ was an i.i.d. random walk too, $\sum^i \zeta_j^{T_k,-,I}-\sum^i \zeta_j^{T_k,-,B}$ would be an i.i.d. random walk, hence $T_{k+1}-T_k=2\sum L_i^{T_k,-}$ would be larger than the integral of an i.i.d. random walk on an interval of length of order $n$. Since such a random walk may go to an height of order $\sqrt{n}$, we would have $T_{k+1}-T_k$ of order $n^{3/2}$, hence $\frac{1}{n^{3/2}}(T_{k+1}-T_{k})$ would be of order 1, hence $\frac{1}{n^{3/2}}(T_{k}-T_{0})$ would be of order $k$. Consequently, it is enough to prove that $\sum^i \zeta_j^{T_k,-,B}$ is close to an i.i.d. random walk. The arguments will differ depending on the evolution of the process prior to $T_k$.

For $k=0$, then $\sum^i \zeta_j^{T_k,-,B}$ will be close to an i.i.d. random walk. Indeed, the $\zeta_j^{T_0,-,B}$ are based on the $\Delta_{T_0,i}$, and if $i$ is at the right of $X_{T_0}$ (respectively at its left), the last move of the process at $i$ before $T_0$ was going to the left (respectively to the right), hence $\Delta_{T_0,i}$ is a step of the $\oplus$-Markov chain at $i$ (respectively the $\ominus$-Markov chain at $i$). Now, at time $T_0=\mathbf{T}_{\lfloor N \theta\rfloor,\lfloor Nx \rfloor}^\pm$, the local times around $X_{T_0}=\lfloor Nx \rfloor\pm1$ are not far from $\lfloor N \theta\rfloor$, hence they are large enough for the $\oplus$- and $\ominus$-Markov chains to be at equilibrium. These Markov chains are also independent for different $i$. We deduce that at the right of $X_{T_0}$, the $\Delta_{T_0,i}$ are i.i.d. with law $\rho_+$, and at the left of $X_{T_0}$, the $\Delta_{T_0,i}$ are i.i.d. with law $\rho_-$. This will imply $\sum^i \zeta_j^{T_0,-,B}$ is an i.i.d. random walk. $\sum^i \zeta_j^{T_k,-,B}$ will also be an i.i.d. random walk if between times $T_0$ and $T_k$, the process $(X_m)_{m\in\mathds{N}}$ never went between $X_{T_k}$ and $X_{T_{k+1}}$, since in this case, for $i$ between $X_{T_k}$ and $X_{T_{k+1}}$ we have $\Delta_{T_0,k}=\Delta_{T_0,i}$. 

Another favorable case is when the ``mesoscopic process'' $(X_{T_k})_{k\in\mathds{N}}$ does a U-turn, that is when $X_{T_{k+1}}=X_{T_k}\pm\lfloor\varepsilon n\rfloor=X_{T_{k-1}}$ (in the following we consider $X_{T_{k+1}}=X_{T_k}-\lfloor\varepsilon n\rfloor=X_{T_{k-1}}$ to fix the notation). Indeed, the $\zeta_i^{T_k,-,B}$ are based on the $\Delta_{T_k,i}$, and in this case $X_{T_k}=X_{T_{k-1}}+\lfloor\varepsilon n\rfloor$, hence the $\Delta_{T_k,i}$ can be deduced from the $\zeta_i^{T_{k-1},+,E}$. Furthermore, the process $\sum^i\zeta_j^{T_{k-1},+,E}$ is roughly a reflected i.i.d. random walk, hence is above an i.i.d. random walk, which allows to control it, hence to control $\sum^i\zeta_j^{T_k,-,B}$. The case of a U-turn is thus tractable. 

However, if the mesoscopic process does not do a U-turn, for example if $X_{T_{k+1}}=X_{T_k}-\lfloor\varepsilon n\rfloor=X_{T_{k-1}}-2\lfloor\varepsilon n\rfloor$, things quickly become more complicated. Indeed, we need to control the $\Delta_{T_{k+1},i}$ on $\{X_{T_k},...,X_{T_k}+\lfloor\varepsilon n\rfloor\}$, since if at some point after time $T_{k+1}$ the mesoscopic process goes from $X_{T_k}$ to $X_{T_k}+\lfloor\varepsilon n\rfloor$, then the environment will be based on these $\Delta_{T_{k+1},i}$. As $T_{k+1}=\beta_{T_k}^-$, the $\Delta_{T_{k+1},i}=\Delta_{\beta_{T_k}^-,i}$ can be deduced from the $\zeta_i^{T_k,-,E}$, so we have to control those. Now, since $X_{T_{k+1}}=X_{T_k}-\lfloor\varepsilon n\rfloor$, between times $T_k$ and $T_{k+1}$ the process $(X_m)_{m\in\mathds{N}}$ may enter $\{X_{T_k},...,X_{T_k}+\lfloor\varepsilon n\rfloor\}$, but will not reach $X_{T_k}+\lfloor\varepsilon n\rfloor$. Let $i_0$ the rightmost site of $\mathds{Z}$ that is reached. We already saw that on $\{X_{T_k},...,X_{T_k}+\lfloor\varepsilon n\rfloor\}$, $\sum^i \zeta_j^{T_k,-,E}$ is an i.i.d. random walk absorbed by $\sum^i \zeta_j^{T_k,-,B}$. For $i \leq i_0$, we have $L_i^{T_k,-}>0$, so $\sum^i \zeta_j^{T_k,-,E}$ is not yet absorbed, thus $\sum^i \zeta_j^{T_k,-,E}$ behaves as an i.i.d. random walk, hence we can control it. For $i>i_0$, since $(X_m)_{m\in\mathds{N}}$ does not reach $i$ between times $T_k$ and $T_{k+1}$, we have $\Delta_{T_{k+1},i}=\Delta_{T_k,i}$, and since $X_{T_k}=X_{T_{k-1}}-\lfloor\varepsilon n\rfloor$, we have $T_k=\beta_{T_{k-1}}^-$, hence the $\Delta_{T_k,i}$ can be deduced from the $\zeta_i^{T_{k-1},-,E}$. We then notice that since we consider $i\in\{X_{T_k},...,X_{T_k}+\lfloor\varepsilon n\rfloor\}$, we have $i \in \{X_{T_{k-1}}-\lfloor\varepsilon n\rfloor,...,X_{T_{k-1}}\}$, and that for such $i$ the process $\sum^i\zeta_j^{T_{k-1},-,E}$ is a reflected i.i.d. random walk, hence is larger than an i.i.d. random walk, therefore we can control it. To sum up, we have two cases, both of which can be controlled, so this will still give a tractable environment for the next time the mesoscopic process goes from $X_{T_k}$ to $X_{T_k}+\lfloor\varepsilon n\rfloor$. 

However, if before that the mesoscopic process makes a visit from $X_{T_k}-\lfloor\varepsilon n\rfloor$ to $X_{T_k}$ and back, then during the shift from $X_{T_k}$ to $X_{T_k}-\lfloor\varepsilon n\rfloor$, $(X_m)_{m\in\mathds{N}}$ may visit some $i\in\{X_{T_k},...,X_{T_k}+\lfloor\varepsilon n\rfloor\}$, which will change their $\Delta_{m,i}$ and give us another case to take into account. Since there is no limit on the number of such visits, the environment on $\{X_{T_k},...,X_{T_k}+\lfloor\varepsilon n\rfloor\}$ can become uncontrollable. In order to solve this problem, we devised an algorithm that keeps track of the control we have on the environment, and used it to prove that whatever the path of the mesoscopic process $(X_{T_k})_{k\in\mathds{N}}$, there is always a positive fraction of its steps in which we can control the environment, hence for which $\frac{1}{n^{3/2}}(T_{k+1}-T_{k})$ is of order 1. This is enough to prove $\frac{1}{n^{3/2}}(T_{k}-T_{0})$ is of order $k$ (Proposition \ref{prop_T_K}).
 
\subsection{Proof of Proposition \ref{prop_continuity}}

In order to prove this proposition, which is done in Section \ref{sec_conclusion_prop}, we need to show $(X_m)_{m\in\mathds{N}}$ fluctuates too quickly for $(\frac{1}{N}X_{\lfloor N^2 t\rfloor})_{t\in[0,+\infty)}$ to have a limit. In order to do that, we reuse some of the techniques developed for the proof of Theorem \ref{thm_main}. If we choose again $T_0=\mathbf{T}_{\lfloor N \theta\rfloor,\lfloor Nx \rfloor}^\pm$, but we take $T_{1}$ the first time $m$ after $T_0$ at which $|X_m-X_{T_0}|=\lfloor\varepsilon N\rfloor$ (instead of $\lfloor\varepsilon n\rfloor$ as in the proof of Theorem \ref{thm_main}), we can prove that $\frac{1}{N^{3/2}}(T_1-T_0)$ converges in distribution (Lemma \ref{lem_continuity}), which implies $T_1-T_0$ is of order $N^{3/2}$. This means the time needed for $(X_m)_{m\in\mathds{N}}$ to move on a scale $N$ is of order $N^{3/2}$, therefore the time needed for $(\frac{1}{N}X_{\lfloor N^2 t\rfloor})_{t\in[0,+\infty)}$ to move on a scale 1 is of order $1/N^{1/2}$. It is thus clear the latter process cannot converge when $N$ tends to $+\infty$.

 \section{Proof of Theorem \ref{thm_main} and Proposition \ref{prop_continuity}}\label{sec_conclusion}
 
 In this section, we give the proofs of Theorem \ref{thm_main} and Proposition \ref{prop_continuity}, conditionally on important results which will be proven in the following sections. We recall the definition of $n$ given in \eqref{eq_scale_n} and that of $(Y_t^N)_{t\in[0,+\infty)}$ spelled out in \eqref{eq_def_YN}. We need to introduce several other objects. Remembering the definition of $\mathbf{T}_{\lfloor N \theta\rfloor,\lfloor Nx \rfloor}^\pm$ given in \eqref{eq_def_T_mi}, for any $\varepsilon>0$ we define (as at the beginning of Section \ref{sec_outline_proof}):
 \[
 T_0=\mathbf{T}_{\lfloor N \theta\rfloor,\lfloor Nx \rfloor}^\pm \text{ and } \forall\, k \in \mathds{N}, 
T_{k+1}=\inf\{k' \geq T_{k}\,|\,|X_{k'}-X_{T_{k}}| = \lfloor\varepsilon n\rfloor\}.
\]
 The $T_k$ depend on $n$ and $\varepsilon$, but we do not write it in the notation to make it lighter. The $T_k$ are ``mesoscopic times''. For $k \in \mathds{N}$, we also set $Z_k^N=\frac{1}{\lfloor\varepsilon n\rfloor}(X_{T_k}-X_{T_0})$, which is a ``mesoscopic walk''. 
 
 We also need to define some ``bad events'' $\mathcal{B}$, $\mathcal{B}_0$, $\mathcal{B}_1,...,\mathcal{B}_6$ such that outside of these bad events, ``the process behaves well''. Since their definition is long, technical, and unnecessary to understand this section, we do not give it here and rather refer to the definitions in Propositions \ref{prop_bound_m}, \ref{prop_law_zeta}, as well as to the beginning of Section \ref{sec_bad_events}. We also need some $\tilde \varepsilon > 0$ which will depend on $\varepsilon$, given by \eqref{eq_def_tildepsilon}. Finally, if $\mu$ is a probability measure and $f$ a function taking real values, we denote $\mu(f)$ the expectation of $f$ under $\mu$.
 
 \subsection{Proof of Theorem \ref{thm_main}}\label{sec_conclusion_thm}
 
 To prove that $(Y_t^N)_{t\in[0,+\infty)}$ converges in distribution in the topology of continuous real processes on $[0,+\infty)$, it is enough to show the two following propositions. 
 
 \begin{proposition}\label{prop_marginals}
  For any $\ell \in \mathds{N}^*$, for any $0 < t_1 < \cdots < t_\ell$, $(Y_{t_1}^N,...,Y_{t_\ell}^N)$ converges in distribution when $N \to +\infty$. 
 \end{proposition}

 \begin{proposition}\label{prop_module}
  For any $\vartheta > 0$, for any $\delta_1, \delta_2 > 0$, there exists $\delta_3 > 0$ such that for $N$ large enough, we have $\mathds{P}(\sup_{0 \leq s,t \leq \vartheta, |s-t|\leq \delta_3}|Y_t^N-Y_s^N| > \delta_1) \leq \delta_2$. 
 \end{proposition}
 
 Given Propositions \ref{prop_marginals} and \ref{prop_module}, all that remains to prove Theorem \ref{thm_main} is to prove the following lemma.
 
 \begin{lemma}\label{lem_lim_not0}
  $(Y_t^N)_{t\in[0,+\infty)}$ does not converge in distribution to the null function in the topology of continuous real processes on $[0,+\infty)$.
 \end{lemma}

 We now prove Propositions \ref{prop_marginals} and \ref{prop_module}, as well as Lemma \ref{lem_lim_not0}.
 
\begin{proof}[Proof of Proposition \ref{prop_marginals}]
We first show $((Y_{t_1}^N,...,Y_{t_\ell}^N))_{N\in\mathds{N}^*}$ is tight. For this part of the proof, we choose $\varepsilon=1$. We fix $\delta > 0$. We notice that for any $K\in\mathds{N}^*$, if $(Y_{t_1}^N,...,Y_{t_\ell}^N) \not\in [-K,K]^\ell$, then $T_K-T_0 \leq \lceil t_\ell n^{3/2}\rceil$. This implies 
\[
\mathds{P}((Y_{t_1}^N,...,Y_{t_\ell}^N) \not\in [-K,K]^\ell)\leq \mathds{P}(\mathcal{B})+\mathds{P}\left(\bigcup_{r=0}^6\mathcal{B}_r\right)+\mathds{P}\left(T_K-T_0 \leq \lceil t_\ell n^{3/2}\rceil,\mathcal{B}^c\cap\bigcap_{r=0}^6\mathcal{B}_r^c\right).
\]
The results in the later sections allow us to prove that this tends to 0 when $N$ tends to $+\infty$. Indeed, Proposition \ref{prop_bound_m} yields $\mathds{P}(\mathcal{B}) \leq e^{-c'n^{((\alpha-1)/4) \wedge (1/10)}}$ when $n$ is large enough, and by assumption $n$ tends to $+\infty$ when $N$ tends to $+\infty$, hence $\mathds{P}(\mathcal{B})$ tends to 0 when $N$ tends to $+\infty$. Similarly, by Proposition \ref{prop_bound_bad_evts}, $\mathds{P}(\bigcup_{r=0}^6\mathcal{B}_r) \leq e^{-c(\ln n)^2}$ when $n$ is large enough, hence $\mathds{P}(\bigcup_{r=0}^6\mathcal{B}_r)$ tends to 0 when $N$ tends to $+\infty$. In addition, by Proposition \ref{prop_T_K}, if we choose $K$ large enough so that $K \geq \frac{240t_\ell}{\tilde\varepsilon^{3/2}r_2}$ and $1/2^K \leq \delta/2$, then when $n$ is large enough we have 
\[
\mathds{P}\left(T_K-T_0 \leq \lceil t_\ell n^{3/2}\rceil,\mathcal{B}^c\cap\bigcap_{r=0}^6\mathcal{B}_r^c\right) 
\leq \mathds{P}\left(T_K-T_0 \leq K\frac{r_2}{120} (\tilde\varepsilon n)^{3/2},\mathcal{B}^c\cap\bigcap_{r=0}^6\mathcal{B}_r^c\right) \leq \delta/2.
\]
Therefore, for such a $K$, when $N$ is large enough $\mathds{P}((Y_{t_1}^N,...,Y_{t_\ell}^N) \not\in [-K,K]^\ell) \leq \delta$, which is enough to prove the tightness of $((Y_{t_1}^N,...,Y_{t_\ell}^N))_{N\in\mathds{N}^*}$.

It remains to prove that all subsequences of $((Y_{t_1}^N,...,Y_{t_\ell}^N))_{N\in\mathds{N}^*}$ that converge do so to the same limit. Let $((Y_{t_1}^{\psi(N)},...,Y_{t_\ell}^{\psi(N)}))_{N\in\mathds{N}^*}$ be a converging subsequence, and $\mu$ be its limit law. Let $f \colon \mathds{R}^\ell \mapsto \mathds{R}$ be a continuous function with compact support. We are going to study $\mu(f)$. Let $\delta_1 > 0$. $f$ is uniformly continuous, hence if we denote $\|(y_1,...,y_\ell)\|_\infty=\max_{1 \leq \ell' \leq \ell}|y_{\ell'}|$ for any $(y_1,...,y_\ell)\in\mathds{R}^\ell$, there exists $\delta_2 > 0$ such that if $y,y'\in\mathds{R}^\ell$ satisfy $\|y-y'\|_\infty \leq \delta_2$ then $|f(y)-f(y')|\leq \delta_1$. For any $\varepsilon > 0$, for any $\ell' \in\{1,...,\ell\}$, we define $\tau_{\ell'}=\sup\{k \in \mathds{N}: T_k-T_0 \leq t_{\ell'}n^{3/2}\}$. Then $T_{\tau_{\ell'}} \leq T_0+t_{\ell'}n^{3/2} < T_{\tau_{\ell'}+1}$, hence $|X_{T_0+\lfloor t_{\ell'}n^{3/2}\rfloor}-X_{T_{\tau_{\ell'}}}| \leq \lfloor \varepsilon n \rfloor$ and $|X_{T_0+\lceil t_{\ell'}n^{3/2}\rceil}-X_{T_{\tau_{\ell'}}}| \leq \lfloor \varepsilon n \rfloor$, which implies $|Y_{t_{\ell'}}^N-\frac{1}{n}(X_{T_{\tau_{\ell'}}}-X_{T_0})| \leq \frac{\lfloor\varepsilon n \rfloor}{n}$, therefore $|Y_{t_{\ell'}}^N-\frac{\lfloor\varepsilon n \rfloor}{n} Z_{\tau_{\ell'}}^N| \leq \varepsilon$. Therefore when $\varepsilon \leq \delta_2$, we have 
\begin{equation}\label{eq_marginals}
\left|f(Y_{t_{1}}^N,...,Y_{t_{\ell}}^N)-f\left(\frac{\lfloor\varepsilon n \rfloor}{n} Z_{\tau_{1}}^N,...,\frac{\lfloor\varepsilon n \rfloor}{n} Z_{\tau_{\ell}}^N\right)\right| \leq \delta_1.
\end{equation}
Let us study $f(\frac{\lfloor\varepsilon n \rfloor}{n} Z_{\tau_{1}}^N,...,\frac{\lfloor\varepsilon n \rfloor}{n} Z_{\tau_{\ell}}^N)$. By Proposition \ref{prop_conv_ZT}, we have that $((Z_k^N,\frac{1}{n^{3/2}}(T_k-T_{k-1})))_{k\in\mathds{N}^*}$ converges in distribution to the law $((\breve Z_k,\breve T_k))_{k\in\mathds{N}^*}$ defined in Definition \ref{def_lim_envts} when $N$ tends to $+\infty$ (in the sense of convergence of the finite-dimensional marginals). Moreover, $(Z_{\tau_1}^N,...,Z_{\tau_\ell}^N)$ is a function of $((Z_k^N,\frac{1}{n^{3/2}}(T_k-T_{k-1})))_{k\in\mathds{N}^*}$, and since by Proposition \ref{prop_conv_ZT} the $\sum_{k'=1}^k \breve T_{k'}$, $k\in\mathds{N}$ have no atoms, almost surely for all $k\in\mathds{N}^*$, $\sum_{k'=1}^k \breve T_{k'} \neq t_{\ell'}$ for all $\ell'\in\{1,...,\ell\}$, therefore almost surely $((\breve Z_k,\breve T_k))_{k\in\mathds{N}^*}$ is a point of continuity of this function. Consequently, $(Z_{\tau_1}^N,...,Z_{\tau_\ell}^N)$ converges in distribution when $N$ tends to $+\infty$. We denote $\mu_\varepsilon$ its limiting law, which is also the limiting law of $\frac{\lfloor\varepsilon n \rfloor}{n} (Z_{\tau_1}^N,...,Z_{\tau_\ell}^N)$. Now, from (\ref{eq_marginals}) we deduce $|\mu(f)-\mu_\varepsilon(f)| \leq \delta_1$. To sum up, for any $\delta_1 > 0$, when $\varepsilon$ is small enough we have $|\mu(f)-\mu_\varepsilon(f)| \leq \delta_1$, hence $\mu(f)=\lim_{\varepsilon \to 0}\mu_\varepsilon(f)$. This means $\mu(f)$ does not depend on the choice of the subsequence, thus $\mu$ does not depend on the choice of the subsequence, which ends the proof.
\end{proof}

\begin{proof}[Proof of Proposition \ref{prop_module}] 
Let $\vartheta > 0$, $\delta_1, \delta_2 > 0$. In this proof, we will set $\varepsilon=\frac{\delta_1}{3}$. Then for any $m,m'\in\mathds{N}$, if there exists $k\in\mathds{N}$ such that $m,m' \in [T_k,T_{k+2}]$, then $|X_m-X_{m'}| \leq 3\varepsilon n$.  This implies that for any $s,t \in [0,\vartheta]$, if $s n^{3/2}+T_0,t n^{3/2}+T_0 \in [T_k,T_{k+2}]$ then $|Y_t^N-Y_s^N|\leq 3\varepsilon = \delta_1$. Now for $\delta_3 > 0$, $K \in \mathds{N}^*$, if $T_K \geq \vartheta n^{3/2}+T_0$ and for each $k \in \{0,...,K-1\}$ we have $T_{k+1}-T_k > \delta_3n^{3/2}$, then for any $s,t \in [0,\vartheta]$ so that $|s-t| \leq \delta_3$ there exists $k \in \mathds{N}$ such that $s n^{3/2}+T_0,t n^{3/2}+T_0 \in [T_k,T_{k+2}]$, so we obtain $|Y_t^N-Y_s^N|\leq\delta_1$, therefore $\sup_{0 \leq s,t \leq \vartheta, |s-t|\leq \delta_3}|Y_t^N-Y_s^N| \leq \delta_1$. We deduce 
\[
 \mathds{P}\left(\sup_{0 \leq s,t \leq \vartheta, |s-t|\leq \delta_3}|Y_t^N-Y_s^N| > \delta_1\right)
 \leq \mathds{P}(T_K < \vartheta n^{3/2}+T_0)+\sum_{k=0}^{K-1}\mathds{P}(T_{k+1}-T_k \leq \delta_3n^{3/2}).
\]

We set $K$ so that $1/2^K \leq \delta_2/8$ and $\vartheta \leq K\frac{r_2}{120}\tilde\varepsilon^{3/2}$. Then we have 
\[
\mathds{P}(T_K < \vartheta n^{3/2}+T_0) \leq \mathds{P}\left(T_K-T_0 < K\frac{r_2}{120}(\tilde\varepsilon n)^{3/2},\mathcal{B}^c \cap \bigcap_{r=0}^{6}\mathcal{B}_r^c\right)+\mathds{P}(\mathcal{B})+\mathds{P}\left(\bigcup_{r=0}^6\mathcal{B}_r\right).
\]
By Proposition \ref{prop_T_K}, when $N$ is large enough, the first term is at most $\delta_2/6$. By Proposition \ref{prop_bound_m}, $\mathds{P}(\mathcal{B})$ tends to 0 when $n$ tends to $+\infty$, so if $N$ is large enough, $\mathds{P}(\mathcal{B}) \leq \delta_2/6$. By Proposition \ref{prop_bound_bad_evts}, $\mathds{P}(\bigcup_{r=0}^6\mathcal{B}_r) \leq e^{-c(\ln n)^2}$ when $n$ is large enough, so $\mathds{P}(\bigcup_{r=0}^6\mathcal{B}_r) \leq \delta_2/6$ when $N$ is large enough. We deduce $\mathds{P}(T_K < \vartheta n^{3/2}+T_0) \leq \delta_2/2$ when $N$ is large enough. 

Furthermore for any $k \in \{0,...,K-1\}$, we notice 
\[
\mathds{P}(T_{k+1}-T_k \leq \delta_3n^{3/2}) = \mathds{P}\left(\frac{1}{n^{3/2}}(T_{k+1}-T_k) \leq \delta_3\right).
\]
Now, by Proposition \ref{prop_conv_ZT}, $\frac{1}{n^{3/2}}(T_{k+1}-T_k)$ converges in distribution to the law $\breve T_{k+1}$ defined in Definition \ref{def_lim_envts}, thus when $N$ is large enough $\mathds{P}(\frac{1}{n^{3/2}}(T_{k+1}-T_k) \leq \delta_3) \leq \mathds{P}(\breve T_{k+1} \leq \delta_3)+\frac{\delta_2}{4K}$. In addition, Proposition \ref{prop_conv_ZT} yields that for any $k\in \{0,...,K-1\}$ we have $\mathds{P}(\breve T_{k+1}=0)=0$, hence we can choose $\delta_3 > 0$ so that for any $k\in \{0,...,K-1\}$, $\mathds{P}(\breve T_{k+1} \leq \delta_3) \leq \frac{\delta_2}{4K}$. For such $\delta_3$, we obtain that for any $k\in \{0,...,K-1\}$ we have $\mathds{P}(T_{k+1}-T_k \leq \delta_3n^{3/2}) \leq \frac{\delta_2}{2K}$ when $N$ is large enough.

Consequently, there exists $\delta_3 > 0$ such that $\mathds{P}(\sup_{0 \leq s,t \leq \vartheta, |s-t|\leq \delta_3}|Y_t^N-Y_s^N| > \delta_1) \leq \delta_2$ when $N$ is large enough. 
\end{proof}

\begin{proof}[Proof of Lemma \ref{lem_lim_not0}.] 
 We assume by contradiction that $(Y_t^N)_{t\in[0,+\infty)}$ converges in distribution to the null function in the topology of continuous real processes on $[0,+\infty)$ when $N$ tends to $+\infty$. Then, by the Skorohod Representation Theorem, there exists a probability space containing random variables $(\hat Y_t^N)_{t\in[0,+\infty)}$ for any $N \in \mathds{N}^*$ so that the $(\hat Y_t^N)_{t\in[0,+\infty)}$ have the same distribution as the $(Y_t^N)_{t\in[0,+\infty)}$, and $(\hat Y_t^N)_{t\in[0,+\infty)}$ converges almost surely to the null function in the topology of continuous real processes on $[0,+\infty)$ when $N$ tends to $+\infty$. Then for any $M>0$ we have $\mathds{P}(\sup_{t\in[0,M]}|\hat Y_t^N| \geq 1/2)$ tends to 0 when $N$ tends to $+\infty$, thus $\mathds{P}(\sup_{t\in[0,M]}|Y_t^N| \geq 1/2)$ tends to 0 when $N$ tends to $+\infty$. In this proof, we set $\varepsilon=1$. Then for any $M>0$ we have that $\mathds{P}(\frac{1}{n^{3/2}}(T_1-T_0) \leq M)$ tends to 0 when $N$ tends to $+\infty$. However, by Proposition \ref{prop_conv_ZT}, $\frac{1}{n^{3/2}}(T_1-T_0)$ converges in distribution to the $\breve T_1$ defined in Definition \ref{def_lim_envts} when $N$ tends to $+\infty$, and $\breve T_1$ has no atoms. This implies that for any $M>0$, we have $\mathds{P}(\breve T_1 \leq M)=0$, which is impossible. This ends the proof. 
\end{proof}

\subsection{Proof of Proposition \ref{prop_continuity}} \label{sec_conclusion_prop}

For any $t \in \mathds{R}^+$, $N \in \mathds{N}^*$, we denote $\hat Y_t^{N}=\frac{1}{N}X_{\lfloor tN^2\rfloor}$. By the definition of the Skorohod topology (see Theorem 10 of Chapter VI of \cite{PollardJ1}), it is enough to prove that for any subsequence $(\hat Y_t^{\psi(N)})_{N \in \mathds{N}^*}$, there exists $\delta_1,\delta_2>0$ so that for any $\vartheta >0$ large enough, for any $M\in\mathds{N}^*$ and any $0=t_0 < \cdots < t_M=\vartheta$, we have 
\begin{align*}
\limsup_{N \to +\infty}\mathds{P}\left(\max_{1 \leq i \leq M}\inf\{\delta_3>0\,|\,\exists\, t \in(t_{i-1},t_{i}]\text{ with }|\hat Y_s^{\psi(N)}-\hat Y_{t_{i-1}}^{\psi(N)}|<\delta_3\text{ if }t_{i-1}\leq s < t \right. \quad & \\ 
\left. \text{ and }|\hat Y_s^{\psi(N)}-\hat Y_{t_{i}}^{\psi(N)}|<\delta_3\text{ if }t\leq s \leq t_{i}\} > \delta_2\right) & \geq \delta_1.
\end{align*}
Moreover, the process $(\hat Y_t^{\psi(N)})_{t\in[0,+\infty)}$ has jumps of size $1/{\psi(N)}$, which tends to 0 when $N$ tends to $+\infty$, so it is enough to show that for any $\vartheta\geq1$, $M\in\mathds{N}^*$ and any $0=t_0 < \cdots < t_M=\vartheta$, we have 
\[
\lim_{N \to +\infty}\mathds{P}\left(\exists\, i \in\{1,...,M\},\max_{t_{i-1} \leq t \leq t_{i}}|\hat Y_t^{\psi(N)}-\hat Y_{t_{i-1}}^{\psi(N)}| > \frac{1}{32\sqrt{2}}\right) = 1.
\]

Let $\vartheta\geq1$, $M\in\mathds{N}^*$ and $0=t_0 < \cdots < t_M=\vartheta$. We notice that if there exist $s,t\in[0,\vartheta]$ so that $|s-t|< \min_{1\leq i \leq M}|t_{i}-t_{i-1}|$ but $|\hat Y_s^{\psi(N)}-\hat Y_{t}^{\psi(N)}| > \frac{1}{8\sqrt{2}}$, then there exists $ i \in\{1,...,M\}$ so that $\max_{t_{i-1} \leq t \leq t_{i}}|\hat Y_t^{\psi(N)}-\hat Y_{t_{i-1}}^{\psi(N)}| > \frac{1}{32\sqrt{2}}$. We will choose $t=\frac{T_0'}{\psi(N)^2}$ and $s=\frac{T_1'}{\psi(N)^2}$, where $T_0'$ and $T_1'$ are equivalents of $T_0,T_1$ which we define now. We set $x=0$, $\theta=\frac{\sqrt{\vartheta}}{2\sqrt{2}}$, $T_0'=\mathbf{T}^-_{\lfloor \psi(N)\theta\rfloor,0}$ (see \eqref{eq_def_T_mi}), and $T_1'=\inf\{m \geq T_0' \,|\, |X_m-X_{T_0'}|=\lfloor \theta \psi(N)/2 \rfloor\}$ (this definition differs slightly from the definition of $T_1$, as $\psi(N)$ replaces $N$ and $\theta \psi(N)/2$ replaces $\varepsilon n$). We then have 
\[
|\hat Y_t^{\psi(N)} - \hat Y_s^{\psi(N)}|=\frac{1}{\psi(N)}\lfloor \theta \psi(N)/2 \rfloor > \frac{\theta}{4} \geq \frac{1}{8\sqrt{2}}
\]
when $N$ is large enough. Consequently, we only have to prove that 
\[
\lim_{N \to +\infty}\mathds{P}\left(\frac{T_1'}{\psi(N)^2} \leq \vartheta, \frac{T_1'}{\psi(N)^2}-\frac{T_0'}{\psi(N)^2} \leq \min_{1\leq i \leq M}|t_{i}-t_{i-1}|\right)=1.
\]
In order to do that, we remark that Corollary 1 of \cite{Toth_et_al2008} states that $\frac{T_0'}{\psi(N)^2}$ converges in probability to $4\theta^2=4(\frac{\sqrt{\vartheta}}{2\sqrt{2}})^2=\frac{\vartheta}{2}$ when $N$ tends to $+\infty$, which implies $\lim_{N\to+\infty}\mathds{P}(\frac{T_0'}{\psi(N)^2} > \frac{3\vartheta}{4})=0$, thus we only have to prove $\lim_{N\to+\infty}\mathds{P}(\frac{T_1'}{\psi(N)^2}-\frac{T_0'}{\psi(N)^2} > \frac{\vartheta}{4} \wedge \min_{1\leq i \leq M}|t_{i}-t_{i-1}|)=0$. Moreover, by Lemma \ref{lem_continuity} $\frac{1}{\psi(N)^{3/2}}(T_1'-T_0')$ converges in distribution when $N$ tends to $+\infty$, which is sufficient and ends the proof. 

\section{Notation and auxiliary random variables}\label{sec_first_variables}

If $a,b \in \mathds{R}$, we set $a \wedge b = \min(a,b)$, $a \vee b = \max(a,b)$, and $a_+=\max(a,0)$. For any set $A$ and any function  $f : A \mapsto \mathds{R}$, we denote $\|f\|_\infty = \sup_{x \in A}|f(x)|$. For any $m \in \mathds{N}$, we define 
\begin{equation}\label{eq_def_F}
\mathcal{F}_m = \sigma(X_0,X_1,\dots,X_m)
\end{equation}

Let $\varepsilon > 0$. $\varepsilon$ may take different values throughout the paper, but the one used will always be clear from the context. Remembering that $\mathbf{T}_{\lfloor N \theta\rfloor,\lfloor Nx \rfloor}^\pm$ was defined in \eqref{eq_def_T_mi}, we recall the following definition already given at the beginning of Section \ref{sec_conclusion}:  
\begin{equation}\label{eq_def_Tk}
T_0=\mathbf{T}_{\lfloor N \theta\rfloor,\lfloor Nx \rfloor}^\pm \text{ and } \forall\, k \in \mathds{N}, 
T_{k+1}=\inf\{k' \geq T_{k}\,|\,|X_{k'}-X_{T_{k}}| = \lfloor\varepsilon n\rfloor\}.
\end{equation}
For any $m \in \mathds{N}$, we also introduce the stopping times (compatible with those defined at the beginning of Section \ref{sec_outline_conv}):
\begin{equation}\label{eq_def_Tk_pm^}
\beta_m^{\pm} = \inf\{m'\geq m\,|\, X_{m'} = X_{m}\pm\lfloor\varepsilon n\rfloor\}.
\end{equation}
Proposition 1 of \cite{Toth_et_al2008} states that almost surely, for any $i\in\mathds{Z}$, $m,m' \in\mathds{N}$, the local time $\ell_{m,i}^\pm$ defined in \eqref{eq_def_ell} will reach $m'$ in finite time, therefore all these stopping times are finite. We recall the definition of the $\Delta_{m,i}$ in \eqref{eq_def_Delta}. 

 \begin{definition}\label{def_local_times}
 For any $m \in \mathds{N}$, we define random variables $\zeta^{m,\pm,B}_i$, $\zeta^{m,\pm,E}_i$ 
  for $i \in\mathds{Z}$ as follows:
  \[
  \zeta^{m,-,B}_i = \left\{\begin{array}{ll}
                          -\Delta_{m,i}-1/2&\text{if }i \leq X_m,\\
                          -\Delta_{m,i}+1/2&\text{if }i > X_m,
                         \end{array}\right.
  \quad\text{and}\quad \zeta^{m,-,E}_i = -\Delta_{\beta_{m}^-,i}+1/2,
 \]
 \[
  \zeta^{m,+,B}_i = \left\{\begin{array}{ll}
                          \Delta_{m,i}+1/2&\text{if }i \leq X_m,\\
                          \Delta_{m,i}-1/2&\text{if }i > X_m,
                         \end{array}\right.
  \quad\text{and}\quad \zeta^{m,+,E}_i = \Delta_{\beta_{m}^+,i}+1/2.
 \]
 We also define $L_i^{m,\pm} = |\{m \leq m' < \beta_m^\pm \,|\, (X_{m'},X_{m'+1}) = (i-1,i)\}|$.
 \end{definition}
 
 The superscript $B$ stands for ``beginning'', and the superscript $E$ for ``end'', since we will use the corresponding random variables respectively at the beginning and at the end of ``steps of the mesoscopic walk $(X_{T_k})_{k \in \mathds{N}}$''.
 
 \begin{observation}\label{obs_rec_temps_local}
 For $i \geq X_m-\lfloor\varepsilon n\rfloor+1$, we have $L_{i+1}^{m,-}=L_i^{m,-}+\zeta^{m,-,E}_i-\zeta^{m,-,B}_i$, and 
 for $i \leq X_m+\lfloor\varepsilon n\rfloor-1$ we have $L_{i}^{m,+}=L_{i+1}^{m,+}+\zeta^{m,+,E}_i-\zeta^{m,+,B}_i$.
\end{observation}
 
 \begin{proof}
 We write the proof for $L_{i+1}^{m,-}$; the argument for $L_{i}^{m,+}$ is similar. We have 
 \[
 L_{i+1}^{m,-} = \ell_{\beta_m^-,i}^+-\ell_{m,i}^+ 
 = (\ell_{\beta_m^-,i}^+-\ell_{\beta_m^-,i}^-) -(\ell_{m,i}^+-\ell_{m,i}^-) + \ell_{\beta_m^-,i}^--\ell_{m,i}^- = -\Delta_{\beta_m^-,i} - (-\Delta_{m,i}) + \ell_{\beta_m^-,i}^--\ell_{m,i}^-.
 \]
 Now, $\ell_{\beta_m^-,i}^--\ell_{m,i}^-$ is the number of times $X$ goes from $i$ to $i-1$ between $m$ and $\beta_m^-$, which 
 is $L_i^{m,-}+1$ if $i \leq X_m$ and $L_i^{m,-}$ if $i > X_m$, hence the result.
 \end{proof}
 
In order to control the behavior of the $\zeta^{m,\pm,B}_i$, $\zeta^{m,\pm,E}_i$, we recall some definitions and properties from \cite{Toth_et_al2008}. 
We define a Markov chain $(\xi(m))_{m \in \mathds{N}}$ on $\mathds{Z}$ by the following transition probabilities:
\begin{equation}\label{eq_def_xi}
 \mathds{P}(\xi(m+1)=\xi(m)+1) = 1-\mathds{P}(\xi(m+1)=\xi(m)-1) = \frac{w(-\xi(m))}{w(\xi(m))+w(-\xi(m))}.
\end{equation}
We notice that for any $i\in\mathds{Z}$ the jump chain of $(\Delta_{m,i})_{m\in\mathds{N}}$ has the law of $(\xi(m))_{m \in \mathds{N}}$. For any $m \in \mathds{N}$, we denote by $\tau_+(m)$ (respectively $\tau_-(m)$) the time of the $m$-th 
upwards (respectively downwards) step of $\xi$: 
\begin{equation}\label{eq_def_tau}
\tau_\pm(0)=0 \text{ and } \forall\, m \in \mathds{N}, 
\tau_\pm(m+1) = \inf\{m' > \tau_\pm(m)\,|\, \xi(m')=\xi(m'-1) \pm 1\}
\end{equation}
which can easily be seen to be finite. 
The processes $(\eta_-(m))_{m \in \mathds{N}}$ and $(\eta_+(m))_{m \in \mathds{N}}$ defined by
\begin{equation}\label{eq_def_eta}
 \eta_{+}(m)=-\xi(\tau_+(m)) \qquad\text{and}\qquad \eta_{-}(m)=\xi(\tau_-(m))
\end{equation}
for any $m \in \mathds{N}$ are Markov chains on $\mathds{Z}$. Moreover, and one can check that $\xi$ and $-\xi$ have the same law, which implies $\eta_+$ and $\eta_-$ have the same law. In what follows, $\eta$ will refer to a Markov chain with this law. 

For any $m \in \mathds{N}$, $i > X_m-\lfloor\varepsilon n\rfloor$, at time $\beta_m^-$ the process $(X_{m'})_{m' \in \mathds{N}}$ is at the left of $i$, thus the last time $(X_{m'})_{m' \in \mathds{N}}$ was at $i$ it went to the left, hence the last step of $(\Delta_{m',i})_{m' \in \mathds{N}}$ before time $\beta_m^-$ was an upwards step. Moreover, the number of upwards steps made by $(\Delta_{m',i})_{m' \in \mathds{N}}$ between times $m$ and $\beta_m^-$ is $|\{m \leq m' \leq \beta_m^-\,|\,(X_{m'},X_{m'+1})=(i,i-1)\}|=\ell_{\beta_m^-,i}^--\ell_{m,i}^-$. This implies $\Delta_{\beta_{m}^-,i} 
= \xi(\tau_+(\ell_{\beta_m^-,i}^--\ell_{m,i}^-)) 
= -\eta(\ell_{\beta_m^-,i}^--\ell_{m,i}^-)$ where $\xi$ starts at $\Delta_{m,i}$, 
so $\eta$ starts at $-\Delta_{m,i}$, with the transitions of $\xi$, $\eta$ independent of 
$\mathcal{F}_m$, $\Delta_{\beta_{m}^-,j}$, $j < i$. Similarly, if $i<X_m+\lfloor\varepsilon n\rfloor$, we have $\Delta_{\beta_{m}^+,i} 
= \eta(\ell_{\beta_m^+,i}^+-\ell_{m,i}^+)$ with $\eta(0)=\Delta_{m,i}$. We deduce that 
\begin{equation}\label{eq_zeta}
\zeta_i^{m,-,E}=\left\{\begin{array}{ll}
                        \eta(L_i^{m,-}+1)+1/2\text{ with }\eta(0) = -\Delta_{m,i} & 
                        \text{if }X_m-\lfloor\varepsilon n\rfloor < i \leq X_m,\\
                        \eta(L_i^{m,-})+1/2\text{ with }\eta(0) = -\Delta_{m,i} & \text{if }i > X_m, \\
                       \end{array}
\right.
\end{equation}
\[
\forall\, i \in \mathds{Z}, \zeta_i^{m,+,E} = \eta(L_{i+1}^{m,+})+1/2\text{ with }\eta(0) = \Delta_{m,i}.
\]

In \cite{Toth_et_al2008}, it was proven that the measure $\rho_-$ defined as follows is the unique invariant 
probability measure of $\eta$:
\begin{equation}\label{eq_def_rho-}
 \forall\, i \in \mathds{Z},\quad 
 \rho_-(i) = \frac{1}{Z(w)} \prod_{j=1}^{\lfloor|2i+1|/2\rfloor} \frac{w(-j)}{w(j)} \quad \text{with} \quad 
 Z(w) = \sum_{i \in \mathds{Z}}\prod_{j=1}^{\lfloor|2i+1|/2\rfloor} \frac{w(-j)}{w(j)}.
\end{equation}
We notice that for any $i\in \mathds{N}$, $\rho_-(-i-1)=\rho_-(i)$, so $\rho_-$ is symmetric with respect to $-1/2$. 
Therefore, we may define the measure $\rho_+$ on $\mathds{Z}$ by 
\begin{equation}\label{eq_def_rho+}
\forall \, i \in \mathds{Z}, \rho_+(i)=\rho_-(i-1)=\rho_-(-i). 
\end{equation}
$\rho_-$ and $\rho_+$ have respective expectations $-1/2$ and $1/2$. We also denote $\rho_0$ the measure 
on $\frac{1}{2}+\mathds{Z}$ defined by 
\begin{equation}\label{eq_def_rho0}
 \forall \, i \in \frac{1}{2}+\mathds{Z}, \rho_0(i)=\rho_-\left(i-\frac{1}{2}\right),
\end{equation}
which has expectation 0. The measure $\rho_0$ is very important, since the law of the $\zeta^{\mathbf{T}_{m,i}^\iota,\pm,B},\zeta^{\mathbf{T}_{m,i}^\iota,\pm,E}$ will be close to $\rho_0$ under ``good conditions''. In particular, these variables will have expectation close to 0. 

\begin{remark}
 We could study our random walk $(X_m)_{m\in\mathds{N}}$ ``starting from a random environment'', that is setting the $\Delta_{0,i}$, $i\in\mathds{Z}$ to random variables instead of setting them to 0, and then evolving $(X_m)_{m\in\mathds{N}}$ and the $(\Delta_{m,i})_{m \in\mathds{N}}$, $i\in\mathds{Z}$ according to the usual rules. This yields a new random walk $(\hat X_m)_{m\in\mathds{N}}$ and an ``environment'' process on $(\hat \Delta_{m,i})_{m \in\mathds{N},i\in\mathds{Z}}$ which evolves as follows:
  \begin{enumerate}
\item We choose the $\hat\Delta_{0,i}$, $i\in\mathds{Z}$ to be independent, with distribution $\rho_-$ for $i<0$, $\rho_+$ for $i>0$ and $\frac{\rho_- + \rho_+}{2}$ for $i=0$. We set $\hat X_0=0$.
\item For any $m \geq 0$, $\mathds{P}(\hat X_{m+1} = \hat X_m + 1)=1-\mathds{P}(\hat X_{m+1} = \hat X_m - 1)=\frac{w(\hat \Delta_{m,\hat X_m})}{w(-\hat \Delta_{m,\hat X_m})+w(\hat \Delta_{m,\hat X_m})}$.
\item For $m \geq 0$, we set $\hat\Delta_{m+1,i} = \hat\Delta_{m,i}$ for $i \ne \hat X_m$ and $\hat \Delta_{m+1,\hat X_m}= \hat \Delta_{m,\hat X_m} - (\hat X_{m+1}-\hat X_m)$.
\end{enumerate}
Let $\mu_0$ be the law of $(\Delta_{0,i})_{i\in\mathds{Z}}$ and for all $m>0$ let $\mu_m$ be the law of $(\Delta_{m,i})_{i\in\mathds{Z}}$ shifted by the ``tagged particle" $\hat X_m$, that is the law of $(\Delta_{m,\hat X_m+i})_{i\in\mathds{Z}}$. Then direct calculation shows that for each $m \geq 0, \mu_m = \mu_0$. So with this particular measure $\mu_0 $ for the initial environment, the distribution of the environment is stationary, hence the increments of $(\hat X_m)_{m\in\mathds{N}}$ are stationary. Though of course knowing that $\hat X_1=1$ will typically mean that the conditional distribution of $(\hat\Delta_{1,i})_{i\in\mathds{Z}}$ shifted by $1$ is not $\mu_0$.

A slight modification of our arguments (to deal with the distribution of $(\hat \Delta_{0,i})_{i\in\mathds{Z}}$) shows that the motion of $(\hat X_m)_{m\in\mathds{N}}$ is governed by Theorem \ref{thm_main}. 
Unlike the motion of the tagged particle in an exclusion process with a non nearest neighbor jump kernel or in high dimension with an initial product measure (see \cite{Kipnis_et_al1986}), our environment does not evolve outside of the position of $\hat X_m$. In this it is like the Markov chain cookie random walk studied by \cite{Kosygina_et_al2022} where the intial distribution of the environment is $\pi_- $ for $i<0$, $\pi_+$ for $i>0$ and $\frac{\pi_+ + \pi_-}{2} = \pi $ for $i=0$. However, in the Markov chain cookie random walk, the ``tagged particle" does have a motion that (under diffusive scaling) converges to a Brownian motion, which is not the case in our model (as the limit has non-Brownian scaling properties). The two models, though similar, thus have a different behavior, and the reason for that is not clear, though obviously the operator for our process does not fall into the domain of Kipnis-Varadhan analysis.
\end{remark}

In order to control the behavior of $\eta$, thus of the $\zeta^{m,\pm,B}_i$, $\zeta^{m,\pm,E}_i$, we will need the following lemma, proved in \cite{Toth_et_al2008}.

\begin{lemma}[Lemma 1 of \cite{Toth_et_al2008}]\label{lem_conv_eta}
 There exist constants $\bar c = \bar c(w) > 0$ and $\bar C = \bar C(w) < \infty$ such that for any $m \geq 0$, 
 \[
 \mathds{P}(\eta(m)=i|\eta(0)=0) \leq \bar C e^{-\bar c |i|} \qquad\text{ and }\qquad
 \sum_{i \in \mathds{Z}}|\mathds{P}(\eta(m)=i|\eta(0)=0)-\rho_-(i)| \leq \bar C e^{-\bar c m}. 
 \]
 \end{lemma}

 We now state two easy coupling lemmas, which we will need in order to define auxiliary random variables.

\begin{lemma}\label{lem_const_coupling}
 For any probability laws $\mu$ and $\nu$ on $\mathds{Z}$, for any random variables $V_\mu$ with law $\mu$ and $U$ 
 independent from $V_\mu$ uniform on $[0,1]$, one can construct a random variable $V_\nu$ of law $\nu$ depending only on 
 $V_\mu$ and $U$ such that $\mathds{P}(V_\mu \neq V_\nu)$ is minimal.
\end{lemma}

\begin{proof}
 We suppose $\mu \neq \nu$, as if $\mu=\nu$ we can take $V_\nu=V_\mu$. 
 The construction is as follows. If $\mu(V_\mu) \leq \nu(V_\mu)$, we set $V_\nu=V_\mu$. If $\mu(V_\mu) > \nu(V_\mu)$, 
 we set $V_\nu=V_\mu$ if $U \in [0,\frac{\nu(V_\mu)}{\mu(V_\mu)}]$, and for any $i \in \mathds{Z}\setminus\{V_\mu\}$, 
 $V_\nu=i$ if $U$ is in
 \[
 \left(\frac{\nu(V_\mu)}{\mu(V_\mu)}+\frac{\mu(V_\mu)-\nu(V_\mu)}{\mu(V_\mu)}
 \frac{\sum_{j<i}(\nu(j)-\mu(j))_+}{\sum_{j\in\mathds{Z}}(\nu(j)-\mu(j))_+},
 \frac{\nu(V_\mu)}{\mu(V_\mu)}+\frac{\mu(V_\mu)-\nu(V_\mu)}{\mu(V_\mu)}
 \frac{\sum_{j \leq i}(\nu(j)-\mu(j))_+}{\sum_{j\in\mathds{Z}}(\nu(j)-\mu(j))_+}\right]. 
 \]
 It is straightforward to check that $V_\nu$ has law $\nu$ and that 
 $\mathds{P}(V_\mu \neq V_\nu)=\sum_{j\in\mathds{Z}}(\mu(j)-\nu(j))_+$, hence is minimal. 
\end{proof}

\begin{lemma}\label{lem_coupling}
 It is possible to couple two processes $\eta$ and $\eta'$ with $\eta'(0)=\eta(0)-1$ so that 
 for any $\ell \in \mathds{N}$, $\eta(\ell)-1 \leq \eta'(\ell) \leq \eta(\ell)$.
\end{lemma}

\begin{proof}
 It is enough to couple $\eta(1)$ and $\eta'(1)$ so that $\eta(1)-1 \leq \eta'(1) \leq \eta(1)$. For this, we 
 set $U$ a random variable uniform on $[0,1]$, and we set $\tilde\eta(1) = i$ when 
 $U \in (\mathds{P}(\tilde\eta(1) \geq i+1),\mathds{P}(\tilde\eta(1) \geq i)]$ for $\tilde \eta = \eta$ or $\eta'$. $\eta$ and $\eta'$ 
 have the right marginal laws. We only have to prove that for any $i \in \mathds{Z}$, 
 \[
 \mathds{P}(\eta'(1) \geq i+1) \leq \mathds{P}(\eta(1) \geq i+1)\qquad\text{ and }\qquad\mathds{P}(\eta(1) \geq i) \leq 
 \mathds{P}(\eta'(1) \geq i-1).
 \]
 Now, for any $i \in \mathds{Z}$, one can check that $\mathds{P}(\tilde\eta(1) \geq i) = \prod_{j=\tilde\eta(0)}^i \frac{w(-j)}{w(j)+w(-j)}$ if $i \geq \tilde\eta(0)$, 
 and $\mathds{P}(\tilde\eta(1) \geq i) =1$ if $i < \tilde\eta(0)$. Since $\eta'(0) < \eta(0)$, 
 we deduce $\mathds{P}(\eta'(1) \geq i+1) \leq \mathds{P}(\eta(1) \geq i+1)$. Now, if $i < \eta(0)$, $i-1 < \eta'(0)$, 
 so $\mathds{P}(\eta(1) \geq i) = \mathds{P}(\eta'(1) \geq i-1) = 1$, and if $i \geq \eta(0)$ we have 
 \[
 \mathds{P}(\eta(1) \geq i) = \prod_{j=\eta(0)}^i \frac{w(-j)}{w(j)+w(-j)} 
 \leq \prod_{j=\eta'(0)}^{i-1} \frac{w(-j)}{w(j)+w(-j)} = \mathds{P}(\eta'(1) \geq i-1)
 \]
 since $\frac{w(-.)}{w(.)+w(-.)}$ is non-increasing. This ends the proof of the lemma. 
\end{proof}

We are now in position to control the laws of the $\Delta_{\mathbf{T}_{m,i}^\iota,j}$ for $m \geq N \theta/2$, $i \in \mathds{Z}$, $\iota \in \{+,-\}$. Heuristically, the $\Delta_{\mathbf{T}_{m,i}^\iota,j}$ are steps of chains $\eta$ or $-\eta$, and these chains have made a large number of steps before time $\mathbf{T}_{m,i}^\iota$ since $m$ is large, hence the $\Delta_{\mathbf{T}_{m,i}^\iota,j}$ will have law close to the invariant measure of $\eta$ or $-\eta$, that is $\rho_-$ or $\rho_+$.  More precisely, we have the following proposition (remember the definitions of $n$ and $\mathcal{F}_m$ given in \eqref{eq_scale_n},\eqref{eq_def_F}). 

\begin{proposition}\label{prop_law_zeta}
For any $m \geq N \theta/2$, $i \in \mathds{Z}$, $\iota \in \{+,-\}$, there exists a collection of random variables 
$(\bar \Delta_{\mathbf{T}_{m,i}^\iota,j})_{j \in \mathds{Z}}$, an event $\mathcal{B}_{0}^{m,i,\iota}$, and constants $C_0=C_0(w,\varepsilon) < \infty$ and $c_0=c_0(w) > 0$, so that when $n$ is large enough, $\mathds{P}(\mathcal{B}_{0}^{m,i,\iota}) \leq C_0 e^{-c_0n^{(\alpha-1)/4}}$, $\mathcal{B}_{0}^{m,i,\iota}$ contains $\{$there exists $i-n^{(\alpha-1)/4}\lfloor \varepsilon n\rfloor-1 \leq j \leq i+n^{(\alpha-1)/4}\lfloor \varepsilon n\rfloor+1, \bar \Delta_{\mathbf{T}_{m,i}^\iota,j} \neq \Delta_{\mathbf{T}_{m,i}^\iota,j}\}$,  $\mathcal{B}_{0}^{m,i,\iota}$ depends only on $\mathcal{F}_{\mathbf{T}_{m,i}^\iota}$ and on random variables independent from $X$, and the $(\bar \Delta_{\mathbf{T}_{m,i}^\iota,j})_{j \in \mathds{Z}}$ are independent with the following laws:
\begin{itemize}
  \item for $\iota=-$, $\bar \Delta_{\mathbf{T}_{m,i}^\iota,j}$ has law $\rho_-$ for $j \leq i-1$ and 
  $\bar \Delta_{\mathbf{T}_{m,i}^\iota,j}$ has law $\rho_+$ for $j \geq i$;
  \item for $\iota=+$, $\bar \Delta_{\mathbf{T}_{m,i}^\iota,j}$ has law $\rho_-$ for $j \leq i$ and 
  $\bar \Delta_{\mathbf{T}_{m,i}^\iota,j}$ has law $\rho_+$ for $j \geq i+1$.
 \end{itemize}
\end{proposition}

\begin{proof}
We write the argument for $\iota=-$; the case $\iota=+$ is similar. 
We begin by constructing $(\bar \Delta_{\mathbf{T}_{m,i}^\iota,j})_{j \in \mathds{Z}}$. This construction is inspired from the one in Section 3.3 of \cite{Toth_et_al2008}. We have $X_{\mathbf{T}_{m,i}^\iota-1}=i$, thus for $j \leq i-1$, the last time before $\mathbf{T}_{m,i}^\iota$ that the process $(X_{m'})_{m'\in\mathds{Z}}$ was at $j$, it went to the right, hence the last step of $(\Delta_{m',j})_{m'\in\mathds{N}}$ before time $\mathbf{T}_{m,i}^\iota$ is an downwards step. Moreover, the number of downwards steps of $(\Delta_{m',j})_{m'\in\mathds{N}}$ before time $\mathbf{T}_{m,i}^\iota$ is $\ell_{\mathbf{T}_{m,i}^-,j}^+$. We deduce that $\Delta_{\mathbf{T}_{m,i}^-,j} = \xi_j(\tau_-(\ell_{\mathbf{T}_{m,i}^-,j}^+)) = 
\eta_{j,-}(\ell_{\mathbf{T}_{m,i}^-,j}^+)$, where the $\eta_{j,-}$ are independent copies of $\eta$ starting from 0 
(and the $\xi_j$ are independent copies of $\xi$ starting from 0). In the same way, 
for $j \geq i$, $\Delta_{\mathbf{T}_{m,i}^-,j} = \xi_j(\tau_+(\ell_{\mathbf{T}_{m,i}^-,j}^-)) = 
-\eta_{j,+}(\ell_{\mathbf{T}_{m,i}^-,j}^-)$ where the $\eta_{j,+}$ are independent, independent from the $\eta_{j,-}$, $j \leq i-1$, 
and start from 0. We will drop the index $+$ or $-$ from the $\eta_{j,\pm}$ for convenience. 
By Lemmas \ref{lem_conv_eta} and \ref{lem_const_coupling}, we can introduce random variables $r_j$ i.i.d. of law $\rho_-$ such that 
for any $j$, $\mathds{P}(\eta_j(\lceil N\theta/4 \rceil) \neq r_j) \leq \bar C e^{-\bar c \theta N/4}$. 
For any $j$, we define another copy of $\eta$, $(\bar \eta_j(\ell))_{\ell \geq \lceil N\theta/4 \rceil}$, 
so that $\bar \eta_j(\lceil N\theta/4 \rceil)=r_j$, if $\eta_j(\lceil N\theta/4 \rceil) = r_j$, 
$\bar \eta_j(\ell) = \eta_j(\ell)$ for any $\ell \geq \lceil N\theta/4 \rceil$, and the $\bar \eta_j$ are 
independent. For $j \leq i-1$, we set $\bar \Delta_{\mathbf{T}_{m,i}^-,j} = 
\bar \eta_j(\ell_{\mathbf{T}_{m,i}^-,j}^+ \vee \lceil N\theta/4 \rceil)$, and for $j \geq i$ we set 
$\bar \Delta_{\mathbf{T}_{m,i}^-,j} = -\bar \eta_j(\ell_{\mathbf{T}_{m,i}^-,j}^- \vee \lceil N\theta/4 \rceil)$. 

To show that the $\bar \Delta_{\mathbf{T}_{m,i}^-,j}$ are independent with the required laws, we notice that 
since $\rho_-$ is invariant for $\eta$, $\bar \eta_j(\ell)$ has law $\rho_-$ for any $\ell \geq \lceil N\theta/4 \rceil$. 
Now, for $j > i$, we notice that $\ell_{\mathbf{T}_{m,i}^-,j}^- = \ell_{\mathbf{T}_{m,i}^-,j-1}^+$ 
or $\ell_{\mathbf{T}_{m,i}^-,j-1}^++1$ depending only on the position of $i$ and $j$ with respect to 0. 
Furthermore, $\ell_{\mathbf{T}_{m,i}^-,j-1}^+=\ell_{\mathbf{T}_{m,i}^-,j-1}^--\Delta_{\mathbf{T}_{m,i}^-,j-1}
=\ell_{\mathbf{T}_{m,i}^-,j-1}^-+\eta_{j-1}(\ell_{\mathbf{T}_{m,i}^-,j-1}^-)$, and we recall that 
$\ell_{\mathbf{T}_{m,i}^-,i}^-=m$, so one can prove by induction that $\ell_{\mathbf{T}_{m,i}^-,j}^-$ depends only on 
the $\eta_{j'}$, $i \leq j' < j$, which are independent from $\bar \eta_j$, therefore $\bar \Delta_{\mathbf{T}_{m,i}^-,j}$ 
is independent from the $\eta_{j'},\bar \eta_{j'}$, $j' < j$, and has law $\rho_+$. The same argument can be used for 
$j \leq i-1$ to show that $\ell_{\mathbf{T}_{m,i}^-,j}^+$ depends only on the 
$\eta_{j'}$, $j < j' < i$ so $\bar \Delta_{\mathbf{T}_{m,i}^-,j}$ has law $\rho_-$ and is independent from 
the $\eta_{j'},\bar \eta_{j'}$, $j' > j$. This implies the $\bar \Delta_{\mathbf{T}_{m,i}^-,j}$ are independent with the required laws. 

We now define $\mathcal{B}_{0}^{m,i,-}$. We set 
\[
\mathcal{B}_{0,1}^{m,i,-} = \{\exists\, j\text{ so that } i-n^{(\alpha-1)/4}\lfloor \varepsilon n\rfloor-1 \leq j \leq i+n^{(\alpha-1)/4}\lfloor \varepsilon n\rfloor+1,\eta_j(\lceil N\theta/4 \rceil) \neq r_j\},
\]
\[
\mathcal{B}_{0,2}^{m,i,-} = \{\exists\, j\text{ so that }i-n^{(\alpha-1)/4}\lfloor \varepsilon n\rfloor 
\leq j \leq i+n^{(\alpha-1)/4}\lfloor \varepsilon n\rfloor, |\Delta_{\mathbf{T}_{m,i}^-,j}| > n^{(\alpha-1)/4}\},
\]
and $\mathcal{B}_{0}^{m,i,-} = \mathcal{B}_{0,1}^{m,i,-} \cup \mathcal{B}_{0,2}^{m,i,-}$. To show that $\mathcal{B}_{0}^{m,i,-}$ 
contains the required event, we notice that for any $j \in \mathds{Z}$, 
$\ell_{\mathbf{T}_{m,i}^-,j}^- = \Delta_{\mathbf{T}_{m,i}^-,j}+\ell_{\mathbf{T}_{m,i}^-,j}^+$ 
and $|\ell_{\mathbf{T}_{m,i}^-,j}^- - \ell_{\mathbf{T}_{m,i}^-,j-1}^+| \leq 1$, so we have 
$|\ell_{\mathbf{T}_{m,i}^-,j-1}^+ - \ell_{\mathbf{T}_{m,i}^-,j}^+| \leq |\Delta_{\mathbf{T}_{m,i}^-,j}|+1$ 
and $|\ell_{\mathbf{T}_{m,i}^-,j+1}^- - \ell_{\mathbf{T}_{m,i}^-,j}^-| \leq |\Delta_{\mathbf{T}_{m,i}^-,j}|+1$. 
In addition, we recall that $\ell_{\mathbf{T}_{m,i}^-,i}^- = m$. We deduce that when $n$ is large enough, 
for any $0 \leq s \leq n^{(\alpha-1)/4}\lfloor \varepsilon n\rfloor$, if 
$|\Delta_{\mathbf{T}_{m,i}^-,j}| \leq n^{(\alpha-1)/4}$ for all $i-s \leq j \leq i-1$, then 
$|\ell_{\mathbf{T}_{m,i}^-,j}^+-m| \leq 4 \varepsilon n^{(\alpha+1)/2}$ for $i-s-1 \leq j \leq i-1$, 
and if $|\Delta_{\mathbf{T}_{m,i}^-,j}| \leq n^{(\alpha-1)/4}$ for all $i \leq j \leq i+s$ then 
$|\ell_{\mathbf{T}_{m,i}^-,j}^--m| \leq 4 \varepsilon n^{(\alpha+1)/2}$ for $i \leq j \leq i+s+1$. 
Consequently, if $(\mathcal{B}_{0,2}^{m,i,-})^c$ is satisfied and $n$ is large enough, 
for all $i-n^{(\alpha-1)/4}\lfloor \varepsilon n\rfloor-1 \leq j \leq i-1$ 
we have $\ell_{\mathbf{T}_{m,i}^-,j}^+ \geq m-4 \varepsilon n^{(\alpha+1)/2} \geq \lceil N\theta/4 \rceil$, 
and for all $i \leq j \leq i+n^{(\alpha-1)/4}\lfloor \varepsilon n\rfloor+1$ we have 
$\ell_{\mathbf{T}_{m,i}^-,j}^- \geq m-4 \varepsilon n^{(\alpha+1)/2} \geq \lceil N\theta/4 \rceil$. 
This implies that if $(\mathcal{B}_{0}^{m,i,-})^c$ is satisfied, for all 
$i-n^{(\alpha-1)/4}\lfloor \varepsilon n\rfloor-1 \leq j \leq i+n^{(\alpha-1)/4}\lfloor \varepsilon n\rfloor+1$, 
$\bar \Delta_{\mathbf{T}_{m,i}^-,j} = \Delta_{\mathbf{T}_{m,i}^-,j}$, thus $\mathcal{B}_{0}^{m,i,-}$ 
contains the required event. 

To see that $\mathcal{B}_{0}^{m,i,-}$ has the required dependencies, we notice that $\mathcal{B}_{0,2}^{m,i,-}$ 
depends on $\mathcal{F}_{\mathbf{T}_{m,i}^-}$. Furthermore, if $(\mathcal{B}_{0,2}^{m,i,-})^c$ is satisfied 
and $n$ is large enough, for all $i-n^{(\alpha-1)/4}\lfloor \varepsilon n\rfloor-1 \leq j \leq i-1$ 
we have $\ell_{\mathbf{T}_{m,i}^-,j}^+ \geq \lceil N\theta/4 \rceil$ and 
for all $i \leq j \leq i+n^{(\alpha-1)/4}\lfloor \varepsilon n\rfloor+1$ we have 
$\ell_{\mathbf{T}_{m,i}^-,j}^- \geq \lceil N\theta/4 \rceil$, so the events $\{\eta_j(\lceil N\theta/4 \rceil) \neq r_j\}$ depend 
only on $\mathcal{F}_{\mathbf{T}_{m,i}^-}$ and on the random variables used to construct the $r_j$.

We now bound the probability of $\mathcal{B}_{0}^{m,i,-}$. By the definition of the $r_j$, when $n$ is large enough, 
$\mathds{P}(\mathcal{B}_{0,1}^{m,i,-}) \leq 3 \varepsilon n^{(\alpha+3)/4} \bar C e^{-\bar c \theta N/4}$. 
Furthermore, when $n$ is large enough, if $\mathcal{B}_{0,2}^{m,i,-}$ 
is satisfied by some $i-n^{(\alpha-1)/4}\lfloor \varepsilon n\rfloor \leq j \leq i-1$ 
(the case $i \leq j \leq i+n^{(\alpha-1)/4}\lfloor \varepsilon n\rfloor$ is similar), and if we consider the largest such $j$, 
then $|\ell_{\mathbf{T}_{m,i}^-,j}^+-m| \leq 4 \varepsilon n^{(\alpha+1)/2}$, so there exists an integer 
$m' \in [m-4 \varepsilon n^{(\alpha+1)/2},m+4 \varepsilon n^{(\alpha+1)/2}]$ such that $|\eta_j(m')| \geq n^{(\alpha-1)/4}$. 
This implies 
\[
 \mathds{P}(\mathcal{B}_{0,2}^{m,i,-}) \leq 
 \sum_{|j-i| \leq n^{(\alpha-1)/4}\lfloor \varepsilon n\rfloor, |m'-m| \leq 4 \varepsilon n^{(\alpha+1)/2}}
 \mathds{P}(|\eta_j(m')| \geq n^{(\alpha-1)/4}) \leq 32\varepsilon^2 n^{(3\alpha+5)/4} \frac{2\bar C}{1-e^{-\bar c}}e^{-\bar c n^{(\alpha-1)/4}},
\]
the latter inequality coming from Lemma \ref{lem_conv_eta}. This ends the proof.
\end{proof}

Proposition \ref{prop_law_zeta} gives us a good control on the $\Delta_{m,i}$ when $m$ is some $\mathbf{T}_{m',i'}^\iota$ with $m' \geq N\theta/2$. However, we will need to understand the $\Delta_{m,i}$ when $m$ is $T_k$, $k\in\mathds{N}$. In order to do that, we establish the following proposition, which states that outside of an event of very small probability, each $T_k$ will be one of the $\mathbf{T}_{m',i'}^\iota$ for some random $m' \geq N\theta/2, i'\in\mathds{Z},\iota\in\{+,-\}$. 

\begin{proposition}\label{prop_bound_m}
 We can define an event $\mathcal{B}$ such that  for any $k \in \mathds{N}^*$, if $\mathcal{B}^c$ occurs and $n$ is large enough, $T_k =\mathbf{T}_{m,i}^+$ or $\mathbf{T}_{m,i}^-$ for some integers $\lfloor N \theta \rfloor-2n^{(\alpha+4)/5} \leq m \leq \lfloor N\theta \rfloor+2n^{(\alpha+4)/5}$ and $\lfloor Nx \rfloor- n^{(\alpha+4)/5} \leq i \leq \lfloor Nx \rfloor+ n^{(\alpha+4)/5}$. In addition, there exists a constant $c'=c'(w)>0$ such that $\mathds{P}(\mathcal{B}) \leq e^{-c' n^{((\alpha-1)/4) \wedge (1/10)}}$ when $n$ is large enough.
\end{proposition}

\begin{proof}
By the definition of the $T_k$, if $n$ is large enough, there exist $m\in\mathds{N}$, $\lfloor Nx \rfloor- n^{(\alpha+4)/5} \leq i \leq \lfloor Nx \rfloor+ n^{(\alpha+4)/5}$ and $\iota\in\{+,-\}$ so that $T_k=\mathbf{T}_{m,i}^\iota$, hence we only have to obtain the property on $m$. Roughly, the idea of the proof is that at $T_0=\mathbf{T}_{\lfloor N\theta\rfloor,\lfloor Nx \rfloor}^\pm$, Proposition \ref{prop_law_zeta} allows us to control the $\Delta_{T_0,j}$, which are tightly linked to the $\ell_{T_0,j}^\pm$, which allows to show that the $\ell_{T_0,j}^\pm$ cannot be too small, thus since $T_k \geq T_0$, the $\ell_{T_k,j}^\pm$ cannot be too small which yields a lower bound on $m$. Moreover, this control on the $\ell_{T_0,j}^\pm$ also implies that they cannot be too large, therefore $\ell_{T_0,\lfloor (N+n^{(\alpha+9)/10})x \rfloor}^\iota \leq \lfloor (N+n^{(\alpha+9)/10})\theta\rfloor$, thus $T_0 \leq \mathbf{T}_{\lfloor (N+n^{(\alpha+9)/10})\theta\rfloor,\lfloor (N+n^{(\alpha+9)/10})x \rfloor}^\iota$. Since the random walk cannot reach $\lfloor (N+n^{(\alpha+9)/10})x \rfloor$ from $\lfloor Nx\rfloor$ between times $T_0$ and $T_k$, this implies $T_k \leq \mathbf{T}_{\lfloor (N+n^{(\alpha+9)/10})\theta\rfloor,\lfloor (N+n^{(\alpha+9)/10})x \rfloor}^\iota$. We can exert the same control on the $\ell^\pm$ at time $\mathbf{T}_{\lfloor (N+n^{(\alpha+9)/10})\theta\rfloor,\lfloor (N+n^{(\alpha+9)/10})x \rfloor}^\iota$ as at time $T_0$, which allows us to prove that $m$ is not too large. 

 We now construct the event $\mathcal{B}$, which will roughly mean ``the $\ell^\pm$ don't behave well''. We suppose without loss of generality that we work with $T_0=\mathbf{T}_{\lfloor N\theta\rfloor,\lfloor Nx \rfloor}^+$ and $x > 0$. For $\tilde\iota \in \{+,-\}$, we define 
 \[
 \mathcal{B}_{\tilde\iota} = \{\exists\, i \in \mathds{Z}\text{ with }|i-\lfloor Nx \rfloor| \leq 2 n^{(\alpha+4)/5}\text{ and }|\ell_{T_0,i}^{\tilde\iota}-\ell_{T_0,y}^{\tilde\iota}+(i-\lfloor Nx \rfloor)/2| > n^{(\alpha+5)/10}\}
 \]
  with $y=\lfloor Nx \rfloor$ if $\tilde\iota=+$ and $y=\lfloor Nx \rfloor+1$ if $\tilde\iota=-$. To shorten the notation, we will write in this proof 
  \begin{equation}\label{eq_bound_m}
  T^{\tilde\iota}=\mathbf{T}_{\lfloor (N+n^{(\alpha+9)/10})\theta\rfloor,\lfloor (N+n^{(\alpha+9)/10})x \rfloor}^{\tilde\iota}.
  \end{equation}
  We define 
  \[
  \mathcal{B}_{\tilde\iota}' = \{\exists\, i \in \mathds{Z}\text{ with }|i-\lfloor (N+n^{(\alpha+9)/10})x \rfloor| \leq 2 n^{(\alpha+4)/5} \qquad\qquad\qquad\qquad\qquad\qquad\qquad\qquad
  \]
  \[
  \qquad\qquad\qquad\text{ and }|\ell_{T^{\tilde\iota},i}^{\tilde\iota}-\ell_{T^{\tilde\iota},\lfloor (N+n^{(\alpha+9)/10})x \rfloor}^{\tilde\iota}+(i-\lfloor (N+n^{(\alpha+9)/10})x \rfloor)/2| > n^{(\alpha+5)/10}\}.
  \]
  If $x<0$, we would replace $(i-\lfloor Nx \rfloor)/2$ by $-(i-\lfloor Nx \rfloor)/2$ in $\mathcal{B}_{\tilde\iota}$, and similarly in $\mathcal{B}_{\tilde\iota}'$. If we had $x=0$, we would replace $\lfloor (N+n^{(\alpha+9)/10})x \rfloor$ by $\lfloor n^{(\alpha+9)/10} \rfloor$, in $\mathcal{B}_{\tilde\iota}$ we would replace $(i-\lfloor Nx \rfloor)/2$ by $|i/2|$, and in $\mathcal{B}_{\tilde\iota}'$ we would replace $(i-\lfloor (N+n^{(\alpha+9)/10})x \rfloor)/2$ by $(|i|-\lfloor n^{(\alpha+9)/10}\rfloor)/2$. Finally, we define $\mathcal{B} = \mathcal{B}_{+} \cup \mathcal{B}_{-} \cup \mathcal{B}_{+}' \cup \mathcal{B}_{-}'$. 
 
 We now prove that if $\mathcal{B}^c$ occurs, $T_k$ has the desired property. We notice that since $T_k=\mathbf{T}_{m,i}^\iota$, we have $m=\ell_{T_k,i}^\iota$. We first prove the lower bound on $m$. We have $m=\ell_{T_k,i}^\iota \geq \ell_{T_0,i}^\iota$. Moreover, $|i-\lfloor Nx \rfloor| \leq 2 n^{(\alpha+4)/5}$ and $\mathcal{B}_{\iota}^c$ occurs, hence $|\ell_{T_0,i}^{\iota}-\ell_{T_0,y}^{\iota}+(i-\lfloor Nx \rfloor)/2| \leq n^{(\alpha+5)/10}$ with $y=\lfloor Nx \rfloor$ if $\iota=+$ and $y=\lfloor Nx \rfloor+1$ if $\iota=-$. Since $T_0=\mathbf{T}_{\lfloor N\theta\rfloor,\lfloor Nx \rfloor}^+$, if $\iota=+$ we have $\ell_{T_0,\lfloor Nx \rfloor}^{+}=\lfloor N\theta\rfloor$, hence $\ell_{T_0,i}^{\iota} \geq \lfloor N\theta\rfloor-n^{(\alpha+5)/10}-n^{(\alpha+4)/5}$. If $\iota=-$, we have $|\ell_{T_0,\lfloor Nx \rfloor+1}^{-}-\ell_{T_0,\lfloor Nx \rfloor}^{+}| \leq 1$ and $\ell_{T_0,\lfloor Nx \rfloor}^{+}=\lfloor N\theta\rfloor$, thus $\ell_{T_0,i}^{\iota} \geq \lfloor N\theta\rfloor-n^{(\alpha+5)/10}-1-n^{(\alpha+4)/5}$. In both cases we get $\ell_{T_0,i}^{\iota} \geq \lfloor N\theta\rfloor-2n^{(\alpha+4)/5}$ when $n$ is large enough, hence $m \geq \lfloor N\theta\rfloor-2n^{(\alpha+4)/5}$.
 
 We now prove the upper bound on $m$. In order to do that, we notice that $|\lfloor (N+n^{(\alpha+9)/10})x \rfloor-\lfloor Nx \rfloor| \leq 2n^{(\alpha+4)/5}$ when $n$ is large enough, thus since $\mathcal{B}_{\iota}^c$ occurs, 
 \[
 \ell_{T_0,\lfloor (N+n^{(\alpha+9)/10})x \rfloor}^{\iota}-\ell_{T_0,y}^{\iota}+(\lfloor (N+n^{(\alpha+9)/10})x \rfloor-\lfloor Nx \rfloor)/2| \leq n^{(\alpha+5)/10}
 \]
 with $|\ell_{T_0,y}^{\iota}-\lfloor N\theta\rfloor| \leq 1$. This implies 
 \[
 \ell_{T_0,\lfloor (N+n^{(\alpha+9)/10})x \rfloor}^{\iota}-\lfloor N\theta\rfloor \leq n^{(\alpha+5)/10}-(\lfloor (N+n^{(\alpha+9)/10})x \rfloor-\lfloor Nx \rfloor)/2+1\leq n^{(\alpha+5)/10}
 \]
 when $n$ is large enough, hence $\ell_{T_0,\lfloor (N+n^{(\alpha+9)/10})x \rfloor}^{\iota} < \lfloor(N+n^{(\alpha+9)/10})\theta\rfloor$, thus $T_0 \leq \mathbf{T}_{\lfloor (N+n^{(\alpha+9)/10})\theta\rfloor,\lfloor (N+n^{(\alpha+9)/10})x \rfloor}^\iota$. Furthermore, between times $T_0$ and $T_k$ the random walk stays at distance at most $k\varepsilon n+1$ of $\lfloor Nx \rfloor$, hence when $n$ is large enough it does not reach $\lfloor (N+n^{(\alpha+9)/10})x \rfloor$, therefore $T_k \leq \mathbf{T}_{\lfloor (N+n^{(\alpha+9)/10})\theta\rfloor,\lfloor (N+n^{(\alpha+9)/10})x \rfloor}^\iota=T^\iota$ (see \eqref{eq_bound_m}), which yields $m=\ell_{T_k,i}^\iota \leq \ell_{T^\iota,i}^\iota$. Now, $|i- \lfloor Nx \rfloor| \leq n^{(\alpha+4)/5}$, thus $|i-\lfloor (N+n^{(\alpha+9)/10})x \rfloor| \leq 2 n^{(\alpha+4)/5}$ when $n$ is large enough, thus since $\mathcal{B}_\iota'$ occurs, $|\ell_{T^\iota,i}^{\iota}-\ell_{T^\iota,\lfloor (N+n^{(\alpha+9)/10})x \rfloor}^{\iota}+(i-\lfloor (N+n^{(\alpha+9)/10})x \rfloor)/2| \leq n^{(\alpha+5)/10}$, so $\ell_{T^\iota,i}^{\iota} \leq \ell_{T^\iota,\lfloor (N+n^{(\alpha+9)/10})x \rfloor}^{\iota}+n^{(\alpha+5)/10}+n^{(\alpha+4)/5}$. In addition, by the definition of $T^\iota$ we have $\ell_{T^\iota,\lfloor (N+n^{(\alpha+9)/10})x \rfloor}^{\iota}= \lfloor (N+n^{(\alpha+9)/10})\theta\rfloor$, thus $m \leq\ell_{T^\iota,i}^{\iota} \leq \lfloor (N+n^{(\alpha+9)/10})\theta\rfloor+n^{(\alpha+5)/10}+n^{(\alpha+4)/5} \leq \lfloor N \theta \rfloor + 2n^{(\alpha+4)/5}$. 
 
 We now prove the bound on $\mathds{P}(\mathcal{B})$ with the help of Proposition \ref{prop_law_zeta}. It is enough to find $c'=c'(w)>0$ so that $\mathds{P}(\mathcal{B}_-) \leq e^{-2 c' n^{((\alpha-1)/4) \wedge (1/10)}}$ when $n$ is large enough, as the probabilities $\mathds{P}(\mathcal{B}_+), \mathds{P}(\mathcal{B}_+'), \mathds{P}(\mathcal{B}_-')$ can be dealt with in the same way. Moreover, by Proposition \ref{prop_law_zeta} we have $\mathds{P}(\mathcal{B}_0^{\lfloor N\theta\rfloor,\lfloor Nx \rfloor,+}) \leq C_0e^{-c_0n^{(\alpha-1)/4}}$ when $n$ is large enough, so it is enough to prove $\mathds{P}(\mathcal{B}_- \cap (\mathcal{B}_0^{\lfloor N\theta\rfloor,\lfloor Nx \rfloor,+})^c) \leq e^{-\breve c_0 n^{1/10}}$ for some constant $\breve c_0=\breve c_0(w) > 0$ when $n$ is large enough. 
 
 In order to do that, we set $i \in \mathds{Z}$ so that $|i-\lfloor Nx \rfloor| \leq 2 n^{(\alpha+4)/5}$. We will write $\ell_{T_0,i}^{-}-\ell_{T_0,\lfloor Nx \rfloor}^{-}+(i-\lfloor Nx \rfloor)/2$ as a sum of i.i.d. random variables as follows. We suppose $i < \lfloor Nx \rfloor$, then 
 \[
 \ell_{T_0,i}^{-}-\ell_{T_0,\lfloor Nx \rfloor}^{-}+(i-\lfloor Nx \rfloor)/2=\sum_{j=\lfloor Nx \rfloor}^{i+1}(\ell_{T_0,j-1}^{-}-\ell_{T_0,j}^{-}-1/2).
 \]
 Now, we recall that $x > 0$ and $n^\alpha \leq N$, so when $n$ is large enough $i \geq 0$, which yields that for any $j \in \{i+1,...,\lfloor Nx \rfloor\}$, $\ell_{T_0,j}^{-}=\ell_{T_0,j-1}^{+}-1$, hence 
 \[
 \ell_{T_0,i}^{-}-\ell_{T_0,\lfloor Nx \rfloor}^{-}+(i-\lfloor Nx \rfloor)/2=\sum_{j=\lfloor Nx \rfloor}^{i+1}(\ell_{T_0,j-1}^{-}-\ell_{T_0,j-1}^{+}+1/2)=\sum_{j=\lfloor Nx \rfloor}^{i+1}(\Delta_{T_0,j-1}+1/2).
 \]
 Moreover, $\ell_{T_0,\lfloor Nx \rfloor+1}^{-}=\ell_{T_0,\lfloor Nx \rfloor}^{+}$, thus $\ell_{T_0,\lfloor Nx \rfloor}^{-}-\ell_{T_0,\lfloor Nx \rfloor+1}^{-}=\ell_{T_0,\lfloor Nx \rfloor}^{-}-\ell_{T_0,\lfloor Nx \rfloor}^{+}=\Delta_{T_0,\lfloor Nx \rfloor}$. We deduce 
 \[
 \ell_{T_0,i}^{-}-\ell_{T_0,\lfloor Nx \rfloor+1}^{-}+(i-\lfloor Nx \rfloor)/2=\Delta_{T_0,\lfloor Nx \rfloor}+\sum_{j=\lfloor Nx \rfloor}^{i+1}(\Delta_{T_0,j-1}+1/2).
 \]
 Now, if $(\mathcal{B}_0^{\lfloor N\theta\rfloor,\lfloor Nx \rfloor,+})^c$ occurs and $n$ is large enough, $\Delta_{T_0,j}=\bar\Delta_{T_0,j}$ for any $j \in \{i,...,\lfloor Nx \rfloor\}$, thus 
 \[
 \ell_{T_0,i}^{-}-\ell_{T_0,\lfloor Nx \rfloor+1}^{-}+(i-\lfloor Nx \rfloor)/2=\bar\Delta_{T_0,\lfloor Nx \rfloor}+\sum_{j=\lfloor Nx \rfloor}^{i+1}(\bar\Delta_{T_0,j-1}+1/2).
 \]
 
 Therefore it is enough to show $\mathds{P}(|\bar\Delta_{T_0,\lfloor Nx \rfloor}+\sum_{j=\lfloor Nx \rfloor}^{i+1}(\bar\Delta_{T_0,j-1}+1/2)| > n^{(\alpha+5)/10}) \leq e^{-2\breve c_0 n^{1/10}}$ for some constant $\breve c_0=\breve c_0(w) > 0$ when $n$ is large enough. Now, by Proposition \ref{prop_law_zeta} $\bar\Delta_{T_0,\lfloor Nx \rfloor}$ has law $\rho_-$ which has exponential tails, so there exists a constant $\bar c_0 =\bar c_0(w) > 0$ so that $\mathds{P}(|\bar\Delta_{T_0,\lfloor Nx \rfloor}|>n^{(\alpha+5)/10}/2) \leq e^{-\bar c_0 n^{(\alpha+5)/10}}$ when $n$ is large enough. Therefore it suffices to prove $\mathds{P}(|\sum_{j=\lfloor Nx \rfloor}^{i+1}(\bar\Delta_{T_0,j-1}+1/2)| > n^{(\alpha+5)/10}/2) \leq e^{-n^{1/10}/3}$ when $n$ is large enough. Furthermore, by Proposition \ref{prop_law_zeta} the $\bar\Delta_{T_0,j-1}$ are i.i.d. with law $\rho_-$, thus the $\bar\Delta_{T_0,j-1}+1/2$ are i.i.d. with law $\rho_0$. 
 
 Consequently, we only have to prove that $\mathds{P}(|\sum_{j=1}^{\lfloor Nx \rfloor-i}\zeta_j| > n^{(\alpha+5)/10}/2) \leq e^{-n^{1/10}/3}$ when $n$ is large enough, where $(\zeta_j)_{j\in\mathds{N}}$ are i.i.d. with law $\rho_0$. Moreover, $\rho_0$ has exponential tails, so $\mathds{E}(e^{s\zeta_1})<+\infty$ when $s > 0$ is small enough. Since $\rho_0$ is symmetric, when $n$ is large enough, 
 \begin{equation}\label{eq_B}
 \begin{split}
 \mathds{P}\left(\left|\sum_{j=1}^{\lfloor Nx \rfloor-i}\zeta_j\right| > \frac{n^{(\alpha+5)/10}}{2}\right) = 2 \mathds{P}\left(\sum_{j=1}^{\lfloor Nx \rfloor-i}\zeta_j > \frac{n^{(\alpha+5)/10}}{2}\right) \qquad\qquad\qquad\qquad \\
 = 2\mathds{P}\left(\exp\left( n^{-(\alpha+4)/10} \sum_{j=1}^{\lfloor Nx \rfloor-i}\zeta_j\right) > \exp(n^{1/10}/2)\right)
 =2e^{-n^{1/10}/2}\mathds{E}\left(\exp\left(n^{-(\alpha+4)/10}\zeta_1\right)\right)^{\lfloor Nx \rfloor-i}.
 \end{split}
 \end{equation}
 We now study $\mathds{E}(\exp(n^{-(\alpha+4)/10}\zeta_1))$. We have 
 \[
 \exp(n^{-(\alpha+4)/10}\zeta_1)= 1+n^{-(\alpha+4)/10}\zeta_1+\frac{1}{2}n^{-(\alpha+4)/5}\zeta_1^2e^{\zeta_1'}
 \]
 with $|\zeta_1'| \leq |n^{-(\alpha+4)/10}\zeta_1|$, hence 
 \[
 \mathds{E}(\exp(n^{-(\alpha+4)/10}\zeta_1))= 1+\mathds{E}\left(\frac{1}{2}n^{-(\alpha+4)/5}\zeta_1^2e^{\zeta_1'}\right) \leq 1+\frac{1}{2}n^{-(\alpha+4)/5}\mathds{E}(\zeta_1^2e^{|n^{-(\alpha+4)/10}\zeta_1|}).
 \]
 Now, since $\rho_0$ has exponential tails, there exists $\tilde c_0=\tilde c_0(w)>0$ and $\tilde C_0=\tilde C_0(w)<+\infty$ so that $\mathds{E}(\zeta_1^2e^{\tilde c_0|\zeta_1|}) \leq \tilde C_0$. When $n$ is large enough, $\mathds{E}(\zeta_1^2e^{|n^{-(\alpha+4)/10}\zeta_1|}) \leq \mathds{E}(\zeta_1^2e^{\tilde c_0|\zeta_1|}) \leq \tilde C_0$, thus $\mathds{E}(\exp(n^{-(\alpha+4)/10}\zeta_1)) \leq 1+n^{-(\alpha+4)/5}\tilde C_0/2 \leq e^{n^{-(\alpha+4)/5}\tilde C_0/2}$. From that and \eqref{eq_B} we obtain 
 \[
 \mathds{P}\left(\left|\sum_{j=1}^{\lfloor Nx \rfloor-i}\zeta_j\right| > n^{(\alpha+5)/10}/2\right) = 2e^{-n^{1/10}/2}e^{(\lfloor Nx \rfloor-i)n^{-(\alpha+4)/5}\tilde C_0/2} 
 \]
 \[
 \leq 2e^{-n^{1/10}/2}e^{2 n^{(\alpha+4)/5}n^{-(\alpha+4)/5}\tilde C_0/2}=2e^{-n^{1/10}/2}e^{\tilde C_0} \leq e^{-n^{1/10}/3}
 \]
 when $n$ is large enough, which ends the proof.
\end{proof}

\begin{remark}
 It is possible to use the main result of \cite{Toth_et_al2008} to craft an event $\mathcal{B}$ for which the proof is much simpler, but such that $\mathcal{B}^c$ only ensures $N\theta/2 \leq m \leq 5N\theta/2$. This is not enough for our purposes, since we will later use union bounds on events indexed by $m$, the probability of each event of order $e^{-c(\ln n)^2}$. 
\end{remark}

We will need some other auxiliary variables. We recall that the $L_{i}^{m,\pm}, \zeta_i^{m,\pm,B}, \zeta_i^{m,\pm,E}$ are defined in Definition \ref{def_local_times}, the $\mathbf{T}_{m,i}^\iota$ in \eqref{eq_def_T_mi} and the $\bar \Delta_{\mathbf{T}_{m,i}^\iota,j}$ in Proposition \ref{prop_law_zeta}. We begin by constructing, for $m \geq N \theta/2$, $i \in \mathds{Z}$, $\iota \in \{+,-\}$, an equivalent of the processes 
$(L^{\mathbf{T}_{m,i}^\iota,-}_j)_{j \in \mathds{Z}}$ 
and $(L^{\mathbf{T}_{m,i}^\iota,+}_j)_{j \in \mathds{Z}}$ 
``when the environment at time $\mathbf{T}_{m,i}^\iota$ is 
$(\bar \Delta_{\mathbf{T}_{m,i}^\iota,j})_{j \in \mathds{Z}}$ instead of $(\Delta_{\mathbf{T}_{m,i}^\iota,j})_{j \in \mathds{Z}}$''. We denote $\bar m = \mathbf{T}_{m,i}^\iota$ and $\bar i = X_{\mathbf{T}_{m,i}^\iota}$ for short. 

We define $(\bar L^{\bar m,-}_j)_{j \in \mathds{Z}}$ as follows. 
For $j \leq \bar i-\lfloor \varepsilon n\rfloor+1$, 
$\bar L^{\bar m,-}_j=0$. By Observation \ref{obs_rec_temps_local}, for any $j \geq \bar i-\lfloor \varepsilon n\rfloor+1$ 
we have $L_{j+1}^{\bar m,-}=L_j^{\bar m,-}+\zeta^{\bar m,-,E}_j-\zeta^{\bar m,-,B}_j$, and by (\ref{eq_zeta}), 
if $\bar i-\lfloor \varepsilon n \rfloor+1 \leq j \leq \bar i$, $\zeta^{\bar m,-,E}_j=\eta(L_j^{\bar m,-}+1)+1/2$ 
with $\eta(0)=-\Delta_{\bar m,j}$, while if $j > \bar i$, 
$\zeta^{\bar m,-,E}_j=\eta(L_j^{\bar m,-})+1/2$ 
with $\eta(0)=-\Delta_{\bar m,j}$. We can define $\bar \eta$ so that $\bar \eta(0)=-\bar \Delta_{\bar m,j}$,
the transitions of $\bar \eta$ are independent from $(\bar \Delta_{\bar m,j'})_{j'\in\mathds{Z}}$,
and $\bar \eta=\eta$ if $(\mathcal{B}_{0}^{m,i,\iota})^c$ is satisfied, $n$ large enough and $|j-\bar i| \leq n^{(\alpha-1)/4}\lfloor\varepsilon n\rfloor$. 
We define $\bar L^{\bar m,-}_j$ by induction 
by setting $\bar L^{\bar m,-}_{j+1}=\bar L^{\bar m,-}_{j}+\bar\eta(\bar L_j^{\bar m,-}+1)+\bar \Delta_{\bar m,j}+1$ 
if $\bar i-\lfloor \varepsilon n\rfloor+1 \leq j \leq \bar i$ and 
$\bar L^{\bar m,-}_{j+1}=\bar L^{\bar m,-}_{j}+\bar\eta(\bar L_j^{\bar m,-})+\bar \Delta_{\bar m,j}$ if $j > \bar i$.

We define $(\bar L^{\bar m,+}_j)_{j \in \mathds{Z}}$ in the same way. $\bar L^{\bar m,+}_j=0$ for 
$j > \bar i+\lfloor \varepsilon n\rfloor$, $\bar L^{\bar m,+}_{\bar i+\lfloor \varepsilon n\rfloor}=1$. 
For any $j < \bar i+\lfloor \varepsilon n\rfloor$, 
$\zeta_j^{\bar m,+,E}=\eta(L_{j+1}^{\bar m,+})+1/2$ with $\eta(0)=\Delta_{\bar m,j}$, 
and we may define $\bar \eta$ so that $\bar \eta(0)=\bar \Delta_{\bar m,j}$, the transitions of $\bar \eta$ 
are independent from $(\bar \Delta_{\bar m,j})_{j\in\mathds{Z}}$, and $\bar \eta=\eta$ if 
$(\mathcal{B}_{0}^{m,i,\iota})^c$ is satisfied, $n$ large enough and $|j-\bar i| \leq n^{(\alpha-1)/4}\lfloor\varepsilon n\rfloor$.  
We then define 
$\bar L_j^{\bar m,+}=\bar L_{j+1}^{\bar m,+}+\bar\eta(\bar L_{j+1}^{\bar m,+})-\bar \Delta_{\bar m,j}+1$ 
if $\bar i < j < \bar i+\lfloor \varepsilon n\rfloor$ and 
$\bar L_j^{\bar m,+}=\bar L_{j+1}^{\bar m,+}+\bar\eta(\bar L_{j+1}^{\bar m,+})-\bar \Delta_{\bar m,j}$ if $j \leq \bar i$. 
When $n$ is large enough, if $(\mathcal{B}_{0}^{m,i,\iota})^c$ is satisfied, $\bar L_j^{\bar m,\pm} = L_j^{\bar m,\pm}$ for any $j \in \mathds{Z}$ with $|j-\bar i| \leq n^{(\alpha-1)/4}\lfloor\varepsilon n\rfloor$. 

Now, for any $m \in \mathds{N}$, we are going to construct random variables $\zeta_j^{m,\pm,I}$, $j\in\mathds{Z}$, independent from $\mathcal{F}_m$ such that 
``when $L$ (more precisely, $L_j^{m,-}$ or $L_{j+1}^{m,+}$) is not too small, $\zeta_j^{m,\pm,I}=\zeta_j^{m,\pm,E}$, and the $\zeta_j^{m,\pm,I}$, $j\in\mathds{Z}$ are i.i.d. with law $\rho_0$'', where $\rho_0$ is defined in \eqref{eq_def_rho0}. 
For $m \geq N \theta/2$, $i \in \mathds{Z}$, $\iota \in \{+,-\}$, we will also define 
random variables $\bar \zeta_j^{m,\pm,I}$, $j\in\mathds{Z}$, independent from $(\bar \Delta_{\mathbf{T}_{m,i}^\iota,j})_{j\in\mathds{Z}}$ and 
equal to the $\zeta_j^{m,\pm,I}$ when $(\mathcal{B}_{0}^{m,i,\iota})^c$ is satisfied. The superscript $I$ stands for ``independent''. 

We begin by constructing the $(\zeta_j^{m,-,I})_{j\in\mathds{Z}}$ where $m \in \mathds{N}$. If $m=\mathbf{T}_{m',i}^\iota$ 
 for some $m' \geq N \theta/2$, $i \in \mathds{Z}$, $\iota \in \{+,-\}$, we construct the $(\bar \zeta_j^{m,-,I})_{j\in\mathds{Z}}$ 
 at the same time. Let $j \in \mathds{Z}$. 
 If $j \leq X_m -\lfloor \varepsilon n \rfloor$, $\zeta_j^{m,-,I}=\bar \zeta_j^{m,-,I}$ will be a random variable of law $\rho_0$ 
 independent from everything else. If $X_m -\lfloor \varepsilon n \rfloor < j \leq X_m$, by (\ref{eq_zeta}) 
 we know that $\zeta_j^{m,-,E}=\eta(L_j^{m,-}+1)+1/2$ with $\eta(0) = -\Delta_{m,j}$, and (if $m=\mathbf{T}_{m',i}^\iota$) we remember the definitions of $\bar \eta$ 
 and $\bar L_j^{m,-}$ given above. We denote $T = \inf\{\ell \geq 0 \,|\, \eta(\ell)=0\}$ and 
 $\bar T = \inf\{\ell \geq 0 \,|\, \bar \eta(\ell)=0\}$. Let $U$ be a random variable uniform on $[0,1]$ independent 
 from everything else. We can apply the construction of Lemma \ref{lem_const_coupling} with 
 $\eta(T+\lfloor (\ln n)^2/2 \rfloor)$ and $U$ to construct a random variable $\Delta$ of law $\rho_-$ and so that 
 $\mathds{P}(\eta(T+\lfloor (\ln n)^2/2 \rfloor) \neq \Delta)$ is minimal, and with $\bar\eta(\bar T+\lfloor (\ln n)^2/2 \rfloor)$ 
 and $U$ to construct a random variable $\bar\Delta$ of law $\rho_-$ and so that 
 $\mathds{P}(\bar\eta(\bar T+\lfloor (\ln n)^2/2 \rfloor) \neq \bar\Delta)$ is minimal. If $(\mathcal{B}_{0}^{m,i,\iota})^c$ is satisfied and $n$ is large enough, 
 we then have $\bar \Delta=\Delta$. 
Then, if $L_j^{m,-}+1-T\geq (\ln n)^2/2$, we define $\zeta_j^{m,-,I}=\eta'(L_j^{m,-}+1-T-\lfloor(\ln n)^2/2\rfloor)+1/2$, 
where $\eta'(0)=\Delta$ and $\eta'(.)=\eta(T+\lfloor (\ln n)^2/2 \rfloor+ \cdot)$ when 
$\eta(T+\lfloor (\ln n)^2/2 \rfloor) = \Delta$. If $L_j^{m,-}+1-T < (\ln n)^2/2$, we set $\zeta_j^{m,-,I} = \hat \zeta$, 
 where $\hat \zeta$ is a random variable of law $\rho_0$ independent of everything else. 
 Similarly, if $\bar L_j^{m,-}+1-\bar T\geq (\ln n)^2/2$, we define 
 $\bar\zeta_j^{m,-,I}=\bar\eta'(\bar L_j^{m,-}+1-\bar T-\lfloor(\ln n)^2/2\rfloor)+1/2$, 
where $\bar\eta'(0)=\bar\Delta$, $\bar\eta'(.)=\bar\eta(\bar T+\lfloor (\ln n)^2/2 \rfloor+ \cdot)$ when 
$\bar \eta(\bar T+\lfloor (\ln n)^2/2 \rfloor) = \bar\Delta$, 
and $\bar\eta'=\eta'$ when $(\mathcal{B}_{0}^{m,i,\iota})^c$ is satisfied and $n$ large enough. 
If $\bar L_j^{m,-}+1-\bar T < (\ln n)^2/2$, we set
$\bar\zeta_j^{m,-,I} = \hat \zeta$. If $j > X_m$, we use the same construction with $L_j^{m,-}$ replacing $L_j^{m,-}+1$. 

 We use a similar construction for the $\zeta_j^{m,+,I}$, $j\in\mathds{Z}$. For any $j \in \mathds{Z}$, 
 $\zeta_j^{m,+,E} = \eta(L_{j+1}^{m,+})+1/2$ with $\eta(0)=\Delta_{m,j}$. 
 We take similar $T$ and $\Delta$, as well as $\bar T$ and $\bar \Delta$ 
 when $i \iota 1-\lfloor \varepsilon n\rfloor \leq j \leq i \iota 1+\lfloor \varepsilon n\rfloor$. 
 If $L_{j+1}^{m,+}-T\geq (\ln n)^2/2$, we define $\zeta_j^{m,+,I}=\eta'(L_{j+1}^{m,+}-T-\lfloor(\ln n)^2/2\rfloor)+1/2$, 
where $\eta'(0)=\Delta$ and $\eta'(.)=\eta(T+\lfloor (\ln n)^2/2 \rfloor+ \cdot)$ when 
$\eta(T+\lfloor (\ln n)^2/2 \rfloor) = \Delta$. If $L_{j+1}^{m,+}-T < (\ln n)^2/2$, we set $\zeta_j^{m,+,I} = \hat \zeta$, 
 where $\hat \zeta$ is a random variable of law $\rho_0$ independent of everything else. 
 In the same way, if $\bar L_{j+1}^{m,+}-\bar T\geq (\ln n)^2/2$, we define 
 $\bar\zeta_j^{m,+,I}=\bar\eta'(\bar L_{j+1}^{m,+}-\bar T-\lfloor(\ln n)^2/2\rfloor)+1/2$, 
where $\bar\eta'(0)=\bar\Delta$, $\bar\eta'(.)=\bar\eta(\bar T+\lfloor (\ln n)^2/2 \rfloor+ \cdot)$ when 
$\bar\eta(\bar T+\lfloor (\ln n)^2/2 \rfloor) = \bar\Delta$, and $\bar\eta'=\eta'$ when $\bar\Delta_{m,j}=\Delta_{m,j}$. 
If $\bar L_{j+1}^{m,+}-\bar T < (\ln n)^2/2$, we set $\bar \zeta_j^{m,+,I} = \hat \zeta$. 

Some properties of the random variables defined thus are stated in the following proposition (the definition of $\mathcal{F}_m$ was given in \eqref{eq_def_F}).

\begin{proposition}\label{prop_def_zetaI}
 For any $m \in \mathds{N}$, $\iota\in\{+,-\}$, 
 $(\zeta_i^{m,\iota,I})_{i \in \mathds{Z}}$ are i.i.d. with law $\rho_0$, independent from 
$\mathcal{F}_m$, and depend only on $\mathcal{F}_{\beta_m^\iota}$ and on a set of random variables independent from 
everything else. Moreover, for any $i \in \mathds{Z}$, $\zeta_i^{m,-,I}$ 
is independent from $\zeta_j^{m,-,B},\zeta_j^{m,-,E}$ for $j < i$, and $\zeta_i^{m,+,I}$ 
is independent from $\zeta_j^{m,+,B},\zeta_j^{m,+,E}$ for $j > i$. 
Furthermore, for any $m \geq N \theta/2$, $i \in \mathds{Z}$, $\iota \in \{+,-\}$, for $\iota'\in\{+,-\}$, 
 $(\bar\zeta_j^{\mathbf{T}_{m,i}^\iota,\iota',I})_{j}$ are i.i.d. with law $\rho_0$, independent from 
$(\bar \Delta_{\mathbf{T}_{m,i}^\iota,j})_{j \in \mathds{Z}}$, for any 
 $i\iota1-\lfloor \varepsilon n\rfloor \leq j \leq i\iota1+\lfloor \varepsilon n\rfloor$, 
 $(\bar\zeta_{j'}^{\mathbf{T}_{m,i}^\iota,-,I})_{j'\geq j}$ is independent from $(\bar L_{j'}^{\mathbf{T}_{m,i}^\iota,-,I})_{j' \leq j}$ 
 and $(\bar\zeta_{j'}^{\mathbf{T}_{m,i}^\iota,+,I})_{j'\leq j}$ is independent 
 from $(\bar L_{j'}^{\mathbf{T}_{m,i}^\iota,+,I})_{j' > j}$. In addition, if $(\mathcal{B}_{0}^{m,i,\iota})^c$ is satisfied 
and $n$ is large enough, $\bar\zeta_j^{\mathbf{T}_{m,i}^\iota,-,I}=\zeta_j^{\mathbf{T}_{m,i}^\iota,-,I}$ 
and $\bar\zeta_j^{\mathbf{T}_{m,i}^\iota,+,I}=\zeta_j^{\mathbf{T}_{m,i}^\iota,+,I}$ 
for any $i\iota1-\lfloor \varepsilon n\rfloor \leq j \leq i\iota1+\lfloor \varepsilon n\rfloor$. 
\end{proposition}

\begin{proof}
We only prove the independence and distribution properties for $\zeta^{m,-,I}$, as the proof is the same for $\zeta^{m,+,I}$, $\bar\zeta^{\mathbf{T}_{m,i}^\iota,-,I}$ and $\bar\zeta^{\mathbf{T}_{m,i}^\iota,+,I}$ and the other claims are clear from the construction. If $j \leq X_m -\lfloor \varepsilon n \rfloor$, the result is clear. If $X_m -\lfloor \varepsilon n \rfloor < j \leq X_m$, we notice that 
$\eta(T+\lfloor (\ln n)^2/2 \rfloor)$ is independent from $T$, $\mathcal{F}_m$, 
 $\zeta_{j'}^{m,-,B},\zeta_j^{m,-,E},\zeta_{j'}^{m,-,I}$ for $j' < j$, so $\Delta$ also is, as well as the transitions of $\eta'$. 
 Consequently, $\eta'(L_j^{m,-}+1-T-\lfloor(\ln n)^2/2\rfloor)$ is independent from 
 $T$, $\mathcal{F}_m$, $\zeta_{j'}^{m,-,B},\zeta_{j'}^{m,-,E},\zeta_{j'}^{m,-,I}$ for $j' < j$ and has law $\rho_-$. We deduce that 
 $\zeta_j^{m,-,I}$ is independent from $\mathcal{F}_m$, $\zeta_{j'}^{m,-,B},\zeta_{j'}^{m,-,E},\zeta_{j'}^{m,-,I}$ 
 for $j' < j$ and has law $\rho_0$. If $j > X_m$, the proof is the same as for $j \leq X_m -\lfloor \varepsilon n \rfloor$. 
\end{proof}
 
\section{Bad events}\label{sec_bad_events}

In this section, we are going to prove that outside of ``bad events'' of small probability, the random variables defined in Section \ref{sec_first_variables} behave well. We remind the reader that $\varepsilon > 0$, that the $L_i^{m,\pm}$, $\zeta_i^{m,\pm,B}$, $\zeta_i^{m,\pm,E}$ are defined in Definition \ref{def_local_times}, and the $\zeta_i^{m,\pm,I}$ just before Proposition \ref{prop_def_zetaI}. For any $m \in \mathds{N}$, we define two sequences $(I^{m,-}(\ell))_{\ell \in \mathds{N}}$ and 
$(I^{m,+}(\ell))_{\ell \in \mathds{N}}$ by $I^{m,-}(0) = X_m-\lfloor \varepsilon n\rfloor$, 
$I^{m,-}(\ell+1)=\inf\{I^{m,-}(\ell) < i < X_m \,|\, L^{m,-}_{i} < (\ln n)^3\}$ for $\ell \in \mathds{N}$ and 
$I^{m,+}(0) = X_m+\lfloor \varepsilon n\rfloor$, 
$I^{m,+}(\ell+1)=\sup\{X_m < i < I^{m,+}(\ell) \,|\, L^{m,+}_{i+1} < (\ln n)^3+1\}$ for $\ell \in \mathds{N}$. 
We also denote $\ell_{\textrm{max}}^{m,-}=\max\{\ell > 0\,|\,I^{m,-}(\ell) < +\infty\}$ 
and $\ell_{\textrm{max}}^{m,+}=\max\{\ell > 0\,|\,I^{m,+}(\ell) > -\infty\}$. 
We define the following events (we stress that they are different from the events defined in Proposition \ref{prop_law_zeta} and its proof). 
\[
 \mathcal{B}_{m,1}^- = \{\exists\, i \in \{X_m-\lfloor n \varepsilon \rfloor+1,\dots,X_m-\lfloor (\ln n)^{8} \rfloor\}, 
 \forall\, j \in \{i,\dots,i+\lfloor (\ln n)^{8} \rfloor\},L_j^{m,-} < (\ln n)^3\}\},
 \]
 \[
 \mathcal{B}_{m,1}^+ = \{\exists\, i \in \{X_m + \lfloor (\ln n)^{8} \rfloor+1,\dots,X_m+\lfloor n \varepsilon \rfloor\}, 
 \forall\, j \in \{i-\lfloor (\ln n)^{8} \rfloor,\dots,i\},L_j^{m,+} < (\ln n)^3+1\}\},
 \]
 \[
 \mathcal{B}_{m,2}^- = \{|\{X_m < i \leq X_m + n^{(\alpha-1)/4}\lfloor\varepsilon n\rfloor \,|\, 0 < L_i^{m,-} < (\ln n)^3\}| \geq (\ln n)^{8}\}, 
 \]
 \[
 \mathcal{B}_{m,2}^+ = \{|\{X_m - n^{(\alpha-1)/4}\lfloor\varepsilon n\rfloor \leq i \leq X_m \,|\, 0 < L_i^{m,+} < (\ln n)^3\}| 
 \geq (\ln n)^{8}\}, 
 \]
 \[
 \mathcal{B}_{m,3}^-=\{\exists\, i\in\{X_m-\lfloor n \varepsilon \rfloor+1,\dots,X_m+\lfloor n \varepsilon \rfloor\}
 \text{ such that }L_i^{m,-} \geq (\ln n)^2\text{ and } \zeta_i^{m,-,E} \neq \zeta_i^{m,-,I}\}, 
 \]
 \[
 \mathcal{B}_{m,3}^+=\{\exists\, i\in\{X_m-\lfloor n \varepsilon \rfloor,\dots,X_m+\lfloor n \varepsilon \rfloor-1\}
 \text{ such that }L_{i+1}^{m,+} \geq (\ln n)^2\text{ and } \zeta_i^{m,+,E} \neq \zeta_i^{m,+,I}\},
\]
 \[
 \mathcal{B}_{m,4}^- = \{|\{X_m - \lfloor \varepsilon n \rfloor < i < X_m \,|\, 0 \leq L_i^{m,-} < (\ln n)^3\}| 
 > (\ln n)^{10}\sqrt{n}\}, 
 \]
 \[
 \mathcal{B}_{m,4}^+ = \{|\{X_m < i < X_m + \lfloor \varepsilon n \rfloor \,|\, 0 \leq L_{i+1}^{m,+} < (\ln n)^3+1\}| 
 > (\ln n)^{10}\sqrt{n}\}, 
 \]
\[
 \mathcal{B}_{m,5}^-=\{\exists\, i\in\{X_m-\lfloor n \varepsilon \rfloor+1,\dots,X_m+\lfloor n \varepsilon \rfloor\}
 \text{ such that }|\zeta_j^{m,-,B}| > (\ln n)^2, |\zeta_j^{m,-,E}| > (\ln n)^2\text{ or }|\zeta_j^{m,-,I}| > (\ln n)^2\},
\]
\[
 \mathcal{B}_{m,5}^+=\{\exists\, i\in\{X_m-\lfloor n \varepsilon \rfloor-1,\dots,X_m+\lfloor n \varepsilon \rfloor-1\}
 \text{ such that }|\zeta_j^{m,+,B}| > (\ln n)^2, |\zeta_j^{m,+,E}| > (\ln n)^2\text{ or }|\zeta_j^{m,+,I}| > (\ln n)^2\},
\]
\[
 \mathcal{B}_{m,6}^-=\left\{\max_{1 \leq \ell_1 \leq \ell_2 \leq \ell_{\textrm{max}}^{m,-}}
 \left|\sum_{\ell=\ell_1}^{\ell_2}\zeta_{I^{m,-}(\ell)}^{m,-,I}\right| > (\ln n)^{7}n^{1/4}
 \text{ or }\max_{1 \leq \ell_1 \leq \ell_2 \leq \ell_{\textrm{max}}^{m,-}}
 \left|\sum_{\ell=\ell_1}^{\ell_2}\zeta_{I^{m,-}(\ell)}^{m,-,B}\right| > (\ln n)^{7}n^{1/4}\right\},
\]
\[
 \mathcal{B}_{m,6}^+=\left\{\max_{1 \leq \ell_1 \leq \ell_2 \leq \ell_{\textrm{max}}^{m,+}}
 \left|\sum_{\ell=\ell_1}^{\ell_2}\zeta_{I^{m,+}(\ell)}^{m,+,I}\right| > (\ln n)^{7}n^{1/4}
 \text{ or }\max_{1 \leq \ell_1 \leq \ell_2 \leq \ell_{\textrm{max}}^{m,+}}
 \left|\sum_{\ell=\ell_1}^{\ell_2}\zeta_{I^{m,+}(\ell)}^{m,+,B}\right| > (\ln n)^{7}n^{1/4}\right\}. 
\]
Moreover, for any $r \in \{1,\dots,6\}$, 
we set $\mathcal{B}_r = \bigcup(\mathcal{B}_{\mathbf{T}_{m,i}^\iota,r}^- 
\cup \mathcal{B}_{\mathbf{T}_{m,i}^\iota,r}^+)$, where the union is on 
$\lfloor N \theta\rfloor-2n^{(\alpha+4)/5} \leq m \leq \lfloor N \theta\rfloor+2n^{(\alpha+4)/5}$, $\lfloor Nx \rfloor- n^{(\alpha+4)/5} \leq i 
\leq \lfloor Nx \rfloor+n^{(\alpha+4)/5}$, $\iota \in \{+,-\}$. 
Finally, we set $\mathcal{B}_0 = \bigcup \mathcal{B}_{0}^{m,i,\iota}$ (see Proposition \ref{prop_law_zeta} for the definition of the $\mathcal{B}_{0}^{m,i,\iota}$), where the union is on the same indexes as before. The goal of this section is to prove that $\mathds{P}(\bigcup_{i=0}^{6}B_i)$ is small (Proposition \ref{prop_bound_bad_evts}). To achieve it, we will deal with each ``bad event'' separately.

\begin{proposition}\label{prop_TCL}
  There exists a constant $c_1 = c_1(w) > 0$ such that 
 when $n$ is large enough, $\mathds{P}(\mathcal{B}_0^c \cap \mathcal{B}_1) \leq e^{-c_1 (\ln n)^2}$ and 
 $\mathds{P}(\mathcal{B}_0^c \cap \mathcal{B}_2) \leq e^{-c_1 (\ln n)^2}$. 
\end{proposition}

\begin{proof}
 Let $\lfloor N \theta\rfloor-2n^{(\alpha+4)/5} \leq m \leq \lfloor N \theta\rfloor+2n^{(\alpha+4)/5}$, $\lfloor Nx \rfloor- n^{(\alpha+4)/5} \leq i 
\leq \lfloor Nx \rfloor-n^{(\alpha+4)/5}$ and $\iota \in\{+,-\}$. We are going to bound the probability of 
$(\mathcal{B}_{0}^{m,i,\iota})^c \cap \mathcal{B}_{\mathbf{T}_{m,i}^\iota,1}^-$ and 
$(\mathcal{B}_{0}^{m,i,\iota})^c \cap \mathcal{B}_{\mathbf{T}_{m,i}^\iota,2}^-$ 
(the $\mathcal{B}_{\mathbf{T}_{m,i}^\iota,r}^+$ can be dealt with in the same way).
We write $\bar i = i\, \iota \, 1 = X_{\mathbf{T}_{m,i}^\iota}$ and $\bar m = \mathbf{T}_{m,i}^\iota$.

By Observation \ref{obs_rec_temps_local}, for any $\bar i -\lfloor \varepsilon n\rfloor < j_1 \leq j_2$, 
$L_{j_2}^{\bar m,-}-L_{j_1}^{\bar m,-} = 
\sum_{j=j_1}^{j_2-1}(\zeta_j^{\bar m,-,E}-\zeta_j^{\bar m,-,B})$. Instead of tackling this sum, we will consider 
a more amenable $\sum_{j=j_1}^{j_2-1}A_j$, where the random variables $A_j$, $j > \bar i-\lfloor \varepsilon n \rfloor$, are defined as follows. 
We fix $j > \bar i-\lfloor \varepsilon n \rfloor$, and we recall the $\bar\Delta_{\bar m,j}$ defined in Proposition \ref{prop_law_zeta} as well as the $\bar L_j^{\bar m,-}$ and $\bar \eta$ introduced before Proposition \ref{prop_def_zetaI}. 
If $\bar i -\lfloor \varepsilon n \rfloor < j \leq \bar i$, 
by Lemma \ref{lem_coupling} we can couple $\bar \eta$ with a chain $\tilde \eta$ such that $\tilde \eta(0)=-\bar\Delta_{\bar m,j}-1$ 
and for all $\ell \geq 0$, $\tilde \eta(\ell) \leq \bar \eta(\ell)$. We then set 
$A_j=\tilde\eta(\bar L_j^{\bar m,-}+1)+\bar\Delta_{\bar m,j}+1$, which is at most $\zeta_j^{\bar m,-,E}-\zeta_j^{\bar m,-,B}$ 
when $(\mathcal{B}_{0}^{m,i,\iota})^c$ is satisfied and $n$ is large enough. 
If $j > \bar i$, we set $A_j = \bar\eta(\bar L_j^{\bar m,-} \vee 1)+\bar\Delta_{\bar m,j}$.
For any $i_0 \in \{\bar i-\lfloor n \varepsilon \rfloor+1,\dots,\bar i-\lfloor (\ln n)^{6} \rfloor\}$, we denote 
$\mathcal{K}_{i_0} = \{\exists\, j \in\{i_0+1,\dots,i_0+\lfloor (\ln n)^{6} \rfloor\},
\sum_{j'=i_0}^{j-1} A_{j'} \geq (\ln n)^3\}$; for $i_0\geq\bar i+1$, we denote $\mathcal{K}_{i_0} = \{\exists\, j \in\{i_0+1,\dots,i_0+\lfloor (\ln n)^{6} \rfloor\},
\sum_{j'=i_0}^{j-1} A_{j'} \leq -(\ln n)^3\}$. Finally, for any $j \geq \bar i-\lfloor n \varepsilon \rfloor+1$, we denote $\mathcal{G}_j = 
\sigma(A_{j'},\bar i-\lfloor n \varepsilon \rfloor+1 \leq j' < j;\bar L^{\bar m,-}_{j'},\bar i-\lfloor n \varepsilon \rfloor+1 \leq j' \leq j)$. 

We are going to prove the following.

\begin{lemma}\label{lem_TCL}
 There exists a constant $\tilde c_1 = \tilde c_1(w) > 0$ such that when $n$ is large enough, for any 
 $i_0 \in \{\bar i-\lfloor n \varepsilon \rfloor+1,\dots,\bar i-\lfloor (\ln n)^{6} \rfloor\}$ or $i_0 \geq \bar i+1$, 
 we have $\mathds{P}(\mathcal{K}_{i_0}^c | \mathcal{G}_{i_0}) \leq e^{-\tilde c_1}$ almost surely. 
\end{lemma}

Let us show that Lemma \ref{lem_TCL} implies sufficient bounds on 
$\mathds{P}((\mathcal{B}_{0}^{m,i,\iota})^c \cap \mathcal{B}_{\mathbf{T}_{m,i}^\iota,1}^-)$ and 
$\mathds{P}((\mathcal{B}_{0}^{m,i,\iota})^c \cap \mathcal{B}_{\mathbf{T}_{m,i}^\iota,2}^-)$.

We begin with $\mathds{P}((\mathcal{B}_{0}^{m,i,\iota})^c \cap \mathcal{B}_{\mathbf{T}_{m,i}^\iota,1}^-)$. 
If $(\mathcal{B}_{0}^{m,i,\iota})^c$ is satisfied, $n$ is large enough and there exists 
$i_0 \in \{\bar i-\lfloor n \varepsilon \rfloor+1,\dots,\bar i-\lfloor (\ln n)^{8} \rfloor\}$ such that for all 
$j \in \{i_0,\dots,i_0+\lfloor (\ln n)^{8} \rfloor\}$, $L_j^{\bar m,-} < (\ln n)^3$, then for all 
$\ell \in\{0,\dots,\lfloor (\ln n)^2\rfloor-1\}$, for all $j \in \{i_0+\ell \lfloor (\ln n)^{6} \rfloor+1,\dots,
i_0+(\ell+1) \lfloor (\ln n)^{6} \rfloor\}$ we have $L_{j}^{\bar m,-} 
< L_{i_0+\ell \lfloor (\ln n)^{6} \rfloor}^{\bar m,-} + (\ln n)^3$, thus for all $\ell \in\{0,\dots,\lfloor (\ln n)^2\rfloor-1\}$, 
$\mathcal{K}_{i_0+\ell \lfloor (\ln n)^{6} \rfloor}^c$ is satisfied. We deduce that when $n$ is large enough, 
\[
 \mathds{P}((\mathcal{B}_{0}^{m,i,\iota})^c \cap \mathcal{B}_{\mathbf{T}_{m,i}^\iota,1}^-) 
 \leq \sum_{i_0= \bar i-\lfloor n \varepsilon \rfloor+1}^{\bar i-\lfloor (\ln n)^{8} \rfloor}
 \mathds{P}\left(\bigcap_{\ell=0}^{\lfloor (\ln n)^2\rfloor-1}\mathcal{K}_{i_0+\ell \lfloor (\ln n)^{6} \rfloor}^c\right) 
 \leq n \varepsilon e^{-\tilde c_1 \lfloor (\ln n)^2\rfloor},
\]
which is enough.

We now deal with $\mathds{P}((\mathcal{B}_{0}^{m,i,\iota})^c \cap \mathcal{B}_{\mathbf{T}_{m,i}^\iota,2}^-)$. 
We define the following random variables when possible: $\tau_1=\inf\{j > \bar i \,|\, 0 < \bar L_j^{\bar m,-} < (\ln n)^3\}$, 
and for $\ell \geq 1$, $\tau_{i+1} = \inf\{j \geq \tau_\ell+\lfloor (\ln n)^{6} \rfloor \,|\, 0 < \bar L_j^{\bar m,-} < (\ln n)^3\}$. 
If $(\mathcal{B}_{0}^{m,i,\iota})^c \cap \mathcal{B}_{\mathbf{T}_{m,i}^\iota,2}^-$ is satisfied and $n$ is large enough, 
$\tau_{\lfloor (\ln n)^2 \rfloor}$ exists, 
$\tau_{\lfloor (\ln n)^2 \rfloor} \leq \bar i + n^{(\alpha-1)/4}\lfloor\varepsilon n\rfloor - \lfloor(\ln n)^{6}\rfloor+1$, and for any 
$\ell \in \{1,\dots,\lfloor (\ln n)^2 \rfloor-1\}$, $\bar L_j^{\bar m,-} > 0$ for 
$j \in \{\tau_\ell,\dots,\tau_\ell+\lfloor (\ln n)^{6} \rfloor\}$, since if $j > \bar i$ is such that $\bar L_j^{\bar m,-} = 0$, 
$\bar L_j^{\bar m,-} = 0$ for all $j' > j$. In addition, if $(\mathcal{B}_{0}^{m,i,\iota})^c$ is satisfied, $n$ is large enough and $j \leq \bar i + n^{(\alpha-1)/4}\lfloor\varepsilon n\rfloor$, 
when $L_j^{\bar m,-} = \bar L_j^{\bar m,-} > 0$ we have $A_j = \zeta_j^{\bar m,-,E}-\zeta_j^{\bar m,-,B}$. We deduce that if 
$(\mathcal{B}_{0}^{m,i,\iota})^c \cap \mathcal{B}_{\mathbf{T}_{m,i}^\iota,2}^-$ is satisfied, 
$\tau_{\lfloor (\ln n)^2 \rfloor}$ exists, $\tau_{\lfloor (\ln n)^2 \rfloor} 
\leq \bar i + n^{(\alpha-1)/4}\lfloor\varepsilon n\rfloor - \lfloor(\ln n)^{6}\rfloor+1$, and for any 
$\ell \in \{1,\dots,\lfloor (\ln n)^2 \rfloor-1\}$, $\mathcal{K}_{\tau_\ell}^c$ occurs. This yields 
 $\mathds{P}((\mathcal{B}_{0}^{m,i,\iota})^c \cap \mathcal{B}_{\mathbf{T}_{m,i}^\iota,2}^-) \leq 
 e^{- \tilde c_1 (\lfloor (\ln n)^2\rfloor-1)}$, which is enough.
 
 We now prove Lemma \ref{lem_TCL}. To proceed, we will need the following claim:
 
 \begin{claim}\label{claim_TCL_constants}
 Let $i_0 \in \{\bar i-\lfloor n \varepsilon \rfloor+1,\dots,\bar i-\lfloor (\ln n)^{6} \rfloor\}$  or  $i_0 \geq \bar i+1$. 
 For any $j \in \{i_0,\dots,i_0+\lfloor(\ln n)^{6}\rfloor-1\}$, $p \geq 1$, 
 $A_j \in \mathrm{L}^p$ and $\mathds{E}(A_j|\mathcal{G}_{j})=0$. 
 Furthermore, there exist constants $\bar c_1=\bar c_1(w) > 0$, $\bar C_1=\bar C_1(w) < \infty$ 
 such that for any $j \in \{i_0,\dots,i_0+\lfloor(\ln n)^{6}\rfloor-1\}$,  
 $\mathds{E}(A_j^2|\mathcal{G}_{j}) \geq \bar c_1$ and $\mathds{E}(|A_j|^3|\mathcal{G}_{j}) \leq \bar C_1$.
 \end{claim}
 
 \begin{proof}[Proof of Claim \ref{claim_TCL_constants}.]
 We suppose $i_0 \in \{\bar i-\lfloor n \varepsilon \rfloor+1,\dots,\bar i-\lfloor (\ln n)^{6} \rfloor\}$; the case $i_0 \geq \bar i+1$ can be dealt with in the same way. Let $j \in \{i_0,\dots,i_0+\lfloor(\ln n)^{6}\rfloor-1\}$. 
 Then $A_j=\tilde\eta(\bar L_j^{\bar m,-}+1)+\bar\Delta_{\bar m,j}+1$ with $\tilde\eta(0)=-\bar\Delta_{\bar m,j}-1$. 
 By Proposition \ref{prop_law_zeta}, $\bar\Delta_{\bar m,j}$ has the law $\rho_-$ defined in \eqref{eq_def_rho-} and is independent of $\mathcal{G}_j$, so the chain $\tilde\eta$ is stationary and independent 
 of $\mathcal{G}_j$. Moreover, $\bar L_j^{\bar m,-}$ is $\mathcal{G}_j$-measurable, so conditionally to 
 $\mathcal{G}_j$, $\tilde\eta(\bar L_j^{\bar m,-}+1)$ has law $\rho_-$. Therefore $\bar\Delta_{\bar m,j}$ and 
 $\tilde\eta(\bar L_j^{\bar m,-}+1)$ have exponential tails, so $A_j \in \mathrm{L}^p$ for any $p \geq 1$. In addition, 
 if we write $A_{j,1}=\bar\Delta_{\bar m,j}+\frac{1}{2}$ and $A_{j,2}=\tilde\eta(\bar L_j^{\bar m,-}+1)+\frac{1}{2}$ for short, 
 conditionally to $\mathcal{G}_j$ both $A_{j,1}$ and $A_{j,2}$ have the law $\rho_0$ defined in \eqref{eq_def_rho0}. This implies 
 $\mathds{E}(A_j|\mathcal{G}_j)=\mathds{E}(A_{j,1}+A_{j,2}|\mathcal{G}_j)=0$. 
 Furthermore,
 \begin{align*}
  \mathds{E}(|A_j|^3|\mathcal{G}_{j}) &\leq \mathds{E}(|A_{j,1}|^3|\mathcal{G}_{j})
  +3\mathds{E}(|A_{j,1}|^2|A_{j,2}||\mathcal{G}_{j})+3\mathds{E}(|A_{j,1}||A_{j,2}|^2|\mathcal{G}_{j})
  +\mathds{E}(|A_{j,2}|^3|\mathcal{G}_{j}) \\
  &\leq \mathds{E}(|A_{j,1}|^3|\mathcal{G}_{j})
  +3\mathds{E}(|A_{j,1}|^4|\mathcal{G}_{j})^{1/2}\mathds{E}(|A_{j,2}|^2|\mathcal{G}_{j})^{1/2} 
  +3\mathds{E}(|A_{j,1}|^2|\mathcal{G}_{j})^{1/2}\mathds{E}(|A_{j,2}|^4|\mathcal{G}_{j})^{1/2}
  +\mathds{E}(|A_{j,2}|^3|\mathcal{G}_{j})
 \end{align*}
 by the Cauchy-Schwarz inequality. Since $\rho_0$ has exponential tails, each of these expectations is bounded, thus 
 $\mathds{E}(|A_j|^3|\mathcal{G}_{j})$ is at most a constant depending on $w$.

 We now deal with the lower bound of $\mathds{E}(A_j^2|\mathcal{G}_{j})$. Since $A_j$ is integer-valued, 
 \[
  \mathds{E}(A_j^2|\mathcal{G}_{j}) \geq \mathds{P}(A_j \neq 0|\mathcal{G}_{j}) 
  \geq \mathds{P}(\tilde\eta(\bar L_j^{\bar m,-}+1) \neq 0|\mathcal{G}_{j},\bar\Delta_{\bar m,j}=-1)
  \mathds{P}(\bar\Delta_{\bar m,j}=-1|\mathcal{G}_{j}).
 \]
Furthermore, $\mathds{P}(\bar\Delta_{\bar m,j}=-1|\mathcal{G}_{j}) = \rho_-(-1)$. In addition, if $\bar\Delta_{\bar m,j}=-1$, 
$\tilde\eta(0)=0$, so by Lemma \ref{lem_conv_eta} there exists $\ell_0 \in \mathds{N}^*$ such that for any $\ell \geq \ell_0$, 
$\mathds{P}(\tilde\eta(\ell) \neq 0|\mathcal{G}_{j},\bar\Delta_{\bar m,j}=-1) \geq \frac{1}{2}\rho_-(\mathds{Z}^*)$. 
Now, for any $1 \leq \ell < \ell_0$, there exists a constant $\bar c_{1,\ell} > 0$ such that 
$\mathds{P}(\tilde\eta(\ell) \neq 0|\mathcal{G}_{j},\bar\Delta_{\bar m,j}=-1) \geq \bar c_{1,\ell}$. 
We deduce 
\[
 \mathds{E}(A_j^2|\mathcal{G}_{j}) \geq \rho_-(-1) \min\left(\frac{1}{2}\rho_-(\mathds{Z}^*),
 \min_{1 \leq \ell < \ell_0}\bar c_{i,\ell}\right) > 0,
\]
which ends the proof of the claim.
 \end{proof}
 
 \begin{proof}[Proof of Lemma \ref{lem_TCL}.]
 Let $i_0 \in \{\bar i-\lfloor n \varepsilon \rfloor+1,\dots,\bar i-\lfloor (\ln n)^{6} \rfloor\}$ or $i_0 \geq \bar i+1$. We denote $i_0'=i_0+\lfloor (\ln n)^{6} \rfloor-1$ for short. Claim \ref{claim_TCL_constants} implies $(\sum_{j'=i_0}^{j-1}A_{j'})_{i_0 \leq j \leq i_0'+1}$ is a martingale with respect to the filtration $(\mathcal{G}_j)_{j>\bar i-\lfloor n \varepsilon \rfloor}$. We would like to use a central limit theorem for martingales to control the law of $\sum_{j'=i_0}^{i_0'}A_{j'}$, but in order to do that we would need $\sum_{j'=i_0}^{i_0'}A_{j'}^2$ to be close to a constant when $n$ is large enough, and we do not control it well enough. We will therefore define another martingale. 
 
 Thanks to Claim \ref{claim_TCL_constants}, for any $j \in \{i_0,\dots,i_0'\}$, we can define $\sigma_j \geq 0$ by $\sigma_j^2 = \mathds{E}(A_j^2|\mathcal{G}_{j})$. We set $j_0 = \inf\{j\in\{i_0,\dots,i_0' \}\,|\, \sum_{j'=i_0}^j \sigma_{j'}^2 \geq \frac{\bar c_1(\ln n)^{6}}{2}\}$, which exists since $\sum_{j'=i_0}^{i_0'}\sigma_{j'}^2 \geq \bar c_1\lfloor(\ln n)^{6}\rfloor$ by Claim \ref{claim_TCL_constants}. We define $\kappa \in [0,1]$ by $\sum_{j=i_0}^{j_0-1}\sigma_{j}^2+\kappa \sigma_{j_0}^2 = \frac{\bar c_1(\ln n)^{6}}{2}$. For any $j \in \{i_0,\dots,i_0'\}$, we also define $\bar A_j=A_j \mathds{1}_{\{j_0 > j\}}+\sqrt{\kappa}A_j \mathds{1}_{\{j_0 = j\}}$ and $\bar \sigma_j \geq 0$ by $\bar \sigma_j^2 = \mathds{E}(\bar A_j^2 | \mathcal{G}_j)=\sigma_j^2\mathds{1}_{\{j_0 > j\}} +\kappa\sigma_j^2\mathds{1}_{\{j_0 = j\}}$, so that $\sum_{j=i_0}^{i_0'}\bar A_j=\sum_{j=i_0}^{j_0-1}A_j+\sqrt{\kappa} A_{j_0}$. This implies that for $i_0 \in \{\bar i-\lfloor n \varepsilon \rfloor+1,\dots,\bar i-\lfloor (\ln n)^{6} \rfloor\}$, if $\sum_{j=i_0}^{i_0'}\bar A_j \geq (\ln n)^3$ then $\mathcal{K}_{i_0}$ occurs, so $\mathds{P}(\mathcal{K}_{i_0}^c|\mathcal{G}_{i_0}) \leq \mathds{P}(\sum_{j=i_0}^{i_0'}\bar A_j < (\ln n)^3|\mathcal{G}_{i_0})$. Similarly, for $i_0 \geq \bar i+1$ we have $\mathds{P}(\mathcal{K}_{i_0}^c|\mathcal{G}_{i_0}) \leq \mathds{P}(\sum_{j=i_0}^{i_0'}\bar A_j > -(\ln n)^3|\mathcal{G}_{i_0})$. Consequently, to prove Lemma \ref{lem_TCL} we only have to find a constant $\tilde c_1 = \tilde c_1(w) > 0$ such that when $n$ is large enough, for any $i_0 \in \{\bar i-\lfloor n \varepsilon \rfloor+1,\dots,\bar i-\lfloor (\ln n)^{6} \rfloor\}$ we have $\mathds{P}(\sum_{j=i_0}^{i_0'}\bar A_j < (\ln n)^3|\mathcal{G}_{i_0}) \leq e^{-\tilde c_1}$ almost surely (the case $i_0 \geq \bar i+1$ can be dealt with in the same way). 
 
 Suppose by contradiction that it is not true. This implies that there exists a sequence $(N(k))_{k\in\mathds{N}}$ tending to $+\infty$ so that for each $k\in\mathds{N}$ there exists the following (the quantities will depend on $k$, but we will not include this dependence in the notation as that would make it too heavy) $\lfloor N(k) \theta\rfloor-2n^{(\alpha+4)/5} \leq m \leq \lfloor N(k) \theta\rfloor+2n^{(\alpha+4)/5}$, $\lfloor N(k) x \rfloor- n^{(\alpha+4)/5} \leq i \leq \lfloor N(k) x \rfloor-n^{(\alpha+4)/5}$, $\iota \in\{+,-\}$ and  
 $i_0 \in \{\bar i-\lfloor n \varepsilon \rfloor+1,\dots,\bar i-\lfloor (\ln n)^{6} \rfloor\}$ so that $\mathds{P}(\sum_{j=i_0}^{i_0'}\bar A_j < (\ln n)^3|\mathcal{G}_{i_0}) > \frac{1+c_1'}{2}$ with positive probability, where $c_1'\in(0,1)$ is the probability that a random variable with law $\mathcal{N}(0,1)$ is at most $\frac{\sqrt{2}}{\sqrt{\bar c_1}}$.
 
 For any $j \geq \bar i-\lfloor n \varepsilon \rfloor+1$, we denote $G_j = ((A_{j'})_{\bar i-\lfloor n \varepsilon \rfloor+1 \leq j' < j},(\bar L^{\bar m,-}_{j'})_{\bar i-\lfloor n \varepsilon \rfloor+1 \leq j' \leq j})$. Since we have $\mathcal{G}_{i_0} = \sigma((A_{j})_{\bar i-\lfloor n \varepsilon \rfloor+1 \leq j < i_0},(\bar L^{\bar m,-}_{j})_{\bar i-\lfloor n \varepsilon \rfloor+1 \leq j \leq i_0})$, since $\mathds{P}(\sum_{j=i_0}^{i_0'}\bar A_j < (\ln n)^3|\mathcal{G}_{i_0}) > \frac{1+c_1'}{2}$ with positive probability, there exists $\omega\in\mathds{Z}^{2(i_0-\bar i+\lfloor n \varepsilon \rfloor)-1}$ such that $\mathds{P}(G_{i_0}=\omega)>0$ and $\mathds{P}(\sum_{j=i_0}^{i_0'}\bar A_j < (\ln n)^3|G_{i_0}=\omega) > \frac{1+c_1'}{2}$. We want to apply a central limit theorem for martingales to the process $(\sum_{j'=i_0}^{j-1}\frac{\sqrt{2}}{\sqrt{\bar c_1}(\ln n)^3}\bar A_{j'})_{i_0 \leq j \leq i_0'+1}$ under the law $\mathds{P}(\cdot|G_{i_0}=\omega)$. We denote this law $\mathds{P}'(.)$ for short, and write $\mathds{E}'(\cdot)$ for the expectation operator. We consider the probability space $(\prod_{k\in\mathds{N}}\Omega_k,\bigotimes_{k\in\mathds{N}}\mathcal{F}'_k,\bigotimes_{k\in\mathds{N}}\mathds{P}_k)$, where for any $k\in \mathds{N}$ the space $(\Omega_k,\mathcal{F}'_k,\mathds{P}_k)$ is a copy of the probability space where the $A_{j}$, $j \geq \bar i-\lfloor n \varepsilon \rfloor+1$, $\bar L^{\bar m,-}_{j}$, $j \geq \bar i-\lfloor n \varepsilon \rfloor+1$ corresponding to $k$ live, with the probability measure $\mathds{P}'$ corresponding to $k$. We denote $\mathcal{P}=\bigotimes_{k\in\mathds{N}}\mathds{P}_k$ and $\mathcal{E}$ the corresponding expectation. For any $k\in\mathds{N}$, we may consider the $(\mathcal{G}_j(k))_{i_0(k) \leq j \leq i_0'(k)}$ and $(\bar A_j(k))_{i_0(k) \leq j \leq i_0'(k)}$ defined as previously, but on the space $(\Omega_k,\mathcal{F}'_k,\mathds{P}_k)$. Possibly through extracting a subsequence, we can assume $n(k)$ is non-decreasing in $k$. For any $k\in\mathds{N}$, $\ell\in\{1,...,\lfloor(\ln n(k))^6\rfloor\}$, we define $\mathcal{G}'_{k,\ell}=(\bigotimes_{k'=0}^k\mathcal{G}_{i_0(k')+\ell-1}(k'))\otimes(\bigotimes_{k'>k}\{\emptyset,\Omega_{k'}\})$. We then have $\mathcal{G}'_{k,\ell} \subset \mathcal{G}'_{k+1,\ell}$. We will use a central limit theorem for martingales with $(\sum_{j'=i_0(k)}^{i_0(k)+j-1}\frac{\sqrt{2}}{\sqrt{\bar c_1}(\ln n(k))^3}\bar A_{j'}(k))_{0 \leq j \leq \lfloor(\ln n(k))^6\rfloor}$. 

To do that, let us prove its assumptions. We first notice that for any $k\in\mathds{N}$, for any random variable $V$ and any $j \geq i_0$ we have $\mathds{E}'(V|\mathcal{G}_j)=\mathds{E}(V|\mathcal{G}_j)$. Indeed, for any $\omega'\in\mathds{Z}^{2(j-i_0)}$ so that $\mathds{P}'(G_j=(\omega,\omega'))>0$ we have 
\[
\mathds{E}'(V|G_j=(\omega,\omega'))\!=\!\frac{\mathds{E}'(V\mathds{1}_{\{G_j=(\omega,\omega')\}})}{\mathds{P}'(G_j=(\omega,\omega'))}\!=\!\frac{\mathds{E}(V\mathds{1}_{\{G_j=(\omega,\omega')\}})}{\mathds{P}(G_{i_0}=\omega)}\frac{\mathds{P}(G_{i_0}=\omega)}{\mathds{P}(G_j=(\omega,\omega'))}\!=\!\frac{\mathds{E}(V\mathds{1}_{\{G_j=(\omega,\omega')\}})}{\mathds{P}(G_j=(\omega,\omega'))}\!=\!\mathds{E}(V|G_j=(\omega,\omega')).
\]
In addition, by Claim \ref{claim_TCL_constants}, for any $j \in \{i_0,...,i_0'\}$ we have $\mathds{E}(A_j|\mathcal{G}_j)=0$, which implies $\mathds{E}(\bar A_j|\mathcal{G}_j)=0$, therefore $\mathds{E}'(\bar A_j|\mathcal{G}_j)=0$. This implies $\mathcal{E}(\bar A_{j}(k)|\mathcal{G}'_{k,j-i_0(k)+1})=0$, so $(\frac{\sqrt{2}}{\sqrt{\bar c_1}(\ln n(k))^3}\bar A_{i_0(k)+\ell-1}(k),\mathcal{G}_{k,\ell})_{k\in\mathds{N},1 \leq \ell \leq \lfloor(\ln n(k))^6\rfloor}$ is a martingale difference array. By Claim \ref{claim_TCL_constants}, for any $k\in\mathds{N}$, $i_0 \leq j \leq i_0'$, $A_j$ is square-integrable with respect to $\mathds{P}$, thus to $\mathds{P}'$, hence $\bar A_j$ also, therefore $\bar A_j(k)$ is square-integrable. Furthermore, $\sum_{\ell=1}^{\lfloor(\ln n(k))^6\rfloor}\mathcal{E}((\bar A_{i_0(k)+\ell-1}(k))^2|\mathcal{G}'_{k,\ell})$ is the same as $\sum_{i=i_0}^{i_0'}\mathds{E}'(\bar A_j^2|\mathcal{G}_j)=\sum_{i=i_0}^{i_0'}\mathds{E}(\bar A_j^2|\mathcal{G}_j)=\frac{\bar c_1 (\ln n)^6}{2}$ by definition of $j_0$ and $\kappa$. This allows to obtain $\sum_{\ell=1}^{\lfloor(\ln n(k))^6\rfloor}\mathcal{E}((\frac{\sqrt{2}}{\sqrt{\bar c_1}(\ln n)^3}\bar A_{i_0(k)+\ell-1}(k))^2|\mathcal{G}'_{k,\ell})=1$. We now prove the conditional Lindeberg condition. Let $\delta>0$, for any $k\in\mathds{N}$ Claim \ref{claim_TCL_constants} yields $\mathds{E}(|\bar A_j|^3|\mathcal{G}_j) \leq \bar C_1$ for all $j\in\{i_0,...,i_0'\}$, therefore 
\[
 \sum_{i=i_0}^{i_0'}\mathds{E}'\left(\left.\left(\frac{\sqrt{2}}{\sqrt{\bar c_1}(\ln n)^3}\bar A_j\right)^2\mathds{1}_{\{|\frac{\sqrt{2}}{\sqrt{\bar c_1}(\ln n)^3}\bar A_j| > \delta\}}\right|\mathcal{G}_j\right)
 = \sum_{i=i_0}^{i_0'}\mathds{E}\left(\left.\left(\frac{\sqrt{2}}{\sqrt{\bar c_1}(\ln n)^3}\bar A_j\right)^2\mathds{1}_{\{|\frac{\sqrt{2}}{\sqrt{\bar c_1}(\ln n)^3}\bar A_j| > \delta\}}\right|\mathcal{G}_j\right)
\]
\[
 \leq \sum_{i=i_0}^{i_0'}\mathds{E}\left(\left.\frac{1}{\delta}\left|\frac{\sqrt{2}}{\sqrt{\bar c_1}(\ln n)^3}\bar A_j\right|^3\right|\mathcal{G}_j\right)
 =\frac{1}{\delta}\frac{2^{3/2}}{\bar c_1^{3/2} (\ln n)^9}\sum_{i=i_0}^{i_0'}\mathds{E}(|\bar A_j|^3|\mathcal{G}_j)
 \leq \frac{1}{\delta}\frac{2^{3/2}}{\bar c_1^{3/2} (\ln n)^9}(i_0'-i_0+1)\bar C_1
 =\frac{1}{\delta}\frac{2^{3/2}\bar C_1}{\bar c_1^{3/2} (\ln n)^3}.
\]
Thus $\sum_{\ell=1}^{\lfloor(\ln n(k))^6\rfloor}\mathcal{E}((\frac{\sqrt{2}}{\sqrt{\bar c_1}(\ln n)^3}\bar A_{i_0(k)+\ell-1}(k))^2\mathds{1}_{\{|\frac{\sqrt{2}}{\sqrt{\bar c_1}(\ln n)^3}\bar A_{i_0(k)+\ell-1}(k)| > \delta\}}|\mathcal{G}'_{k,\ell}) \leq \frac{1}{\delta}\frac{2^{3/2}\bar C_1}{\bar c_1^{3/2} (\ln n)^3}$, hence it converges to 0 in probability, which is the conditional Lindeberg condition. Consequently, by the central limit theorem for martingales found as Corollary 3.1 of \cite{Hall_Heyde_martingales}, $\sum_{\ell=1}^{\lfloor(\ln n(k))^6\rfloor}\frac{\sqrt{2}}{\sqrt{\bar c_1}(\ln n)^3}\bar A_{i_0(k)+\ell-1}(k)$ converges in distribution to $\mathcal{N}(0,1)$. This implies that when $k$ is large enough, $\mathcal{P}(\sum_{\ell=1}^{\lfloor(\ln n(k))^6\rfloor}\frac{\sqrt{2}}{\sqrt{\bar c_1}(\ln n)^3}\bar A_{i_0(k)+\ell-1}(k) < \frac{\sqrt{2}}{\sqrt{\bar c_1}}) \leq \frac{1+c_1'}{2}$, hence $\mathds{P}'(\sum_{j=i_0}^{i'_0}\bar A_j < (\ln n)^3) \leq \frac{1+c_1'}{2}$. However, that contradicts the fact that $\mathds{P}(\sum_{j=i_0}^{i_0'}\bar A_j < (\ln n)^3|G_{i_0}=\omega) > \frac{1+c_1'}{2}$, hence our assumption was wrong, which ends the proof of the lemma.
 \end{proof}
\end{proof}

\begin{lemma}\label{lem_zeta_E_and_I}
There exists a constant $c_3=c_3(w) > 0$ such that 
 when $n$ is large enough, $\mathds{P}(\mathcal{B}_0^c \cap \mathcal{B}_3) \leq e^{-c_3 (\ln n)^2}$. 
\end{lemma}

\begin{proof}
Let $\lfloor N \theta\rfloor-2n^{(\alpha+4)/5} \leq m \leq \lfloor N \theta\rfloor+2n^{(\alpha+4)/5}$, $\lfloor Nx \rfloor-n^{(\alpha+4)/5} \leq i 
\leq \lfloor Nx \rfloor+n^{(\alpha+4)/5}$, $\iota \in \{+,-\}$. 
We denote $\bar m = \mathbf{T}_{m,i}^\iota$ and $\bar i = X_{\mathbf{T}_{m,i}^\iota}$. It is enough to find 
constants $\tilde C_3=\tilde C_3(w) < +\infty$ and $\tilde c_3=\tilde c_3(w) > 0$ such that 
 when $n$ is large enough, for any $j \in \{\bar i - \lfloor \varepsilon n \rfloor+1,\dots,\bar i + \lfloor \varepsilon n \rfloor\}$, 
 $\mathds{P}((\mathcal{B}_0^{m,i,\iota})^c \cap \{L_j^{\bar m,-} \geq (\ln n)^2,
 \zeta_j^{\bar m,-,E}\neq \zeta_j^{\bar m,-,I}\}) \leq \tilde C_3 e^{-\tilde c_3 (\ln n)^2}$ 
 and for any $j \in \{\bar i - \lfloor \varepsilon n \rfloor,\dots,\bar i + \lfloor \varepsilon n \rfloor-1\}$, 
 $\mathds{P}((\mathcal{B}_0^{m,i,\iota})^c \cap \{L_{j+1}^{\bar m,+} \geq (\ln n)^2,
 \zeta_j^{\bar m,+,E}\neq \zeta_j^{\bar m,+,I}\}) \leq \tilde C_3 e^{-\tilde c_3 (\ln n)^2}$. 
 We will write the proof for the $\zeta_j^{\bar m,-,E}$ with $j \in \{\bar i - \lfloor \varepsilon n \rfloor+1,\dots,\bar i\}$; 
 the other cases can be dealt with in the same way. 
 
 We use the notation of the construction of the $\zeta_j^{\bar m,-,I}$. With this notation, 
 $\zeta_j^{\bar m,-,E}$ can be different from $\zeta_j^{\bar m,-,I}$ only if $\eta(T+\lfloor(\ln n)^2/2\rfloor) \neq \Delta$ 
 or $L_i^{\bar m,-}+1-T < (\ln n)^2/2$. This yields that if $L_j^{\bar m,-} \geq (\ln n)^2$, 
 $\zeta_j^{\bar m,-,E}$ can be different from $\zeta_j^{\bar m,-,I}$ only if $\eta(T+\lfloor(\ln n)^2/2\rfloor) \neq \Delta$ 
 or $T > (\ln n)^2/2+1$. Therefore it is enough to bound $\mathds{P}(\eta(T+\lfloor(\ln n)^2/2\rfloor) \neq \Delta)$ 
 and $\mathds{P}((\mathcal{B}_0^{m,i,\iota})^c \cap \{T > (\ln n)^2/2+1\})$. $\Delta$ was chosen so to have 
 $\mathds{P}(\eta(T+\lfloor(\ln n)^2/2\rfloor) \neq \Delta)$ minimal, so by Lemma \ref{lem_conv_eta}, 
 $\mathds{P}(\eta(T+\lfloor(\ln n)^2/2\rfloor) \neq \Delta) \leq \bar C e^{-\bar c \lfloor(\ln n)^2/2\rfloor}$, 
 which is enough. It remains to bound $\mathds{P}((\mathcal{B}_0^{m,i,\iota})^c \cap \{T > (\ln n)^2/2+1\})$. 
 In order to bound $\mathds{P}((\mathcal{B}_0^{m,i,\iota})^c \cap \{T > (\ln n)^2/2+1\})$, we consider the chain $\xi$ 
 so that $\eta$ corresponds to the $\eta_-$ of $\xi$ (see \eqref{eq_def_xi}, \eqref{eq_def_tau}, \eqref{eq_def_eta}). 
 We notice $\xi(0)=\eta(0)=-\Delta_{\bar m,j}$. 
 We denote $T'=\inf\{\ell > 0\,|\,\xi(\ell-1)=1,\xi(\ell)=0\}$; we then have $T \leq T'$, so it is enough to 
 find constants $\bar C_3=\bar C_3(w) < +\infty$ and $\bar c_3=\bar c_3(w) > 0$ such that when $n$ is large enough, 
 $\mathds{P}((\mathcal{B}_0^{m,i,\iota})^c \cap \{T' > (\ln n)^2/2+1\})\leq \bar C_3 e^{-\bar c_3 (\ln n)^2}$. 
 
 In order to do that, we will notice that if we denote $i_w=\min\{i'\in \mathds{N}^*\,|\,w(i')\neq w(-i')\}-1$, then on $\{-i_w,...,i_w\}$ the chain $\xi$ behaves like a simple random walk, while outside $\{-i_w,...,i_w\}$ the chain $\xi$ is biased towards 0. We consider the successive times at which $\xi$ is at $-i_w$ or $i_w$: $E_0=\inf\{\ell \geq 0 \,|\, \xi(\ell)=i_w$ or $-i_w\}$, and for any $\ell > 1$, $E_\ell = \{\ell' > E_{\ell-1} \,|\,
 \xi(\ell')=i_w$ or $-i_w\}$. After each of these times, $\xi$ may try to go to 1 and then to 0. Therefore, if $T'$ is large, one of the following happens: $\xi$ did not reach $\{-i_w,i_w\}$ quickly enough at the beginning to have spare time to make a lot of tries, or it did not come back to $\{-i_w,i_w\}$ many times afterwards to make other tries, or there were many tries but they all failed. Let us formalize this. We denote $p_w=\frac{w(-i_w-1)}{w(i_w+1)+w(-i_w-1)} \in (0,1/2)$. We will also need a constant $\bar c_3= \bar c_3(w) > 0$ that we will define 
 later. We set $\mathcal{A}_1=\{|\xi(0)|>\frac{1-2p_w}{2}\lfloor(\ln n)^2/4\rfloor\}$, 
 $\mathcal{A}_2=\{E_0 > (\ln n)^2/4\}$, $\mathcal{A}_3 = \{E_{\lfloor \bar c_3 (\ln n)^2\rfloor}-E_0 > (\ln n)^2/4\}$ 
 and $\mathcal{A}_4=\{T' > E_{\lfloor \bar c_3 (\ln n)^2\rfloor}\}$. 
 We have $\{T' > (\ln n)^2/2+1\} \subset \mathcal{A}_2 \cup \mathcal{A}_3 \cup \mathcal{A}_4$, 
 hence we have
 \begin{equation}\label{eq_B3}
  \mathds{P}((\mathcal{B}_0^{m,i,\iota})^c \cap \{T' > (\ln n)^2/2+1\})
  \leq \mathds{P}((\mathcal{B}_0^{m,i,\iota})^c \cap \mathcal{A}_1)+\mathds{P}(\mathcal{A}_1^c \cap \mathcal{A}_2)
  +\mathds{P}(\mathcal{A}_3)+\mathds{P}(\mathcal{A}_4).
 \end{equation}
 Each of these four terms admits an exponential bound which is rather easy to prove, hence we postpone the proof to the appendix.
\end{proof}

\begin{proposition}\label{prop_excursions}
There exists a constant $c_4=c_4(w,\varepsilon) > 0$ such that 
 when $n$ is large enough, $\mathds{P}(\mathcal{B}_0^c \cap \mathcal{B}_1^c \cap \mathcal{B}_3^c \cap \mathcal{B}_4) 
 \leq e^{-c_4 (\ln n)^2}$. 
\end{proposition}

The proof of Proposition \ref{prop_excursions} uses rather classical techniques, therefore we include only a sketch here and put the full proof in the appendix.

\begin{proof}[Proof sketch of Proposition \ref{prop_excursions}.]
Let $\lfloor N \theta\rfloor-2n^{(\alpha+4)/5} \leq m \leq \lfloor N \theta\rfloor+2n^{(\alpha+4)/5}$, $\lfloor Nx \rfloor-n^{(\alpha+4)/5} \leq i 
\leq \lfloor Nx \rfloor+n^{(\alpha+4)/5}$, $\iota \in \{+,-\}$. 
We denote $\bar m = \mathbf{T}_{m,i}^\iota$. We give the sketch only for the $-$ case, as the argument for the $+$ case is the same. Since $(\mathcal{B}_{\bar m,1}^-)^c$ occurs, each ``excursion of $L^{\bar m,-}$ below $(\ln n)^3$'' 
has length at most $(\ln n)^8$, hence to have $\mathcal{B}_{\bar m,4}^-$ we need at least $(\ln n)^2\sqrt{n}$ ``excursions of $L^{\bar m,-}$ below $(\ln n)^3$'', hence $(\ln n)^2\sqrt{n}-1$ ``excursions of $L^{\bar m,-}$ above $(\ln n)^3$''. On $(\mathcal{B}_{\bar m,3}^-)^c$, when $L_j^{\bar m,-} \geq (\ln n)^3$ we have by Observation \ref{obs_rec_temps_local} that $L_{j+1}^{\bar m,-}-L_j^{\bar m,-}=\zeta^{\bar m,-,E}_j-\zeta^{\bar m,-,B}_j=\zeta^{\bar m,-,I}_j-\zeta^{\bar m,-,B}_j$, hence $L_j^{\bar m,-}$ is roughly an i.i.d. random walk. Therefore each ``excursion of $L^{\bar m,-}$ above $(\ln n)^3$'' has probability roughly $\frac{1}{\sqrt{n}}$ to have length at least $n$ conditional on the past ``excursions'', thus to be the last ``excursion'' we see as we only consider an interval of size $\varepsilon n$. Therefore the probability of seeing $(\ln n)^2\sqrt{n}-1$ ``excursions'' has the appropriate bound.
\end{proof}

\begin{lemma}\label{lem_bound_zeta}
There exists a constant $c_5=c_5(w) > 0$ such that 
 when $n$ is large enough, $\mathds{P}(\mathcal{B}_0^c \cap \mathcal{B}_5) \leq e^{-c_5 (\ln n)^2}$. 
\end{lemma}

\begin{proof}
 Let $\lfloor N \theta\rfloor-2n^{(\alpha+4)/5} \leq m \leq \lfloor N \theta\rfloor+2n^{(\alpha+4)/5}$, $\lfloor Nx \rfloor-n^{(\alpha+4)/5} \leq i \leq \lfloor Nx \rfloor+n^{(\alpha+4)/5}$, $\iota \in \{+,-\}$. 
 We denote $\bar m = \mathbf{T}_{m,i}^\iota$ and $\bar i = X_{\mathbf{T}_{m,i}^\iota}$. 
 It is enough to find constants $\tilde C_5=\tilde C_5(w) < +\infty$ 
 and $\tilde c_5=\tilde c_5(w) > 0$ such that when $n$ is large enough, 
 $\mathds{P}((\mathcal{B}_0^{m,i,\iota})^c \cap \{|\zeta| > (\ln n)^2\}) 
 \leq \tilde C_5e^{-\tilde c_5(\ln n)^2}$ for $\zeta \in \{\zeta_j^{\bar m,-,B},\zeta_j^{\bar m,-,E},\zeta_j^{\bar m,-,I} \,|\, 
 j \in \{\bar i -\lfloor \varepsilon n\rfloor+1,\dots,\bar i+\lfloor\varepsilon n\rfloor\}\}
 \cup \{\zeta_j^{\bar m,+,B},\zeta_j^{\bar m,+,E},\zeta_j^{\bar m,+,I} \,|\, 
 j \in \{\bar i -\lfloor \varepsilon n\rfloor,\dots,\bar i+\lfloor\varepsilon n\rfloor-1\}\}$. The $\zeta_j^{\bar m,\pm,I}$ are easy to handle, since by Proposition \ref{prop_def_zetaI} they have the law $\rho_0$ defined in \eqref{eq_def_rho0}, which has exponential tails.
 The $\zeta_j^{\bar m,\pm,B}$ also are easy to deal with. Indeed, by Proposition \ref{prop_law_zeta} and Definition \ref{def_local_times}, 
 if $(\mathcal{B}_0^{m,i,\iota})^c$ occurs, $\zeta_j^{\bar m,\pm,B}$ or $\zeta_j^{\bar m,\pm,B}\mp1$ is equal to a random variable of law $\rho_0$, 
 and $\rho_0$ has exponential tails. We now consider $\zeta_j^{\bar m,-,E}$ with 
 $j \in \{\bar i -\lfloor \varepsilon n\rfloor+1,\dots,\bar i+\lfloor\varepsilon n\rfloor\}$ (the $\zeta_j^{\bar m,+,E}$ 
 can be dealt with in the same way). Thanks to (\ref{eq_zeta}), 
 $\zeta_j^{\bar m,-,E} = \eta(L_j^{\bar m,-}+1)+1/2$ or $\eta(L_j^{\bar m,-})+1/2$ (depending on $j$) 
 with $\eta(0)=-\Delta_{\bar m,j}$. Recalling the definitions before Proposition \ref{prop_def_zetaI}, 
 if $(\mathcal{B}_0^{m,i,\iota})^c$ occurs and $n$ is large enough, 
 we have $\zeta_j^{\bar m,-,E} = \bar\eta(\bar L_j^{\bar m,-}+1)+1/2$ or $\bar \eta(\bar L_j^{\bar m,-})+1/2$ depending on $j$. 
 Remembering Proposition \ref{prop_law_zeta}, if $j$ is such that $\bar\Delta_{\bar m,j}$ has law $\rho_+$, $-\bar\Delta_{\bar m,j}$ has law $\rho_-$, 
 so $\bar\eta(\bar L_j^{\bar m,-}+1)+1/2$ and $\bar \eta(\bar L_j^{\bar m,-})+1/2$ have law $\rho_0$, 
 which is enough. Now, if $j$ is such that $\bar \Delta_{\bar m,j}$ has law $\rho_-$, $-\bar\Delta_{\bar m,j}-1$ has law $\rho_-$. 
 By Lemma \ref{lem_coupling}, we can couple $\bar\eta$ with a process $\bar\eta'$ so that $\bar\eta'(0)=-\bar\Delta_{\bar m,j}-1$ 
 and $\bar\eta(\ell)-1 \leq \bar\eta'(\ell) \leq \bar\eta(\ell)$ for any $\ell \in \mathds{N}$. 
 Then $\bar\eta'(\bar L_j^{\bar m,-})$ and $\bar\eta'(\bar L_j^{\bar m,-}+1)$ have law $\rho_-$, which has exponential tails, 
 hence the result. 
\end{proof}

\begin{lemma}\label{lem_bound_sums}
There exists a constant $c_6 > 0$ such that when $n$ is large enough, 
$\mathds{P}(\mathcal{B}_0^c \cap \mathcal{B}_4^c \cap \mathcal{B}_6) \leq e^{-c_6(\ln n)^2}$. 
\end{lemma}

\begin{proof}
 Let $\lfloor N \theta\rfloor-2n^{(\alpha+4)/5} \leq m \leq \lfloor N \theta\rfloor+2n^{(\alpha+4)/5}$, $\lfloor Nx \rfloor-n^{(\alpha+4)/5} \leq i \leq \lfloor Nx \rfloor+n^{(\alpha+4)/5}$, $\iota \in \{+,-\}$. It is enough to find 
constants $\tilde C_6=\tilde C_6(w) < +\infty$ and $\tilde c_6 > 0$ such that 
when $n$ is large enough $\mathds{P}((\mathcal{B}_0^{m,i,\iota})^c \cap (\mathcal{B}_{\mathbf{T}_{m,i}^\iota,4}^-)^c 
\cap \mathcal{B}_{\mathbf{T}_{m,i}^\iota,6}^-) \leq \tilde C_6e^{-\tilde c_6(\ln n)^2}$ and 
$\mathds{P}((\mathcal{B}_0^{m,i,\iota})^c \cap (\mathcal{B}_{\mathbf{T}_{m,i}^\iota,4}^+)^c 
\cap \mathcal{B}_{\mathbf{T}_{m,i}^\iota,6}^+) \leq 
 \tilde C_6e^{-\tilde c_6(\ln n)^2}$. Let us do it for $\mathcal{B}_{\mathbf{T}_{m,i}^\iota,6}^-$; the case 
 $\mathcal{B}_{\mathbf{T}_{m,i}^\iota,6}^+$ is similar. 
 We denote $\bar m = \mathbf{T}_{m,i}^\iota$ and $\bar i = X_{\mathbf{T}_{m,i}^\iota}$. 
 We introduce a sequence $(\bar I^{\bar m,-}(\ell))_{\ell \in \mathds{N}}$ ``like $(I^{\bar m,-}(\ell))_{\ell \in \mathds{N}}$, 
 but for $\bar L^{\bar m,-}$'' (the $\bar L^{\bar m,-}$ were defined before Proposition \ref{prop_def_zetaI}): $\bar I^{\bar m,-}(0) = \bar i-\lfloor \varepsilon n\rfloor$, and for any $\ell \in \mathds{N}$, 
 $\bar I^{\bar m,-}(\ell+1)=\inf\{\bar I^{\bar m,-}(\ell) < j < \bar i \,|\, \bar L^{\bar m,-}_{j} < (\ln n)^3\}$. 
 If $(\mathcal{B}_0^{m,i,\iota})^c$ occurs and $n$ is large enough, $\bar I^{\bar m,-}(\ell)=I^{\bar m,-}(\ell)$ 
 for any $\ell \in \mathds{N}$, hence for $\ell \leq \ell^{\bar m,-}_{\textrm{max}}$, 
 $\zeta_{I^{\bar m,-}(\ell)}^{\bar m,-,B}=-\bar\Delta_{\bar m,\bar I^{\bar m,-}(\ell)}-1/2$ 
 and $\zeta_{I^{\bar m,-}(\ell)}^{\bar m,-,I}=\bar\zeta_{\bar I^{\bar m,-}(\ell)}^{\bar m,-,I}$ (the $\bar\zeta_{j}^{\bar m,-,I}$ were also defined before Proposition \ref{prop_def_zetaI}).
 By abuse of notation, the $-\bar\Delta_{\bar m,\bar I^{\bar m,-}(\ell)}-1/2$ and $\bar\zeta_{\bar I^{\bar m,-}(\ell)}^{\bar m,-,I}$ 
 for $\ell > \ell^{\bar m,-}_{\textrm{max}}$ will be i.i.d. random variables with law $\rho_0$ (defined in \eqref{eq_def_rho0}) independent from everything else. 
 We notice that if $(\mathcal{B}_{\bar m,4}^-)^c$ occurs, $\ell^{\bar m,-}_{\textrm{max}} \leq (\ln n)^{10}\sqrt{n}$. 
 Consequently, when $n$ is large enough, 
 \[
 \mathds{P}((\mathcal{B}_0^{m,i,\iota})^c \cap (\mathcal{B}_{\mathbf{T}_{m,i}^\iota,4}^-)^c 
 \cap \mathcal{B}_{\mathbf{T}_{m,i}^\iota,6}^-) 
 \]
 \[
 \leq \mathds{P}\left(\left\{\max_{1 \leq \ell_1 \leq \ell_2 \leq (\ln n)^{10}\sqrt{n}}
 \left|\sum_{\ell=\ell_1}^{\ell_2}\bar\zeta_{\bar I^{\bar m,-}(\ell)}^{\bar m,-,I}\right| > (\ln n)^{7}n^{1/4}
 \text{ or }\max_{1 \leq \ell_1 \leq \ell_2 \leq (\ln n)^{10}\sqrt{n}}
 \left|\sum_{\ell=\ell_1}^{\ell_2}(-\bar\Delta_{\bar m,\bar I^{\bar m,-}(\ell)}-1/2)\right| > (\ln n)^{7}n^{1/4}\right\}\right)\!.
\]
 Moreover, thanks to Proposition \ref{prop_def_zetaI}, for any $\ell \geq 1$, $\bar\zeta_{\bar I^{\bar m,-}(\ell)}^{\bar m,-,I}$ has law $\rho_0$ and is independent from 
 $(\bar\zeta_{\bar I^{\bar m,-}(\ell')}^{\bar m,-,I})_{1 \leq \ell' < \ell}$. In addition, thanks to Proposition \ref{prop_law_zeta}, for any $\ell \geq 1$, $-\bar\Delta_{\bar m,\bar I^{\bar m,-}(\ell)}-1/2$ has law $\rho_0$ and is independent from 
 $(-\bar\Delta_{\bar m,\bar I^{\bar m,-}(\ell')}-1/2)_{1 \leq \ell' < \ell}$. 
 Consequently, it is enough to find constants $\tilde C_6'=\tilde C_6'(w) < +\infty$ and 
 $\tilde c_6 > 0$ such that when $n$ is large enough, if $(\zeta_\ell)_{\ell \in \mathds{N}}$ is a sequence 
 of i.i.d. random variables with law $\rho_0$, $\mathds{P}(\max_{1 \leq \ell_1 \leq \ell_2 \leq (\ln n)^{10}\sqrt{n}}
 |\sum_{\ell=\ell_1}^{\ell_2}\zeta_\ell| > (\ln n)^{7}n^{1/4}) \leq \tilde C_6'e^{-\tilde c_6(\ln n)^2}$. 
 
 Let $1 \leq \ell_1 \leq \ell_2 \leq (\ln n)^{10}\sqrt{n}$, we will study 
 $\mathds{P}(|\sum_{\ell=\ell_1}^{\ell_2}\zeta_\ell| > (\ln n)^{7}n^{1/4})$. Since $\rho_0$ is symmetric 
 with respect to 0 and by the Markov inequality, 
 \begin{equation}\label{eq_bad_evt_conv}
 \begin{split}
 \mathds{P}\left(\left|\sum_{\ell=\ell_1}^{\ell_2}\zeta_\ell\right| > (\ln n)^{7}n^{1/4}\right) \leq 
 2\mathds{P}\left(\sum_{\ell=\ell_1}^{\ell_2}\zeta_\ell > (\ln n)^{7}n^{1/4}\right) 
 = 2\mathds{P}\left(\exp\left(\frac{1}{(\ln n)^5n^{1/4}}
 \sum_{\ell=\ell_1}^{\ell_2}\zeta_\ell\right) > \exp\left((\ln n)^2\right)\right)\\
 \leq 2e^{-(\ln n)^2}\mathds{E}\left(\exp\left(\frac{1}{(\ln n)^5n^{1/4}} \sum_{\ell=\ell_1}^{\ell_2}\zeta_\ell\right)\right)
 = 2e^{-(\ln n)^2} \prod_{\ell=\ell_1}^{\ell_2} \mathds{E}\left(\exp\left(\frac{1}{(\ln n)^5n^{1/4}} \zeta_\ell \right)\right), 
 \qquad\qquad\qquad
 \end{split}
 \end{equation}
 so we have to study $\mathds{E}(\exp(\frac{1}{(\ln n)^5n^{1/4}}\zeta))$ where $\zeta$ has law $\rho_0$. 
 Now, we can write that $\exp(\frac{1}{(\ln n)^5n^{1/4}}\zeta)=1+\frac{\zeta}{(\ln n)^5n^{1/4}}
 +\frac{\zeta^2}{2(\ln n)^{10}\sqrt{n}}e^{\zeta'}$, where $|\zeta'| \leq |\frac{\zeta}{(\ln n)^5n^{1/4}}|$, 
 hence 
 \[
 \mathds{E}\left(\exp\left(\frac{1}{(\ln n)^5n^{1/4}}\zeta\right)\right)
 =1+\mathds{E}\left(\frac{\zeta^2}{2(\ln n)^{10}\sqrt{n}}e^{\zeta'}\right)
 \leq 1+\frac{1}{2(\ln n)^{10}\sqrt{n}}\mathds{E}\left(\zeta^2\exp\left(\left|\frac{\zeta}{(\ln n)^5n^{1/4}}\right|\right)\right).
 \]
 Furthermore, $\rho_0$ has exponential tails, so there exist constants 
 $\bar c_6 = \bar c_6(w) > 0$ and $\bar C_6 = \bar C_6(w) < +\infty$ such that 
 $\mathds{E}(\zeta^2e^{\bar c_6 |\zeta|}) \leq \bar C_6$. When $n$ is large enough, $|\frac{\zeta}{(\ln n)^5n^{1/4}}| 
 \leq \bar c_6 |\zeta|$, so $\mathds{E}(\exp(\frac{1}{(\ln n)^5n^{1/4}}\zeta)) \leq 1+\frac{\bar C_6}{2(\ln n)^{10}\sqrt{n}} 
 \leq \exp(\frac{\bar C_6}{2(\ln n)^{10}\sqrt{n}})$. By (\ref{eq_bad_evt_conv}), we deduce that when $n$ is large enough, 
 $\mathds{P}(|\sum_{\ell=\ell_1}^{\ell_2}\zeta_\ell| > (\ln n)^{7}n^{1/4}) \leq 2e^{\bar C_6/2}e^{-(\ln n)^2}$, 
 which suffices. 
\end{proof}

The results of this section can be summed up by the following proposition.

\begin{proposition}\label{prop_bound_bad_evts}
 There exists a constant $c=c(w,\varepsilon)>0$ such that when $n$ is large enough, $\mathds{P}(\bigcup_{r=0}^6 \mathcal{B}_r) \leq e^{-c (\ln n)^2}$. 
\end{proposition}

\begin{proof}
We can write 
\[
 \mathds{P}\left(\bigcup_{r=0}^6 \mathcal{B}_r\right) 
  \leq \mathds{P}(\mathcal{B}_0) + \mathds{P}(\mathcal{B}_0^c \cap \mathcal{B}_1) 
 + \mathds{P}(\mathcal{B}_0^c \cap \mathcal{B}_2) + \mathds{P}(\mathcal{B}_0^c \cap \mathcal{B}_3) 
  + \mathds{P}(\mathcal{B}_0^c \cap \mathcal{B}_1^c \cap \mathcal{B}_3^c \cap \mathcal{B}_4)
 + \mathds{P}(\mathcal{B}_0^c \cap \mathcal{B}_5) + \mathds{P}(\mathcal{B}_0^c \cap \mathcal{B}_4^c \cap \mathcal{B}_6).
\]
Proposition \ref{prop_law_zeta} implies that when $n$ is large enough, $\mathds{P}(\mathcal{B}_0) \leq e^{-c_0n^{(\alpha-1)/4}/2}$. By Proposition \ref{prop_TCL}, $\mathds{P}(\mathcal{B}_0^c \cap \mathcal{B}_1), \mathds{P}(\mathcal{B}_0^c \cap \mathcal{B}_2) \leq e^{-c_1 (\ln n)^2}$ when $n$ is large enough. By Lemma \ref{lem_zeta_E_and_I}, $\mathds{P}(\mathcal{B}_0^c \cap \mathcal{B}_3) \leq e^{-c_3 (\ln n)^2}$ when $n$ is large enough. By Proposition \ref{prop_excursions}, $\mathds{P}(\mathcal{B}_0^c \cap \mathcal{B}_1^c \cap \mathcal{B}_3^c \cap \mathcal{B}_4) \leq e^{-c_4 (\ln n)^2}$ when $n$ is large enough. By Lemma \ref{lem_bound_zeta}, $\mathds{P}(\mathcal{B}_0^c \cap \mathcal{B}_5) \leq e^{-c_5 (\ln n)^2}$ when $n$ is large enough. By Lemma \ref{lem_bound_sums}, $\mathds{P}(\mathcal{B}_0^c \cap \mathcal{B}_4^c \cap \mathcal{B}_6) \leq e^{-c_6 (\ln n)^2}$ when $n$ is large enough. We deduce that if $c=\frac{1}{2}\min(c_1,c_3,c_4,c_5,c_6)$, then $\mathds{P}(\bigcup_{r=0}^6 \mathcal{B}_r) \leq e^{-c (\ln n)^2}$ when $n$ is large enough. 
\end{proof}

\section{A discrete reflected random walk}\label{sec_reflected_RW}

We recall that $\varepsilon > 0$, that the $\zeta_j^{m,\pm,B}$, $\zeta_j^{m,\pm,E}$ were defined in Definition \ref{def_local_times} and the $\zeta_j^{m,\pm,I}$ before Proposition \ref{prop_def_zetaI}. Our goal in this section is to prove that $\sum \zeta_j^{m,-,E}$ (with a corresponding statement for $\sum \zeta_j^{m,+,E}$) behaves roughly as a ``random walk reflected on $\sum \zeta_j^{m,-,B}$''. In order to do that, we will introduce a discrete process $S^{m,-,I}$ that is roughly ``the random walk $\sum \zeta_j^{m,-,I}$ reflected on $\sum \zeta_j^{m,-,B}$'' (Definition \ref{def_SI}), and prove that if the bad events $\mathcal{B}_{m,1}^-,...,\mathcal{B}_{m,6}^-$ defined at the beginning of Section \ref{sec_bad_events} do not occur, then $S^{m,-,I}$ is very close to $\sum \zeta_j^{m,-,E}$ (Proposition \ref{prop_sym}). When $\sum \zeta_j^{m,-,E}$ is far above $\sum \zeta_j^{m,-,B}$, it will evolve similarly to $S^{m,-,I}$, thus the hard part will be to deal with what happens near $\sum \zeta_j^{m,-,B}$. We begin by recalling the definition of the reflected Brownian motion (the definition of a discrete-time reflected random walk is similar).  

\begin{definition}\label{def_reflected_BM}
 Let $a <b$ be real numbers, $f : [a,b] \mapsto \mathds{R}$ a continuous function, and $(W_t)_{t\in[a,b]}$ a Brownian motion so that $W_a \geq f(a)$. The reflection of $W$ on $f$ is the process $W'$ defined as follows. If $W_a=f(a)$, for all $t\in[a,b]$ we set $W_t'=W_t+\sup_{a \leq s \leq t}(f(s)-W_s)$. If $W_a > f(a)$, if $t_0$ denotes $b \wedge \inf\{t\in[a,b] \,|\, W_t=f(t)\}$, for $t\in[a,t_0]$ we set $W_t'=W_t$, and for $t\in[t_0,b]$ we set $W_t'=W_t+\sup_{t_0 \leq s \leq t}(f(s)-W_s)$. If $f$ is random, a Brownian motion reflected on $f$ without further precision will be the reflection on $f$ of a Brownian motion independent of $f$.
\end{definition}

We now introduce the following notation.

\begin{definition}\label{def_environments}
 For any $m \in \mathds{N}$, we will define processes $(S^{m,-,B}_i)_{i > X_m-\lfloor \varepsilon n \rfloor}$, 
 $(S^{m,-,E}_i)_{i > X_m-\lfloor \varepsilon n \rfloor}$, $(S^{m,+,B}_i)_{i \leq X_m+\lfloor \varepsilon n \rfloor}$, 
 $(S^{m,+,E}_i)_{i \leq X_m+\lfloor \varepsilon n \rfloor}$ so that for $\Xi=B$ or $E$, 
 \begin{gather*}
 S^{m,-,\Xi}_{X_m-\lfloor \varepsilon n \rfloor+1}=0 \text{ and }
 \forall\, i \geq X_m-\lfloor \varepsilon n \rfloor+1, S^{m,-,\Xi}_{i+1} = S^{m,-,\Xi}_i+\zeta^{m,-,\Xi}_i \\
 S^{m,+,\Xi}_{X_m+\lfloor \varepsilon n \rfloor} = 0  \text{ and }
 \forall\, i \leq X_m+\lfloor \varepsilon n \rfloor-1, S^{m,+,\Xi}_i = S^{m,+,\Xi}_{i+1}+\zeta^{m,+,\Xi}_i. 
 \end{gather*}
 We then have $L_i^{m,-}=S^{m,-,E}_i-S^{m,-,B}_i$ for $i > X_m-\lfloor \varepsilon n \rfloor$ 
 and $L_i^{m,+}=S^{m,+,E}_i-S^{m,+,B}_i+1$ for $i \leq X_m+\lfloor \varepsilon n \rfloor$. 
\end{definition}

\begin{definition}\label{def_SI}
 For any $m \in \mathds{N}$, we define the processes $(S^{m,-,I}_i)_{X_m-\lfloor \varepsilon n \rfloor < i \leq X_m}$ 
 and $(S^{m,+,I}_i)_{X_m < i \leq X_m+\lfloor \varepsilon n \rfloor}$ by 
 \begin{gather*}
  S^{m,-,I}_{X_m-\lfloor \varepsilon n \rfloor+1}=0 \text{ and }
  \forall\, i \in \{X_m-\lfloor \varepsilon n \rfloor+1,\dots,X_m-1\}, 
  S^{m,-,I}_{i+1}=\left\{\begin{array}{ll}
                       S^{m,-,I}_i+\zeta_i^{m,-,I} & \text{if }S^{m,-,I}_i+\zeta_i^{m,-,I} \geq S^{m,-,B}_{i+1},\\
                       S^{m,-,B}_{i+1} & \text{otherwise},
                              \end{array}\right.\\
  S^{m,+,I}_{X_m+\lfloor \varepsilon n \rfloor}=0 \text{ and }
  \forall\, i \in \{X_m+1,\dots,X_m+\lfloor \varepsilon n \rfloor-1\}, 
  S^{m,+,I}_{i}=\left\{\begin{array}{ll}
                       S^{m,+,I}_{i+1}+\zeta_i^{m,+,I} & 
                       \text{if }S^{m,+,I}_{i+1}+\zeta_i^{m,+,I} \geq S^{m,+,B}_{i},\\
                       S^{m,+,B}_{i} & \text{otherwise}.
                              \end{array}\right.
 \end{gather*}
\end{definition}

The following lemma shows that ``$S^{m,\pm,I}$ is the random walk $\sum \zeta_j^{m,\pm,I}$ reflected on $\sum \zeta_j^{m,\pm,B}$''.

\begin{lemma}\label{lem_SI_max}
 For any $m\in\mathds{N}$, we have that for all $i\in\{X_m-\lfloor \varepsilon n \rfloor+1,...,X_m\}$, 
 \[
 S^{m,-,I}_i=\sum_{j=X_m-\lfloor \varepsilon n \rfloor+1}^{i-1}\zeta_j^{m,-,I}+\max_{X_m-\lfloor \varepsilon n \rfloor+1 \leq j \leq i}\left(S^{m,-,B}_{j}-\sum_{j'=X_m-\lfloor \varepsilon n \rfloor+1}^{j-1}\zeta_{j'}^{m,-,I}\right),
 \]
 and for all $i\in\{X_m+1,\dots,X_m+\lfloor \varepsilon n \rfloor\}$,
 \[
 S^{m,+,I}_i=\sum_{j=i}^{X_m+\lfloor \varepsilon n \rfloor-1}\zeta_j^{m,+,I}+\max_{i\leq j \leq X_m+\lfloor \varepsilon n \rfloor}\left(S^{m,+,B}_{j}-\sum_{j'=j}^{X_m+\lfloor \varepsilon n \rfloor-1}\zeta_{j'}^{m,+,I}\right).
 \]
\end{lemma}

\begin{proof}
 Let $m \in \mathds{N}$. We will write the proof for $S^{m,-,I}$; the same argument also applies to $S^{m,+,I}$. To shorten the notation, we will drop the exponents $m,-$, and write $i_1=X_m-\lfloor \varepsilon n \rfloor+1$, $i_2=X_m$. We thus want to prove that for each $i\in\{i_1,...,i_2\}$ we have 
 \[
 S^{I}_i=\sum_{j=i_1}^{i-1}\zeta_j^{I}+\max_{i_1 \leq j \leq i}\left(S^{B}_{j}-\sum_{j'=i_1}^{j-1}\zeta_{j'}^{I}\right).
 \]
 We will prove it by induction on $i$. For $i=i_1$, this comes from the definition of the processes. Now let $i\in\{i_1,...,i_2-1\}$ so that $S^{I}_i=\sum_{j=i_1}^{i-1}\zeta_j^{I}+\max_{i_1 \leq j \leq i}(S^{B}_{j}-\sum_{j'=i_1}^{j-1}\zeta_{j'}^{I})$. There are two possibilities. 
 
 The first possibility is $S_i^I+\zeta_i^I \geq S_{i+1}^B$. In this case, $S_{i+1}^I=S_i^I+\zeta_i^I$. Moreover, we have $\sum_{j=i_1}^{i-1}\zeta_j^{I}+\max_{i_1 \leq j \leq i}(S^{B}_{j}-\sum_{j'=i_1}^{j-1}\zeta_{j'}^{I})+\zeta_i^I\geq S_{i+1}^B$, hence $\max_{i_1 \leq j \leq i}(S^{B}_{j}-\sum_{j'=i_1}^{j-1}\zeta_{j'}^{I}) \geq S_{i+1}^B-\sum_{j=i_1}^{i}\zeta_{j}^{I}$, thus $\max_{i_1 \leq j \leq i+1}(S^{B}_{j}-\sum_{j'=i_1}^{j-1}\zeta_{j'}^{I})=\max_{i_1 \leq j \leq i}(S^{B}_{j}-\sum_{j'=i_1}^{j-1}\zeta_{j'}^{I})$, so $\sum_{j=i_1}^{i}\zeta_{j}^{I}+\max_{i_1 \leq j \leq i+1}(S^{B}_{j}-\sum_{j'=i_1}^{j-1}\zeta_{j'}^{I})=\sum_{j=i_1}^{i-1}\zeta_{j}^{I}+\max_{i_1 \leq j \leq i}(S^{B}_{j}-\sum_{j'=i_1}^{j-1}\zeta_{j'}^{I})+\zeta_i^I=S_i^I+\zeta_i^I=S_{i+1}^{I}$, which is what we want. 
 
 The other possibility is $S_i^I+\zeta_i^I < S_{i+1}^B$. In this case, $S_{i+1}^I=S_{i+1}^B$. Furthermore, we have $\sum_{j=i_1}^{i-1}\zeta_j^{I}+\max_{i_1 \leq j \leq i}(S^{B}_{j}-\sum_{j'=i_1}^{j-1}\zeta_{j'}^{I})+\zeta_i^I < S_{i+1}^B$, so $\max_{i_1 \leq j \leq i}(S^{B}_{j}-\sum_{j'=i_1}^{j-1}\zeta_{j'}^{I}) < S_{i+1}^B-\sum_{j=i_1}^{i}\zeta_j^{I}$, hence $\max_{i_1 \leq j \leq i+1}(S^{B}_{j}-\sum_{j'=i_1}^{j-1}\zeta_{j'}^{I}) = S_{i+1}^B-\sum_{j=i_1}^{i}\zeta_j^{I}$. We deduce $\sum_{j=i_1}^{i}\zeta_j^{I}+\max_{i_1 \leq j \leq i+1}(S^{B}_{j}-\sum_{j'=i_1}^{j-1}\zeta_{j'}^{I})=\sum_{j=i_1}^{i}\zeta_j^{I}+S_{i+1}^B-\sum_{j=i_1}^{i}\zeta_j^{I}=S_{i+1}^B=S_{i+1}^I$, which is the desired result. 
\end{proof}

The following proposition is the main result of the section: if the bad events do not occur, $S^{m,\pm,I}$ is close to $S^{m,\pm,E}$.

\begin{proposition}\label{prop_sym}
 When $n$ is large enough, for any $m \in \mathds{N}$, if $\bigcap_{r=1}^6 (\mathcal{B}_{m,r}^-)^c$ occurs then 
 for all $i \in \{X_m-\lfloor \varepsilon n \rfloor+1,\dots,X_m\}$, $S^{m,-,I}_i-(\ln n)^8 n^{1/4} \leq 
 S^{m,-,E}_i \leq S^{m,-,I}_i+\lceil(\ln n)^3\rceil$, and if $\bigcap_{r=1}^6 (\mathcal{B}_{m,r}^+)^c$ occurs then 
 for all $i \in \{X_m+1,\dots,X_m+\lfloor\varepsilon n\rfloor\}$, $S^{m,+,I}_i-(\ln n)^8 n^{1/4} \leq 
 S^{m,+,E}_i \leq S^{m,+,I}_i+\lceil(\ln n)^3\rceil$.
\end{proposition}

\begin{proof}
 Let $m \in \mathds{N}$. We will write the proof for $S^{m,-,E}$; the same argument also applies to $S^{m,+,E}$. 
 In order to lighten the notation, we will drop the exponents $m,-$, and write $i_1=X_m-\lfloor \varepsilon n \rfloor+1$, $i_2=X_m$. 
 
 The idea of the proof is that when $S_i^E \geq S_i^B+(\ln n)^3$, then $L_i \geq (\ln n)^3$ by Definition \ref{def_environments}, thus since $\mathcal{B}_{m,3}^-$ holds we have $\zeta_i^E=\zeta_i^I$, therefore $S_{i+1}^E=S_i^E+\zeta_i^E=S_i^E+\zeta_i^I$. Now, if $S_i^I$ is not too close to $S_i^B$ we also have $S^I_{i+1}=S_i^I+\zeta_i^I$, so $S^E$ and $S^I$ evolve in the same way. Consequently, the difference between $S^E$ and $S^I$ comes only from the $i$ such that $L_i < (\ln n)^3$, and the fact the bad events do not occur will imply the difference thus accrued is small. In order to make this argument work, we need to show that when $L_i \geq (\ln n)^3$, $S_i^I$ is not too close to $S_i^B$. However, it may not actually be the case for all $i$. To solve this problem, we will actually use the aforementioned argument with some processes $S'$ and $S''$, which will respectively be close to $S^E$ and $S^I$.
 
 We begin by proving that $S^E$ is close to the auxiliary process $S'$ defined 
 for $i \in \{i_1,\dots,i_2\}$ by $S_i' = \max(S^E_i,S^B_i+\lceil(\ln n)^3\rceil)$. 
 For the $i$ such that $S_i'=S^E_i$, it is obvious. If $i$ is such that $S_i'=S^B_i+\lceil(\ln n)^3\rceil$, 
 then we have $S^E_i \leq S^B_i+\lceil(\ln n)^3\rceil$. Moreover, by Definition \ref{def_environments} 
 $S^E_i-S^B_i=L_i \geq 0$, so $S^B_i \leq S^E_i \leq S^B_i+\lceil(\ln n)^3\rceil$, which means 
 $S_i' - \lceil(\ln n)^3\rceil \leq S^E_i \leq S_i'$. We deduce 
 \begin{equation}\label{eq_reflected_easy}
 \forall\, i \in \{i_1,\dots,i_2\}, \quad S_i' - \lceil(\ln n)^3\rceil \leq S^E_i \leq S_i'.
 \end{equation}
 
 We now prove that $S^I$ is close to an auxiliary process $S''$ which will be ``the random walk $\sum\zeta_i^I$ reflected on $S^{B}+\lceil(\ln n)^3\rceil$''. More precisely, $S_i''$ is defined for 
 $i \in \{i_1,\dots,i_2\}$ as follows: $S_{i_1}'' = \lceil(\ln n)^3\rceil$, and for any $i \in \{i_1,\dots,i_2-1\}$, 
 \[
 S_{i+1}''=\left\{\begin{array}{ll}
                S_i'' +\zeta_i^I &\text{if }S_i'' +\zeta_i^I \geq S^{B}_{i+1}+\lceil(\ln n)^3\rceil,\\
                S^{B}_{i+1}+\lceil(\ln n)^3\rceil&\text{otherwise.}
               \end{array}\right.
\]
 Since $S^I$ is ``the random walk $\sum\zeta_i^I$ reflected on $S^{B}$'' and $S''$ is ``the random walk $\sum\zeta_i^I$ reflected on $S^{B}+\lceil(\ln n)^3\rceil$'', we can expect $S^I$ and $S''$ to be close. We are going to prove by induction on $i \in \{i_1,\dots,i_2\}$ that 
 $S_i^I \leq S_i'' \leq S_i^I + \lceil(\ln n)^3\rceil$. It is true for $i=i_1$ by the definition of the processes. 
 We now suppose it is true for some $i \in \{i_1,\dots,i_2-1\}$ and prove it for $i+1$. If $S_i^I+\zeta_i^I \geq S^B_{i+1}$ and 
 $S_i''+\zeta_i^I \geq S^B_{i+1}+\lceil(\ln n)^3\rceil$, then $S_{i+1}^I - S_{i+1}'' = 
 (S_i^I+\zeta_i^I)-(S_i''+\zeta_i^I)=S_i^I-S_i''$, which is enough. 
 If $S_i^I+\zeta_i^I \geq S^B_{i+1}$ and $S_i''+\zeta_i^I < S^B_{i+1}+\lceil(\ln n)^3\rceil$, 
 then $S_{i+1}^I = S_i^I+\zeta_i^I \leq S_i''+\zeta_i^I < S^B_{i+1}+\lceil(\ln n)^3\rceil=S_{i+1}''$, 
 thus $S_{i+1}^I \leq S_{i+1}''$, and $S_{i+1}'' = S^B_{i+1}+\lceil(\ln n)^3\rceil \leq 
 S_i^I+\zeta_i^I+\lceil(\ln n)^3\rceil=S_{i+1}^I+\lceil(\ln n)^3\rceil$, which is enough. 
 If $S_i^I+\zeta_i^I < S^B_{i+1}$ and $S_i''+\zeta_i^I \geq S^B_{i+1}+\lceil(\ln n)^3\rceil$, 
 then $S_i''-S_i^I > \lceil(\ln n)^3\rceil$, so this case is impossible. Finally, 
 if $S_i^I+\zeta_i^I < S^B_{i+1}$ and $S_i''+\zeta_i^I < S^B_{i+1}+\lceil(\ln n)^3\rceil$, 
 then $S_{i+1}^I=S_{i+1}^B$ and $S_{i+1}'' = S^{B}_{i+1}+\lceil(\ln n)^3\rceil$, which is enough. We deduce that 
 \begin{equation}\label{eq_reflected_medium}
 \forall\, i \in \{i_1,\dots,i_2\}, \quad S_i^I \leq S_i'' \leq S_i^I + \lceil(\ln n)^3\rceil.
 \end{equation}

 We are now able to show that the only difference between $S'$ and $S''$ comes from the $i$ such that $L_i < (\ln n)^3$. We denote $\ell(i_1)=0$, and for any $i \in \{i_1+1,\dots,i_2\}$, 
 $\ell(i)=|\{j \in \{i_1,\dots,i-1\} \,|\, L_j < (\ln n)^3\}|$. 
 We are going to prove the following by induction on $i \in \{i_1,\dots,i_2\}$:
 \begin{equation}\label{eq_reflected_hard}
  S_i' \leq S_i'' \leq S_i' +\left(\max_{1 \leq \ell_1 \leq \ell(i)}
  \sum_{\ell_2=\ell_1}^{\ell(i)}\left(\zeta_{I(\ell_2)}^I-\zeta^B_{I(\ell_2)}\right)\right)_+,
 \end{equation}
 where the maximum is 0 if $\ell(i)=0$. For $i=i_1$, we have 
 $S_{i_1}'=S_{i_1}''=\lceil(\ln n)^3\rceil$, so (\ref{eq_reflected_hard}) holds. 
 Now, let $i \in \{i_1,\dots,i_2-1\}$ and suppose (\ref{eq_reflected_hard}) holds for $i$. We will prove that it holds 
 also for $i+1$. 
 
 We first consider the case $L_i \geq (\ln n)^3$. \\
 In this case, $\ell(i+1)=\ell(i)$, so it is enough to 
 prove $0 \leq S_{i+1}'' - S_{i+1}' \leq S_i'' - S_i'$. We notice first that since $L_i \geq (\ln n)^3$, 
 $S_i^E-S_i^B=L_i \geq \lceil(\ln n)^3\rceil$, so $S_i^E \geq S_i^B+\lceil(\ln n)^3\rceil$, so $S_i'=S_i^E$. 
 We also notice that since $L_i \geq (\ln n)^2$ and $(\mathcal{B}_{m,3}^-)^c$ occurs, $\zeta_i^E=\zeta_i^I$. \\
We begin by assuming $L_{i+1} \geq (\ln n)^3$. Then $S^E_{i+1}-S^B_{i+1} \geq \lceil(\ln n)^3\rceil$, so 
$S_{i+1}' = S^E_{i+1}$. This implies $S_{i+1}'=S_i^E+\zeta_i^E=S_i'+\zeta_i^I$. 
Moreover, $S_i''+\zeta_i^I \geq S_i'+\zeta_i^I 
= S_{i+1}' = S^E_{i+1} \geq S^B_{i+1}+\lceil(\ln n)^3\rceil$, so $S_{i+1}''=S_i''+\zeta_i^I$. 
This yields $S_{i+1}'' - S_{i+1}' = S_i'' - S_i'$, which is enough. \\
We now assume $L_{i+1} < (\ln n)^3$. Then $S^E_{i+1} - S^B_{i+1} < \lceil (\ln n)^3 \rceil$, so $S_{i+1}' = S^B_{i+1}+\lceil (\ln n)^3 \rceil$. 
If $S_i'' +\zeta_i^I < S^{B}_{i+1}+\lceil(\ln n)^3\rceil$, $S_{i+1}'' = S^{B}_{i+1}+\lceil(\ln n)^3\rceil$, 
so $S_{i+1}'' - S_{i+1}' = 0$, which is enough. If $S_i'' +\zeta_i^I \geq S^{B}_{i+1}+\lceil(\ln n)^3\rceil$, 
$S_{i+1}'' = S_i'' +\zeta_i^I$.
Furthermore, $S_{i+1}' \geq S_{i+1}^E=S_i^E+\zeta_i^E=S_i'+\zeta_i^I$. 
We deduce $S_{i+1}''-S_{i+1}' \leq S_i''+\zeta_i^I -(S_i'+\zeta_i^I) = S_i'' -S_i'$. 
In addition, $S_{i+1}'' = S_i'' +\zeta_i^I \geq S^{B}_{i+1}+\lceil(\ln n)^3\rceil=S_{i+1}'$, 
so $S_{i+1}''-S_{i+1}' \geq 0$, which is enough. Consequently, (\ref{eq_reflected_hard}) holds for $i+1$ in the case 
$L_i \geq (\ln n)^3$. 

We now consider the case $L_i < (\ln n)^3$. \\
We first show that $S_{i+1}'' \geq S_{i+1}'$. If $L_{i+1} < (\ln n)^3$, $S_{i+1}^E < S_{i+1}^B+\lceil(\ln n)^3\rceil$ 
so $S_{i+1}' = S_{i+1}^B+\lceil(\ln n)^3\rceil \leq S_{i+1}''$. If $L_{i+1} \geq (\ln n)^3$, we notice that 
$L_{i+1}=L_i+\zeta_i^E-\zeta_i^B$ by Observation \ref{obs_rec_temps_local}. Moreover, since $(\mathcal{B}_{m,5}^-)^c$ 
occurs, we have $|\zeta_i^E|,|\zeta_i^B| \leq (\ln n)^2$, so $L_i \geq (\ln n)^3-2(\ln n)^2 \geq (\ln n)^2$ when 
$n$ is large enough. Thus, since $(\mathcal{B}_{m,3}^-)^c$ occurs, $\zeta_i^E=\zeta_i^I$. We deduce 
$S^{B}_{i+1}+\lceil(\ln n)^3\rceil \leq S_{i+1}^E = S_i^E + \zeta_i^I \leq S_i' + \zeta_i^I 
\leq S_i'' + \zeta_i^I$, so $S_{i+1}'' = S_i'' + \zeta_i^I$. Furthermore, since $S_{i+1}^E \geq 
S^{B}_{i+1}+\lceil(\ln n)^3\rceil$, $S_{i+1}' = S_{i+1}^E \leq S_{i+1}''$. Therefore $S_{i+1}' \leq S_{i+1}''$ 
in all cases. \\
 We now show that we have 
 \[
 S_{i+1}'' - S_{i+1}' \leq \left(\max_{1 \leq \ell_1 \leq \ell(i+1)}
  \sum_{\ell_2=\ell_1}^{\ell(i+1)}\left(\zeta_{I(\ell_2)}^I-\zeta^B_{I(\ell_2)}\right)\right)_{+} = \left(\max_{1 \leq \ell_1 \leq \ell(i)+1}
  \sum_{\ell_2=\ell_1}^{\ell(i)+1}\left(\zeta_{I(\ell_2)}^I-\zeta^B_{I(\ell_2)}\right)\right)_{+}
  \]
  since $\ell(i+1)=\ell(i)+1$. 
  If $S_i'' +\zeta_i^I < S^{B}_{i+1}+\lceil(\ln n)^3\rceil$, $S_{i+1}'' = S^{B}_{i+1}+\lceil(\ln n)^3\rceil 
  \leq S_{i+1}'$, so $S_{i+1}''-S_{i+1}' \leq 0$, which is enough. Hence we consider the case 
  $S_i'' +\zeta_i^I \geq S^{B}_{i+1}+\lceil(\ln n)^3\rceil$. We have $S_{i+1}' \geq S^{B}_{i+1}+\lceil(\ln n)^3\rceil$, 
  so $S_{i+1}'' - S_{i+1}' \leq S_i'' +\zeta_i^I - S^{B}_{i+1} - \lceil(\ln n)^3\rceil 
  = S_i'' +\zeta_i^I - S^{B}_{i} - \lceil(\ln n)^3\rceil -\zeta_i^B$. Furthermore, since $L_i < (\ln n)^3$, 
  $S_i' = S_i^B+\lceil(\ln n)^3\rceil$, so we get $S_{i+1}'' - S_{i+1}' \leq 
  S_i'' - S_i' +\zeta_i^I -\zeta_i^B$. In addition, $I(\ell(i+1))=I(\ell(i)+1)=i$, thus we have 
  \begin{equation}\label{eq_reflected_aux}
  S_{i+1}'' - S_{i+1}' \leq S_i'' - S_i' +\zeta_{I(\ell(i)+1)}^I -\zeta_{I(\ell(i)+1)}^B. 
  \end{equation}
  We first assume $\max_{1 \leq \ell_1 \leq \ell(i)}\sum_{\ell_2=\ell_1}^{\ell(i)}(\zeta_{I(\ell_2)}^I-\zeta^B_{I(\ell_2)}) \geq 0$. 
  Then 
  \[
  \max_{1 \leq \ell_1 \leq \ell(i)}\sum_{\ell_2=\ell_1}^{\ell(i)+1}\left(\zeta_{I(\ell_2)}^I-\zeta^B_{I(\ell_2)}\right) \geq 
  \zeta_{I(\ell(i)+1)}^I -\zeta_{I(\ell(i)+1)}^B,
  \]
  so we have 
  \[
  \max_{1 \leq \ell_1 \leq \ell(i)+1}
  \sum_{\ell_2=\ell_1}^{\ell(i)+1}\left(\zeta_{I(\ell_2)}^I-\zeta^B_{I(\ell_2)}\right) = 
  \max_{1 \leq \ell_1 \leq \ell(i)}\sum_{\ell_2=\ell_1}^{\ell(i)+1}\left(\zeta_{I(\ell_2)}^I-\zeta^B_{I(\ell_2)}\right).
  \]
  Therefore, by (\ref{eq_reflected_hard}) and (\ref{eq_reflected_aux}), we obtain $S_{i+1}'' - S_{i+1}' \leq 
  \max_{1 \leq \ell_1 \leq \ell(i)+1}\sum_{\ell_2=\ell_1}^{\ell(i)+1}(\zeta_{I(\ell_2)}^I-\zeta^B_{I(\ell_2)})$, 
  which is enough. \\
  We now assume that $\max_{1 \leq \ell_1 \leq \ell(i)}
  \sum_{\ell_2=\ell_1}^{\ell(i)}(\zeta_{I(\ell_2)}^I-\zeta^B_{I(\ell_2)}) \leq 0$. 
  Then (\ref{eq_reflected_hard}) yields $S_i'' - S_i' \leq 0$, so by (\ref{eq_reflected_aux}) 
  $S_{i+1}'' - S_{i+1}' \leq \zeta_{I(\ell(i)+1)}^I -\zeta_{I(\ell(i)+1)}^B$. In addition, 
  \[
  \max_{1 \leq \ell_1 \leq \ell(i)} \sum_{\ell_2=\ell_1}^{\ell(i)+1}\left(\zeta_{I(\ell_2)}^I-\zeta^B_{I(\ell_2)}\right) \leq 
  \zeta_{I(\ell(i)+1)}^I -\zeta_{I(\ell(i)+1)}^B,
  \]
  so 
  \[
  \max_{1 \leq \ell_1 \leq \ell(i)+1} 
  \sum_{\ell_2=\ell_1}^{\ell(i)+1}\left(\zeta_{I(\ell_2)}^I-\zeta^B_{I(\ell_2)}\right) = \zeta_{I(\ell(i)+1)}^I -\zeta_{I(\ell(i)+1)}^B.
  \]
  We deduce $S_{i+1}'' - S_{i+1}' \leq \max_{1 \leq \ell_1 \leq \ell(i)+1}
  \sum_{\ell_2=\ell_1}^{\ell(i)+1}(\zeta_{I(\ell_2)}^I-\zeta^B_{I(\ell_2)})$, which is enough. \\
  Consequently, (\ref{eq_reflected_hard}) holds for $i+1$ in the case $L_i < (\ln n)^3$. 
  
  We deduce that (\ref{eq_reflected_hard}) holds for any $i \in \{i_1,\dots,i_2\}$. Moreover, for any $i \in \{i_1,\dots,i_2\}$, 
  since $(\mathcal{B}_{m,6}^-)^c$ occurs, we have 
  \[
  \left(\max_{1 \leq \ell_1 \leq \ell(i)}\sum_{\ell_2=\ell_1}^{\ell(i)}(\zeta_{I(\ell_2)}^I-\zeta^B_{I(\ell_2)})\right)_+ 
  \leq \left|\max_{1 \leq \ell_1 \leq \ell(i)}\sum_{\ell_2=\ell_1}^{\ell(i)}(\zeta_{I(\ell_2)}^I-\zeta^B_{I(\ell_2)})\right| 
  \leq \max_{1 \leq \ell_1 \leq \ell(i)}\left|\sum_{\ell_2=\ell_1}^{\ell(i)}\zeta_{I(\ell_2)}^I\right|
  + \max_{1 \leq \ell_1 \leq \ell(i)}\left|\sum_{\ell_2=\ell_1}^{\ell(i)}\zeta^B_{I(\ell_2)}\right| 
  \]
  \[
  \leq \max_{1 \leq \ell_1 \leq \ell_2 \leq \ell_{\textrm{max}}}\left|\sum_{\ell=\ell_1}^{\ell_2}\zeta_{I(\ell)}^I\right|
  + \max_{1 \leq \ell_1 \leq \ell_2 \leq \ell_{\textrm{max}}}\left|\sum_{\ell=\ell_1}^{\ell_2}\zeta^B_{I(\ell)}\right|
  \leq 2 (\ln n)^7 n^{1/4},
  \]
  therefore $S_i' \leq S_i'' \leq S_i' +2 (\ln n)^7 n^{1/4}$. From this, (\ref{eq_reflected_easy}) and 
  (\ref{eq_reflected_medium}), we deduce that for any $i \in \{i_1,\dots,i_2\}$, 
  $S_i^I-2 (\ln n)^7 n^{1/4}-\lceil(\ln n)^3\rceil \leq S_i^E \leq S_i^I + \lceil(\ln n)^3\rceil$, 
  so when $n$ is large enough, $S_i^I -(\ln n)^8 n^{1/4} \leq S_i^E \leq S_i^I + \lceil(\ln n)^3\rceil$, 
  which ends the proof. 
\end{proof}
 
 \section{Lower bounds on the $T_K-T_0$}\label{sec_lower_bounds_Tk}
 
 We recall the stopping times $T_k$ defined in \eqref{eq_def_Tk}, as well as the ``bad events'' $\mathcal{B}$ defined in Proposition \ref{prop_bound_m} and $\mathcal{B}_0,...,\mathcal{B}_6$ defined at the beginning of Section \ref{sec_bad_events}. The goal of the current section is to prove that if the bad events do not happen, then for any $K \in \mathds{N}$, $T_K-T_0$ is at least of order $K n^{3/2}$: there exists a constant $\delta>0$ so that $\mathds{P}(T_K-T_0 < \delta K n^{3/2},\mathcal{B}^c \cap \bigcap_{r=0}^6 \mathcal{B}_r^c) \leq \frac{1}{2^K}$ (Proposition \ref{prop_T_K}). We stress that we will not try to prove that each $T_{k+1}-T_k$, $k\in\{0,...,K-1\}$ is large, since it is very possible that for some $k$ the configuration at time $T_k$ is bad enough to prevent it. However, a combinatorial argument will allow us to prove that a constant proportion of the $k$ satisfy that $T_{k+1}-T_k$ is large, which will be enough. This is one of the hardest parts of the work, and the most novel one. Let us give some ideas of the proof.
 
 Remember that $\varepsilon > 0$, that the $\beta_k^\pm$ were defined in \eqref{eq_def_Tk_pm^} and the $\zeta_i^{m,\pm,B},\zeta_i^{m,\pm,E},L_i^{m,\pm}$ in Definition \ref{def_local_times}. For any $k\in\{0,...,K-1\}$, if (say) $X_{T_{k+1}}=X_{T_{k}}-\lfloor\varepsilon n\rfloor$, then $T_{k+1}-T_{k} = \beta_{T_{k}}^- - T_{k} \geq \sum_{i=i_1}^{i_2} L_{i+1}^{T_{k},-}$ for any $\{i_1,...,i_2\} \subset \{X_{T_{k+1}}+1,...,X_{T_{k}}\}$. Now, if $i\in\{i_1,...,i_2\}$, Observation \ref{obs_rec_temps_local} yields $L_{i+1}^{T_{k},-}=L_{i_1}^{T_{k},-}+\sum_{j=i_1}^{i}(\zeta_j^{T_{k},-,E}-\zeta_j^{T_{k},-,B})$, and $L_{i_1}^{T_{k},-} \geq 0$, thus $L_{i+1}^{T_{k},-}\geq\sum_{j=i_1}^{i}(\zeta_j^{T_{k},-,E}-\zeta_j^{T_{k},-,B})$, so we obtain $T_{k+1}-T_{k} \geq \sum_{i=i_1}^{i_2}\sum_{j=i_1}^{i} (\zeta_j^{T_{k},-,E}-\zeta_j^{T_{k},-,B}) = \sum_{i=i_1}^{i_2}\sum_{j=i_1}^{i} \zeta_j^{T_{k},-,E}-\sum_{i=i_1}^{i_2}\sum_{j=i_1}^{i}\zeta_j^{T_{k},-,B}$. Therefore, if for some constant $\delta>0$ we have that $\sum_{i=i_1}^{i_2}\sum_{j=i_1}^{i} \zeta_j^{T_{k},-,E} \geq \delta n^{3/2}$ and $\sum_{i=i_1}^{i_2}\sum_{j=i_1}^{i} \zeta_j^{T_{k},-,B} \leq - \delta n^{3/2}$, then it guarantees $T_{k+1}-T_{k} \geq 2\delta n^{3/2}$. If this is true for a positive fraction of the $k\in\{0,...,K-1\}$, then $T_K-T_0$ will be of order $K n^{3/2}$. 
 
 The $\sum_{i=i_1}^{i_2}\sum_{j=i_1}^{i} \zeta_j^{T_{k},-,E}$ will be rather easy to control if we remember the $\zeta_i^{m,-,I}$ constructed just before Proposition \ref{prop_def_zetaI} and the $S^{m,-,B}$, $S^{m,-,E}$, $S^{m,-,I}$ defined in Definitions \ref{def_environments} and \ref{def_SI}. Indeed, Proposition \ref{prop_sym} indicates that $\sum_{i=i_1}^{i_2}\sum_{j=i_1}^{i} \zeta_j^{T_{k},-,E} = \sum_{i=i_1}^{i_2}(S_{i}^{T_{k},-,E}-S_{i_1-1}^{T_{k},-,E})$ will be close to $\sum_{i=i_1}^{i_2}(S_{i}^{T_{k},-,I}-S_{i_1-1}^{T_{k},-,I})$. Now, $S^{T_{k},-,I}$ is ``the random walk $\sum \zeta_j^{T_{k},-,I}$ reflected on $S^{T_{k},-,B}$'', hence $\sum_{i=i_1}^{i_2}(S_{i}^{T_{k},-,I}-S_{i_1-1}^{T_{k},-,I}) \geq \sum_{i=i_1}^{i_2}\sum_{j=i_1}^{i} \zeta_j^{T_{k},-,I}$, so it is enough to prove that we have $\sum_{i=i_1}^{i_2}\sum_{j=i_1}^{i} \zeta_j^{T_{k},-,I}\geq \delta n^{3/2}$. Since by Proposition \ref{prop_def_zetaI} the $\zeta_j^{T_{k},-,I}$ are i.i.d. with the law $\rho_0$ defined in \eqref{eq_def_rho0}, $\sum_{i=i_1}^{i_2}\sum_{j=i_1}^{i} \zeta_j^{T_{k},-,I}$ is basically the integral of the i.i.d. random walk $\sum_{j=i_1}^{i} \zeta_j^{T_{k},-,I}$ on the interval $\{i_1,...,i_2\}$, so if $i_2-i_1$ is of order $n$, there is a positive probability to have $\sum_{j=i_1}^{i} \zeta_j^{T_{k},-,I}$ of order $\sqrt{n}$, hence to have $\sum_{i=i_1}^{i_2}\sum_{j=i_1}^{i} \zeta_j^{T_{k},-,I}\geq \delta n^{3/2}$. 
 
 However, we also have to control the $\sum_{i=i_1}^{i_2}\sum_{j=i_1}^{i} \zeta_j^{T_{k},-,B}$, which depend on the $\Delta_{T_{k},j}$ (defined in \eqref{eq_def_Delta}), and this is harder. If $X_{T_{k-1}}=X_{T_{k}}-\lfloor\varepsilon n\rfloor=X_{T_{k+1}}$ (i.e. the mesoscopic process $(X_{T_{k'}})_{k'\in\mathds{N}}$ is doing a U-turn), then for $j\in\{X_{T_{k+1}}+1,...,X_{T_{k}}\}$ we have $\zeta_j^{T_{k},-,B}=\zeta_j^{T_{k-1},+,E}$, which we can then deal with in the same way as the $\zeta_j^{T_{k},-,E}$. However, if the mesoscopic process is not doing a U-turn, the state of the $\Delta_{T_{k},j}$ will depend on the previous history of the process. To keep track of it, we will use an algorithm to associate to each time $k \in\{0,...,K\}$ a configurations of \emph{states of the edges of $\mathds{Z}$}. The edges $(z,z+1)$ will be in any of the four following states:
 \begin{itemize}
  \item \emph{Clean}. This is the case in which $(X_m)_{m\in\mathds{N}}$ did not visit any $j \in \{X_{T_0}+\lfloor\varepsilon n\rfloor z+1,...,X_{T_0}+\lfloor\varepsilon n\rfloor (z+1)-1\}$ since time $T_0$, so the corresponding $\Delta_{j}$ are still the $\Delta_{T_0,j}$, which we can control by Proposition \ref{prop_law_zeta}.
  \item \emph{Usable}. This is the case in which there was some $k'$ so that $X_{T_{k'-1}}=X_{T_0}+\lfloor\varepsilon n\rfloor(z+1)$, $X_{T_{k'}}=X_{T_0}+\lfloor\varepsilon n\rfloor z$ and $X_{T_{k'+1}}=X_{T_0}+\lfloor\varepsilon n\rfloor(z-1)$ (or symmetrically $X_{T_{k'-1}}=X_{T_0}+\lfloor\varepsilon n\rfloor z$, $X_{T_{k'}}=X_{T_0}+\lfloor\varepsilon n\rfloor(z+1)$, $X_{T_{k'+1}}=X_{T_0}+\lfloor\varepsilon n\rfloor(z+2)$), and $(X_m)_{m\in\mathds{N}}$ did not visit $\{X_{T_0}+\lfloor\varepsilon n\rfloor z,...,X_{T_0}+\lfloor\varepsilon n\rfloor(z+1)\}$ since. At time $T_{k'}$, the $\Delta_j$ for $j\in\{X_{T_0}+\lfloor\varepsilon n\rfloor z,...,X_{T_0}+\lfloor\varepsilon n\rfloor(z+1)\}$ correspond to the $\zeta_j^{T_{k'-1},-,E}$, and between times $T_{k'}$ and $T_{k'+1}$ the process $(X_m)_{m\in\mathds{N}}$ visited some such $j$, but not all, so at time $T_{k'+1}$ the $\Delta_j$ of the sites such visited correspond to the $\zeta_j^{T_{k'},-,E}$, while the $\Delta_j$ of the sites not visited still correspond to the $\zeta_j^{T_{k'-1},-,E}$. Consequently, the $\Delta_{j}$ may correspond to the $\zeta_j^{T_{k'},-,E}$ or the $\zeta_j^{T_{k'-1},-,E}$, which we will be able to control since there are only two possibilities.
  \item \emph{Usable-clean}. This is the case in which ``the mesoscopic process made a U-turn just at the left of $z$ or at the right of $z+1$, but never approached $z$ or $z+1$ otherwise'': there was some $k'$ so that $X_{T_{k'+1}}=X_{T_{k'-1}}=X_{T_0}+\lfloor\varepsilon n\rfloor (z-1)$ and $X_{T_{k'}}=X_{T_0}+\lfloor\varepsilon n\rfloor z$ (or symmetrically $X_{T_{k'+1}}=X_{T_{k'-1}}=X_{T_0}+\lfloor\varepsilon n\rfloor(z+2)$, $X_{T_{k'}}=X_{T_0}+\lfloor\varepsilon n\rfloor (z+1)$), but none of the other $X_{T_{k''}}$ was $X_{T_0}+\lfloor\varepsilon n\rfloor z$ or $X_{T_0}+\lfloor\varepsilon n\rfloor (z+1)$. In this case, since time $T_0$, the process $(X_m)_{m\in\mathds{N}}$ could only visit $\{X_{T_0}+\lfloor\varepsilon n\rfloor z,...,X_{T_0}+\lfloor\varepsilon n\rfloor (z+1)\}$ between times $T_{k'}$ and $T_{k'+1}$, and did not visit all the sites. The $\Delta_j$ of the sites that were visited correspond to the $\zeta_j^{T_{k'},-,E}$, and the $\Delta_j$ of the sites that were not visited are still the $\Delta_{T_0,j}$. There are still only two possibilities that we can control. 
  \item \emph{Dirty}. This covers all the other cases, in which we will not be able to control the $\Delta_{j}$. 
 \end{itemize}
 Consequently, if the edge $(z,z+1)$ is clean, usable or usable-clean at the step corresponding to $T_k$, the $\Delta_{T_k,j}$ on $\{X_{T_0}+\lfloor\varepsilon n\rfloor z,...,X_{T_0}+\lfloor\varepsilon n\rfloor (z+1)\}$ can be controlled, hence the $\zeta_j^{T_k,-,B}$ can. We will show that whatever the path of the mesoscopic process $(X_{T_{k'}})_{0 \leq k' \leq K}$, a positive fraction of the edges it crosses will be clean, usable or usable-clean at the time of crossing, so a positive fraction of the steps will give us a lower bound $T_{k'}-T_{k'-1} \geq 2\delta n^{3/2}$, which is enough to prove $T_K-T_0$ is of order $K n^{3/2}$. 

\medskip

In order to write the rigorous proof, we will need some notation for the ``trajectory'' of the mesoscopic process $(X_{T_k})_{k\in\mathds{N}}$. Let $K \in \mathds{N}^*$. A \emph{path of length $K$} is a sequence $\gamma = (z_0,z_1,\dots,z_K)$ with 
$z_0=0$, $z_k \in \mathds{Z}$ and $|z_k-z_{k-1}|=1$ for any $k \in \{1,\dots,K\}$. We say that \emph{$X$ follows $\gamma$} when $X_{T_k}=X_{T_0} + \lfloor\varepsilon n\rfloor z_k$ for all $k\in \{0,\dots,K\}$. 

 Some of the $\zeta_i^{T_k,\pm,B}$ we need to control will depend on the $\Delta_{T_0,i}$, but their exact definition depends on if we want to work with $\zeta_i^{T_k,-,B}$ or $\zeta_i^{T_k,+,B}$, which depends on the path of the mesoscopic process. Moreover, it is more practical to work with the $\bar\Delta_{T_0,i}$ defined in Proposition \ref{prop_law_zeta}, as this proposition gives us their law. Consequently, for any $k \in \{0,...,K-1\}$ we define $(\hat \zeta_i^{\gamma,k})_{i \in \mathds{Z}}$ thus:
 \begin{itemize}
 \item if $z_{k+1}=z_{k}-1$, 
 \[
  \hat\zeta^{\gamma,k}_i = \left\{\begin{array}{ll}
                          -\bar\Delta_{T_0,i}-1/2&\text{if }i \leq X_{T_0}+\lfloor\varepsilon n\rfloor z_{k},\\
                          -\bar\Delta_{T_{0},i}+1/2&\text{if }i > X_{T_0}+\lfloor\varepsilon n\rfloor z_{k} ;
                         \end{array}\right. 
 \]
 \item if $z_{k+1}=z_{k}+1$, 
 \[
  \hat\zeta^{\gamma,k}_i = \left\{\begin{array}{ll}
                          \bar\Delta_{T_{0},i}+1/2&\text{if }i \leq X_{T_0}+\lfloor\varepsilon n\rfloor z_{k},\\
                          \bar\Delta_{T_{0},i}-1/2&\text{if }i > X_{T_0}+\lfloor\varepsilon n\rfloor z_{k}.
                         \end{array}\right.
 \]
 \end{itemize}
 
 Since we may use the $\hat \zeta_i^{\gamma,k}$ instead of the $\zeta_i^{T_k,\pm,B}$, we will need to replace the $\zeta_i^{T_k,\pm,I}$ by random variables that are independent from the $\hat \zeta_i^{\gamma,k}$, hence from the $\bar\Delta_{T_{0},i}$. We had a construction in Proposition \ref{prop_def_zetaI} that gave appropriate replacements for the $\zeta_i^{T_0,\pm,I}$, but not for the $\zeta_i^{T_k,\pm,I}$ with $k>0$. Finding good replacements for the $\zeta_i^{T_k,\pm,I}$ for all $k\in\mathds{N}$ is the goal of the following proposition (we recall that $\mathcal{B}_0^{\lfloor N \theta\rfloor,\lfloor Nx \rfloor,\pm}$ was defined in Proposition \ref{prop_law_zeta}).
 
 \begin{proposition}
 For any $k \in \{0,...,K-1\}$, we can define random variables $(\zeta_i^{\gamma,k})_{i \in \mathds{Z}}$ with the following properties. The $\zeta_i^{\gamma,k}$, $i \in \mathds{Z}$ are i.i.d. with law $\rho_0$ and $(\zeta_i^{\gamma,k})_{i \in \mathds{Z}}$ is independent from $(\bar \Delta_{T_0,i})_{i \in \mathds{Z}}$ and $(\zeta_i^{\gamma,k'})_{i \in \mathds{Z}}$, $k' < k$. In addition, if $n$ is large enough, $X$ follows $\gamma$ and $(\mathcal{B}_0^{\lfloor N \theta\rfloor,\lfloor Nx \rfloor,\pm})^c$ occurs, then for any $k \in \{0,...,K-1\}$, $(\zeta_i^{\gamma,k})_{i \in \mathds{Z}} = (\zeta_i^{T_k,\iota,I})_{i \in \mathds{Z}}$, where $\iota=+$ if $z_{k+1}=z_k+1$ and $\iota=-$ if $z_{k+1}=z_k-1$.
 \end{proposition}
 
 \begin{proof}
We can define a process $(\tilde X_m)_{m \geq T_0}$ which is ``like $(X_m)_{m \geq T_0}$, but such that the environment at time $T_0$ is $(\bar \Delta_{T_0,i})_{i \in \mathds{Z}}$''. It is defined so that $\tilde X_{T_0}=X_{T_0}$, $(\tilde \Delta_{T_0,i})_{i \in \mathds{Z}}=(\bar \Delta_{T_0,i})_{i \in \mathds{Z}}$, for all $m \geq T_0$, 
\[
 \mathds{P}(\tilde X_{m+1}=\tilde X_m+1) = 1-\mathds{P}(\tilde X_{m+1}=\tilde X_m-1) = \frac{w(\tilde\Delta_{m,\tilde X_m})}{w(\tilde\Delta_{m,\tilde X_m})+w(-\tilde\Delta_{m,\tilde X_m})},
\]
\[
 \tilde\Delta_{m+1,\tilde X_m} = \left\{\begin{array}{ll}
                     \tilde\Delta_{m,\tilde X_m}-1 & \text{if } \tilde X_{m+1}=\tilde X_m+1 \\
                     \tilde\Delta_{m,\tilde X_m}+1 & \text{if } \tilde X_{m+1}=\tilde X_m-1
                    \end{array}\right.
 \text{ and } \tilde\Delta_{m+1,i}=\tilde\Delta_{m,i}\text{ for all }i \neq \tilde X_m,
\]
the transitions of $(\tilde X_m)_{m \geq T_0}$ are independent from $(\bar \Delta_{T_0,i})_{i \in \mathds{Z}}$, and for any $k \in \mathds{N}$, if $n$ is large enough and $(\mathcal{B}_0^{\lfloor N \theta\rfloor,\lfloor Nx \rfloor,\pm})^c$ occurs then $(\tilde X_m)_{T_0 \leq m \leq T_k}=(X_m)_{T_0 \leq m \leq T_k}$. Moreover, we define the following stopping times: $T_0^\gamma = T_0$ and for $k\in \{1,\dots,K\}$, $T_k^\gamma = \inf\{m \geq T_{k-1}^\gamma \,|\, \tilde X_{m}=\tilde X_{T_0} + \lfloor\varepsilon n\rfloor z_k\}$. If $X$ follows $\gamma$, $(\mathcal{B}_0^{\lfloor N \theta\rfloor,\lfloor Nx \rfloor,\pm})^c$ occurs and $n$ is large enough, then $T_k^\gamma = T_k$ for all $k \in \{0,...,K\}$. The $(\zeta_i^{\gamma,k})_{i \in \mathds{Z}}$ will then be defined for the process $(\tilde X_m)_{m \geq T_0}$ as the $(\zeta_i^{T_k^\gamma,\iota,I})_{i \in \mathds{Z}}$ are defined for the process $(X_m)_{m \geq T_0}$, where $\iota=+$ if $z_{k+1}=z_k+1$ and $\iota=-$ if $z_{k+1}=z_k-1$, with the construction given before Proposition \ref{prop_def_zetaI}. 
 \end{proof}
 
 In order to lower bound the $\sum_{i=i_1}^{i_2}\sum_{j=i_1}^{i} \zeta_j^{T_{k},-,E}$ and the $\sum_{i=i_1}^{i_2}\sum_{j=i_1}^{i} \zeta_j^{T_{k},-,B}$ (as well as the symmetric quantities when $X_{T_{k+1}}=X_{T_{k}}+\lfloor\varepsilon n\rfloor$), we will need to lower bound the $\sum_{i=i_1}^{i_2}\sum_{j=i_1}^i \zeta_j^{\gamma,k}$, the $\sum_{i=i_1}^{i_2}\sum_{j=i_1}^i \hat\zeta_j^{\gamma,k}$ and the $\sum_{i=i_1}^{i_2}\sum_{j=i_1}^i -\hat\zeta_j^{\gamma,k}$ (as well as the symmetric quantities). We introduce the necessary notation to do that. We denote $r_2=\mathds{E}(\zeta^2)$ where $\zeta$ has law $\rho_0$. We set 
 \begin{equation}\label{eq_def_tildepsilon}
 0 <\tilde\varepsilon < \min\left(\frac{\varepsilon}{8},\frac{1}{2^{15}61\ln 2}\varepsilon,-\frac{\ln(\frac{31}{32})}{2^9 61\ln 2}\varepsilon,\frac{-\ln(1-2^{-10})}{240\ln 2}\varepsilon\right).
 \end{equation}
 For any path $\gamma$ of length $K$, for any $k \in \{0,...,K-1\}$, for any interval $I = \{i_1,...,i_2\}$ of $\mathds{Z}$ with $i_2-i_1=\lceil \tilde\varepsilon n\rceil-1$, we define the following events:
 \[
  \mathcal{W}_{\gamma,k,I}^\leftarrow = \left\{\sum_{i=i_1}^{i_2}\sum_{j=i_1}^i \zeta_j^{\gamma,k} \geq \frac{r_2}{6}(\tilde\varepsilon n)^{3/2}\right\},
  \mathcal{W}_{\gamma,k,I}^\rightarrow = \left\{\sum_{i=i_1}^{i_2}\sum_{j=i}^{i_2} \zeta_j^{\gamma,k} \geq \frac{r_2}{6}(\tilde\varepsilon n)^{3/2}\right\},
  \]
  \[
  \mathcal{W}_{\gamma,k,I}^{+,\leftarrow} = \left\{\sum_{i=i_1}^{i_2}\sum_{j=i_1}^i \hat\zeta_j^{\gamma,k} \geq \frac{r_2}{6}(\tilde\varepsilon n)^{3/2}\right\},
  \mathcal{W}_{\gamma,k,I}^{-,\leftarrow} = \left\{\sum_{i=i_1}^{i_2}\sum_{j=i_1}^i -\hat\zeta_j^{\gamma,k} \geq \frac{r_2}{6}(\tilde\varepsilon n)^{3/2}\right\},
 \]
\[
  \mathcal{W}_{\gamma,k,I}^{+,\rightarrow} = \left\{\sum_{i=i_1}^{i_2}\sum_{j=i}^{i_2} \hat\zeta_j^{\gamma,k} \geq \frac{r_2}{6} (\tilde\varepsilon n)^{3/2}\right\},
  \mathcal{W}_{\gamma,k,I}^{-,\rightarrow} = \left\{\sum_{i=i_1}^{i_2}\sum_{j=i}^{i_2} -\hat\zeta_j^{\gamma,k} \geq \frac{r_2}{6} (\tilde\varepsilon n)^{3/2}\right\}.
 \]
 
 \begin{lemma}\label{lem_proba_W}
  When $n$ is large enough, for any $K \in \mathds{N}^*$, for any path $\gamma$ of length $K$, for any $k \in \{0,...,K-1\}$, for any interval $I = \{i_1,...,i_2\}$ of $\mathds{Z}$ with $i_2-i_1=\lceil \tilde\varepsilon n\rceil-1$, $\mathds{P}(\mathcal{W}_{\gamma,k,I}^\leftarrow) \geq \frac{1}{32}$ and $\mathds{P}(\mathcal{W}_{\gamma,k,I}^\rightarrow) \geq \frac{1}{32}$. Moreover, if $z_k \leq 0$ and $I \subset (-\infty,X_{T_0}+\lfloor\varepsilon n\rfloor z_{k}-1]$, or if $z_k \geq 0$ and $I \subset [X_{T_0}+\lfloor\varepsilon n\rfloor z_{k}+1,+\infty)$, then we have $\mathds{P}(\mathcal{W}_{\gamma,k,I}^{+,\leftarrow}),\mathds{P}(\mathcal{W}_{\gamma,k,I}^{-,\leftarrow}), \mathds{P}(\mathcal{W}_{\gamma,k,I}^{+,\rightarrow}), \mathds{P}(\mathcal{W}_{\gamma,k,I}^{-,\rightarrow}) > \frac{1}{32}$.
 \end{lemma}
 
 \begin{proof}
  The $\zeta_i^{\gamma,k}$, $i \in I$ are i.i.d. with law $\rho_0$. 
  Furthermore, if $z_k \leq 0$ and $I \subset (-\infty,X_{T_0}+\lfloor\varepsilon n\rfloor z_{k}-1]$, or if $z_k \geq 0$ and $I \subset [X_{T_0}+\lfloor\varepsilon n\rfloor z_{k}+1,+\infty)$, by Proposition \ref{prop_law_zeta} the $\hat\zeta_i^{\gamma,k}$, $i \in I$ are i.i.d. with law $\rho_0$. Therefore it is enough to show that when $n$ is large enough, if $\zeta_i$, $i \in \{1,...,\lceil \tilde\varepsilon n\rceil\}$ are i.i.d. with law $\rho_0$ and we denote $S = \sum_{i=1}^{\lceil \tilde\varepsilon n\rceil}\sum_{j=i}^{\lceil \tilde\varepsilon n\rceil} \zeta_j$, then $\mathds{P}(S\geq \frac{r_2}{6}(\tilde\varepsilon n)^{3/2}) \geq \frac{1}{32}$. $S$ is symmetric, so $\mathds{P}(S\geq \frac{r_2}{6}(\tilde\varepsilon n)^{3/2}) = \frac{1}{2}\mathds{P}(|S|\geq \frac{r_2}{6}(\tilde\varepsilon n)^{3/2})= \frac{1}{2}\mathds{P}(S^2\geq \frac{r_2}{6}(\tilde\varepsilon n)^3)$, thus it is enough to show $\mathds{P}(S^2\geq \frac{r_2}{6}(\tilde\varepsilon n)^3) \geq \frac{1}{16}$. In order to do that, we notice that $S=\sum_{i=1}^{\lceil \tilde\varepsilon n\rceil} i\zeta_i$ and $\rho_0$ has expectation 0, hence $\mathds{E}(S^2) = \sum_{i=1}^{\lceil \tilde\varepsilon n\rceil}i^2 r_2 = \frac{\lceil \tilde\varepsilon n\rceil(\lceil \tilde\varepsilon n\rceil+1)(2\lceil \tilde\varepsilon n\rceil+1)}{6}r_2 \geq \frac{r_2}{3}(\tilde\varepsilon n)^{3/2}$ and 
  \[
  \mathds{E}(S^4)=3\sum_{i=1}^{\lceil \tilde\varepsilon n\rceil}\sum_{j=1}^{\lceil \tilde\varepsilon n\rceil} i^2j^2r_2^2+\sum_{i=1}^{\lceil \tilde\varepsilon n\rceil}i^4 (\mathds{E}(\zeta^4)-3r_2^2) 
  \]
  \[
  = 3\left(\frac{\lceil \tilde\varepsilon n\rceil(\lceil \tilde\varepsilon n\rceil+1)(2\lceil \tilde\varepsilon n\rceil+1)}{6}\right)^2r_2^2+\frac{6\lceil\tilde\varepsilon n\rceil^5+15\lceil\tilde\varepsilon n\rceil^4+10\lceil\tilde\varepsilon n\rceil^3-\lceil\tilde\varepsilon n\rceil}{30}(\mathds{E}(\zeta^4)-3r_2^2) 
  \]
   is smaller than $4 \mathds{E}(S^2)^2$ when $n$ is large enough. We deduce $\mathds{P}(S^2\geq \frac{r_2}{6}(\tilde\varepsilon n)^3) \geq \mathds{P}(S^2\geq \frac{\mathds{E}(S^2)}{2})$, hence by the Paley-Zygmund inequality, $\mathds{P}(S^2\geq \frac{r_2}{6}(\tilde\varepsilon n)^3) \geq \frac{1}{4}\frac{\mathds{E}(S^2)^2}{\mathds{E}(S^4)} \geq \frac{1}{16}$. 
 \end{proof}
 
 We are now in position to write down the algorithm mentioned at the beginning of the section, which for each time $k\in\{0,...,K\}$ yields a configuration of states of the edges of $\mathds{Z}$ in which the edges can be clean, usable, usable-clean or dirty depending on the control we have on them. Let $K \in \mathds{N}^*$. For any path $\gamma = (z_0,z_1,\dots,z_K)$ of length $K$, at the same time as the configurations of states of the edges, we will define a sequence of random variables $(\Theta_k^\gamma)_{0 \leq k \leq K-1}$ so that for any $k \in \{0,...,K-1\}$, $\Theta_k^\gamma \in \{0,1,*\}$. As we will show later in Proposition \ref{prop_test_to_times}, they will be defined so that that if $X$ follows $\gamma$, $\Theta_k^\gamma=1$ (as well as an additional condition) and $\mathcal{B} \cap (\bigcap_{r=1}^{6}\mathcal{B}_0^c)$ occurs, then $T_k-T_{k-1} \geq \frac{r_2}{6}(\tilde\varepsilon n)^{3/2}$.
 
 For any edge $(z,z+1)$ of $\mathds{Z}$, we denote $I(z,z+1)$ the collection of intervals composed of the $\{X_{T_0}+\lfloor\varepsilon n\rfloor z + \lceil \tilde\varepsilon n\rceil (m-1) + 1 ,...,X_{T_0}+\lfloor\varepsilon n\rfloor z + \lceil \tilde\varepsilon n\rceil m\}$ for $m \in \{1,...,2\lfloor \frac{\varepsilon}{4\tilde\varepsilon}\rfloor\}$. We also denote respectively $I_l(z,z+1)$ and $I_r(z,z+1)$ the collections of the $\{X_{T_0}+\lfloor\varepsilon n\rfloor z + \lceil \tilde\varepsilon n\rceil (m-1) + 1 ,...,X_{T_0}+\lfloor\varepsilon n\rfloor z + \lceil \tilde\varepsilon n\rceil m\}$ respectively for $m \in \{1,...,\lfloor \frac{\varepsilon}{4\tilde\varepsilon}\rfloor\}$ and $m \in \{\lfloor \frac{\varepsilon}{4\tilde\varepsilon}\rfloor+1,...,2\lfloor \frac{\varepsilon}{4\tilde\varepsilon}\rfloor\}$. When $n$ is large enough, the intervals of $I(z,z+1)$ are contained in $\{X_{T_0}+\lfloor\varepsilon n\rfloor z + 1 ,...,X_{T_0}+\lfloor\varepsilon n\rfloor (z+1) -1\}$. 
 
 We now define the $(\Theta_k^\gamma)_{0 \leq k \leq K-1}$ as follows. For any $k \in \{0,...,K-1\}$, we say the \emph{$k$-th step of $\gamma$} is the passage from $z_{k}$ to $z_{k+1}$. We will decompose the path in \emph{stages} of one or two steps at the end of which we update the states of the edges of $\mathds{Z}$. At time $k=0$, all the edges of $\mathds{Z}$ are clean. Let $k=0$ or let $k \in \{1,...,K-2\}$ and suppose the last step of a stage of $\gamma$ is the step $k-1$. We suppose $z_{k+1}=z_k+1$ (if $z_{k+1}=z_k-1$, the definition is similar, with all the arrows reversed in the events and $r,l$ exchanged). We define the next stage as follows, depending on the state of the edges at time $k$.
 
 \emph{Case $(z_k,z_{k+1})$ clean.} \\
 In this case, the stage will encompass only step $k$. We then define $\Theta_{k}^\gamma$ as the indicator of $\bigcup_{I \in I(z_k,z_{k+1})}(\mathcal{W}_{\gamma,k,I}^\rightarrow \cap \mathcal{W}_{\gamma,k,I}^{-,\rightarrow})$, we say $\Theta_k^\gamma$ is of \emph{type C}, and the edges $(z_k,z_{k+1})$, $(z_k,z_k-1)$ become dirty at time $k+1$.
 
 \emph{Case $(z_k,z_{k+1})$ dirty.} \\
 In this case, the stage will encompass steps $k$ and $k+1$, and there will be different cases. \\
 If $z_{k+2}=z_k$, we set $\Theta_{k}^\gamma=*$ and $\Theta_{k+1}^\gamma$ as the indicator of $\bigcup_{I \in I(z_k,z_{k+1})}(\mathcal{W}_{\gamma,k,I}^\leftarrow \cap \mathcal{W}_{\gamma,k+1,I}^\leftarrow)$. We also say $\Theta_{k+1}^\gamma$ is of \emph{type D}. After the stage, at time $k+2$, $(z_k,z_{k+1})$ and its two neighboring edges become dirty. \\
 We now assume $z_{k+2} \neq z_k$, i.e. $z_{k+2}=z_{k+1}+1$. Then there will be different cases depending on the state of $(z_{k+1},z_{k+2})$ at time $k$. \\
 \emph{Case $(z_{k+1},z_{k+2})$ dirty.} Then we set $\Theta_{k}^\gamma=*$, and $\Theta_{k+1}^\gamma$ as the indicator of $\{|\{I\in I_l(z_k,z_{k+1})|\mathcal{W}_{\gamma,k,I}^\leftarrow\}| \geq {}\frac{\varepsilon}{2^9\tilde \varepsilon}\} \cap \{|\{I\in I_r(z_k,z_{k+1})|\mathcal{W}_{\gamma,k+1,I}^\leftarrow\}| \geq \frac{\varepsilon}{2^9\tilde \varepsilon}\}$. We then say that $\Theta_{k+1}^\gamma$ is of \emph{type A'}. At time $k+2$, the edges $(z_k-1,z_k)$ and $(z_{k+1},z_{k+2})$ become dirty. Moreover, if $\Theta_{k+1}^\gamma=1$, we say the stage is a \emph{stage with wait} and the edge $(z_k,z_{k+1})$ becomes usable at time $k+2$. If, in addition to having $\Theta_{k+1}^\gamma=1$, we also have that $(z_k,z_{k}-1)$ was dirty at time $k$, we say the stage is \emph{dirty}. \\
 \emph{Case $(z_{k+1},z_{k+2})$ clean.} Then we set $\Theta_{k+1}^\gamma$ as the indicator of $\bigcup_{I \in I(z_{k+1},z_{k+2})}(\mathcal{W}_{\gamma,k+1,I}^\rightarrow \cap \mathcal{W}_{\gamma,k+1,I}^{-,\rightarrow})$ and we say $\Theta_{k+1}^\gamma$ is of type C. $(z_{k+1},z_{k+2})$ then becomes dirty at time $k+2$. If $(z_k-1,z_k)$ is not clean at time $k$, it becomes dirty at time $k+2$ and we set $\Theta_{k}^\gamma=*$. If $(z_k-1,z_k)$ is clean at time $k$, then we set $\Theta_{k}^\gamma$ as the indicator of $\{|\{I\in I_l(z_k-1,z_k)|\mathcal{W}_{\gamma,k,I}^{+,\leftarrow}\}| \geq \frac{\varepsilon}{2^9\tilde \varepsilon}\} \cap \{|\{I\in I_r(z_k-1,z_k)|\mathcal{W}_{\gamma,k,I}^\leftarrow\}| \geq \frac{\varepsilon}{2^9\tilde \varepsilon}\}$ and we say $\Theta_{k}^\gamma$ is of \emph{type B'}. If $\Theta_{k}^\gamma=1$ then $(z_k-1,z_k)$ becomes usable-clean at time $k+2$, otherwise it becomes dirty. \\
 \emph{Case $(z_{k+1},z_{k+2})$ usable (respectively usable-clean).} In this case, there exists $k' \leq k$ so that $(z_{k+1},z_{k+2})$ became usable (respectively usable-clean) at time $k'$, and we consider the largest such $k'$. We then have $\Theta_{k'-1}^\gamma=1$ (respectively $\Theta_{k'-2}^\gamma=1$), so the sets $\mathcal{E}_r =\{I\in I_r(z_{k+1},z_{k+2})|\mathcal{W}_{\gamma,k'-2,I}^\rightarrow\}$ and $\mathcal{E}_l =\{I\in I_l(z_{k+1},z_{k+2})|\mathcal{W}_{\gamma,k'-1,I}^\rightarrow\}$ (respectively $\mathcal{E}_r =\{I\in I_r(z_{k+1},z_{k+2})|\mathcal{W}_{\gamma,k'-2,I}^{+,\rightarrow}\}$ and $\mathcal{E}_l =\{I\in I_l(z_{k+1},z_{k+2})|\mathcal{W}_{\gamma,k'-2,I}^{\rightarrow}\}$) have at least $\frac{\varepsilon}{2^9\tilde \varepsilon}$ elements. We then define $\Theta_{k+1}^\gamma$ as the indicator of $(\bigcup_{I \in \mathcal{E}_l}\mathcal{W}_{\gamma,k+1,I}^\rightarrow) \cap (\bigcup_{I \in \mathcal{E}_r}\mathcal{W}_{\gamma,k+1,I}^\rightarrow)$ and say $\Theta_{k+1}^\gamma$ is of \emph{type A} (respectively of \emph{type B}). Both $(z_{k},z_{k+1})$ and $(z_{k+1},z_{k+2})$ become dirty at time $k+2$. Moreover, if $(z_k-1,z_k)$ is not clean at time $k$, it becomes dirty and we set $\Theta_{k}^\gamma=*$. If $(z_k-1,z_k)$ is clean at time $k$, then we set $\Theta_{k}^\gamma$ as the indicator of $\{|\{I\in I_l(z_k-1,z_k)|\mathcal{W}_{\gamma,k,I}^{+,\leftarrow}\}| \geq \frac{\varepsilon}{2^9\tilde \varepsilon}\} \cap \{|\{I\in I_r(z_k-1,z_k)|\mathcal{W}_{\gamma,k,I}^\leftarrow\}| \geq \frac{\varepsilon}{2^9\tilde \varepsilon}\}$ and we say $\Theta_{k}^\gamma$ is of type B'. If $\Theta_{k}^\gamma=1$ then $(z_k-1,z_k)$ becomes usable-clean at time $k+2$, otherwise it becomes dirty.
 
 \emph{Case $(z_k,z_{k+1})$ usable (respectively usable-clean).} \\
 In this case, the stage will encompass only step $k$. Moreover, there exists $k' \leq k$ such that $(z_k,z_{k+1})$ became usable (respectively usable-clean) at time $k'$, and we consider the largest such $k'$. We then have $\Theta_{k'-1}^\gamma=1$ (respectively $\Theta_{k'-2}^\gamma=1$), so the sets $\mathcal{E}_r =\{I\in I_r(z_k,z_{k+1})|\mathcal{W}_{\gamma,k'-2,I}^\rightarrow\}$ and $\mathcal{E}_l =\{I\in I_l(z_k,z_{k+1})|\mathcal{W}_{\gamma,k'-1,I}^\rightarrow\}$ (respectively $\mathcal{E}_r =\{I\in I_r(z_k,z_{k+1})|\mathcal{W}_{\gamma,k'-2,I}^{+,\rightarrow}\}$ and $\mathcal{E}_l =\{I\in I_l(z_k,z_{k+1})|\mathcal{W}_{\gamma,k'-2,I}^{\rightarrow}\}$) have at least $\frac{\varepsilon}{2^9\tilde \varepsilon}$ elements. We then define $\Theta_{k}^\gamma$ as the indicator of $(\bigcup_{I \in \mathcal{E}_l}\mathcal{W}_{\gamma,k,I}^\rightarrow) \cap (\bigcup_{I \in \mathcal{E}_r}\mathcal{W}_{\gamma,k,I}^\rightarrow)$ and say $\Theta_{k}^\gamma$ is of type A (respectively of type B). Both $(z_{k},z_{k+1})$ and $(z_{k},z_{k}-1)$ become dirty at time $k+1$.
 
 If this algorithm does not yield a value for $\Theta_{K-1}^\gamma$, we set $\Theta_{K-1}^\gamma = *$. 
 
 \begin{proposition}\label{prop_test_to_times}
  For any $K \in \mathds{N}^*$, for any path $\gamma$ of length $K$, if $X$ follows $\gamma$, $\mathcal{B}^c \cap (\bigcap_{r=0}^{6}\mathcal{B}_r^c)$ occurs and $n$ is large enough, then for any $k\in\{0,...,K-1\}$, if $\Theta_k^\gamma$ is of type A, B, C or D and $\Theta_k^\gamma=1$ then $T_{k+1}-T_k \geq \frac{r_2}{6}(\tilde\varepsilon n)^{3/2}$. 
 \end{proposition}

 \begin{proof}
 Let us assume that $X$ follows $\gamma$, $\mathcal{B}^c \cap (\bigcap_{r=0}^{6}\mathcal{B}_r^c)$ occurs and $n$ is large enough. We notice that since $\bigcap_{r=0}^{6}\mathcal{B}_r^c$ occurs, $(\mathcal{B}_0^{\lfloor N\theta\rfloor,\lfloor Nx\rfloor,\pm})^c$ occurs. In particular, by Proposition \ref{prop_law_zeta}, for any $\lfloor Nx\rfloor-n^{(\alpha-1)/4}\lfloor\varepsilon n\rfloor-1 \leq i \leq \lfloor Nx\rfloor+n^{(\alpha-1)/4}\lfloor\varepsilon n\rfloor+1$, hence for any $k\in\{0,...,K-1\}$ and $i \in \{X_{T_0}+\lfloor\varepsilon n\rfloor (z_k-1),...,X_{T_0}+\lfloor\varepsilon n\rfloor (z_k+1)\}$, we have $\bar \Delta_{T_0,i}=\Delta_{T_0,i}$. Furthermore, by Proposition \ref{prop_bound_m}, since $\mathcal{B}^c$ occurs, for any $k\in\{0,...,K-1\}$ we have $T_k=\mathbf{T}_{m,i}^+$ or $\mathbf{T}_{m,i}^-$ for some integers $\lfloor N\theta\rfloor-2n^{(\alpha+4)/5} \leq m \leq \lfloor N\theta\rfloor+2n^{(\alpha+4)/5}$ and $i \in [\lfloor Nx \rfloor-n^{(\alpha+4)/5},\lfloor Nx \rfloor+n^{(\alpha+4)/5}]$. Therefore, since $\bigcap_{r=0}^{6}\mathcal{B}_r^c$ occurs, $\bigcap_{r=1}^6 (\mathcal{B}_{T_k,r}^-)^c$ and $\bigcap_{r=1}^6 (\mathcal{B}_{T_k,r}^+)^c$ occur. Set $k\in\{0,...,K-1\}$ and suppose $\Theta_k^\gamma=1$. We will deal with the possible types of $\Theta_k^\gamma$ separately.
 
 \emph{Case $\Theta_k^\gamma$ of type A.} 
 
 We suppose $z_{k+1}=z_k+1$, the other case can be dealt with in the same way. In this case, the edge $(z_k,z_{k+1})$ was usable at time $k$. We denote $k'$ the biggest integer below $k$ such that $(z_k,z_{k}+1)$ became usable at time $k'$. Then the path $\gamma$ did not cross $(z_k,z_{k}+1)$ between times $k'-1$ and $k$, and was always strictly below $z_k$ between these times. Moreover, for any $k''\in \mathds{N}$, $m\in\{T_{k''},...,T_{k''+1}\}$, by definition of $T_{k''+1}$ we have $X_m \in \{X_{T_{k''}}-\lfloor\varepsilon n\rfloor,...,X_{T_{k''}}+\lfloor\varepsilon n\rfloor\}$. Since $X$ follows $\gamma$, this implies that for any $m \in \{T_{k'},...,T_k-1\}$, $X_m < X_{T_0}+\lfloor\varepsilon n\rfloor z_k$, so for any $i \geq X_{T_0}+\lfloor\varepsilon n\rfloor z_k$, $\Delta_{T_k,i}= \Delta_{T_{k'},i}$. There will be two different cases (we recall the notation $\mathcal{E}_l$, $\mathcal{E}_r$ introduced when defining the $\Theta_k^\gamma$ of type A). 
 
 We first assume $L_{X_{T_0}+\lfloor\varepsilon n\rfloor z_k+\lfloor\frac{\varepsilon}{4\tilde\varepsilon}\rfloor+1}^{T_{k'-1},-}=0$. \\
 We notice that since $\Theta_k^\gamma=1$, there exists $I=\{i_1,...,i_2\} \in \mathcal{E}_r$ such that $\mathcal{W}_{\gamma,k,I}^\rightarrow$ occurs. Since $I \in \mathcal{E}_r$, $\mathcal{W}_{\gamma,k'-2,I}^\rightarrow$ also occurs. This yields $\sum_{i=i_1}^{i_2}\sum_{j=i}^{i_2} \zeta_j^{\gamma,k} \geq \frac{r_2}{6}(\tilde\varepsilon n)^{3/2}$ and $\sum_{i=i_1}^{i_2}\sum_{j=i}^{i_2} \zeta_j^{\gamma,k'-2} \geq \frac{r_2}{6}(\tilde\varepsilon n)^{3/2}$. Now, since $(\mathcal{B}_0^{\lfloor N\theta\rfloor,\lfloor Nx\rfloor,\pm})^c$ occurs, $n$ is large enough and $X$ follows $\gamma$, by the definition of the $\zeta_i^{\gamma,k}$ we get that for any $j \in I$, $\zeta_j^{\gamma,k}=\zeta_j^{T_k,+,I}$ and $\zeta_j^{\gamma,k'-2}=\zeta_j^{T_{k'-2},-,I}$. We deduce 
 \[
 \sum_{i=i_1}^{i_2}\sum_{j=i}^{i_2} \zeta_j^{T_k,+,I} \geq \frac{r_2}{6}(\tilde\varepsilon n)^{3/2}\text{ and }\sum_{i=i_1}^{i_2}\sum_{j=i}^{i_2} \zeta_j^{T_{k'-2},-,I} \geq \frac{r_2}{6}(\tilde\varepsilon n)^{3/2}.
 \]
 In addition, by the definition of $S_i^{T_k,+,I}$, for any $i \in \{i_1,...,i_2\}$, $S_i^{T_k,+,I}-S_{i_2+1}^{T_k,+,I} \geq \sum_{j=i}^{i_2} \zeta_j^{T_k,+,I}$. Moreover, since $\bigcap_{r=1}^6 (\mathcal{B}_{T_k,r}^+)^c$ occurs, Proposition \ref{prop_sym} yields $S_i^{T_k,+,E} \geq S_i^{T_k,+,I} - (\ln n)^8 n^{1/4}$ and $S_{i_2+1}^{T_k,+,E} \leq S_{i_2+1}^{T_k,+,I}+\lceil(\ln n)^3\rceil$, so 
 \[
S_i^{T_k,+,E}-S_{i_2+1}^{T_k,+,E} \geq S_i^{T_k,+,I}-S_{i_2+1}^{T_k,+,I} - (\ln n)^8 n^{1/4}-\lceil(\ln n)^3\rceil \geq \sum_{j=i}^{i_2} \zeta_j^{T_k,+,I}- (\ln n)^8 n^{1/4}-\lceil(\ln n)^3\rceil,
\]
which implies $\sum_{i=i_1}^{i_2}(S_i^{T_k,+,E}-S_{i_2+1}^{T_k,+,E}) \geq \frac{r_2}{12}(\tilde\varepsilon n)^{3/2}$, that is $\sum_{i=i_1}^{i_2}\sum_{j=i}^{i_2}\zeta_j^{T_k,+,E}\geq \frac{r_2}{12}(\tilde\varepsilon n)^{3/2}$. By the same arguments, we have $\sum_{i=i_1}^{i_2}(S_{i_2+1}^{T_{k'-2},-,E}-S_{i}^{T_{k'-2},-,E}) \geq \frac{r_2}{12}(\tilde\varepsilon n)^{3/2}$, that is $\sum_{i=i_1}^{i_2}\sum_{j=i}^{i_2}\zeta_j^{T_{k'-2},-,E}\geq \frac{r_2}{12}(\tilde\varepsilon n)^{3/2}$. \\
Now, for any $j \in I$, $\zeta_j^{T_{k'-2},-,E} = -\Delta_{\bar m,j}+1/2$ where $\bar m=\beta_{T_{k'-2}}^- = T_{k'-1}$ since $X$ follows $\gamma$, so $\zeta_j^{T_{k'-2},-,E} = -\Delta_{T_{k'-1},j}+1/2$. Furthermore, since $L_{X_{T_0}+\lfloor\varepsilon n\rfloor z_k+\lfloor\frac{\varepsilon}{4\tilde\varepsilon}\rfloor+1}^{T_{k'-1},-}=0$, for any $\bar m \in \{T_{k'-1},...,T_{k'}\}$ we have $X_{\bar m} \leq X_{T_0}+\lfloor\varepsilon n\rfloor z_k+\lfloor\frac{\varepsilon}{4\tilde\varepsilon}\rfloor$ so if $j \in I$ we have $\Delta_{T_{k'-1},j}= \Delta_{T_{k'},j}=\Delta_{T_{k},j}$. We deduce that for any $j \in I$, we have $\zeta_j^{T_{k'-2},-,E} = -\Delta_{T_k,j}+1/2=-\zeta_j^{T_k,+,B}$. Therefore $\sum_{i=i_1}^{i_2}\sum_{j=i}^{i_2}\zeta_j^{T_{k'-2},-,E}\geq \frac{r_2}{12}(\tilde\varepsilon n)^{3/2}$ becomes $\sum_{i=i_1}^{i_2}\sum_{j=i}^{i_2}-\zeta_j^{T_k,+,B}\geq \frac{r_2}{12}(\tilde\varepsilon n)^{3/2}$. Since $\sum_{i=i_1}^{i_2}\sum_{j=i}^{i_2}\zeta_j^{T_k,+,E}\geq \frac{r_2}{12}(\tilde\varepsilon n)^{3/2}$, we get $\sum_{i=i_1}^{i_2}\sum_{j=i}^{i_2}(\zeta_j^{T_k,+,E}-\zeta_j^{T_k,+,B})\geq \frac{r_2}{6}(\tilde\varepsilon n)^{3/2}$. By Observation \ref{obs_rec_temps_local}, this yields $\sum_{i=i_1}^{i_2}(L_{i}^{T_k,+}-L_{i_2+1}^{T_k,+})\geq \frac{r_2}{6}(\tilde\varepsilon n)^{3/2}$. Now, $L_{i_2+1}^{T_k,+} \geq 0$, hence $\sum_{i=i_1}^{i_2}L_{i}^{T_k,+}\geq \frac{r_2}{6}(\tilde\varepsilon n)^{3/2}$. This implies $\beta_{T_k}^+-T_k \geq \frac{r_2}{6}(\tilde\varepsilon n)^{3/2}$, thus $T_{k+1}-T_k \geq \frac{r_2}{6}(\tilde\varepsilon n)^{3/2}$. 
 
 We now assume $L_{X_{T_0}+\lfloor\varepsilon n\rfloor z_k+\lfloor\frac{\varepsilon}{4\tilde\varepsilon}\rfloor+1}^{T_{k'-1},-}\neq 0$. \\
 Since $\Theta_k^\gamma=1$, there exists $I=\{i_1,...,i_2\} \in \mathcal{E}_l$ such that $\mathcal{W}_{\gamma,k,I}^\rightarrow$ occurs.  Since $I \in \mathcal{E}_l$, $\mathcal{W}_{\gamma,k'-1,I}^\rightarrow$ also occurs. This yields $\sum_{i=i_1}^{i_2}\sum_{j=i}^{i_2} \zeta_j^{\gamma,k} \geq \frac{r_2}{6}(\tilde\varepsilon n)^{3/2}$ and $\sum_{i=i_1}^{i_2}\sum_{j=i}^{i_2} \zeta_j^{\gamma,k'-1} \geq \frac{r_2}{6}(\tilde\varepsilon n)^{3/2}$. Since $(\mathcal{B}_0^{\lfloor N\theta\rfloor,\lfloor Nx\rfloor,\pm})^c$ occurs, $n$ is large enough and $X$ follows $\gamma$, for any $j \in I$ we have $\zeta_j^{\gamma,k}=\zeta_j^{T_k,+,I}$ and $\zeta_j^{\gamma,k'-1}=\zeta_j^{T_{k'-1},-,I}$, hence $\sum_{i=i_1}^{i_2}\sum_{j=i}^{i_2} \zeta_j^{T_k,+,I} \geq \frac{r_2}{6}(\tilde\varepsilon n)^{3/2}$ and $\sum_{i=i_1}^{i_2}\sum_{j=i}^{i_2} \zeta_j^{T_{k'-1},-,I} \geq \frac{r_2}{6}(\tilde\varepsilon n)^{3/2}$. From $\sum_{i=i_1}^{i_2}\sum_{j=i}^{i_2} \zeta_j^{T_k,+,I} \geq \frac{r_2}{6}(\tilde\varepsilon n)^{3/2}$ we can deduce $\sum_{i=i_1}^{i_2}\sum_{j=i}^{i_2} \zeta_j^{T_k,+,E} \geq \frac{r_2}{12}(\tilde\varepsilon n)^{3/2}$ by the same arguments as before. However, we cannot do the same with $\sum_{i=i_1}^{i_2}\sum_{j=i}^{i_2} \zeta_j^{T_{k'-1},-,I}$, as that would require Proposition \ref{prop_sym}, that relies on $j \in \{X_{T_{k'-1}}-\lfloor\varepsilon n\rfloor+1,...,X_{T_{k'-1}}\}$, which is not the case for $i \in I$. However, $(\mathcal{B}_{T_{k'-1},3}^-)^c$ occurs, so for each $j \in I$ such that $L_j^{T_{k'-1},-} \geq (\ln n)^2$, we have $\zeta_j^{T_{k'-1},-,I}=\zeta_j^{T_{k'-1},-,E}$. Furthermore, $L_{X_{T_0}+\lfloor\varepsilon n\rfloor z_k+\lfloor\frac{\varepsilon}{4\tilde\varepsilon}\rfloor+1}^{T_{k'-1},-}\neq 0$, hence the random walk $X$ went from $X_{T_0}+\lfloor\varepsilon n\rfloor z_k$ to $X_{T_0}+\lfloor\varepsilon n\rfloor z_k+\lfloor\frac{\varepsilon}{4\tilde\varepsilon}\rfloor+1$ between times $T_{k'-1}$ and $T_{k'}$, which implies $L_j^{T_{k'-1},-} > 0$ for each $j \in I$. In addition, $(\mathcal{B}_{T_{k'-1},2}^-)^c$ occurs, thus $|\{j\in I \,|\, 0 < L_j^{T_{k'-1},-} < (\ln n)^2\}| < (\ln n)^8$, so $|\{j\in I \,|\,L_j^{T_{k'-1},-} < (\ln n)^2\}| < (\ln n)^8$. Finally, $(\mathcal{B}_{T_{k'-1},5}^-)^c$ occurs, hence for any $j \in I$ we have $|\zeta_j^{T_{k'-1},-,E}|,|\zeta_j^{T_{k'-1},-,I}| \leq (\ln n)^2$. We deduce 
 \[
 \sum_{i=i_1}^{i_2}\sum_{j=i}^{i_2} \zeta_j^{T_{k'-1},-,E} \geq \sum_{i=i_1}^{i_2}\sum_{j=i}^{i_2} \zeta_j^{T_{k'-1},-,I}-2\lceil\tilde\varepsilon n\rceil(\ln n)^{10} \geq \frac{r_2}{6}(\tilde\varepsilon n)^{3/2}-2\lceil\tilde\varepsilon n\rceil(\ln n)^{10} \geq \frac{r_2}{12}(\tilde\varepsilon n)^{3/2}
 \]
  when $n$ is large enough. Now, for any $j \in I$, $\zeta_j^{T_{k'-1},-,E} = -\Delta_{\bar m,j}+1/2$ with $\bar m = \beta_{T_{k'-1}}^-=T_{k'}$, hence $\zeta_j^{T_{k'-1},-,E} = -\Delta_{T_{k'},j}+1/2= -\Delta_{T_{k},j}+1/2=-\zeta_{j}^{T_k,+,B}$. This yields $\sum_{i=i_1}^{i_2}\sum_{j=i}^{i_2} -\zeta_{j}^{T_k,+,B} \geq \frac{r_2}{12}(\tilde\varepsilon n)^{3/2}$. Since we also proved $\sum_{i=i_1}^{i_2}\sum_{j=i}^{i_2} \zeta_j^{T_k,+,E} \geq \frac{r_2}{12}(\tilde\varepsilon n)^{3/2}$, we can end the proof as in the previous case.
 
 \emph{Case $\Theta_k^\gamma$ of type B.} 
 
 We suppose $z_{k+1}=z_k+1$, the other case can be dealt with in the same way. In this case, $(z_k,z_{k+1})$ was usable-clean at time $k$. We denote $k'$ the (only) integer below $k$ such that $(z_k,z_{k+1})$ became usable-clean at time $k'$. Then the path $\gamma$ remained below $z_k$ up to time $k$, and the only time before $k$ at which the path reached $z_k$ is time $k'-2$. Since $X$ follows $\gamma$, this implies that for any $i \in \{X_{T_0}+\lfloor\varepsilon n\rfloor z_k,...,X_{T_0}+\lfloor\varepsilon n\rfloor (z_k+1)\}$, $\Delta_{T_{k'-2},i}=\Delta_{T_{0},i}$, and $\Delta_{T_{k},i}=\Delta_{T_{k'-1},i}$. If $L_{X_{T_0}+\lfloor\varepsilon n\rfloor z_k+\lfloor\frac{\varepsilon}{4\tilde\varepsilon}\rfloor+1}^{T_{k'-2},-}\neq0$, we can prove our result using the same method as in the similar case when $\Theta_k^\gamma$ of type A, replacing $T_{k'-1}$ by $T_{k'-2}$. We now deal with the case $L_{X_{T_0}+\lfloor\varepsilon n\rfloor z_k+\lfloor\frac{\varepsilon}{4\tilde\varepsilon}\rfloor+1}^{T_{k'-2},-}=0$. Since $\Theta_k^\gamma=1$, there exists $I=\{i_1,...,i_2\} \in \mathcal{E}_r$ such that $\mathcal{W}_{\gamma,k,I}^\rightarrow$ occurs. Since $I \in \mathcal{E}_r$, $\mathcal{W}_{\gamma,k'-2,I}^{+,\rightarrow}$ also occurs. This yields $\sum_{i=i_1}^{i_2}\sum_{j=i}^{i_2} \zeta_j^{\gamma,k} \geq \frac{r_2}{6}(\tilde\varepsilon n)^{3/2}$ and $\sum_{i=i_1}^{i_2}\sum_{j=i}^{i_2} \hat\zeta_j^{\gamma,k'-2} \geq \frac{r_2}{6}(\tilde\varepsilon n)^{3/2}$. From the first inequality we can deduce $\sum_{i=i_1}^{i_2}\sum_{j=i}^{i_2} \zeta_j^{T_k,+,E} \geq \frac{r_2}{12}(\tilde\varepsilon n)^{3/2}$ as in the case $\Theta_k^\gamma$ of type A. Now, by the definition of the $\hat\zeta_j^{\gamma,k'-2}$, for any $j \in I$ we have $\hat\zeta_j^{\gamma,k'-2}= -\bar\Delta_{T_0,j}+1/2=-\Delta_{T_0,j}+1/2$, thus $\hat\zeta_j^{\gamma,k'-2}=-\Delta_{T_{k'-2},j}+1/2$. Now, since $L_{X_{T_0}+\lfloor\varepsilon n\rfloor z_k+\lfloor\frac{\varepsilon}{4\tilde\varepsilon}\rfloor+1}^{T_{k'-2},-}=0$, $X$ did not visit $j$ between times $T_{k'-2}$ and $T_{k'-1}$, hence $\Delta_{T_{k'-2},j}=\Delta_{T_{k'-1},j}=\Delta_{T_k,j}$, so $\hat\zeta_j^{\gamma,k'-2}=-\Delta_{T_k,j}+1/2=-\zeta_{j}^{T_k,+,B}$. Therefore $\sum_{i=i_1}^{i_2}\sum_{j=i}^{i_2} \hat\zeta_j^{\gamma,k'-2} \geq \frac{r_2}{6}(\tilde\varepsilon n)^{3/2}$ yields $\sum_{i=i_1}^{i_2}\sum_{j=i}^{i_2} -\zeta_{j}^{T_k,+,B} \geq \frac{r_2}{6}(\tilde\varepsilon n)^{3/2}$. We can now conclude as in the case $\Theta_k^\gamma$ of type A. 
 
 \emph{Case $\Theta_k^\gamma$ of type C.} 
 
 We suppose $z_{k+1}=z_k+1$, the other case can be dealt with in the same way. Since $\Theta_k^\gamma=1$, there exists $I \in I(z_k,z_{k+1})$ such that $\mathcal{W}_{\gamma,k,I}^\rightarrow \cap \mathcal{W}_{\gamma,k,I}^{-,\rightarrow}$, which yields $\sum_{i=i_1}^{i_2}\sum_{j=i}^{i_2} \zeta_j^{\gamma,k} \geq \frac{r_2}{6}(\tilde\varepsilon n)^{3/2}$ and $\sum_{i=i_1}^{i_2}\sum_{j=i}^{i_2} -\hat\zeta_j^{\gamma,k} \geq \frac{r_2}{6}(\tilde\varepsilon n)^{3/2}$. Since $X$ follows $\gamma$, $(\mathcal{B}_0^{\lfloor N\theta\rfloor,\lfloor Nx\rfloor,\pm})^c$ occurs and $n$ is large enough, for any $j \in I$ we have $\zeta_j^{\gamma,k}=\zeta_j^{T_k,+,I}$, so $\sum_{i=i_1}^{i_2}\sum_{j=i}^{i_2} \zeta_j^{T_k,+,I} \geq \frac{r_2}{6}(\tilde\varepsilon n)^{3/2}$. We can now use the same arguments as in the case $\Theta_k^\gamma$ of type A to deduce $\sum_{i=i_1}^{i_2}\sum_{j=i}^{i_2} \zeta_j^{T_k,+,E} \geq \frac{r_2}{12}(\tilde\varepsilon n)^{3/2}$. Moreover, for any $j \in I$ we have $\hat\zeta_j^{\gamma,k} = \bar\Delta_{T_0,j}-1/2 = \Delta_{T_0,j}-1/2$. In addition, since $\Theta_k^\gamma$ is of type C, $(z_k,z_{k+1})$ was clean at time $k$, hence the path $\gamma$ stayed strictly below $z_k$ until time $k$, thus $\Delta_{T_0,j} = \Delta_{T_k,j}$, hence $\hat\zeta_j^{\gamma,k} = \Delta_{T_k,j}-1/2=\zeta_j^{T_k,+,B}$. Consequently, $\sum_{i=i_1}^{i_2}\sum_{j=i}^{i_2} -\hat\zeta_j^{\gamma,k} \geq \frac{r_2}{6}(\tilde\varepsilon n)^{3/2}$ implies $\sum_{i=i_1}^{i_2}\sum_{j=i}^{i_2} -\zeta_j^{T_k,+,B} \geq \frac{r_2}{6}(\tilde\varepsilon n)^{3/2}$. We can now end the proof as in the case $\Theta_k^\gamma$ of type A.
 
 \emph{Case $\Theta_k^\gamma$ of type D.}
 
 We suppose $z_{k+1}=z_k+1$, the other case can be dealt with in the same way. Then since $\Theta_k^\gamma=1$, there exists $I \in I(z_k,z_{k+1})$ such that $\mathcal{W}_{\gamma,k-1,I}^\rightarrow \cap \mathcal{W}_{\gamma,k,I}^\rightarrow$ occurs. This yields $\sum_{i=i_1}^{i_2}\sum_{j=i}^{i_2} \zeta_j^{\gamma,k-1} \geq \frac{r_2}{6}(\tilde\varepsilon n)^{3/2}$ and $\sum_{i=i_1}^{i_2}\sum_{j=i}^{i_2} \zeta_j^{\gamma,k} \geq \frac{r_2}{6}(\tilde\varepsilon n)^{3/2}$. Since $X$ follows $\gamma$, $(\mathcal{B}_0^{\lfloor N\theta\rfloor,\lfloor Nx\rfloor,\pm})^c$ occurs and $n$ is large enough, for any $j \in I$ we have $\zeta_j^{\gamma,k-1}=\zeta_j^{T_{k-1},-,I}$ and $\zeta_j^{\gamma,k}=\zeta_j^{T_{k},+,I}$, so we get $\sum_{i=i_1}^{i_2}\sum_{j=i}^{i_2} \zeta_j^{T_{k-1},-,I} \geq \frac{r_2}{6}(\tilde\varepsilon n)^{3/2}$ and $\sum_{i=i_1}^{i_2}\sum_{j=i}^{i_2} \zeta_j^{T_{k},+,I} \geq \frac{r_2}{6}(\tilde\varepsilon n)^{3/2}$. From the second inequality we can deduce that $\sum_{i=i_1}^{i_2}\sum_{j=i}^{i_2} \zeta_j^{T_{k},+,E} \geq \frac{r_2}{12}(\tilde\varepsilon n)^{3/2}$ by the same arguments as in the case $\Theta_k^\gamma$ of type A; we can also apply them to the first inequality to obtain $\sum_{i=i_1}^{i_2}\sum_{j=i}^{i_2} \zeta_j^{T_{k-1},-,E} \geq \frac{r_2}{12}(\tilde\varepsilon n)^{3/2}$. Now, for any $j \in I$, $\zeta_j^{T_{k-1},-,E}=-\Delta_{\bar m,j}+1/2$ with $\bar m =\beta_{T_{k-1}}^-=T_k$, hence $\zeta_j^{T_{k-1},-,E}=-\Delta_{T_k,j}+1/2=-\zeta_j^{T_k,+,B}$. Therefore we have $\sum_{i=i_1}^{i_2}\sum_{j=i}^{i_2} -\zeta_j^{T_k,+,B} \geq \frac{r_2}{12}(\tilde\varepsilon n)^{3/2}$. We can now conclude as in the case $\Theta_k^\gamma$ of type A.
 \end{proof}
 
 In light of Proposition \ref{prop_test_to_times}, we want to prove that for any $K\in\mathds{N}^*$, for any path $\gamma$ of length $K$, the probability that there are not enough $k\in\{0,...,K-1\}$ so that $\Theta_k^\gamma$ is of type A, B, C or D and $\Theta_k^\gamma=1$ is very weak. A sequence of $\{0,1,*\}^K$ that is a possible value of $(\Theta_k^\gamma)_{0 \leq k \leq K-1}$ will be called an \emph{admissible sequence} for $\gamma$. Since the states of the edges of $\mathds{Z}$ at time $k$ depend only on the path and of the $\Theta_{k'}^\gamma$, $k'<k$, and since the states of the edges at time $k$ determine whether $\Theta_k^\gamma=*$, we have the following lemma. 
 
 \begin{lemma}\label{lem_nb_sequences}
  For any $K \in \mathds{N}^*$, there are at most $2^K$ admissible sequences for any given path of length $K$.
 \end{lemma}
 
 For any $K \in \mathds{N}^*$, for any path $\gamma$ of length $K$, we call $A(\gamma)$ the set of admissible sequences for $\gamma$. We also call $A'(\gamma)$ the set of \emph{bad admissible sequences}, that is the $(t_k)_{0 \leq k \leq K-1}\in A(\gamma)$ such that $|\{k \in\{0,...,K-1\}|t_k=0\}| \geq K/20$. All admissible sequences that are not bad will contain enough $k\in\{0,...,K-1\}$ so that $\Theta_k^\gamma$ is of type A, B, C or D and $\Theta_k^\gamma=1$, as established by the following lemma.
 
 \begin{lemma}\label{lem_seq_comb}
  For any $K \in \mathds{N}^*$, for any path $\gamma$ of length $K$, if $(\Theta_k^\gamma)_{0 \leq k \leq K-1}$ is not bad, we have $|\{k\in\{0,...,K-1\}|\,\Theta_k^\gamma$ is of type A, B, C or D and $\Theta_k^\gamma=1\}| \geq K/20$.
 \end{lemma}
 
 \begin{proof}
We notice that at each stage of the path without wait (the notion of a stage with wait was defined in the algorithm), we get either a $\Theta_k^\gamma$ which is 0 or a $\Theta_k^\gamma$ of type A, B, C or D. Since $(\Theta_k^\gamma)_{0 \leq k \leq K-1}$ is not bad, $|\{k \in\{0,...,K-1\}|\Theta_k^\gamma=0\}| < K/20$, so if there are at least $K/10$ stages without wait, $|\{k\in\{0,...,K-1\}|\,\Theta_k^\gamma$ is of type A, B, C or D and $\Theta_k^\gamma=1\}| \geq K/20$. Therefore it is enough to prove that there are at least $K/10$ stages without wait. 
 
 If there are at least $K/10$ stages with wait that are not dirty, we notice that each of these stages has to follow a stage without wait, so there are at least $K/10$ stages without wait.
 
 If there are less than $K/10$ stages with wait that are not dirty, we call $K_d$ the number of dirty stages with wait, $K_{nd}$ the number of stages with wait that are not dirty, and $K_{ww}$ the number of stages without wait. There are at least $K/2$ stages in the path (since all the edges are initially clean, the first stage is one-step long), hence $K_d+K_{nd}+K_{ww} \geq K/2$. By assumption, $K_{nd} \leq K/10$, hence $K_d+K_{ww} \geq K/2-K_{nd}\geq K/2-K/10 = 2K/5$. Now, for each stage without wait, the number of dirty edges of $\mathds{Z}$ increases by at most 3, for each dirty stage with wait, the number of dirty edges of $\mathds{Z}$ decreases by 1, and for each stage with wait that is not dirty, the number of dirty edges of $\mathds{Z}$ does not change. We deduce that $K_d \leq 3K_{ww}$, so $K_d+K_{ww} \geq 2K/5$ implies $4 K_{ww} \geq 2K/5$, thus $K_{ww} \geq K/10$, which means there are at least $K/10$ stages without wait, which ends the proof. 
 \end{proof}
 
 It now remains to prove that the probability of a bad admissible sequence to occur is very small, which is the following proposition.
 
 \begin{proposition}\label{prop_prob_bad_seq}
  When $n$ is large enough, for any $K \in \mathds{N}^*$, for any path $\gamma$ of length $K$, for any $(t_k)_{0 \leq k \leq K-1}\in A'(\gamma)$, we have $\mathds{P}(\forall \, k \in \{0,...,K-1\}, \Theta_{k}^{\gamma}=t_{k}) \leq 1/8^K$.
 \end{proposition}
 
 \begin{proof}
  If we know that $\Theta_k^\gamma = t_k$ for $0 \leq k \leq K-1$, it determines the type of the $\Theta_k^\gamma$, $k \in \{0,...,K-1\}$; if under these conditions $\Theta_k^\gamma$ is of a given type, we will say that $t_k$ is of this type. For any $k \in\{0,...,K-1\}$, we denote $\mathcal{P}_{k}^\gamma$ the event $\{\forall\ k' \in \{0,...,k\}, \Theta_{k'}^{\gamma}=t_{k'}\}$. Since $(t_k)_{0 \leq k \leq K-1}\in A'(\gamma)$, there are at least $K/20$ integers $k \in\{0,...,K-1\}$ such that $t_k=0$. Consequently, it is enough to prove that for any $k\in\{0,...,K-1\}$, if $t_k$ is of type A, B, B' or C then $\mathds{P}(\Theta_k^\gamma=0|\mathcal{P}_{k-1}^\gamma) \leq 2^{-60}$ and if $t_k$ is of type A' or D then $\mathds{P}(\Theta_k^\gamma=0|\mathcal{P}_{k-2}^\gamma) \leq 2^{-60}$ (where $\mathcal{P}_{-1}^\gamma$ denotes the whole universe). Let $k \in\{0,...,K-1\}$.
  
  \emph{Case $t_k$ of type A'.} 
  
  We suppose $z_{k+1}=z_{k}+1$; the other case can be dealt with in the same way. In this case, knowing $\gamma$ and $\Theta_{k'}^\gamma=t_{k'}$, $k' \leq k-2$ is enough to know $\Theta_{k}^\gamma$ is of type A', so $\mathds{P}(\Theta_k^\gamma=0|\mathcal{P}_{k-2}^\gamma) = \mathds{P}(\{|\{I\in I_l(z_{k-1},z_{k})|\mathcal{W}_{\gamma,k-1,I}^\leftarrow\}| < \frac{\varepsilon}{2^9\tilde \varepsilon}\} \cup \{|\{I\in I_r(z_{k-1},z_{k})|\mathcal{W}_{\gamma,k,I}^\leftarrow\}| < \frac{\varepsilon}{2^9\tilde \varepsilon}\}|\mathcal{P}_{k-2}^\gamma)$. Moreover, $\mathcal{P}_{k-2}^\gamma$ depends only on the $\zeta_i^{\gamma,k'}$, $\hat\zeta_i^{\gamma,k'}$ with $k' \leq k-2$, $i \in \mathds{Z}$, hence on the $\zeta_i^{\gamma,k'}$, $\bar\Delta_{T_0,i}$ with $k' \leq k-2$, $i \in \mathds{Z}$. In addition, the $\mathcal{W}_{\gamma,k-1,I}^\leftarrow$, $\mathcal{W}_{\gamma,k,I}^\leftarrow$ depend only on the $\zeta_i^{\gamma,k}, \zeta_i^{\gamma,k-1}$, $i \in \mathds{Z}$, which are by construction independent from the $\zeta_i^{\gamma,k'}$, $\bar\Delta_{T_0,i}$ with $k' \leq k-2$, $i \in \mathds{Z}$, hence from $\mathcal{P}_{k-2}^\gamma$. We deduce that $\mathds{P}(\Theta_k^\gamma=0|\mathcal{P}_{k-2}^\gamma) = \mathds{P}(\{|\{I\in I_l(z_{k-1},z_{k})|\mathcal{W}_{\gamma,k-1,I}^\leftarrow\}| < \frac{\varepsilon}{2^9\tilde \varepsilon}\} \cup \{|\{I\in I_r(z_{k-1},z_{k})|\mathcal{W}_{\gamma,k,I}^\leftarrow\}| < \frac{\varepsilon}{2^9\tilde \varepsilon}\})$. Therefore it is enough to prove $\mathds{P}(|\{I\in I_r(z_{k-1},z_{k})|\mathcal{W}_{\gamma,k,I}^\leftarrow\}| < \frac{\varepsilon}{2^9\tilde \varepsilon}) \leq 2^{-61}$ (as $\mathds{P}(|\{I\in I_l(z_{k-1},z_{k})|\mathcal{W}_{\gamma,k-1,I}^\leftarrow\}| < \frac{\varepsilon}{2^9\tilde \varepsilon})$ can be dealt with in the same way). Moreover, we can write 
  \[
  \mathds{P}\left(|\{I\in I_r(z_{k-1},z_{k})|\mathcal{W}_{\gamma,k,I}^\leftarrow\}| < \frac{\varepsilon}{2^9\tilde \varepsilon}\right) = \mathds{P}\left(\sum_{I\in I_r(z_{k-1},z_{k})}\mathds{1}_{\mathcal{W}_{\gamma,k,I}^\leftarrow} < \frac{\varepsilon}{2^9\tilde \varepsilon}\right),
  \]
  the fact that the $I\in I_r(z_{k-1},z_{k})$ are disjoint implies the $\mathds{1}_{\mathcal{W}_{\gamma,k,I}^\leftarrow}$ are independent, and we have $\mathds{E}(\mathds{1}_{\mathcal{W}_{\gamma,k,I}^\leftarrow}) \geq \frac{1}{32}$ by Lemma \ref{lem_proba_W}, therefore by the Hoeffding inequality, 
  \[
  \mathds{P}\left(|\{I\in I_r(z_{k-1},z_{k})|\mathcal{W}_{\gamma,k,I}^\leftarrow\}| < \frac{\varepsilon}{2^9\tilde \varepsilon}\right)
  \]
  \[
  \leq \mathds{P}\left(\sum_{I\in I_r(z_{k-1},z_{k})}\mathds{1}_{\mathcal{W}_{\gamma,k,I}^\leftarrow}-\mathds{E}\left(\sum_{I\in I_r(z_{k-1},z_{k})}\mathds{1}_{\mathcal{W}_{\gamma,k,I}^\leftarrow}\right) < \frac{\varepsilon}{2^9\tilde \varepsilon}-\frac{1}{32}\left\lfloor \frac{\varepsilon}{4\tilde\varepsilon}\right\rfloor\right) 
  \leq \exp\left(-\frac{2(\frac{1}{32}\lfloor \frac{\varepsilon}{4\tilde\varepsilon}\rfloor-\frac{\varepsilon}{2^9\tilde \varepsilon})^2}{\lfloor \frac{\varepsilon}{4\tilde\varepsilon}\rfloor}\right).
  \]
  Since $\tilde \varepsilon \leq \frac{\varepsilon}{8}$, $\lfloor \frac{\varepsilon}{4\tilde\varepsilon}\rfloor \geq \frac{\varepsilon}{8\tilde\varepsilon}$, so $\frac{1}{32}\lfloor \frac{\varepsilon}{4\tilde\varepsilon}\rfloor-\frac{\varepsilon}{2^9\tilde \varepsilon} \geq \frac{1}{32} \frac{\varepsilon}{8\tilde\varepsilon}-\frac{\varepsilon}{2^9\tilde \varepsilon} = \frac{\varepsilon}{2^9\tilde \varepsilon}$. This implies 
  \[
  \mathds{P}\left(|\{I\in I_r(z_{k-1},z_{k})|\mathcal{W}_{\gamma,k,I}^\leftarrow\}| < \frac{\varepsilon}{2^9\tilde \varepsilon}\right) \leq \exp\left(-\frac{2(\frac{\varepsilon}{2^9\tilde \varepsilon})^2}{\lfloor \frac{\varepsilon}{4\tilde\varepsilon}\rfloor}\right) 
  \leq \exp\left(-2\left(\frac{\varepsilon}{2^9\tilde \varepsilon}\right)^2\frac{4\tilde\varepsilon}{\varepsilon}\right)=\exp\left(-\frac{\varepsilon}{2^{15}\tilde\varepsilon}\right) \leq 2^{-61}
  \]
  since $\tilde\varepsilon \leq \frac{1}{2^{15}61\ln 2}\varepsilon$. This ends the proof for this case.
 
  \emph{Case $t_k$ of type B'.} 
  
  We suppose $z_{k+1}=z_{k}+1$; the other case can be dealt with in the same way. In this case, knowing $\gamma$ and $\Theta_{k'}^\gamma=t_{k'}$, $k' \leq k-1$ is enough to know $\Theta_{k}^\gamma$ is of type B', hence $\mathds{P}(\Theta_k^\gamma=0|\mathcal{P}_{k-1}^\gamma) = \mathds{P}(\{|\{I\in I_l(z_k-1,z_k)|\mathcal{W}_{\gamma,k,I}^{+,\leftarrow}\}| < \frac{\varepsilon}{2^9\tilde \varepsilon}\} \cup \{|\{I\in I_r(z_k-1,z_k)|\mathcal{W}_{\gamma,k,I}^\leftarrow\}| < \frac{\varepsilon}{2^9\tilde \varepsilon}\}|\mathcal{P}_{k-1}^\gamma)$. Furthermore, since $t_k$ is of type B', $(z_k-1,z_k)$ is clean at time $k$, which means the path $\gamma$ ``never used edge $(z_k-1,z_k)$ before time $k$'', hence when $n$ is large enough, $\mathcal{P}_{k-1}^\gamma$ depends only on $\zeta_i^{\gamma,k'}$, $\bar\Delta_{T_0,i}$ with $k' \leq k-1$, $i \not\in \{X_{T_0}+\lfloor \varepsilon n\rfloor(z_k-1),...,X_{T_0}+\lfloor \varepsilon n\rfloor z_k\}$, while the $\mathcal{W}_{\gamma,k,I}^{+,\leftarrow}, \mathcal{W}_{\gamma,k,I}^\leftarrow$ considered here depend only on the $\zeta_i^{\gamma,k}$, $\bar\Delta_{T_0,i}$ with $i \in \{X_{T_0}+\lfloor \varepsilon n\rfloor(z_k-1),...,X_{T_0}+\lfloor \varepsilon n\rfloor z_k\}$, which are independent from the former, thus from $\mathcal{P}_{k-1}^\gamma$. This yields $\mathds{P}(\Theta_k^\gamma=0|\mathcal{P}_{k-1}^\gamma) = \mathds{P}(\{|\{I\in I_l(z_k-1,z_k)|\mathcal{W}_{\gamma,k,I}^{+,\leftarrow}\}| < \frac{\varepsilon}{2^9\tilde \varepsilon}\} \cup \{|\{I\in I_r(z_k-1,z_k)|\mathcal{W}_{\gamma,k,I}^\leftarrow\}| < \frac{\varepsilon}{2^9\tilde \varepsilon}\})$, so it is enough to prove that $\mathds{P}(|\{I\in I_l(z_k-1,z_k)|\mathcal{W}_{\gamma,k,I}^{+,\leftarrow}\}| < \frac{\varepsilon}{2^9\tilde \varepsilon}) \leq 2^{-61}$ and $\mathds{P}(|\{I\in I_r(z_k-1,z_k)|\mathcal{W}_{\gamma,k,I}^\leftarrow\}| < \frac{\varepsilon}{2^9\tilde \varepsilon}) \leq 2^{-61}$. This can be done in the same way as for the case $t_k$ of type A', noticing that since $(z_k-1,z_k)$ is clean at time $k$, the path $\gamma$ did not cross the edge $(z_k-1,z_k)$ before time $k$, thus $z_k \leq 0$ and the intervals $I$ we consider are contained in $(-\infty,X_{T_0}+\lfloor\varepsilon n\rfloor z_{k}-1]$, so we can use Lemma \ref{lem_proba_W}. 
  
  \emph{Case $t_k$ of type A.} 
  
  We suppose $z_{k+1}=z_{k}+1$; the other case can be dealt with in the same way. We will use the notation $k'$, $\mathcal{E}_r$ and $\mathcal{E}_l$ introduced when describing the $\Theta_k^\gamma$ of type A. Knowing $\gamma$ and $\Theta_{k''}^\gamma=t_{k''}$, $k'' \leq k-1$ is enough to know $\Theta_{k}^\gamma$ is of type A and to determine $k'$, hence $\mathds{P}(\Theta_k^\gamma=0|\mathcal{P}_{k-1}^\gamma) = \mathds{P}((\bigcap_{I \in \mathcal{E}_l}(\mathcal{W}_{\gamma,k,I}^\rightarrow)^c) \cup (\bigcap_{I \in \mathcal{E}_r}(\mathcal{W}_{\gamma,k,I}^\rightarrow)^c)|\mathcal{P}_{k-1}^\gamma)$. Therefore it is enough to prove $\mathds{P}(\bigcap_{I \in \mathcal{E}_l}(\mathcal{W}_{\gamma,k,I}^\rightarrow)^c|\mathcal{P}_{k-1}^\gamma) \leq 2^{-61}$, as $\mathds{P}(\bigcap_{I \in \mathcal{E}_r}(\mathcal{W}_{\gamma,k,I}^\rightarrow)^c|\mathcal{P}_{k-1}^\gamma) \leq 2^{-61}$ can be proven in the same way. If $\mathcal{P}_{k-1}^\gamma$ occurs, $|\mathcal{E}_l| \geq \frac{\varepsilon}{2^9\tilde \varepsilon}$, so 
  \begin{equation}\label{eq_proba_tests}
  \mathds{P}\left(\left.\bigcap_{I \in \mathcal{E}_l}(\mathcal{W}_{\gamma,k,I}^\rightarrow)^c\right|\mathcal{P}_{k-1}^\gamma\right)
  = \sum_{E \subset I_l(z_k,z_{k+1}),|E| \geq \frac{\varepsilon}{2^9\tilde \varepsilon}} \mathds{P}\left(\left.\bigcap_{I \in \mathcal{E}_l}(\mathcal{W}_{\gamma,k,I}^\rightarrow)^c,\mathcal{E}_l=E\right|\mathcal{P}_{k-1}^\gamma\right).
  \end{equation}
  Now, for any $E \subset I_l(z_k,z_{k+1})$ with $|E| \geq \frac{\varepsilon}{2^9\tilde \varepsilon}$, we have $\mathds{P}(\bigcap_{I \in \mathcal{E}_l}(\mathcal{W}_{\gamma,k,I}^\rightarrow)^c,\mathcal{E}_l=E|\mathcal{P}_{k-1}^\gamma)=\mathds{P}(\bigcap_{I \in E}(\mathcal{W}_{\gamma,k,I}^\rightarrow)^c,\mathcal{E}_l=E|\mathcal{P}_{k-1}^\gamma)$. Moreover, $\mathcal{P}_{k-1}^\gamma$ and $\{\mathcal{E}_l=E\}$ depend only on the $\zeta_i^{\gamma,k''}$, $\bar\Delta_{T_0,i}$ with $k'' \leq k-1$, $i \in \mathds{Z}$, while the $\mathcal{W}_{\gamma,k,I}^\rightarrow$ depend on the $\zeta_i^{\gamma,k}$, $i \in \mathds{Z}$, which are independent from the former. This implies $\mathds{P}(\bigcap_{I \in E}(\mathcal{W}_{\gamma,k,I}^\rightarrow)^c,\mathcal{E}_l=E|\mathcal{P}_{k-1}^\gamma) = \mathds{P}(\bigcap_{I \in E}(\mathcal{W}_{\gamma,k,I}^\rightarrow)^c)\mathds{P}(\mathcal{E}_l=E|\mathcal{P}_{k-1}^\gamma)$, hence equation (\ref{eq_proba_tests}) becomes 
  \[
   \mathds{P}\left(\left.\bigcap_{I \in \mathcal{E}_l}(\mathcal{W}_{\gamma,k,I}^\rightarrow)^c\right|\mathcal{P}_{k-1}^\gamma\right)
  = \sum_{E \subset I_l(z_k,z_{k+1}),|E| \geq \frac{\varepsilon}{2^9\tilde \varepsilon}} \mathds{P}\left(\bigcap_{I \in E}(\mathcal{W}_{\gamma,k,I}^\rightarrow)^c\right)\mathds{P}(\mathcal{E}_l=E|\mathcal{P}_{k-1}^\gamma),
  \]
  so it is enough to prove that for any $E \subset I_l(z_k,z_{k+1})$ with $|E| \geq \frac{\varepsilon}{2^9\tilde \varepsilon}$, $\mathds{P}(\bigcap_{I \in E}(\mathcal{W}_{\gamma,k,I}^\rightarrow)^c) \leq 2^{-61}$. Now let $E$ be such a set, then the $I \in E$ are disjoint hence the $\mathcal{W}_{\gamma,k,I}^\rightarrow$ are independent, thus $\mathds{P}(\bigcap_{I \in E}(\mathcal{W}_{\gamma,k,I}^\rightarrow)^c) = \prod_{I \in E}\mathds{P}((\mathcal{W}_{\gamma,k,I}^\rightarrow)^c) \leq (\frac{31}{32})^{\frac{\varepsilon}{2^9\tilde \varepsilon}}$ by Lemma \ref{lem_proba_W} and $|E| \geq \frac{\varepsilon}{2^9\tilde \varepsilon}$. Since $\tilde\varepsilon \leq -\frac{\ln(\frac{31}{32})}{2^9 61\ln 2}\varepsilon$, we indeed obtain $\mathds{P}(\bigcap_{I \in E}(\mathcal{W}_{\gamma,k,I}^\rightarrow)^c) \leq 2^{-61}$. 
  
  \emph{Case $t_k$ of type B.} 
  
  This case can be dealt with using the same arguments as for the case $t_k$ of type A.
  
  \emph{Case $t_k$ of type C.}  
  
  We suppose $z_{k+1}=z_{k}+1$; the other case can be dealt with in the same way. In this case, knowing $\gamma$ and $\Theta_{k'}^\gamma=t_{k'}$, $k' \leq k-1$ is enough to know $\Theta_{k}^\gamma$ is of type C, hence $\mathds{P}(\Theta_k^\gamma=0|\mathcal{P}_{k-1}^\gamma) = \mathds{P}(\bigcap_{I \in I(z_k,z_{k+1})}((\mathcal{W}_{\gamma,k,I}^\rightarrow)^c \cup (\mathcal{W}_{\gamma,k,I}^{-,\rightarrow})^c)|\mathcal{P}_{k-1}^\gamma)$. Furthermore, since $t_k$ is of type C, $(z_k,z_{k+1})$ is clean at time $k$, which means the path $\gamma$ ``never used edge $(z_k,z_{k+1})$ before time $k$'', hence when $n$ is large enough, $\mathcal{P}_{k-1}^\gamma$ depends only on $\zeta_i^{\gamma,k'}$, $\bar\Delta_{T_0,i}$ with $k' \leq k-1$, $i \not\in \{X_{T_0}+\lfloor \varepsilon n\rfloor z_k,...,X_{T_0}+\lfloor \varepsilon n\rfloor z_{k+1}\}$, while the $\mathcal{W}_{\gamma,k,I}^\rightarrow, \mathcal{W}_{\gamma,k,I}^{-,\rightarrow}$ we consider depend only on $\zeta_i^{\gamma,k}$, $\bar\Delta_{T_0,i}$ with $i\in \{X_{T_0}+\lfloor \varepsilon n\rfloor z_k,...,X_{T_0}+\lfloor \varepsilon n\rfloor z_{k+1}\}$, which are independent from the former thus from $\mathcal{P}_{k-1}^\gamma$. This implies $\mathds{P}(\Theta_k^\gamma=0|\mathcal{P}_{k-1}^\gamma) = \mathds{P}(\bigcap_{I \in I(z_k,z_{k+1})}((\mathcal{W}_{\gamma,k,I}^\rightarrow)^c \cup (\mathcal{W}_{\gamma,k,I}^{-,\rightarrow})^c))$. In addition, the $I \in I(z_k,z_{k+1})$ are disjoint hence the $(\mathcal{W}_{\gamma,k,I}^\rightarrow)^c \cup (\mathcal{W}_{\gamma,k,I}^{-,\rightarrow})^c$ are independent, so $\mathds{P}(\Theta_k^\gamma=0|\mathcal{P}_{k-1}^\gamma) = \prod_{I \in I(z_k,z_{k+1})}\mathds{P}((\mathcal{W}_{\gamma,k,I}^\rightarrow)^c \cup (\mathcal{W}_{\gamma,k,I}^{-,\rightarrow})^c)$. Now, let $I \in I(z_k,z_{k+1})$, we have $\mathds{P}((\mathcal{W}_{\gamma,k,I}^\rightarrow)^c \cup (\mathcal{W}_{\gamma,k,I}^{-,\rightarrow})^c) =1-\mathds{P}(\mathcal{W}_{\gamma,k,I}^\rightarrow \cap \mathcal{W}_{\gamma,k,I}^{-,\rightarrow})=1-\mathds{P}(\mathcal{W}_{\gamma,k,I}^\rightarrow)\mathds{P}(\mathcal{W}_{\gamma,k,I}^{-,\rightarrow})$ as $\mathcal{W}_{\gamma,k,I}^\rightarrow$ is independent from $\mathcal{W}_{\gamma,k,I}^{-,\rightarrow}$ (they depend respectively on $\zeta_i^{\gamma,k}$ and $\bar\Delta_{T_0,i}$). Furthermore, $(z_k,z_{k+1})$ is clean at time $k$, thus the path $\gamma$ never crossed edge $(z_k,z_{k+1})$ before time $k$, hence $z_k \geq 0$, and we have $I \subset [X_{T_0}+\lfloor\varepsilon n\rfloor z_k+1,+\infty)$, so we can apply Lemma \ref{lem_proba_W} to $\mathcal{W}_{\gamma,k,I}^{-,\rightarrow}$, as well as to $\mathcal{W}_{\gamma,k,I}^\rightarrow$, which yields $\mathds{P}((\mathcal{W}_{\gamma,k,I}^\rightarrow)^c \cup (\mathcal{W}_{\gamma,k,I}^{-,\rightarrow})^c)=1-\mathds{P}(\mathcal{W}_{\gamma,k,I}^\rightarrow)\mathds{P}(\mathcal{W}_{\gamma,k,I}^{-,\rightarrow}) \leq 1-(\frac{1}{32})^2=1-2^{-10}$. We deduce $\mathds{P}(\Theta_k^\gamma=0|\mathcal{P}_{k-1}^\gamma) \leq (1-2^{-10})^{2\lfloor \frac{\varepsilon}{4\tilde\varepsilon}\rfloor} \leq (1-2^{-10})^{\frac{\varepsilon}{4\tilde\varepsilon}} \leq 2^{-60}$ since $\tilde\varepsilon \leq \frac{-\ln(1-2^{-10})}{240\ln 2}\varepsilon$. 
  
  \emph{Case $t_k$ of type D.} 
  
  We assume $z_{k+1}=z_{k}+1$; the other case can be dealt with in the same way (beware: the definition of type D was detailed for $z_{k+1}=z_{k}-1$). In this case, knowing $\gamma$ and $\Theta_{k'}^\gamma=t_{k'}$, $k' \leq k-2$ is enough to know $\Theta_{k}^\gamma$ is of type D, hence $\mathds{P}(\Theta_k^\gamma=0|\mathcal{P}_{k-2}^\gamma) = \mathds{P}(\bigcap_{I \in I(z_{k+1},z_{k})}((\mathcal{W}_{\gamma,k-1,I}^\rightarrow)^c \cup (\mathcal{W}_{\gamma,k,I}^\rightarrow)^c)|\mathcal{P}_{k-2}^\gamma)$. Moreover, $\mathcal{P}_{k-2}^\gamma$ depends only on the $\zeta_i^{\gamma,k'}$, $\bar\Delta_{T_0,i}$ with $k' \leq k-2$, $i \in \mathds{Z}$, while the $\mathcal{W}_{\gamma,k-1,I}^\rightarrow$, $\mathcal{W}_{\gamma,k,I}^\rightarrow$ depend only on the $\zeta_i^{\gamma,k-1},\zeta_i^{\gamma,k}$ with $i\in\mathds{Z}$, which are independent from the former, hence from $\mathcal{P}_{k-2}^\gamma$. We deduce $\mathds{P}(\Theta_k^\gamma=0|\mathcal{P}_{k-2}^\gamma) = \mathds{P}(\bigcap_{I \in I(z_{k+1},z_{k})}((\mathcal{W}_{\gamma,k-1,I}^\rightarrow)^c \cup (\mathcal{W}_{\gamma,k,I}^\rightarrow)^c))$, which can be bounded by the same arguments as in the case $t_k$ of type C.
 \end{proof}
 
 We are now able to conclude. Proposition \ref{prop_test_to_times} and Lemma \ref{lem_seq_comb} allow to deduce that for any $K \in \mathds{N}^*$, when $n$ is large enough, if $T_K-T_0 < \frac{K}{20}\frac{r_2}{6}(\tilde\varepsilon n)^{3/2}$ and $\mathcal{B}^c \cap (\bigcap_{r=0}^{6}\mathcal{B}_r^c)$ occurs, there exists a path $\gamma$ of length $K$ so that $(\Theta_k)_{0 \leq k \leq K-1} \in A'(\gamma)$. In addition, there are $2^K$ possible paths of length $K$, therefore Lemma \ref{lem_nb_sequences} and Proposition \ref{prop_prob_bad_seq} yield the following. 
 
 \begin{proposition}\label{prop_T_K}
  For any $K\in \mathds{N}^*$, for $n$ large enough, 
  \[
  \mathds{P}\left(T_K-T_0 < K\frac{r_2}{120}(\tilde\varepsilon n)^{3/2},\mathcal{B}^c \cap \bigcap_{r=0}^{6}\mathcal{B}_r^c\right) \leq \frac{1}{2^K}.
  \]
 \end{proposition}

 \section{The limit process of the environments}\label{sec_limit_processes}
 
 In Section \ref{sec_conv_limit_processes}, we will need to prove the joint convergence in distribution of the position of our random walk at times $T_0,...T_K$ and of ``environment'' processes depending on the $\Delta_{T_k,j}$, $k\in\{0,...,K\}$ (see Definition \ref{def_envts_bis}). In order to show this convergence, we will need some results on the limit process, the ``limit process of the environments''. We believe said limit process to be of independent interest. In Section \ref{sec_BM_results}, we will prove some results on Brownian motions reflected on and absorbed by general barriers (we recall the Definition \ref{def_reflected_BM} of the reflected Brownian motion), which are interesting in their own right and which we will need to apply to the limit process of the environments. In Section \ref{sec_process_above}, we give the definition of the limit process of the environnements and prove that the results of Section \ref{sec_BM_results} can actually be applied to it. 
 
 \subsection{Brownian motion results}\label{sec_BM_results}
 
 Let us set some notation. The Brownian motions in the subsection will all have the same variance, which can be any positive real. Our barrier will be a continuous function $f:[-1,1]\mapsto\mathds{R}$. We suppose for notational convenience (and with no loss of generality) that $f(0)=0$. We consider a process $(\tilde W_t^-)_{t\in[-1,1]}$ which is a Brownian motion $(W_t^-)_{t\in[-1,1]}$ reflected on $f$ above $f$ on $[-1,0]$, starting with $\tilde W_{-1}^- = f(-1)$, and absorbed by the barrier $f$ on $[0,1]$. We denote $\sigma_- = \inf \{t\geq0 \,|\, \tilde W_t^-= f(t)\}$ the absorption time, and $p_-=\mathds{P}(\sigma_- <1)$ the probability of absorption. Similarly, we consider a process $(\tilde W_t^+)_{t\in[-1,1]}$ which is a Brownian motion starting with $\tilde W_{1}^+ = f(1)$, reflected on $f$ above $f$ on $[0,1]$, and absorbed by $f$ on $[-1,0]$. We denote $\sigma_+ = \sup \{t\leq0 \,|\, \tilde W_t^+= f(t)\}$ the absorption time, and $p_+=\mathds{P}(\sigma_+ >-1)$ the probability of absorption. We want to understand when we have $p_-+p_+=1$.
 
 \begin{proposition}\label{prop_Z}
 We always have $p_-+p_+\geq1$. Moreover, we define a random variable $Z$ as follows: let $\bar W^-$ and $\bar W^+$ be two independent Brownian motions on $[0,1]$ with $\bar W_0^-=\bar W_0^+=0$, we set $Z=\sup_{0 \leq t \leq 1}(\bar W^-_t+f(-t))+\inf_{0 \leq t \leq 1}(\bar W_t^+-f(t))$. Then we have $p_-+p_+=1$ if and only if $\mathds{P}(Z=0)=0$. 
 \end{proposition}
 
 \begin{proof}
 By definition, for any $t\in[-1,0]$ we have $\tilde W_t^-=W_{t}^- +\sup_{-1\leq s \leq t}(f(s)-W_s^-)=\sup_{-1\leq s \leq t}(f(s)+W_{t}^--W_s^-)$, and for $t\in[0,\sigma \wedge 1]$ we have $\tilde W_t^-=\tilde W_0^-+(W^-_t-W^-_0)$. Therefore we have $\mathds{P}(\sigma_-=1)\leq\mathds{P}(\tilde W_0^-+(W^-_1-W^-_0)=f(1))$, while $(W^-_t-W^-_0)_{t\in[0,1]}$ is a Brownian motion independent from $\tilde W_0^-$, hence $\mathds{P}(\sigma_-=1)=0$. This implies $p_-=\mathds{P}(\sigma_-\leq1)$. In addition, $\sigma_-\leq 1$ when $\inf_{0 \leq t \leq 1}(\tilde W^-_t-f(t)) \leq 0$, thus when $\inf_{0 \leq t \leq 1}(\tilde W_0^- + (W^-_t-W^-_0)-f(t)) \leq 0$, that is $\tilde W_0^- + \inf_{0 \leq t \leq 1}((W^-_t-W^-_0)-f(t)) \leq 0$ which can be written as $\sup_{-1\leq t \leq 0}(f(t)+W_{0}^--W_t^-) + \inf_{0 \leq t \leq 1}((W^-_t-W^-_0)-f(t)) \leq 0$. This implies $p_-=\mathds{P}(\sup_{-1\leq t \leq 0}(f(t)+W_{0}^--W_t^-) + \inf_{0 \leq t \leq 1}((W^-_t-W^-_0)-f(t)) \leq 0)=\mathds{P}(Z \leq 0)$. Now, $p_+$ corresponds to the $p_-$ associated to the function $\bar f:[-1,1]\mapsto\mathds{R}$ defined by $\bar f(t)=f(-t)$ for any $s\in[-1,1]$. This yields 
 \[
 p_+  =\mathds{P}\left(\sup_{0 \leq t \leq 1}(\bar W^-_t+\bar f(-t))+\inf_{0 \leq t \leq 1}(\bar W_t^+- \bar f(t)) \leq 0\right)
 \]
 \[
 =\mathds{P}\left(\sup_{0 \leq t \leq 1}(\bar W^-_t+f(t))+\inf_{0 \leq t \leq 1}(\bar W_t^+- f(-t)) \leq 0\right) 
  = \mathds{P}\left(\inf_{0 \leq t \leq 1}(-\bar W^-_t-f(t))+\sup_{0 \leq t \leq 1}(-\bar W_t^++ f(-t)) \geq 0\right),
 \]
 but $\inf_{0 \leq t \leq 1}(-\bar W^-_t-f(t))+\sup_{0 \leq t \leq 1}(-\bar W_t^++ f(-t)) \geq 0$ has the same law as $Z$, so $p_+=\mathds{P}(Z\geq0)$. Since we also have $p_-=\mathds{P}(Z \leq 0)$, we always have $p_-+p_+\geq1$, and we have $p_-+p_+=1$ if and only if $\mathds{P}(Z=0)=0$. 
\end{proof}

In order to get both a more practical condition for having $p_-+p_+=1$ than the one in Proposition \ref{prop_Z} and auxiliary results that will be useful in Section \ref{sec_conv_limit_processes}, we need to introduce some stopping times. Let $(W_t)_{t\in[0,1]}$ a Brownian motion, and $g : [0,1] \mapsto \mathds{R}$ a continuous function. For any $\delta\in\mathds{R}$, we define $\sigma(\delta)=\inf\{t\in[0,1]\,|\,W_t\leq g(t)+\delta\}$, the inf being infinite when the set is empty.

\begin{lemma}\label{lem_conv_abs_times}
 For any continuous function $g : [0,1] \mapsto \mathds{R}$ (possibly random) so that $g(0)<W_0$ almost surely, we have that $\sigma(\delta)$ converges in probability to $\sigma(0)$ as $\delta$ tends to 0.
\end{lemma}

\begin{proof}
We first suppose $g$ and $W_0$ are deterministic and $g(0)<W_0$. It is enough to prove that for any $a>0$, $\mathds{P}(|\sigma(\delta)-\sigma(0)|>a)$ tends to 0 when $\delta$ tends to 0. We will treat $\delta >0$, the negative case is handled similarly. For any $a>0$, for any $\delta>0$, we notice that $\sigma(0)\geq\sigma(\delta)$, so if $|\sigma(\delta)-\sigma(0)|>a$ then $\sigma(0)-\sigma(\delta)>a$, so there exists a non-negative integer $i \leq \lfloor 1/a \rfloor+1$ so that $\sigma(\delta)\leq ia$ and $\sigma(0)> ia$. We deduce 
\begin{equation}\label{eq_conv_abs_times}
 \mathds{P}(|\sigma(\delta)-\sigma(0)|>a)\leq \left(\left\lfloor \frac{1}{a} \right\rfloor+2\right)\max_{t\in[0,1]}\mathds{P}(\sigma(\delta)\leq t, \sigma(0)> t). 
\end{equation}
We thus need to study the $\mathds{P}(\sigma(\delta)\leq t, \sigma(0)> t)$. For any $\delta>0$, we consider a Brownian motion $(W_t^\delta)_{t\in[0,1]}$ starting from $W_0+\delta$, independent from $(W_t)_{t\in[0,1]}$ until they meet, and then coalescing with $(W_t)_{t\in[0,1]}$. We also denote $\sigma'(\delta)=\inf\{t\in[0,1]\,|\,W_t^\delta\leq g(t)+\delta\}$. Since $\delta>0$, we have $W_t^\delta \geq W_t$ for any $t\in[0,1]$, thus we have $\sigma(\delta)\leq\sigma'(\delta)$. Moreover, $\sigma'(\delta)$ has the same law as $\sigma(0)$. We deduce that for $t\in[0,1]$, denoting $T^\delta$ the time of coalescence of $(W_t)_{t\in[0,1]}$ and $(W_t^\delta)_{t\in[0,1]}$, 
\[
 \mathds{P}(\sigma(\delta)\leq t, \sigma(0)> t)=\mathds{P}(\sigma(\delta)\leq t)-\mathds{P}(\sigma(0)\leq t)
 \]
 \[
 =\mathds{P}(\sigma(\delta)\leq t)-\mathds{P}(\sigma'(\delta)\leq t)=\mathds{P}(\sigma(\delta)\leq t, \sigma'(\delta)> t) \leq \mathds{P}(T^\delta>\sigma(\delta)). 
 \]
 From this and (\ref{eq_conv_abs_times}) we deduce $\mathds{P}(|\sigma(\delta)-\sigma(0)|>a)\leq (\lfloor \frac{1}{a} \rfloor+2)\mathds{P}(T^\delta>\sigma(\delta))$, so it is enough to prove $\mathds{P}(T^\delta>\sigma(\delta))$ tends to 0 when $\delta$ tends to 0. To do that, we denote $\delta_0=\frac{W_0-g(0)}{2}>0$. When $\delta\leq\delta_0$ we have $\sigma(\delta)\geq\sigma(\delta_0)$ hence $\mathds{P}(T^\delta>\sigma(\delta)) \leq \mathds{P}(T^\delta>\sigma(\delta_0))$. Now, $\sigma(\delta_0)>0$ and $T^\delta$ converges in probability to 0 when $\delta$ tends to 0, therefore $\lim_{\delta\to0}\mathds{P}(T^\delta>\sigma(\delta_0))=0$, which ends the proof when $g$ and $W_0$ are deterministic. If $g$ and $W_0$ are random, we notice that for any $a>0$, $\mathds{P}(|\sigma(\delta)-\sigma(0)|>a)=\mathds{E}(\mathds{P}(|\sigma(\delta)-\sigma(0)|>a\,|\,g,W_0))$, and that for any value of $g$ and $W_0$ so that $g(0)<W_0$, we have $\lim_{\delta\to0}\mathds{P}(|\sigma(\delta)-\sigma(0)|>a\,|\,g,W_0)=0$, hence $\mathds{P}(|\sigma(\delta)-\sigma(0)|>a\,|\,g,W_0)$ converges almost surely to 0 when $\delta$ tends to 0, therefore $\lim_{\delta\to0}\mathds{P}(|\sigma(\delta)-\sigma(0)|>a)=0$. 
\end{proof}

Lemma \ref{lem_conv_abs_times} allows us to prove the following condition, more practical than the one in Proposition \ref{prop_Z}. 

\begin{proposition}\label{prop_better_than_Z}
 If $\mathds{P}(\tilde W_0^->f(0))=1$, then $p_-+p_+=1$.
\end{proposition}

\begin{proof}
Let us assume $\mathds{P}(\tilde W_0^->f(0))=1$. We recall that by Proposition \ref{prop_Z}, proving $\mathds{P}(Z=0)=0$ is enough to prove $p_-+p_+=1$. Now, by definition $\tilde W_0^-=W_0^-+\sup_{-1 \leq t \leq 0}(f(t)-W_t^-)=\sup_{-1 \leq t \leq 0}(f(t)-W_t^-+W_0^-)$ which has the same law as $\sup_{0 \leq t \leq 1}(\bar W^-_t+f(-t))$, so 
\[
\mathds{P}(Z=0)=\mathds{P}\left(\tilde W_0^-+\inf_{0 \leq t \leq 1}(\bar W_t^+-f(t))=0\right)=\mathds{P}\left(\inf_{0 \leq t \leq 1}(\tilde W_0^-+\bar W_t^+-f(t))=0\right).
\]
We use the notation of Lemma \ref{lem_conv_abs_times} with the process $(\tilde W_0^-+\bar W_t^+)_{t\in[0,1]}$ replacing $(W_t)_{t\in[0,1]}$ and the restriction of $f$ to $[0,1]$ replacing $g$. We then have $\mathds{P}(Z=0)\leq\mathds{P}(\sigma(0)<+\infty,\forall\,\delta<0,\sigma(\delta)=+\infty)$. Now, since $\tilde W_0^-+\bar W_0^+=\tilde W_0^-$ and $\mathds{P}(\tilde W_0^->f(0))=1$, Lemma \ref{lem_conv_abs_times} implies $\sigma(\delta)$ converges in probability to $\sigma(0)$ when $\delta$ tends to 0, hence $\mathds{P}(\sigma(0)<+\infty,\forall\,\delta<0,\sigma(\delta)=+\infty)=0$, therefore $\mathds{P}(Z=0)=0$, which ends the proof.
\end{proof}

We are going to establish another criterion for having $p_-+p_+=1$, which will not be used in this paper but has independent interest. Proposition \ref{prop_Z} stated that $p_-+p_+=1$ if and only if $\mathds{P}(Z=0)=0$, and we saw in the proof of Proposition \ref{prop_better_than_Z} that $\mathds{P}(Z=0)=\mathds{P}(\inf_{0 \leq t \leq 1}(\tilde W_0^-+\bar W_t^+-f(t))=0)$, and that this was 0 if $\mathds{P}(\tilde W_0^->f(0))=1$. Therefore $p_-+p_+>1$ if and only if $\mathds{P}(\tilde W_0^- = f(0))>0$ and with strictly positive probability a Brownian motion $(W_t)_{t\in[0,1]}$ starting at $0$ satisfies $W_t \geq f(t)$ for $0 \leq t \leq 1$. Now, recall that a function $f:[0,1] \mapsto \mathds{R}$ with $f(0)=0$ is called a \emph{lower function} if $\mathds{P}(\forall \, 0 \leq t \leq 1, W_t \geq f(t))>0$. So for example, if $0 < \epsilon < 2$ and $W$ is a standard Brownian motion, a continuous function equivalent to $-\sqrt {(2 + \epsilon ) t \ln(\ln(\frac{1}{t}))}$ around 0 is a lower function (indeed, the Law of the Iterated Logarithm implies there exists $\delta>0$ so that $\mathds{P}(\forall \, 0 \leq t \leq \delta, W_t \geq \sqrt{\frac{1}{1+\varepsilon}}f(t))>0$, and the Forgery Theorem (Theorem 38 of \cite{Freedman_Brownian_Motion}) implies $\mathds{P}\left(\forall \, \delta \leq t \leq 1, W_t-W_\delta \geq f(t)-f(\delta)-\left(\sqrt{\frac{1}{1+\varepsilon}}-1\right)f(\delta)\right)>0$), but a function equivalent to $-\sqrt {(2 - \epsilon ) t \ln(\ln(\frac{1}{t}))}$ is not (for refinement see \cite{Erdos1942}). Furthermore, $\mathds{P}(\tilde W_0^- = f(0))>0$ if and only if $\mathds{P}(\sup_{-1 \leq t \leq 0}(f(t)-W_t^-+W_0^-)=0)>0$, which is the case if and only if the function $: t \mapsto -f(-t)$ is a lower function. We deduce the following criterion.

\begin{proposition}\label{prop_Z_lower_fcts}
$p_-+p_+>1$ if and only if the functions $f_1,f_2:[0,1] \mapsto \mathds{R}$ defined by $f_1(t)=f(t)$ and $f_2(t)= -f(-t)$ for $t\in[0,1]$ are both lower functions.
\end{proposition}
 
 \subsection{The limit process of the environments} \label{sec_process_above}
 
 In this section, the variance of all Brownian motions will be the variance of the law $\rho_0$ defined in \eqref{eq_def_rho0}. Moreover, we have as usual $\varepsilon > 0$. The limit process of the environments will be the following.
 
 \begin{definition}\label{def_lim_envts}
 $W^0$ will be a two-sided Brownian motion with $W_0^0=0$. We denote $\breve Z_0=0$. Let $k\in\mathds{N}$, and suppose that $W^{k'}$, $\breve Z_{k'}$ are defined for any $k' \in \{0,...,k\}$, we construct $W^{k+1}$ as follows. 
 
 We consider a continuous process $(V_t^{k,-})_{t\in[-\varepsilon,\varepsilon]}$ defined as follows: $V_{-\varepsilon}^{k,-}=W_{-\varepsilon}^k$, $(V_t^{k,-})_{t\in[-\varepsilon,0]}$ is a Brownian motion above $W^k$ reflected on $W^k$, and $(V_t^{k,-})_{t\in[0,\varepsilon]}$ is a Brownian motion absorbed by $W^k$. Let $\sigma_{k,-}=\inf\{t\geq0\,|\,V_{t}^{k,-}=W_{t}^k\}$ be the absorption time, and $p_{k,-}=\mathds{P}(\sigma_{k,-}<\varepsilon|W^k)$ the probability of absorption. Similarly, let $(V_t^{k,+})_{t\in[-\varepsilon,\varepsilon]}$ so that $V_{\varepsilon}^{k,+}=W_{\varepsilon}^k$, $(V_{\varepsilon-t}^{k,+})_{t\in[0,\varepsilon]}$ is a Brownian motion reflected on $(W_{\varepsilon-t}^{k})_{t\in[0,\varepsilon]}$ above $(W_{\varepsilon-t}^{k})_{t\in[0,\varepsilon]}$ and $(V_{-t}^{k,+})_{t\in[0,\varepsilon]}$ is a Brownian motion absorbed by $(W_{-t}^{k})_{t\in[0,\varepsilon]}$, 
 let $\sigma_{k,+}=\sup\{t\leq0\,|\,V_{t}^{k,+}=W_{t}^k\}$ be the absorption time, and set $p_{k,+}=\mathds{P}(\sigma_{k,+}>-\varepsilon|W^k)$. 
 
 Then, independently from the $W^{k'}$, $k' \in \{0,...,k\}$, we set $\breve Z_{k+1}=\breve Z_k-1$ with probability $p_{k,-}$ and $\breve Z_{k+1}=\breve Z_k+1$ with probability $1-p_{k,-}$. 
 \begin{itemize}
 \item If $\breve Z_{k+1}=\breve Z_k-1$, $W^{k+1}$ is defined as follows. For $t \in (-\infty,0]\cup[2\varepsilon,+\infty)$, we set $W_t^{k+1} = W_{t-\varepsilon}^{k}-W_{-\varepsilon}^k$. Moreover, we define a process $(\bar W_t^{k,-})_{t\in[-\varepsilon,\varepsilon]}$ thus: $\bar W_{-\varepsilon}^{k,-}=W_{-\varepsilon}^k$, $(\bar W_t^{k,-})_{t\in[-\varepsilon,0]}$ is a Brownian motion above $W^k$ reflected on $W^k$, and $(\bar W_t^{k,-})_{t\in[0,\varepsilon]}$ is a Brownian motion absorbed by $W^k$, but $(\bar W_t^{k,-})_{t\in[-\varepsilon,\varepsilon]}$ is conditioned to coalesce with $W_k$ before time $\varepsilon$. Then for any $t\in[0,2\varepsilon]$, we set $W_{t}^{k+1}=\bar W_{t-\varepsilon}^{k,-}-W_{-\varepsilon}^k$. In addition, we set $\breve T_{k+1}=2\int_{-\varepsilon}^{\varepsilon}(\bar W_t^{k,-}-W_t^k)\mathrm{d}t$.
 \item If $\breve Z_{k+1}=\breve Z_k+1$, the definition is similar. If $t\in(-\infty,-2\varepsilon]\cup[0,+\infty)$, we set $W_t^{k+1}=W_{t+\varepsilon}^k-W_\varepsilon^k$. We also define a process $(\bar W_t^{k,+})_{t\in[-\varepsilon,\varepsilon]}$ so that $\bar W_{\varepsilon}^{k,+}=W_{\varepsilon}^k$, $(\bar W_{\varepsilon-t}^{k,+})_{t\in[0,\varepsilon]}$ is a Brownian motion above $(W_{\varepsilon-t}^k)_{t\in[0,\varepsilon]}$ reflected on $(W_{\varepsilon-t}^k)_{t\in[0,\varepsilon]}$, and $(\bar W_{-t}^{k,+})_{t\in[0,\varepsilon]}$ is a Brownian motion absorbed by $(W_{-t}^k)_{t\in[0,\varepsilon]}$, conditioned to coalesce. Then for $t\in[-2\varepsilon,0]$, we set $W^{k+1}_t=\bar W_{t+\varepsilon}^{k,+}-W_\varepsilon^k$. In addition, we set $\breve T_{k+1}=2\int_{-\varepsilon}^{\varepsilon}(\bar W_t^{k,+}-W_t^k)\mathrm{d}t$.
 \end{itemize}
 \end{definition}
 
 \begin{remark}
 $V^{k,-}$ corresponds roughly to the limit of $\frac{1}{\sqrt{n}}\sum \zeta_i^{T_k,-,E}$, and $W^k$ to the limit of $\frac{1}{\sqrt{n}}\sum \zeta_i^{T_k,-,B}$ (see Definition \ref{def_local_times}). $(\breve Z_k)_{k \in \mathds{N}}$ is the ``mesoscopic walk embedded in the limit process of $Y^N$'' (see \eqref{eq_def_YN} and below).
 \end{remark}
 
 The limit process of the environments satisﬁes the following property, whose proof is given in the appendix.
 
 \begin{lemma}\label{lem_no_atoms}
  For any $k\in\mathds{N}^*$, the random variables $\breve T_k$ and $\sum_{k'=1}^k \breve T_{k'}$ have no atoms. 
 \end{lemma}
 
 We want to apply the results of Section \ref{sec_BM_results} to the limit process of the environments. However, to use them, we need the Brownian motion ($\tilde W^-$ in Proposition \ref{prop_better_than_Z} or $W$ in Lemma \ref{lem_conv_abs_times}) to be strictly above the barrier ($f$ in Proposition \ref{prop_better_than_Z} or $g$ in Lemma \ref{lem_conv_abs_times}) at 0. Hence we have to prove such a result for the processes defined in Definition \ref{def_lim_envts}, which is the following. 
 
 \begin{proposition}\label{prop_envts_above_barrier}
  For any $k\in\mathds{N}$, we have $\mathds{P}(V_0^{k,-}>W_0^k)=1$ and $\mathds{P}(V_0^{k,+}>W_0^k)=1$.
 \end{proposition}
 
 The rest of this section is devoted to the proof of Proposition \ref{prop_envts_above_barrier}. The idea is to prove that the law of $W^k$ in some small interval $[-\bar\varepsilon,\bar\varepsilon]$ around 0 is ``close'' to that of a Brownian motion, or of a Brownian motion reflected on a Brownian motion. Indeed, we can prove that a Brownian motion like $V^{k,\pm}$ reflected on such a process is almost surely strictly above it at time 0 (Lemma \ref{lem_reflected_above}). 

 We need to define some notation. For any $\bar\varepsilon > 0$, let $(W_t)_{t\in[-\bar\varepsilon,\bar\varepsilon]}$ a two-sided Brownian motion with $W_0=0$. We denote its law $\mu_{\bar\varepsilon}$. We will also denote $\mu_{-,\bar\varepsilon}$ the law of $(W_t')_{t\in[-\bar\varepsilon,\bar\varepsilon]}$ so that ``at the left of 0, $W'$ is a Brownian motion, and at the right of 0, $W'$ is a Brownian motion reflected on $W$''; more rigorously, $(W_t')_{t\in[-\bar\varepsilon,0]}=(W_t)_{t\in[-\bar\varepsilon,0]}$ and $(W_t')_{t\in[0,\bar\varepsilon]}$ is a Brownian motion reflected on $(W_t)_{t\in[0,\bar\varepsilon]}$ above $(W_t)_{t\in[0,\bar\varepsilon]}$ so that $W_0'=0$. Similarly, we will denote $\mu_{+,\bar\varepsilon}$ the law of $(W_t')_{t\in[-\bar\varepsilon,\bar\varepsilon]}$ so that ``at the right of 0, $W'$ is a Brownian motion, and at the left of 0, $W'$ is a Brownian motion reflected on $W$'', that is $(W_t')_{t\in[0,\bar\varepsilon]}=(W_t)_{t\in[0,\bar\varepsilon]}$ and $(W_{-t}')_{t\in[0,\bar\varepsilon]}$ is a Brownian motion reflected on $(W_{-t})_{t\in[0,\bar\varepsilon]}$ above $(W_{-t})_{t\in[0,\bar\varepsilon]}$ so that $W_0'=0$. Finally, for any $k\in\mathds{N}$, we denote by $\mu_{z,\bar\varepsilon}^k$ ``the law of $W^k$ in a window of size $2\bar\varepsilon$ around $z\varepsilon$'', that is the law of $(W_{z\varepsilon+t}^k-W_{z\varepsilon}^k)_{t\in[-\bar\varepsilon,\bar\varepsilon]}$. 

 Now, for any $\bar\varepsilon>0$, we denote $\mathbf{F}_{\bar\varepsilon}$ the set of real non-negative bounded functions defined on the space of continuous functions $:[-\bar\varepsilon,\bar\varepsilon] \mapsto \mathds{R}$. If $\mu$ is the law of a continuous stochastic process $(W_t)_{t\in[-\bar\varepsilon,\bar\varepsilon]}$ and $f\in\mathbf{F}_{\bar\varepsilon}$, we denote by $\mu(f)$ or $\mu(f((W_t)_{t\in[-\bar\varepsilon,\bar\varepsilon]}))$ the expectation of $f((W_t)_{t\in[-\bar\varepsilon,\bar\varepsilon]})$ under the law $\mu$. For any $f \in \mathbf{F}_{\bar\varepsilon}$, for any process $(W_t)_{t\in[-\bar\varepsilon,\bar\varepsilon]}$, we denote $f((W_t)_{t\in[-\bar\varepsilon,\bar\varepsilon]})=\tilde f((W_t)_{t\in[-\bar\varepsilon,0]},(W_t)_{t\in[0,\bar\varepsilon]})$. We then have the following proposition, which indicates that for any $k\in\mathds{N}$, the law of $W^k$ is ``close'' to an appropriate law.
 
\begin{proposition}\label{prop_rec_little_ints}
 For any $z\in\mathds{Z}$, $\bar\varepsilon>0$ we have $\mu_{z,\bar\varepsilon}^0=\mu_{\bar\varepsilon}$, and for all $k\in\mathds{N}^*$, for all $\delta>0$, there exists $\bar\varepsilon>0$ so that, for any $f\in\mathbf{F}_{\bar\varepsilon}$, for any $z\in\mathds{Z}\setminus\{0\}$ we have $\mu_{z,\bar\varepsilon}^k(f)\leq 2^k\mu_{\bar\varepsilon}(f)+\delta\|f\|_\infty$, and $\mu_{0,\bar\varepsilon}^k(f) \leq 2^{k-1}(\mu_{-,\bar\varepsilon}(f)+\mu_{+,\bar\varepsilon}(f))+\delta\|f\|_\infty$. 
\end{proposition}

The following lemma indicates that if the law of $W^k$ around time 0 is close to an appropriate law, we have the desired property $\mathds{P}(V_0^{k,\pm}>W_0^k)=1$. Lemma \ref{lem_little_ints_enough} together with Proposition \ref{prop_rec_little_ints} prove Proposition \ref{prop_envts_above_barrier}, and Lemma \ref{lem_little_ints_enough} is also used in the proof of Proposition \ref{prop_rec_little_ints}.

\begin{lemma}\label{lem_little_ints_enough}
 If for any $\bar\varepsilon > 0$ we have $\mu_{0,\bar\varepsilon}^0=\mu_{\bar\varepsilon}$, then $\mathds{P}(V_0^{0,-}>W_0^0)=1$ and $\mathds{P}(V_0^{0,+}>W_0^0)=1$. Moreover, for any $k\in\mathds{N}^*$, if for any $\delta>0$ there exists $\bar\varepsilon>0$ so that for any $f\in\mathbf{F}_{\bar\varepsilon}$ we have $\mu_{0,\bar\varepsilon}^k(f) \leq 2^{k-1}(\mu_{-,\bar\varepsilon}(f)+\mu_{+,\bar\varepsilon}(f))+\delta\|f\|_\infty$, then $\mathds{P}(V_0^{k,-}>W_0^k)=1$ and $\mathds{P}(V_0^{k,+}>W_0^k)=1$. 
\end{lemma}

 In order to prove Lemma \ref{lem_little_ints_enough}, we need to show that a Brownian motion reflected on a process with law $\mu_{\bar\varepsilon}$, $\mu_{-,\bar\varepsilon}$ or $\mu_{+,\bar\varepsilon}$ will almost surely be strictly above it at time 0, which is the following lemma.
 
 \begin{lemma}\label{lem_reflected_above}
  For any $\bar\varepsilon>0$, we denote by $(W_t)_{t\in[-\bar\varepsilon,\bar\varepsilon]}$ a process with law $\mu_{\bar\varepsilon}$, $\mu_{-,\bar\varepsilon}$ or $\mu_{+,\bar\varepsilon}$, and by $(W_t')_{t\in[-\bar\varepsilon,\bar\varepsilon]}$ a Brownian motion reflected on $(W_t)_{t\in[-\bar\varepsilon,\bar\varepsilon]}$ such that $W_{-\bar\varepsilon}' \geq W_{-\bar\varepsilon}$. Then for any $\delta>0$, there exists $0 < \bar\varepsilon' \leq \bar\varepsilon$ so that $\mathds{P}(\forall\,t\in[-\bar\varepsilon',\bar\varepsilon'], W_t'>W_t)\geq 1-\delta$. 
 \end{lemma}

 \begin{proof}
 We begin by introducing some notation. We denote by $(W''_t)_{t\in[-\bar\varepsilon,\bar\varepsilon]}$ the Brownian motion so that $(W'_t)_{t\in[-\bar\varepsilon,\bar\varepsilon]}$ is the reflection of $(W''_t)_{t\in[-\bar\varepsilon,\bar\varepsilon]}$ on $(W_t)_{t\in[-\bar\varepsilon,\bar\varepsilon]}$. We notice that if $(\tilde W_t)_{t\in[0,1]}$ is a Brownian motion with $\tilde W_0=0$, there exists some finite $M>0$ so that $\mathds{P}(\max_{0 \leq t \leq 1}|\tilde W_t| \leq M/3)>0$. We denote $i_0=\lceil\frac{-\ln(\bar\varepsilon)}{2 \ln 2}\rceil$ (then $2^{-2i_0} \leq \bar\varepsilon$). It will be enough to prove that 
 \begin{equation}\label{eq_reflected_BM}
 \mathds{P}(\exists\, i \geq i_0\text{ so that }\forall\, t \in [-2^{-2i},0], |W_t| \leq 2^{-i}M
 \text{ and }W''_0-W''_{-2^{-2i}} \geq (2M+1)2^{-i})=1.
 \end{equation}
 Indeed, then there almost surely exists $i \geq i_0$ so that $\forall\, t \in [-2^{-2i},0], |W_t| \leq 2^{-i}M$ and $W''_0-W''_{-2^{-2i}} \geq (2M+1)2^{-i}$. Then $(W_t')_{t\in[-2^{-2i},0]}$ is above the Brownian motion $(W_t''-W_{-2^{-2i}}''+W_{-2^{-2i}})_{t\in[-2^{-2i},0]}$ reflected on $(W_t)_{t\in[-2^{-2i},0]}$, itself above the Brownian motion $(W_t''-W_{-2^{-2i}}''+W_{-2^{-2i}})_{t\in[-2^{-2i},0]}$. Therefore $W'_0 \geq W_0''-W_{-2^{-2i}}''+W_{-2^{-2i}} \geq (2M+1)2^{-i}-M2^{-i}=(M+1) 2^{-i}\geq W_0+2^{-i}> W_0$. We deduce $\mathds{P}(W_0' > W_0)=1$. Now let $\delta>0$. Since $\mathds{P}(W_0'=W_0)=0$, there exists $\delta_1 > 0$ so that $\mathds{P}(W_0'-W_0 < \delta_1) \leq \delta/2$. Furthermore, the processes $(W_t)_{t\in[-\bar\varepsilon,\bar\varepsilon]}$ and $(W_t')_{t\in[-\bar\varepsilon,\bar\varepsilon]}$ are continuous, hence there exists $0<\bar\varepsilon'<\bar\varepsilon$ so that $\mathds{P}(\forall \, t \in [-\bar\varepsilon',\bar\varepsilon'], |(W_t'-W_t)-(W_0'-W_0)|\leq \delta_1/2) \geq 1-\delta/2$. We then have $\mathds{P}(\forall \, t \in [-\bar\varepsilon',\bar\varepsilon'], W_t'>W_t) \geq 1-\delta$, which is Lemma \ref{lem_reflected_above}. 
 
 Consequently, we only have to prove \eqref{eq_reflected_BM}. We will prove 
 \[
 \mathds{P}(|\{i \in \mathds{N}\,|\, i \geq i_0, \forall\, t \in [-2^{-2i},0], |W_t| \leq 2^{-i}M,W''_0-W''_{-2^{-2i}} \geq (2M+1)2^{-i}\}|=+\infty)=1.
 \]
 By Blumenthal 0-1 law, this event has probability 0 or 1, so it is enough to prove that it has positive probability. Now, 
 \[
 \mathds{P}(|\{i \in \mathds{N} \,|\, i \geq i_0, \forall\, t \in [-2^{-2i},0], |W_t| \leq 2^{-i}M,W''_0-W''_{-2^{-2i}} \geq (2M+1)2^{-i}\}|=+\infty)
 \]
 \[
 =\mathds{P}\left(\bigcap_{i \geq i_0}\bigcup_{j \geq i}\{\forall\, t \in [-2^{-2j},0], |W_t| \leq 2^{-j}M,W''_0-W''_{-2^{-2j}} \geq (2M+1)2^{-j}\}\right)
 \]
 \[
 =\lim_{i \to +\infty}\mathds{P}\left(\bigcup_{j \geq i}\{\forall\, t \in [-2^{-2j},0], |W_t| \leq 2^{-j}M,W''_0-W''_{-2^{-2j}} \geq (2M+1)2^{-j}\}\right)
 \]
 \[
 \geq \liminf_{i \to +\infty}\mathds{P}(\forall\, t \in [-2^{-2i},0], |W_t| \leq 2^{-i}M,W''_0-W''_{-2^{-2i}} \geq (2M+1)2^{-i}).
 \]
 Consequently, it is enough to find a positive lower bound for the latter term. In addition, $W$ and $W''$ are independent, hence 
 \[
 \mathds{P}(\forall\, t \in [-2^{-2i},0], |W_t| \leq 2^{-i}M,W''_0-W''_{-2^{-2i}} \geq (2M+1)2^{-i}) 
 \]
 \[
 = \mathds{P}(\forall\, t \in [-2^{-2i},0], |W_t| \leq 2^{-i}M)\mathds{P}(W''_0-W''_{-2^{-2i}} \geq (2M+1)2^{-i}).
 \]
 Moreover, by scaling invariance of the Brownian motion, $\mathds{P}(W''_0-W''_{-2^{-2i}} \geq (2M+1)2^{-i})=\mathds{P}(W''_0-W''_{-1} \geq 2M+1)$, which is positive and independent on $i$. Therefore we only have to find a positive lower bound for the $\mathds{P}(\forall\, t \in [-2^{-2i},0], |W_t| \leq 2^{-i}M)$. If $(W_t)_{t\in[-\bar\varepsilon,\bar\varepsilon]}$ has law $\mu_{\bar\varepsilon}$ or $\mu_{-,\bar\varepsilon}$, $(W_{-t})_{t\in[0,\bar\varepsilon]}$ is a Brownian motion, so by scaling invariance, $\mathds{P}(\forall\, t \in [-2^{-2i},0], |W_t| \leq 2^{-i}M)=\mathds{P}(\max_{0 \leq t \leq 1}|\tilde W_t| \leq M)>0$, which is enough. If $(W_t)_{t\in[-\bar\varepsilon,\bar\varepsilon]}$ has law $\mu_{+,\bar\varepsilon}$, we may say $(W_{-t})_{t\in[0,\bar\varepsilon]}$ is a Brownian motion $(W_{-t}^1)_{t\in[0,\bar\varepsilon]}$ with $W_0^1=0$ reflected on an independent Brownian motion $(W_{-t}^2)_{t\in[0,\bar\varepsilon]}$ with $W_0^2=0$. As before, $\mathds{P}(\forall\, t \in [-2^{-2i},0], |W_t^1| \leq 2^{-i}M/3)=\mathds{P}(\forall\, t \in [-2^{-2i},0], |W_t^2| \leq 2^{-i}M/3)=\mathds{P}(\max_{0 \leq t \leq 1}|\tilde W_t| \leq M/3)>0$, thus $\mathds{P}(\forall\, t \in [-2^{-2i},0], |W_t^1|,|W_t^2| \leq 2^{-i}M/3)$ is constant and positive. Now, if for all $t \in [-2^{-2i},0]$ we have $|W_t^1|,|W_t^2| \leq 2^{-i}M/3$, then for all $t \in [-2^{-2i},0]$ we have $W_t=W_t^1+\sup_{t \leq s \leq 0}(W_s^2-W_s^1)$, hence $|W_t| \leq 2^{-i}M$. This implies that $\mathds{P}(\forall\, t \in [-2^{-2i},0], |W_t| \leq 2^{-i}M)$ is bounded from below by a positive constant, which ends the proof of Lemma \ref{lem_reflected_above}.
 \end{proof}

 We are now in position to prove Proposition \ref{prop_rec_little_ints} and Lemma \ref{lem_little_ints_enough}. 

\begin{proof}[Proof of Lemma \ref{lem_little_ints_enough}.]
We only spell out the proof for $k\in\mathds{N}^*$ and $\mathds{P}(V_0^{k,+}>W_0^k)=1$, as the other cases can be dealt with in the same way. We are going to prove that for any $\delta>0$ we have $\mathds{P}(V_0^{k,+}=W_0^k) \leq \delta$, which is enough. Let $\delta>0$. We recall that $(V_{-t}^{k,+})_{t\in[-\varepsilon,\varepsilon]}$ is a Brownian motion reflected and absorbed by $(W_{-t}^k)_{t\in[-\varepsilon,\varepsilon]}$ (see Definition \ref{def_lim_envts}). We may consider that it was constructed as the reflection and absorption of the Brownian motion $(\dot V_{-t}^{k,+})_{t\in[-\varepsilon,\varepsilon]}$. Let $\bar\varepsilon\in(0,\varepsilon)$, and let us denote by $(V_t^{k,+,\bar\varepsilon})_{t\in[-\bar\varepsilon,\bar\varepsilon]}$ the process defined so that $(V_{-t}^{k,+,\bar\varepsilon})_{t\in[-\bar\varepsilon,\bar\varepsilon]}$ is the Brownian motion $(\dot V^{k,+}_{-t}-\dot V^{k,+}_{\bar\varepsilon}+W_{\bar\varepsilon}^k)_{t\in[-\bar\varepsilon,\bar\varepsilon]}$ reflected on $(W_{-t}^k)_{t\in[-\bar\varepsilon,\bar\varepsilon]}$ and above it. It is ``the same Brownian motion as $(V_t^{k,+})_{t\in[-\bar\varepsilon,\bar\varepsilon]}$, but starting from a lower point (and without absorption)'', so if $V_0^{k,+,\bar\varepsilon}>W_0^k$ then $V_0^{k,+}>W_0^k$. We deduce $\mathds{P}(V_0^{k,+}=W_0^k) \leq \mathds{P}(V_0^{k,+,\bar\varepsilon}=W_0^k)$. We now introduce some temporary notation: for any measure $\mu$ defined on the space of continuous processes on $[-\bar\varepsilon,\bar\varepsilon]$, $(W_t)_{t\in[-\bar\varepsilon,\bar\varepsilon]}$ will be a process of law $\mu$, and $(W_t')_{t\in[-\bar\varepsilon,\bar\varepsilon]}$ will be defined so that $W_{\bar\varepsilon}'=W_{\bar\varepsilon}$ and $(W_{-t}')_{t\in[-\bar\varepsilon,\bar\varepsilon]}$ is a Brownian motion reflected on $(W_{-t})_{t\in[-\bar\varepsilon,\bar\varepsilon]}$ above it. We then have 
\[
\mathds{P}(V_0^{k,+}=W_0^k) \leq \mathds{E}(\mathds{P}(V_0^{k,+,\bar\varepsilon}=W_0^k|(W_{t}^k)_{t\in[-\bar\varepsilon,\bar\varepsilon]}))=\mu_{0,\bar\varepsilon}^k(\mathds{P}(W_0'=W_0|W)).
\]
We now choose $\bar\varepsilon$ so that for any $f\in\mathbf{F}_{\bar\varepsilon}$ we have $\mu_{0,\bar\varepsilon}^k(f) \leq 2^{k-1}(\mu_{-,\bar\varepsilon}(f)+\mu_{+,\bar\varepsilon}(f))+(\delta/2)\|f\|_\infty$ (we can choose $\bar\varepsilon < \varepsilon$ since it is easy to see that if the property holds for $\bar\varepsilon$ it also holds for all smaller $\bar\varepsilon$). We then have 
\[
\mathds{P}(V_0^{k,+}=W_0^k) \leq 2^{k-1}(\mu_{-,\bar\varepsilon}(\mathds{P}(W_0'=W_0|W))+\mu_{+,\bar\varepsilon}(\mathds{P}(W_0'=W_0|W)))+\delta/2.
\]
This implies that for any $\bar\varepsilon'\in(0,\bar\varepsilon)$, we have that $\mathds{P}(V_0^{k,+}=W_0^k)$ is smaller than 
\[
2^{k-1}(\mu_{-,\bar\varepsilon}(\mathds{P}(\exists\, t\in[-\bar\varepsilon',\bar\varepsilon'],W_t'=W_t|W))+\mu_{+,\bar\varepsilon}(\mathds{P}(\exists\, t\in[-\bar\varepsilon',\bar\varepsilon'],W_t'=W_t|W)))+\delta/2.
\]
Now, by Lemma \ref{lem_reflected_above}, noticing that if $(W_t)_{t\in[-\bar\varepsilon,\bar\varepsilon]}$ has law $\mu_{\pm,\bar\varepsilon}$ then $(W_{-t})_{t\in[-\bar\varepsilon,\bar\varepsilon]}$ has law $\mu_{\mp,\bar\varepsilon}$, there exists $0<\bar\varepsilon'\leq \bar\varepsilon$ so that $\mu_{-,\bar\varepsilon}(\mathds{P}(\exists\, t \in [-\bar\varepsilon',\bar\varepsilon'], W_t' \leq W_t|W)) \leq \delta/2^{k+1}$ and $\mu_{+,\bar\varepsilon}(\mathds{P}(\exists\, t \in [-\bar\varepsilon',\bar\varepsilon'], W_t' \leq W_t|W)) \leq \delta/2^{k+1}$. This implies $\mathds{P}(V_0^{k,+}=W_0^k) \leq 2^{k-1}(\delta/2^{k+1}+\delta/2^{k+1})+\delta/2=\delta$, which ends the proof. 
\end{proof}

\begin{proof}[Proof of Proposition \ref{prop_rec_little_ints}.]
In order to shorten the notation in this proof, for any $k\in\mathds{N}$, any $z\in\mathds{Z}$ and any real numbers $a<a'$, we will denote the process $(W_{z\varepsilon+t}^{k}-W_{z\varepsilon}^{k})_{t\in[a,a']}$ by $W_{[a,a']}^{k,z}$. We will prove Proposition \ref{prop_rec_little_ints} by induction on $k$. Here is a rough sketch of the proof. The idea is that if the statement of the proposition is true for $k$ and if, say, $\breve Z_{k+1}=\breve Z_k+1$, then for any $z\not\in \{0,1,2\}$, $W_{[-\bar\varepsilon,\bar\varepsilon]}^{k+1,z}$ is $W_{[-\bar\varepsilon,\bar\varepsilon]}^{k,z+1}$ which we control by the induction hypothesis. Moreover, for $z=-2$, we notice that $\bar W^{k,+}$ is conditioned to coalesce with $W^k$ before time $-\varepsilon$, so if we choose $\bar\varepsilon$ small enough, with high probability $\bar W^{k,+}$ coalesces with $W^k$ before time $-\varepsilon+\bar\varepsilon$, thus $W_{[-\bar\varepsilon,\bar\varepsilon]}^{k+1,-2}=W_{[-\bar\varepsilon,\bar\varepsilon]}^{k,-1}$ which we control by the induction hypothesis. Furthermore, for $z=-1$, $W_{[-\bar\varepsilon,\bar\varepsilon]}^{k+1,-1}$ is $(\bar W_{t}^{k,+}-\bar W_{0}^{k,+})_{t\in[-\bar\varepsilon,\bar\varepsilon]}$. Now, by the induction hypothesis, $W_{[-\bar\varepsilon,\bar\varepsilon]}^{k,0}$ has a ``good'' law, hence Lemma \ref{lem_little_ints_enough} implies that $\bar W^{k,+}$ is strictly above $W^k$ at 0 thus around 0, hence $\bar W^{k,+}$ behaves like an unconstrained Brownian motion around 0, so  $W_{[-\bar\varepsilon,\bar\varepsilon]}^{k+1,-1}$ has the right law. Finally, for $z=0$, we notice that $W_{[-\bar\varepsilon,\bar\varepsilon]}^{k+1,0}$ is $W_{[0,\bar\varepsilon]}^{k,1}$ at the right of 0 and a Brownian motion reflected on $W_{[-\bar\varepsilon,0]}^{k,1}$ at the right of 0, and by the induction hypothesis $W_{[0,\bar\varepsilon]}^{k,1}$ has a law close to that of a Brownian motion, so the law of $W_{[-\bar\varepsilon,\bar\varepsilon]}^{k+1,0}$ is close to $\mu_{+,\bar\varepsilon}$. 

 We now begin the induction. For $k=0$, Definition \ref{def_lim_envts} yields that $W^0$ is a two-sided Brownian motion, which implies that for any $z\in\mathds{Z}$, $\bar\varepsilon>0$ we have $\mu_{z,\bar\varepsilon}^0=\mu_{\bar\varepsilon}$. Now let $k\in\mathds{N}$ and suppose the statement of Proposition \ref{prop_rec_little_ints} for $k$ holds. Let $\delta>0$ and $z\in\mathds{Z}$. We first notice that by the induction hypothesis and Lemma \ref{lem_little_ints_enough} we have $\mathds{P}(V_0^{k,-}>W_0^k)=1$, so $\mathds{P}(V_0^{k,-}>W_0^k|W^k)=1$ almost surely, hence by Proposition \ref{prop_better_than_Z} we have $p_{k,-}+p_{k,+}=1$ almost surely. As explained above, we will use different arguments depending on the value of $z$.

\emph{Case $z \not\in\{-2,-1,0,1,2\}$.} 

By Definition \ref{def_lim_envts}, given $W^k$, with probability $p_{k,-}$ we have $(W_{z\varepsilon+t}^{k+1}-W_{z\varepsilon}^{k+1})_{t\in[-\varepsilon,\varepsilon]}=(W_{(z-1)\varepsilon+t}^k-W_{(z-1)\varepsilon}^k)_{t\in[-\varepsilon,\varepsilon]}$ and with probability $1-p_{k,-}$ we have $(W_{z\varepsilon+t}^{k+1}-W_{z\varepsilon}^{k+1})_{t\in[-\varepsilon,\varepsilon]}=(W_{(z+1)\varepsilon+t}^k-W_{(z+1)\varepsilon}^k)_{t\in[-\varepsilon,\varepsilon]}$. Hence for any $0<\bar\varepsilon<\varepsilon$, $f\in\mathbf{F}_{\bar\varepsilon}$, we have $\mathds{E}(f(W_{[-\bar\varepsilon,\bar\varepsilon]}^{k+1,z})|W^k)=p_{k,-}f(W_{[-\bar\varepsilon,\bar\varepsilon]}^{k,z-1})+(1-p_{k,-})f(W_{[-\bar\varepsilon,\bar\varepsilon]}^{k,z+1})$, so $\mathds{E}(f(W_{[-\bar\varepsilon,\bar\varepsilon]}^{k+1,z})) \leq \mathds{E}(f(W_{[-\bar\varepsilon,\bar\varepsilon]}^{k,z-1}))+\mathds{E}(f(W_{[-\bar\varepsilon,\bar\varepsilon]}^{k,z+1}))$, that is $\mu_{z,\bar\varepsilon}^{k+1}(f)\leq\mu_{z-1,\bar\varepsilon}^k(f)+\mu_{z+1,\bar\varepsilon}^k(f)$. Now, we notice $z-1,z+1 \neq 0$, so by the induction hypothesis, there exists some $\bar\varepsilon_1 > 0$ (which does not depend on $z$) so that for any $g\in\mathbf{F}_{\bar\varepsilon_1}$ we have $\mu_{z-1,\bar\varepsilon_1}^k(g)\leq 2^k\mu_{\bar\varepsilon_1}(g)+(\delta/2)\|g\|_\infty$ and $\mu_{z+1,\bar\varepsilon_1}^k(g)\leq 2^k\mu_{\bar\varepsilon_1}(g)+(\delta/2)\|g\|_\infty$. We deduce that if $\bar\varepsilon \leq \bar\varepsilon_1$, we have $\mu_{z,\bar\varepsilon}^{k+1}(f)\leq 2^{k+1}\mu_{\bar\varepsilon}(f)+\delta\|f\|_\infty$. 

\emph{Case $z=\pm 2$.} 

We only treat the case $z=-2$, as the case $z=2$ is similar. Given $W^k$, with probability $p_{k,-}$ we have $W_{[-\varepsilon,\varepsilon]}^{k+1,-2}=W_{[-\varepsilon,\varepsilon]}^{k,-3}$ and with probability $1-p_{k,-}$ we have $W_{[-\varepsilon,0]}^{k+1,-2}=W_{[-\varepsilon,0]}^{k,-1}$ and $W_{[0,\varepsilon]}^{k+1,-2}=(\bar W_{t-\varepsilon}^{k,+}-W_{-\varepsilon}^k)_{t\in[0,\varepsilon]}$. Consequently, if $0<\bar\varepsilon<\varepsilon$ and $f\in\mathbf{F}_{\bar\varepsilon}$, we have 
\begin{equation}\label{eq_rec_little_ints_2}
\mathds{E}(f(W_{[-\bar\varepsilon,\bar\varepsilon]}^{k+1,-2})|W^k) = p_{k,-}f(W_{[-\bar\varepsilon,\bar\varepsilon]}^{k,-3})+(1-p_{k,-})\mathds{E}(\tilde f(W_{[-\bar\varepsilon,0]}^{k,-1},(\bar W_{t-\varepsilon}^{k,+}-W_{-\varepsilon}^k)_{t\in[0,\bar\varepsilon]})|W^k). 
\end{equation}
Now, given $W^k$, by definition $\bar W^{k,+}$ has the law of $V^{k,+}$ conditioned to coalesce with $W^k$ before time $-\varepsilon$, an event denoted by $\{\sigma_{k,+}>-\varepsilon\}$ and satisfying $\mathds{P}(\sigma_{k,+}>-\varepsilon|W^k)=p_{k,+}=1-p_{k,-}$. This implies 
\[
\mathds{E}(\tilde f(W_{[-\bar\varepsilon,0]}^{k,-1},(\bar W_{t-\varepsilon}^{k,+}-W_{-\varepsilon}^k)_{t\in[0,\bar\varepsilon]})|W^k)
=\frac{1}{1-p_{k,-}}\mathds{E}(\tilde f(W_{[-\bar\varepsilon,0]}^{k,-1},(V_{t-\varepsilon}^{k,+}-W_{-\varepsilon}^k)_{t\in[0,\bar\varepsilon]})\mathds{1}_{\{\sigma_{k,+}>-\varepsilon\}}|W^k),
\]
therefore (\ref{eq_rec_little_ints_2}) implies 
\[
\mathds{E}(f(W_{[-\bar\varepsilon,\bar\varepsilon]}^{k+1,-2})|W^k) 
 \leq f(W_{[-\bar\varepsilon,\bar\varepsilon]}^{k,-3}) + \mathds{E}(\tilde f(W_{[-\bar\varepsilon,0]}^{k,-1},(V_{t-\varepsilon}^{k,+}-W_{-\varepsilon}^k)_{t\in[0,\bar\varepsilon]})\mathds{1}_{\{\sigma_{k,+}>-\varepsilon\}}|W^k), 
 \]
 hence $\mathds{E}(f(W_{[-\bar\varepsilon,\bar\varepsilon]}^{k+1,-2})) \leq \mathds{E}(f(W_{[-\bar\varepsilon,\bar\varepsilon]}^{k,-3})) + \mathds{E}(\tilde f(W_{[-\bar\varepsilon,0]}^{k,-1},(V_{t-\varepsilon}^{k,+}-W_{-\varepsilon}^k)_{t\in[0,\bar\varepsilon]})\mathds{1}_{\{\sigma_{k,+}>-\varepsilon\}})$. We now choose $\bar\varepsilon_2'>0$ so that $\mathds{P}(\sigma_{k,+}\in(-\varepsilon,-\varepsilon+\bar\varepsilon_2']) \leq \delta/3$, and assume $\bar\varepsilon \leq \bar\varepsilon_2'$. We then have 
 \[
  \mathds{E}(f(W_{[-\bar\varepsilon,\bar\varepsilon]}^{k+1,-2})) \leq \mathds{E}(f(W_{[-\bar\varepsilon,\bar\varepsilon]}^{k,-3})) + \mathds{E}(\tilde f(W_{[-\bar\varepsilon,0]}^{k,-1},(V_{t-\varepsilon}^{k,+}-W_{-\varepsilon}^k)_{t\in[0,\bar\varepsilon]})\mathds{1}_{\{\sigma_{k,+}>-\varepsilon+\bar\varepsilon_2'\}})+(\delta/3)\|f\|_\infty
 \]
\[
 = \mathds{E}(f(W_{[-\bar\varepsilon,\bar\varepsilon]}^{k,-3})) + \mathds{E}(\tilde f(W_{[-\bar\varepsilon,0]}^{k,-1},(W_{t-\varepsilon}^{k}-W_{-\varepsilon}^k)_{t\in[0,\bar\varepsilon]})\mathds{1}_{\{\sigma_{k,+}>-\varepsilon+\bar\varepsilon_2'\}})+(\delta/3)\|f\|_\infty
\]
since for $t \leq \sigma_{k,+}$ we have $V_t^{k,+}=W_t^k$. We deduce $\mathds{E}(f(W_{[-\bar\varepsilon,\bar\varepsilon]}^{k+1,-2})) \leq \mathds{E}(f(W_{[-\bar\varepsilon,\bar\varepsilon]}^{k,-3})) + \mathds{E}(f(W_{[-\bar\varepsilon,\bar\varepsilon]}^{k,-1}))+(\delta/3)\|f\|_\infty$. Now, by the induction hypothesis, there exists $\bar\varepsilon_2''>0$ so that for any $g\in\mathbf{F}_{\bar\varepsilon_2''}$ we have $\mu_{-3,\bar\varepsilon_2''}^k(g)\leq 2^k\mu_{\bar\varepsilon_2''}(g)+(\delta/3)\|g\|_\infty$ and $\mu_{-1,\bar\varepsilon_2''}^k(g)\leq 2^k\mu_{\bar\varepsilon_2''}(g)+(\delta/3)\|g\|_\infty$. Thus, setting $\bar\varepsilon_2=\min(\bar\varepsilon_2',\bar\varepsilon_2'')>0$, if we have $\bar\varepsilon \leq \bar\varepsilon_2$, then $\mathds{E}(f(W_{[-\bar\varepsilon,\bar\varepsilon]}^{k+1,-2})) \leq 2^k\mu_{\bar\varepsilon}(f)+(\delta/3)\|f\|_\infty + 2^k\mu_{\bar\varepsilon}(f)+(\delta/3)\|f\|_\infty+(\delta/3)\|f\|_\infty$, that is $\mu_{-2,\bar\varepsilon}^{k+1}(f) \leq 2^{k+1}\mu_{\bar\varepsilon}(f)+\delta\|f\|_\infty$. 

\emph{Case $z=\pm 1$.}

We only treat the case $z=-1$, as the case $z=1$ is similar. Given $W^k$, with probability $p_{k,-}$ we have $W_{[-\varepsilon,\varepsilon]}^{k+1,-1}=W_{[-\varepsilon,\varepsilon]}^{k,-2}$ and with probability $1-p_{k,-}$ we have $W_{[-\varepsilon,\varepsilon]}^{k+1,-1}=(\bar W_{t}^{k,+}-\bar W_{0}^{k,+})_{t\in[-\varepsilon,\varepsilon]}$. Therefore, if $0<\bar\varepsilon<\varepsilon$ and $f\in\mathbf{F}_{\bar\varepsilon}$, we have 
\begin{equation}\label{eq_rec_little_ints_1}
 \mathds{E}(f(W_{[-\bar\varepsilon,\bar\varepsilon]}^{k+1,-1})|W^k)=p_{k,-}f(W_{[-\bar\varepsilon,\bar\varepsilon]}^{k,-2})+(1-p_{k,-})\mathds{E}(f((\bar W_{t}^{k,+}-\bar W_{0}^{k,+})_{t\in[-\bar\varepsilon,\bar\varepsilon]})|W^k). 
\end{equation}
Now, given $W^k$, by definition $\bar W^{k,+}$ has the law of $V^{k,+}$ conditioned to coalesce with $W^k$ before time $-\varepsilon$, an event denoted by $\{\sigma_{k,+}>-\varepsilon\}$ and satisfying $\mathds{P}(\sigma_{k,+}>-\varepsilon|W^k)=p_{k,+}=1-p_{k,-}$. This yields 
\[
\mathds{E}(f((\bar W_{t}^{k,+}-\bar W_{0}^{k,+})_{t\in[-\bar\varepsilon,\bar\varepsilon]})|W^k)=\frac{1}{1-p_{k,-}}\mathds{E}(f((V_{t}^{k,+}-V_{0}^{k,+})_{t\in[-\bar\varepsilon,\bar\varepsilon]})\mathds{1}_{\{\sigma_{k,+}>-\varepsilon\}}|W^k) 
\]
\[
\leq \frac{1}{1-p_{k,-}}\mathds{E}(f((V_{t}^{k,+}-V_{0}^{k,+})_{t\in[-\bar\varepsilon,\bar\varepsilon]})|W^k).
\]
Hence (\ref{eq_rec_little_ints_1}) implies $\mathds{E}(f(W_{[-\bar\varepsilon,\bar\varepsilon]}^{k+1,-1})|W^k)\leq f(W_{[-\bar\varepsilon,\bar\varepsilon]}^{k,-2})+\mathds{E}(f((V_{t}^{k,+}-V_{0}^{k,+})_{t\in[-\bar\varepsilon,\bar\varepsilon]})|W^k)$, thus we have $\mathds{E}(f(W_{[-\bar\varepsilon,\bar\varepsilon]}^{k+1,-1}))\leq \mathds{E}(f(W_{[-\bar\varepsilon,\bar\varepsilon]}^{k,-2}))+\mathds{E}(f((V_{t}^{k,+}-V_{0}^{k,+})_{t\in[-\bar\varepsilon,\bar\varepsilon]}))$. Now, we recall $(V_{-t}^{k,+})_{t\in[-\varepsilon,\varepsilon]}$ is a Brownian motion reflected and absorbed by $(W_{-t}^k)_{t\in[-\varepsilon,\varepsilon]}$; let us say it was constructed as the reflection and absorption of the Brownian motion $(\dot V^{k,+}_{-t})_{t\in[-\varepsilon,\varepsilon]}$. For any $0<\bar\varepsilon'<\varepsilon$, we denote $\mathcal{S}_{\bar\varepsilon'}=\{\forall\, t \in [-\bar\varepsilon',\bar\varepsilon'], V_t^{k,+} > W_t^k\}$. If $\bar\varepsilon \leq \bar\varepsilon'$, we then have $\mathds{E}(f(W_{[-\bar\varepsilon,\bar\varepsilon]}^{k+1,-1}))\leq \mathds{E}(f(W_{[-\bar\varepsilon,\bar\varepsilon]}^{k,-2}))+\mathds{E}(f((\dot V_{t}^{k,+}-\dot V_{0}^{k,+})_{t\in[-\bar\varepsilon,\bar\varepsilon]})\mathds{1}_{\mathcal{S}_{\bar\varepsilon'}})+\|f\|_\infty\mathds{P}((\mathcal{S}_{\bar\varepsilon'})^c)$, hence 
\begin{equation}\label{eq_rec_little_ints_1bis}
 \mathds{E}(f(W_{[-\bar\varepsilon,\bar\varepsilon]}^{k+1,-1}))\leq \mathds{E}(f(W_{[-\bar\varepsilon,\bar\varepsilon]}^{k,-2}))+\mathds{E}(f((\dot V_{t}^{k,+}-\dot V_{0}^{k,+})_{t\in[-\bar\varepsilon,\bar\varepsilon]}))+\|f\|_\infty\mathds{P}((\mathcal{S}_{\bar\varepsilon'})^c). 
\end{equation}

We now need to deal with $\mathds{P}((\mathcal{S}_{\bar\varepsilon'})^c)$. Let $\bar\varepsilon''\in(\bar\varepsilon',\varepsilon)$, and let us denote $(V_t^{k,+,\bar\varepsilon''})_{t\in[-\bar\varepsilon'',\bar\varepsilon'']}$ the process defined so that $(V_{-t}^{k,+,\bar\varepsilon''})_{t\in[-\bar\varepsilon'',\bar\varepsilon'']}$ is the Brownian motion $(\dot V^{k,+}_{-t}-\dot V^{k,+}_{\bar\varepsilon''}+W_{\bar\varepsilon''}^k)_{t\in[-\bar\varepsilon'',\bar\varepsilon'']}$ reflected on $(W_{-t}^k)_{t\in[-\bar\varepsilon'',\bar\varepsilon'']}$ and above it. It is ``the same Brownian motion as $(V_t^{k,+})_{t\in[-\bar\varepsilon'',\bar\varepsilon'']}$, but starting from a lower point (and without absorption)'', so if $\{\forall\, t \in [-\bar\varepsilon',\bar\varepsilon'], V_t^{k,+,\bar\varepsilon''} > W_t^k\}$ occurs, then $\{\forall\, t \in [-\bar\varepsilon',\bar\varepsilon'], V_t^{k,+} > W_t^k\}$ occurs. This implies $\mathds{P}((\mathcal{S}_{\bar\varepsilon'})^c) \leq \mathds{P}(\exists\, t \in [-\bar\varepsilon',\bar\varepsilon'], V_t^{k,+,\bar\varepsilon''} \leq W_t^k)$. We now introduce a temporary notation. For any measure $\mu$ on continuous processes defined on $[-\bar\varepsilon'',\bar\varepsilon'']$, $(W_t)_{t\in[-\bar\varepsilon'',\bar\varepsilon'']}$ will be a process with law $\mu$, and $(W'_{-t})_{t\in[-\bar\varepsilon'',\bar\varepsilon'']}$ will be a Brownian motion reflected on $(W_{-t})_{t\in[-\bar\varepsilon'',\bar\varepsilon'']}$ and above it with $W_{\bar\varepsilon''}'=W_{\bar\varepsilon''}$. We then have 
\[
\mathds{P}((\mathcal{S}_{\bar\varepsilon'})^c) \leq \mathds{E}(\mathds{P}(\exists\, t \in [-\bar\varepsilon',\bar\varepsilon'], V_t^{k,+,\bar\varepsilon''} \leq W_t^k|W_{[-\bar\varepsilon'',\bar\varepsilon'']}^{k,0}))
=\mu_{0,\bar\varepsilon''}^k(\mathds{P}(\exists\, t \in [-\bar\varepsilon',\bar\varepsilon'], W_t' \leq W_t|(W_t)_{t\in[-\bar\varepsilon'',\bar\varepsilon'']})).
\]
Now, by the induction hypothesis, there exists $\bar\varepsilon_3'\in(0,\varepsilon)$ so that for any $g\in\mathbf{F}_{\bar\varepsilon_3'}$, for any $z\in\mathds{Z}\setminus\{0\}$ we have $\mu_{z,\bar\varepsilon_3'}^k(g)\leq 2^k\mu_{\bar\varepsilon_3'}(g)+(\delta/3)\|g\|_\infty$ and $\mu_{0,\bar\varepsilon_3'}^k(g)\leq 2^{k-1}(\mu_{-,\bar\varepsilon_3'}(g)+\mu_{+,\bar\varepsilon_3'}(g))+(\delta/3)\|g\|_\infty$ (if $k=0$, we instead have $\mu_{0,\bar\varepsilon_3'}^k(g)=\mu_{\bar\varepsilon_3'}(g)$, but the argument will work in the same way). We then choose $\bar\varepsilon''= \bar\varepsilon_3'$ and assume $\bar\varepsilon' \leq \bar\varepsilon_3'$. Then we have 
\begin{align*}
\mathds{P}((\mathcal{S}_{\bar\varepsilon'})^c) \leq 2^{k-1}(&\mu_{-,\bar\varepsilon_3'}(\mathds{P}(\exists\, t \in [-\bar\varepsilon',\bar\varepsilon'], W_t' \leq W_t|(W_t)_{t\in[-\bar\varepsilon_3',\bar\varepsilon_3']})) \\ 
 & +\mu_{+,\bar\varepsilon_3'}(\mathds{P}(\exists\, t \in [-\bar\varepsilon',\bar\varepsilon'], W_t' \leq W_t|(W_t)_{t\in[-\bar\varepsilon_3',\bar\varepsilon_3']})))+\delta/3.
\end{align*}
Now, by Lemma \ref{lem_reflected_above}, noticing that if $(W_t)_{t\in[-\bar\varepsilon_3',\bar\varepsilon_3']}$ has law $\mu_{\pm,\bar\varepsilon_3'}$ then $(W_{-t})_{t\in[-\bar\varepsilon_3',\bar\varepsilon_3']}$ has law $\mu_{\mp,\bar\varepsilon_3'}$, there exists $0<\bar\varepsilon_3\leq \bar\varepsilon_3'$ so that 
\begin{align*}
& \mu_{-,\bar\varepsilon_3'}(\mathds{P}(\exists\, t \in [-\bar\varepsilon_3,\bar\varepsilon_3], W_t' \leq W_t|(W_t)_{t\in[-\bar\varepsilon_3',\bar\varepsilon_3']})) \leq \delta/(3 \cdot 2^k), \\
& \mu_{+,\bar\varepsilon_3'}(\mathds{P}(\exists\, t \in [-\bar\varepsilon_3,\bar\varepsilon_3], W_t' \leq W_t|(W_t)_{t\in[-\bar\varepsilon_3',\bar\varepsilon_3']})) \leq \delta/(3 \cdot 2^k).
\end{align*}
This implies $\mathds{P}((\mathcal{S}_{\bar\varepsilon_3})^c) \leq 2^{k-1}(\delta/(3 \cdot 2^k)+\delta/(3 \cdot 2^k))+\delta/3=2\delta/3$. 

This and (\ref{eq_rec_little_ints_1bis}) imply that if $\bar\varepsilon \leq \bar\varepsilon_3$ we have 
\[
\mathds{E}(f(W_{[-\bar\varepsilon,\bar\varepsilon]}^{k+1,-1}))\leq \mathds{E}(f(W_{[-\bar\varepsilon,\bar\varepsilon]}^{k,-2}))+\mathds{E}(f((\dot V_{t}^{k,+}-\dot V_{0}^{k,+})_{t\in[-\bar\varepsilon,\bar\varepsilon]}))+(2\delta/3)\|f\|_\infty,
\]
which means $\mu_{-1,\bar\varepsilon}^{k+1}(f) \leq \mu_{-2,\bar\varepsilon}^{k}(f)+\mu_{\bar\varepsilon}(f)+(2\delta/3)\|f\|_\infty$. Now, since $\bar\varepsilon \leq \bar\varepsilon_3$ we have $\bar\varepsilon \leq \bar\varepsilon_3'$ hence $\mu_{-2,\bar\varepsilon}^{k}(f) \leq 2^k\mu_{\bar\varepsilon}(f)+(\delta/3)\|f\|_\infty$. We deduce that if $\bar\varepsilon \leq \bar\varepsilon_3$, we have $\mu_{-1,\bar\varepsilon}^{k+1}(f) \leq 2^k\mu_{\bar\varepsilon}(f)+(\delta/3)\|f\|_\infty+\mu_{\bar\varepsilon}(f)+(2\delta/3)\|f\|_\infty$, hence $\mu_{-1,\bar\varepsilon}^{k+1}(f) \leq 2^{k+1}\mu_{\bar\varepsilon}(f)+\delta\|f\|_\infty$. 

\emph{Case $z=0$.} 

Let $0<\bar\varepsilon<\varepsilon$, $f\in\mathbf{F}_{\bar\varepsilon}$. Definition \ref{def_environments} indicates that given $W^k$, with probability $p_{k,-}$ we have $W_{[-\varepsilon,0]}^{k+1,0}=W_{[-\varepsilon,0]}^{k,-1}$ and $W_{[0,\varepsilon]}^{k+1,0}=(\bar W_{t-\varepsilon}^{k,-}-W_{-\varepsilon}^k)_{t\in[0,\varepsilon]}$, and with probability $1-p_{k,-}$ we have $W_{[-\varepsilon,0]}^{k+1,0}=(\bar W_{t+\varepsilon}^{k,+}-W_{\varepsilon}^k)_{t\in[-\varepsilon,0]}$ and $W_{[0,\varepsilon]}^{k+1,0}=W_{[0,\varepsilon]}^{k,1}$. Consequently,  
\begin{equation}\label{eq_rec_little_ints_0}
\mathds{E}(f(W_{[-\bar\varepsilon,\bar\varepsilon]}^{k+1,0})|W^k) 
= p_{k,-}\mathds{E}(\tilde f(W_{[-\bar\varepsilon,0]}^{k,-1},(\bar W_{t-\varepsilon}^{k,-}-W_{-\varepsilon}^k)_{t\in[0,\bar\varepsilon]})|W^k) 
 +(1-p_{k,-})\mathds{E}(\tilde f((\bar W_{t+\varepsilon}^{k,+}-W_{\varepsilon}^k)_{t\in[-\bar\varepsilon,0]},W_{[0,\bar\varepsilon]}^{k,1})|W^k).
\end{equation}
Now, given $W^k$, by definition $\bar W^{k,-}$ has the law of $V^{k,-}$ conditioned to coalesce with $W^k$ before time $\varepsilon$, an event denoted by $\{\sigma_{k,-}<\varepsilon\}$ and satisfying $\mathds{P}(\sigma_{k,-}<\varepsilon|W^k)=p_{k,-}$, so we have 
\[
\mathds{E}(\tilde f(W_{[-\bar\varepsilon,0]}^{k,-1},(\bar W_{t-\varepsilon}^{k,-}-W_{-\varepsilon}^k)_{t\in[0,\bar\varepsilon]})|W^k)=\frac{1}{p_{k,-}}\mathds{E}(\tilde f(W_{[-\bar\varepsilon,0]}^{k,-1},(V_{t-\varepsilon}^{k,-}-W_{-\varepsilon}^k)_{t\in[0,\bar\varepsilon]})\mathds{1}_{\{\sigma_{k,-}<\varepsilon\}}|W^k).
\]
Similarly, 
\[
\mathds{E}(\tilde f((\bar W_{t+\varepsilon}^{k,+}-W_{\varepsilon}^k)_{t\in[-\bar\varepsilon,0]},W_{[0,\bar\varepsilon]}^{k,1})|W^k)=\frac{1}{p_{k,+}}\mathds{E}(\tilde f((V_{t+\varepsilon}^{k,+}-W_{\varepsilon}^k)_{t\in[-\bar\varepsilon,0]},W_{[0,\bar\varepsilon]}^{k,1})\mathds{1}_{\{\sigma_{k,+}>-\varepsilon\}}|W^k).
\]
Furthermore, $p_{k,-}+p_{k,+}=1$, so (\ref{eq_rec_little_ints_0}) implies 
\begin{align*}
 \mathds{E}(\tilde f(W_{[-\bar\varepsilon,0]}^{k+1,0},W_{[0,\bar\varepsilon]}^{k+1,0})|W^k)=&\mathds{E}(\tilde f(W_{[-\bar\varepsilon,0]}^{k,-1},(V_{t-\varepsilon}^{k,-}-W_{-\varepsilon}^k)_{t\in[0,\bar\varepsilon]})\mathds{1}_{\{\sigma_{k,-}<\varepsilon\}}|W^k) \\
 &+\mathds{E}(\tilde f((V_{t+\varepsilon}^{k,+}-W_{\varepsilon}^k)_{t\in[-\bar\varepsilon,0]},W_{[0,\bar\varepsilon]}^{k,1})\mathds{1}_{\{\sigma_{k,+}>-\varepsilon\}}|W^k)
\end{align*}
so 
\begin{equation}\label{eq_rec_little_ints_0bis}
 \mathds{E}(\tilde f(W_{[-\bar\varepsilon,0]}^{k+1,0},W_{[0,\bar\varepsilon]}^{k+1,0}))\leq\mathds{E}(\tilde f(W_{[-\bar\varepsilon,0]}^{k,-1},(V_{t-\varepsilon}^{k,-}-W_{-\varepsilon}^k)_{t\in[0,\bar\varepsilon]}))+\mathds{E}(\tilde f((V_{t+\varepsilon}^{k,+}-W_{\varepsilon}^k)_{t\in[-\bar\varepsilon,0]},W_{[0,\bar\varepsilon]}^{k,1})).
\end{equation}
Let us deal with $\mathds{E}(\tilde f(W_{[-\bar\varepsilon,0]}^{k,-1},(V_{t-\varepsilon}^{k,-}-W_{-\varepsilon}^k)_{t\in[0,\bar\varepsilon]}))$. In order to do that, we introduce temporary notation. For any measure $\mu$ on continuous processes defined on $[-\bar\varepsilon,\bar\varepsilon]$, $(W_t)_{t\in[-\bar\varepsilon,\bar\varepsilon]}$ will be a process with law $\mu$, and $(W'_t)_{t\in[0,\bar\varepsilon]}$ will be defined thus: $W'_{0}=W_{0}$ and $(W'_t)_{t\in[0,\bar\varepsilon]}$ is a Brownian motion reflected on $(W_t)_{t\in[0,\bar\varepsilon]}$ and above it. We then have 
\[
 \mathds{E}(\tilde f(W_{[-\bar\varepsilon,0]}^{k,-1},(V_{t-\varepsilon}^{k,-}-W_{-\varepsilon}^k)_{t\in[0,\bar\varepsilon]}))
 = \mathds{E}(\mathds{E}(\tilde f(W_{[-\bar\varepsilon,0]}^{k,-1},(V_{t-\varepsilon}^{k,-}-W_{-\varepsilon}^k)_{t\in[0,\bar\varepsilon]})|W_{[-\bar\varepsilon,\bar\varepsilon]}^{k,-1}))
 \]
 \[
  =\mu_{-1,\bar\varepsilon}^k(\mathds{E}(\tilde f(W_{[-\bar\varepsilon,0]},(W'_{t}-W'_{0})_{t\in[0,\bar\varepsilon]})|W)).
\]
Now, by the induction hypothesis there exists some $\bar\varepsilon_4 > 0$ so that for any $g \in \mathbf{F}_{\bar\varepsilon_4}$ we have $\mu_{-1,\bar\varepsilon_4}^k(g)\leq 2^k\mu_{\bar\varepsilon_4}(g)+(\delta/2)\|g\|_\infty$, therefore if $\bar\varepsilon\leq\bar\varepsilon_4$, we have 
\[
\mathds{E}(\tilde f(W_{[-\bar\varepsilon,0]}^{k,-1},(V_{t-\varepsilon}^{k,-}-W_{-\varepsilon}^k)_{t\in[0,\bar\varepsilon]})) \leq 2^k\mu_{\bar\varepsilon}(\mathds{E}(\tilde f(W_{[-\bar\varepsilon,0]},(W'_{t}-W'_{0})_{t\in[0,\bar\varepsilon]})|W))+(\delta/2)\|\tilde f\|_\infty
=2^k\mu_{-,\bar\varepsilon}(f)+(\delta/2)\|f\|_\infty.
\]
Similarly, if $\bar\varepsilon\leq\bar\varepsilon_4$, we have 
\[
\mathds{E}(\tilde f((V_{t+\varepsilon}^{k,+}-W_{\varepsilon}^k)_{t\in[-\bar\varepsilon,0]},W_{[0,\bar\varepsilon]}^{k,1}))\leq 2^k\mu_{+,\bar\varepsilon}(f)+(\delta/2)\|f\|_\infty.
\]
Consequently, (\ref{eq_rec_little_ints_0bis}) implies that if $\bar\varepsilon\leq\bar\varepsilon_4$, we have 
\[
\mathds{E}(\tilde f(W_{[-\bar\varepsilon,0]}^{k+1,0},W_{[0,\bar\varepsilon]}^{k+1,0}))\leq 2^k\mu_{-,\bar\varepsilon}(f)+(\delta/2)\|f\|_\infty+2^k\mu_{+,\bar\varepsilon}(f)+(\delta/2)\|f\|_\infty,
\]
that is $\mu^{k+1}_{0,\bar\varepsilon}(f)\leq 2^k\mu_{-,\bar\varepsilon}(f)+2^k\mu_{+,\bar\varepsilon}(f)+\delta \|f\|_\infty$. 

To conclude, if we set $\bar\varepsilon = \min(\bar\varepsilon_1,\bar\varepsilon_2,\bar\varepsilon_3,\bar\varepsilon_4,\varepsilon/2)>0$, for any $f\in\mathbf{F}_{\bar\varepsilon}$, for any $z\in\mathds{Z}\setminus\{0\}$ we have $\mu_{z,\bar\varepsilon}^{k+1}(f)\leq 2^{k+1}\mu_{\bar\varepsilon}(f)+\delta\|f\|_\infty$, and $\mu^{k+1}_{0,\bar\varepsilon}(f)\leq 2^k\mu_{-,\bar\varepsilon}(f)+2^k\mu_{+,\bar\varepsilon}(f)+\delta \|f\|_\infty$, which ends the proof of Proposition \ref{prop_rec_little_ints}. 
\end{proof}
 
 \section{Convergence of the mesoscopic quantities}\label{sec_conv_limit_processes}
 
 In order to prove the main results of this work, Theorem \ref{thm_main} and Proposition \ref{prop_continuity}, we need to prove the convergence of the ``mesoscopic'' quantities, that is the $\frac{1}{n}(X_{T_{k+1}}-X_{T_{k}})$ and $\frac{1}{n^{3/2}}(T_{k+1}-T_{k})$ (we remind the reader that the $T_k$ are defined in \eqref{eq_def_Tk}). For $\varepsilon > 0$, for any $k\in\mathds{N}$, we recall the following definition already given at the beginning of Section \ref{sec_conclusion}:
 \begin{equation}\label{eq_def_Zk}
 Z_k^N = \frac{1}{\lfloor\varepsilon n\rfloor}(X_{T_k}-X_{T_0}).
 \end{equation}
 Then $(Z_k^N)_{k\in\mathds{N}}$ is a nearest-neighbor random walk on $\mathds{Z}$. The result we will need to prove Theorem \ref{thm_main} is the following.
 
\begin{proposition}\label{prop_conv_ZT}
 For any $\varepsilon > 0$, $K\in\mathds{N}^*$, the random variable $(Z_1^N,...,Z_K^N,\frac{1}{n^{3/2}}(T_1-T_0),\frac{1}{n^{3/2}}(T_2-T_1),...,\frac{1}{n^{3/2}}(T_K-T_{K-1}))$ converges in distribution to $(\breve Z_1,...,\breve Z_K,\breve T_1,...,\breve T_K)$ (defined as in Definition \ref{def_lim_envts}) when $N$ tends to $+\infty$. Moreover, the $\breve T_k$ and $\sum_{k'=1}^k \breve T_{k'}$, $k\in \{1,...,K\}$, have no atoms. 
\end{proposition}

To prove Proposition \ref{prop_continuity}, we need a weaker but analogous result. If $\psi : N \mapsto N$ is so that $\psi(N)$ tends to $+\infty$ when $N$ tends to $+\infty$, if $\theta>\frac{1}{2\sqrt{2}}$, $T_0'=\mathbf{T}^-_{\lfloor \psi(N)\theta\rfloor,0}$ (defined in \eqref{eq_def_T_mi}), and $T_1'=\inf\{m \geq T_0' \,|\, |X_m-X_{T_0'}|=\lfloor \theta \psi(N)/2 \rfloor\}$, we have the following, which will be proven at the end of the section.

\begin{lemma}\label{lem_continuity}
 $\frac{1}{\psi(N)^{3/2}}(X_{T_1'}-X_{T_0'})$ converges in distribution when $N$ tends to $+\infty$. 
\end{lemma}

 In order to prove Proposition \ref{prop_conv_ZT}, we notice that for $k\in\{0,...,K-1\}$, we have $Z_{k+1}^N=Z_k^N-1$ if and only if $(X_{T_k+m})_{m\in\mathds{N}}$ reaches $X_{T_k}-\lfloor \varepsilon n \rfloor$ before $X_{T_k}+\lfloor \varepsilon n \rfloor$, which means $L^{T_k,-}_{X_{T_k}+\lfloor \varepsilon n \rfloor}=0$ (see Definition \ref{def_local_times}). In addition, in this case one can check that $T_{k+1}-T_k=\lfloor \varepsilon n \rfloor+2\sum_{i\in\mathds{Z}}L_i^{T_k,-}$. We thus wish to study $L_i^{T_k,-}$. Moreover, remembering Definition \ref{def_environments}, for $i > X_{T_k}-\lfloor \varepsilon n \rfloor$, we have $L_i^{T_k,-}=S_i^{T_k,-,E}-S_i^{T_k,-,B}$, and it so happens that $(S_i^{T_k,-,E})_i$ is close to a random walk reflected on $(S_i^{T_k,-,B})_i$ when $i \in \{X_{T_k}-\lfloor \varepsilon n \rfloor+1,...,X_{T_k}\}$ and absorbed by $(S_i^{T_k,-,B})_i$ when $i \geq X_{T_k}$. Therefore, we are going to study the limit of the processes $(S_i^{T_k,-,B})_i$, which can be considered as ``environments'' in which the $(S_i^{T_k,-,E})_{i}$ evolve. In order to have more practical notation, the precise environment process we will study is the following (we recall the definition of the $\Delta_{m,i}$ in \eqref{eq_def_Delta}).
 
 \begin{definition}\label{def_envts_bis}
  For any $k\in\mathds{N}$, the environment process at time $T_k$, $(E_{k,i}^N)_{i\in\mathds{Z}}$, is defined by $E_{k,i}^N=\sum_{j=X_{T_k}+i}^{X_{T_k}}(\Delta_{T_k,j}+1/2)$ for $i \leq 0$ and $E_{k,i}^N=\sum_{j=X_{T_k}+1}^{X_{T_k}+i-1}(-\Delta_{T_k,j}+1/2)$ for $i \geq 1$.
 \end{definition}
 
 For any family of real-valued discrete processes $(H_i^N)_{i\in\mathds{Z}}$, any real numbers $a<b$, we will write ``$(H_{nt}^N)_{t\in[a,b]}$'' as a shorthand for ``the linear interpolation of $(H_{\lfloor nt \rfloor}^N)_{t\in[a,b]}$''. For any $k\in\mathds{N}$, $\frac{1}{n^{3/2}}(T_{k+1}-T_k)$ can be written as as a function of $Z_{k}^N$, $Z_{k+1}^N$, $(\frac{1}{\sqrt{n}}E_{k,nt}^N)_{t\in[-2\varepsilon,2\varepsilon]}$ and $(\frac{1}{\sqrt{n}}E_{k+1,nt}^N)_{t\in[-2\varepsilon,2\varepsilon]}$. Consequently, it will be enough to prove that the quantity $(Z_1^N,...,Z_k^N,(\frac{1}{\sqrt{n}}E_{0,nt}^N)_{t\in[-a,a]},...,(\frac{1}{\sqrt{n}}E_{k,nt}^N)_{t\in[-a,a]})$ converges in distribution when $N$ tends to $+\infty$ to prove Proposition \ref{prop_conv_ZT}. This is the following proposition. 
 
 \begin{proposition}\label{prop_conv_envts}
  For any $k\in\mathds{N}$, for any $a > 0$, we have that $(Z_1^N,...,Z_k^N,(\frac{1}{\sqrt{n}}E_{0,nt}^N)_{t\in[-a,a]},...,(\frac{1}{\sqrt{n}}E_{k,nt}^N)_{t\in[-a,a]})$ converges in distribution to the quantity $(\breve Z_1,...,\breve Z_k,(W_t^0)_{t\in[-a,a]},...,(W_t^k)_{t\in[-a,a]})$ when $N$ tends to $+\infty$. 
 \end{proposition}

 We first prove Proposition \ref{prop_conv_ZT} given Proposition \ref{prop_conv_envts}.
 
 \begin{proof}[Proof of Proposition \ref{prop_conv_ZT}.]
 Let $\varepsilon > 0$, $K\in\mathds{N}^*$. For any $k\in\{0,...,K-1\}$, we will write $\frac{1}{n^{3/2}}(T_{k+1}-T_{k})$ as a function of $Z_{k}^N$, $Z_{k+1}^N$, $(\frac{1}{\sqrt{n}}E_{k,nt}^N)_{t\in[-2\varepsilon,2\varepsilon]}$ and $(\frac{1}{\sqrt{n}}E_{k+1,nt}^N)_{t\in[-2\varepsilon,2\varepsilon]}$. Indeed, if $Z_{k+1}^N=Z_{k}^N-1$, we have 
 \[
 T_{k+1}-T_{k}=2\sum_{i=X_{T_k}-\lfloor \varepsilon n \rfloor+1}^{X_{T_k}+\lfloor \varepsilon n \rfloor}L_i^{T_k,-}+\lfloor \varepsilon n \rfloor=2\sum_{i=-\lfloor \varepsilon n \rfloor+1}^{\lfloor \varepsilon n \rfloor}(E_{k+1,i+\lfloor \varepsilon n\rfloor}^N+E^N_{k,-\lfloor \varepsilon n\rfloor+1}-E_{k,i}^N)+\lfloor \varepsilon n \rfloor,
 \]
 while if $Z_{k+1}^N=Z_{k}^N+1$, we have 
 \[
 T_{k+1}-T_k=2\sum_{i=X_{T_k}-\lfloor \varepsilon n \rfloor+1}^{X_{T_k}+\lfloor \varepsilon n \rfloor}L_i^{T_k,+}-\lfloor \varepsilon n \rfloor=2\sum_{i=-\lfloor \varepsilon n \rfloor+1}^{\lfloor \varepsilon n \rfloor}(E_{k+1,i-\lfloor \varepsilon n\rfloor}^N-E_{k+1,0}^N+1-(E_{k,i}^N-E_{k,\lfloor \varepsilon n\rfloor}^N))-\lfloor \varepsilon n \rfloor
 \]
 \[
 =2\sum_{i=-\lfloor \varepsilon n \rfloor+1}^{\lfloor \varepsilon n \rfloor}(E_{k+1,i-\lfloor \varepsilon n\rfloor}^N-E_{k+1,0}^N-E_{k,i}^N+E_{k,\lfloor \varepsilon n\rfloor}^N)+3\lfloor \varepsilon n \rfloor.
 \]
 Therefore, if for any $z,z'\in\mathds{Z}$, $f,g$ continuous real functions on $[-2\varepsilon,2\varepsilon]$ we define $F_{N}(Z,Z',f,g)$ as
 \begin{align*} 
 &\mathds{1}_{\{z'=z-1\}} \left(\frac{2}{n}\sum_{i=-\lfloor \varepsilon n \rfloor+1}^{\lfloor \varepsilon n \rfloor}\left(g\left(\frac{i+\lfloor \varepsilon n\rfloor}{n}\right)+f\left(\frac{-\lfloor \varepsilon n\rfloor+1}{n}\right)-f\left(\frac{i}{n}\right)\right)+\frac{\lfloor \varepsilon n \rfloor}{n^{3/2}}\right) \\
 + & \mathds{1}_{\{z'=z+1\}} \left(\frac{2}{n}\sum_{i=-\lfloor \varepsilon n \rfloor+1}^{\lfloor \varepsilon n \rfloor}\left(g\left(\frac{i-\lfloor \varepsilon n\rfloor}{n}\right)-g(0)-f\left(\frac{i}{n}\right)+f\left(\frac{\lfloor \varepsilon n\rfloor}{n}\right)\right)+3\frac{\lfloor \varepsilon n \rfloor}{n^{3/2}}\right),
\end{align*}
then 
\[
\frac{1}{n^{3/2}}(T_{k+1}-T_k)=F_N\left(Z_k^N,Z_{k+1}^N,\left(\frac{1}{\sqrt{n}}E_{k,nt}^N\right)_{t\in[-2\varepsilon,2\varepsilon]},\left(\frac{1}{\sqrt{n}}E_{k+1,nt}^N\right)_{t\in[-2\varepsilon,2\varepsilon]}\right).
\]
Now, thanks to Proposition \ref{prop_conv_envts}, $(Z_1^N,...,Z_K^N,(\frac{1}{\sqrt{n}}E_{0,nt}^N)_{t\in[-2\varepsilon,2\varepsilon]},...,(\frac{1}{\sqrt{n}}E_{K,nt}^N)_{t\in[-2\varepsilon,2\varepsilon]})$ converges in distribution to $(\breve Z_1,...,\breve Z_K,(W_t^0)_{t\in[-2\varepsilon,2\varepsilon]},...,(W_t^K)_{t\in[-2\varepsilon,2\varepsilon]})$ when $N$ tends to $+\infty$. The convergence in distribution of Proposition \ref{prop_conv_ZT} follows easily. Furthermore, Lemma \ref{lem_no_atoms} yields that the $\breve T_k$ and the $\sum_{k'=1}^k \breve T_{k'}$, $k\in\{1,...,K\}$ have no atoms, which ends the proof of Proposition \ref{prop_conv_ZT}.
\end{proof}

It now remains only to prove Proposition \ref{prop_conv_envts}.
 
 \begin{proof}[Proof of Proposition \ref{prop_conv_envts}.]
 We recall the convention already used in Section \ref{sec_process_above}: all the Brownian motions have variance equal to the variance of the law $\rho_0$ defined in \eqref{eq_def_rho0}. Let us prove the proposition by induction on $k$. For $k=0$, for any $a>0$, we notice that Proposition \ref{prop_law_zeta} implies that if $(\mathcal{B}_0^{\lfloor N \theta \rfloor,\lfloor Nx \rfloor,\pm})^c$ occurs and $n$ is large enough, for any $X_{T_0}-\lceil a n \rceil \leq i \leq X_{T_0}+\lceil a n \rceil$ we have $\Delta_{T_0,i}=\bar\Delta_{T_0,i}$. Moreover, $\lim_{N\to+\infty}\mathds{P}((\mathcal{B}_0^{\lfloor N \theta \rfloor,\lfloor Nx \rfloor,\pm})^c)=1$. Furthermore, for any $i < X_{T_0}$, $\bar\Delta_{T_0,i}+1/2$ has law $\rho_0$, for any $i > X_{T_0}$, $-\bar\Delta_{T_0,i}+1/2$ has law $\rho_0$, $\bar \Delta_{T_0,X_{T_0}}+1/2$ has law $\rho_0$ or $\rho_0$ translated by $+1$, and these variables are independent. Therefore $(\frac{1}{\sqrt{n}}E_{0,nt}^N)_{t\in[-a,a]}$ converges to $(W_t^0)_{t\in[-a,a]}$ by Donsker's invariance principle. 
 
 We now set $k\in\mathds{N}$ and suppose the proposition is true for $k$. We will prove it for $k+1$. Let $a>0$.  We will study processes corresponding to ``the environment at the first time after $T_k$ at which the process reaches $X_{T_k}-\lfloor \varepsilon n \rfloor$'' and ``the environment at the first time after $T_k$ at which the process reaches $X_{T_k}+\lfloor \varepsilon n \rfloor$'', and prove they have suitable convergences in distribution. From the convergence in distribution of these two processes we will deduce the convergence in distribution of $Z_{k+1}^N$ and $(\frac{1}{\sqrt{n}}E_{k+1,nt}^N)_{t\in[-a,a]}$. The ``environment at the first time after $T_k$ at which the process reaches $X_{T_k}-\lfloor \varepsilon n \rfloor$'' is defined as follows. Remembering Definition \ref{def_local_times}, we define the process $(E_{k,i}^{N,-})_{i\in\mathds{Z}}$ by $E_{k,i}^{N,-}=E_{k,-\lfloor \varepsilon n \rfloor+1}^N+\sum_{j=X_{T_k}-\lfloor \varepsilon n \rfloor+1}^{X_{T_k}+i-1}\zeta_j^{T_k,-,E}$ for $i > -\lfloor \varepsilon n \rfloor$ (so $E_{k,i}^{N,-}=E_{k,-\lfloor \varepsilon n \rfloor+1}^N+S_{X_{T_k}+i}^{T_k,-,E}$ if we recall Definition \ref{def_environments}) and $E_{k,i}^{N,-}=E_{k,-\lfloor \varepsilon n \rfloor+1}^N$ for $i \leq -\lfloor \varepsilon n \rfloor$. We also define $\sigma_{k,-}^N=\inf\{i> 0 \,|\, L_{X_{T_k}+i}^{T_k,-}=0\}$, noticing that $Z_{k+1}^N=Z_k^N-1$ if and only if $X_{T_k}-\lfloor \varepsilon n\rfloor$ is reached before $X_{T_k}+\lfloor \varepsilon n\rfloor$, that is if and only if $\sigma_{k,-}^N \leq \lfloor \varepsilon n\rfloor$. 
 
 We want to prove the convergence in distribution of $(\frac{1}{n}\sigma_{k,-}^N,(\frac{1}{\sqrt{n}}E_{k,nt}^{N,-})_{t\in[-\varepsilon,\varepsilon]})$ to a target process $(\sigma_{k,-},(V_t^{k,-})_{t\in[-\varepsilon,\varepsilon]})$ where $V^{k,-}$ is a Brownian motion reflected on $W^k$ on $[-\varepsilon,0]$ and absorbed by $W^k$ on $[0,\varepsilon]$, while $\sigma_{k,-}$ is the absorption time. In order to do that, we will define another auxiliary process $(\tilde E_{k,i}^{N,-})_{i\in\mathds{Z}}$. We will first prove that $(\frac{1}{\sqrt{n}}\tilde E_{k,nt}^{N,-})_{t\in[-\varepsilon,\varepsilon]}$ converges in distribution to a Brownian motion reflected on $W^k$ on $[-\varepsilon,0]$ and free on $[0,\varepsilon]$. After that, we will write $(\frac{1}{n}\sigma_{k,-}^N,(\frac{1}{\sqrt{n}}E_{k,nt}^{N,-})_{t\in[-\varepsilon,\varepsilon]})$ as a function of $(\frac{1}{\sqrt{n}}\tilde E_{k,nt}^{N,-})_{t\in[-\varepsilon,\varepsilon]}$ to deduce the convergence of the former. The process $(\tilde E_{k,i}^{N,-})_{i\in\mathds{Z}}$ is defined as follows: remembering that the $\xi_j^{m,-,I}$ were constructed just before Proposition \ref{prop_def_zetaI}, for $i > -\lfloor \varepsilon n \rfloor$, we set
 \[
 \tilde E_{k,i}^{N,-}=E_{k,-\lfloor \varepsilon n \rfloor+1}^N+\sum_{j=X_{T_k}-\lfloor \varepsilon n \rfloor+1}^{X_{T_k}+(i-1)\wedge\sigma_{k,-}^N}\zeta_j^{T_k,-,E}+\sum_{j=X_{T_k}+(i-1)\wedge\sigma_{k,-}^N+1}^{X_{T_k}+i-1}\zeta_j^{T_k,-,I},
 \]
 and when $i \leq -\lfloor \varepsilon n \rfloor$ we set $\tilde E_{k,i}^{N,-}=E_{k,-\lfloor \varepsilon n \rfloor+1}^N$. In order to have shorter notation, we will also write $\Xi^N=(Z_1^N,...,Z_k^N,(\frac{1}{\sqrt{n}}E_{0,nt}^N)_{t\in[-a-\varepsilon,a+\varepsilon]},...,(\frac{1}{\sqrt{n}}E_{k,nt}^N)_{t\in[-a-\varepsilon,a+\varepsilon]})$ and $\Xi=(\breve Z_1,...,\breve Z_k,(W_t^0)_{t\in[-a-\varepsilon,a+\varepsilon]},...,(W_t^k)_{t\in[-a-\varepsilon,a+\varepsilon]})$.
 
 \begin{claim}\label{claim_conv_mixed_evts}
  $(\Xi^N,(\frac{1}{\sqrt{n}}\tilde E_{k,nt}^{N,-})_{t\in[-\varepsilon,\varepsilon]})$ converges in distribution to $(\Xi,(\tilde W_t^k)_{t\in[-\varepsilon,\varepsilon]})$ when $N$ tends to $+\infty$, where the process $(\tilde W_t^k)_{t\in[-\varepsilon,\varepsilon]}$ is a Brownian motion with $\tilde W^k_{-\varepsilon}=W^k_{-\varepsilon}$  reflected above $W^k$ on $W^k$ on $[-\varepsilon,0]$ and free on $[0,\varepsilon]$.
 \end{claim}
 
 \begin{proof}[Proof of claim \ref{claim_conv_mixed_evts}.] 
  We will introduce two auxiliary processes, $(\breve E_{k,i}^{N,-})_{i\in\mathds{Z}}$ and $(E_{k,i}^{N,-,I})_{i\in\mathds{Z}}$. The process $(\breve E_{k,i}^{N,-})_{i\in\mathds{Z}}$ will represent ``the random walk $(E_{k,i}^{N,-,I})_{i\in\mathds{Z}}$ reflected on the environment $(E_{k,i}^N)_{i\in\mathds{Z}}$ until time 0 and free after time 0'', and so will have the right convergence in distribution towards our target. The process $(\tilde E_{k,i}^{N,-})_{i\in\mathds{Z}}$ will be close to $(\breve E_{k,i}^{N,-})_{i\in\mathds{Z}}$, which will allow us to prove it satisfies the same convergence in distribution. We define $(E_{k,i}^{N,-,I})_{i\in\mathds{Z}}$ as follows: for any $i \leq -\lfloor \varepsilon n \rfloor$ we set $E_{k,i}^{N,-,I}=E_{k,-\lfloor \varepsilon n \rfloor+1}^N$, and for any $i>-\lfloor \varepsilon n \rfloor$ we set $E_{k,i}^{N,-,I}=E_{k,-\lfloor \varepsilon n \rfloor+1}^N+\sum_{j=X_{T_k}-\lfloor \varepsilon n \rfloor+1}^{X_{T_k}+i-1}\zeta_j^{T_k,-,I}$. We define $(\breve E_{k,i}^{N,-})_{i\in\mathds{Z}}$ as follows: for any $i \leq -\lfloor \varepsilon n \rfloor$ we set $\breve E_{k,i}^{N,-}=E_{k,-\lfloor \varepsilon n \rfloor+1}^N$, for any $i\in\{-\lfloor \varepsilon n \rfloor+1,...,0\}$ we set $\breve E_{k,i}^{N,-}=E_{k,i}^{N,-,I}+\max_{-\lfloor \varepsilon n \rfloor+1 \leq j \leq i}(E_{k,j}^N-E_{k,j}^{N,-,I})$, and for any $i>0$ we set $\breve E_{k,i}^{N,-}=\breve E_{k,0}^{N,-}+\sum_{j=X_{T_k}}^{X_{T_k}+i-1}\zeta_j^{T_k,-,I}$. 
  
  We begin by studying the convergence of $(\breve E_{k,i}^{N,-})_{i\in\mathds{Z}}$. We notice that $\Xi^N$ is $\mathcal{F}_{T_k}$-mesurable (see \eqref{eq_def_F} for the definition of $\mathcal{F}_m$), and that by Proposition \ref{prop_def_zetaI} the $(\zeta_i^{T_k,-,I})_{i\in\mathds{Z}}$ are independent from $\mathcal{F}_{T_k}$ and i.i.d. with law $\rho_0$, hence the $(\zeta_{X_{T_k}+i}^{T_k,-,I})_{i\in\mathds{Z}}$ are independent from $\mathcal{F}_{T_k}$ and i.i.d. with law $\rho_0$. Therefore, Donsker's invariance principle yields that $(\Xi^N,(\frac{1}{\sqrt{n}}E_{k,nt}^{N,-,I})_{t\in[-\varepsilon,\varepsilon]})$ converges in distribution to $(\Xi,(W_t^{k,I})_{t\in[-\varepsilon,\varepsilon]})$ when $N$ tends to $+\infty$, where $W^{k,I}_{-\varepsilon}=W_{-\varepsilon}^k$ and $W^{k,I}-W_{-\varepsilon}^k$ is a Brownian motion independent from $\Xi$. We can define $(\tilde W_t^k)_{t\in[-\varepsilon,\varepsilon]}$ by $\tilde W_t^k=W_t^{k,I}+\max_{-\varepsilon \leq s \leq t}(W_s^k-W_s^{k,I})$ when $t\in[-\varepsilon,0]$ and $\tilde W_t^k=\tilde W_0^k+W_t^{k,I}-W_0^{k,I}$ when $t\in[0,\varepsilon]$. Then $(\tilde W_t^k)_{t\in[-\varepsilon,\varepsilon]}$ is a Brownian motion with $\tilde W^k_{-\varepsilon}=W^k_{-\varepsilon}$ reflected above $W^k$ on $W^k$ on $[-\varepsilon,0]$ and free on $[0,\varepsilon]$, and $(\Xi^N,(\frac{1}{\sqrt{n}}\breve E_{k,nt}^{N,-})_{t\in[-\varepsilon,\varepsilon]})$ converges in distribution to $(\Xi,(\tilde W_t^{k})_{t\in[-\varepsilon,\varepsilon]})$ when $N$ tends to $+\infty$.
  
  We now prove that $(\tilde E_{k,i}^{N,-})_{i\in\mathds{Z}}$ is close to $(\breve E_{k,i}^{N,-})_{i\in\mathds{Z}}$. For any $i \leq -\lfloor \varepsilon n \rfloor$ we have $\tilde E_{k,i}^{N,-}=\breve E_{k,i}^{N,-}$ by definition of the processes. We now deal with $i\in\{-\lfloor \varepsilon n \rfloor+1,...,0\}$. Firstly, we notice that for any $j > -\lfloor \varepsilon n \rfloor$, recalling Definition \ref{def_environments}, we have 
  \[
  E_{k,j}^{N,-,I}-E_{k,j}^N=\sum_{j'=X_{T_k}-\lfloor \varepsilon n \rfloor+1}^{X_{T_k}+j-1}(\zeta_{j'}^{T_k,-,I}-\zeta_{j'}^{T_k,-,B})=\sum_{j'=X_{T_k}-\lfloor \varepsilon n \rfloor+1}^{X_{T_k}+j-1}\zeta_{j'}^{T_k,-,I} - S_{X_{T_k}+j}^{T_k,-,B}.
  \]
  Recalling Definition \ref{def_SI}, the definition of $(\breve E_{k,i}^{N,-})_{i\in\mathds{Z}}$ and Lemma \ref{lem_SI_max} then imply 
  \begin{align*}
   \breve E_{k,i}^{N,-}&=E_{k,-\lfloor \varepsilon n \rfloor+1}^N+\sum_{j=X_{T_k}-\lfloor \varepsilon n \rfloor+1}^{X_{T_k}+i-1}\zeta_j^{T_k,-,I}+\max_{-\lfloor \varepsilon n \rfloor+1 \leq j \leq i}\left(S_{X_{T_k}+j}^{T_k,-,B}-\sum_{j'=X_{T_k}-\lfloor \varepsilon n \rfloor+1}^{X_{T_k}+j-1}\zeta_{j'}^{T_k,-,I}\right) \\
   &=E_{k,-\lfloor \varepsilon n \rfloor+1}^N+S_{X_{T_k}+i}^{T_k,-,I}.
  \end{align*}
  Consequently, we have 
  \[
  \breve E_{k,i}^{N,-}-\tilde E_{k,i}^{N,-}=S_{X_{T_k}+i}^{T_k,-,I}-\sum_{j=X_{T_k}-\lfloor \varepsilon n \rfloor+1}^{X_{T_k}+i-1}\zeta_j^{T_k,-,E}=S_{X_{T_k}+i}^{T_k,-,I}-S_{X_{T_k}+i}^{T_k,-,E}.
  \]
  We recall that the ``bad events'' $\mathcal{B}$, $\mathcal{B}_0$, $\mathcal{B}_{m,1}^-,...,\mathcal{B}_{m,6}^-$, $\mathcal{B}_1,...,\mathcal{B}_6$ were defined in Propositions \ref{prop_bound_m}, \ref{prop_law_zeta} and at the beginning of Section \ref{sec_bad_events}. Now, by Proposition \ref{prop_sym}, if $n$ is large enough (not depending on $T_k$ or $i$), if $\bigcap_{r=1}^6(\mathcal{B}_{T_k,r}^-)^c$ occurs then $|S_{X_{T_k}+i}^{T_k,-,I}-S_{X_{T_k}+i}^{T_k,-,E}| \leq (\ln n)^8 n^{1/4}$. Furthermore, by Proposition \ref{prop_bound_m}, if $\mathcal{B}^c$ occurs and $n$ is large enough, $T_k =\mathbf{T}_{m,i}^+$ or $\mathbf{T}_{m,i}^-$ (see \eqref{eq_def_T_mi} for the definition of $\mathbf{T}_{m,i}^\pm$) for some integers $\lfloor N \theta \rfloor-2n^{(\alpha+4)/5} \leq m \leq \lfloor N\theta \rfloor+2n^{(\alpha+4)/5}$ and $\lfloor Nx \rfloor- n^{(\alpha+4)/5} \leq i \leq \lfloor Nx \rfloor+ n^{(\alpha+4)/5}$, hence if $\mathcal{B}^c\cap\bigcap_{r=0}^6 \mathcal{B}_r^c$ occurs and $n$ is large enough, $\bigcap_{r=1}^6(\mathcal{B}_{T_k,r}^-)^c$ occurs. Therefore, if $\mathcal{B}^c\cap\bigcap_{r=0}^6 \mathcal{B}_r^c$ occurs and $n$ is large enough, $|\breve E_{k,i}^{N,-}-\tilde E_{k,i}^{N,-}|\leq (\ln n)^8 n^{1/4}$ for all $i\in\{-\lfloor \varepsilon n \rfloor+1,...,0\}$.
  
  We now deal with the case $i \in \{1,...,\lfloor \varepsilon n \rfloor+1\}$. We can then write 
  \begin{equation}\label{eq_reflected_close}
  \begin{split}
   \breve E_{k,i}^{N,-}-\tilde E_{k,i}^{N,-} = 
   \breve E_{k,0}^{N,-}+\sum_{j=X_{T_k}}^{X_{T_k}+i-1}\zeta_j^{T_k,-,I}-\left(\tilde E_{k,0}^{N,-}+\sum_{j=X_{T_k}}^{X_{T_k}+(i-1)\wedge\sigma_{k,-}^N}\zeta_j^{T_k,-,E}+\sum_{j=X_{T_k}+(i-1)\wedge\sigma_{k,-}^N+1}^{X_{T_k}+i-1}\zeta_j^{T_k,-,I}\right) \\
    = \breve E_{k,0}^{N,-}-\tilde E_{k,0}^{N,-}+\sum_{j=X_{T_k}}^{X_{T_k}+(i-1)\wedge\sigma_{k,-}^N}\left(\zeta_j^{T_k,-,I}-\zeta_j^{T_k,-,E}\right).\qquad\qquad\qquad\qquad\qquad\qquad
  \end{split}
  \end{equation}
 We assume $\mathcal{B}^c\cap\bigcap_{r=0}^6 \mathcal{B}_r^c$ occurs and $n$ is large enough so it implies $\bigcap_{r=1}^6(\mathcal{B}_{T_k,r}^-)^c$ occurs. Since $(\mathcal{B}_{T_k,3}^-)^c$ occurs, for any $j\in\{X_{T_k},...,X_{T_k}+\lfloor \varepsilon n\rfloor\}$ such that $L_j^{T_k,-} \geq (\ln n)^2$, we have $\zeta_j^{T_k,-,I}=\zeta_j^{T_k,-,E}$, and since $(\mathcal{B}_{T_k,5}^-)^c$ occurs, for any $j\in\{X_{T_k},...,X_{T_k}+\lfloor \varepsilon n\rfloor\}$ we have $|\zeta_j^{T_k,-,I}|,|\zeta_j^{T_k,-,E}| \leq (\ln n)^2$. We deduce 
 \[
 \left|\sum_{j=X_{T_k}}^{X_{T_k}+(i-1)\wedge\sigma_{k,-}^N}(\zeta_j^{T_k,-,I}-\zeta_j^{T_k,-,E})\right| 
 \leq 2 (\ln n)^2 \left|\left\{j\in\{X_{T_k},...,X_{T_k}+(i-1)\wedge\sigma_{k,-}^N\}\left|L_j^{T_k,-} < (\ln n)^2\right.\right\}\right| 
 \]
 \[
  \leq 2 (\ln n)^2 \left(|\{j\in\{X_{T_k},...,X_{T_k}+\lfloor \varepsilon n\rfloor\}\,|\,0 < L_j^{T_k,-} < (\ln n)^2\}| +1\right).
 \]
 Now, since $(\mathcal{B}_{T_k,2}^-)^c$ occurs, if $n$ is large enough, $|\{j\in\{X_{T_k},...,X_{T_k}+\lfloor \varepsilon n\rfloor\}\,|\,0 < L_j^{T_k,-} < (\ln n)^2\}| < (\ln n)^8$, hence $|\sum_{j=X_{T_k}}^{X_{T_k}+(i-1)\wedge\sigma_{k,-}^N}(\zeta_j^{T_k,-,I}-\zeta_j^{T_k,-,E})| \leq 3 (\ln n)^{10}$. Moreover, we already proved that if $\mathcal{B}^c\cap\bigcap_{r=0}^6 \mathcal{B}_r^c$ occurs and $n$ is large enough, $|\breve E_{k,0}^{N,-}-\tilde E_{k,0}^{N,-}|\leq (\ln n)^8 n^{1/4}$. Thus (\ref{eq_reflected_close}) yields that if $\mathcal{B}^c\cap\bigcap_{r=0}^6 \mathcal{B}_r^c$ occurs and $n$ is large enough, $|\breve E_{k,i}^{N,-}-\tilde E_{k,i}^{N,-}|\leq (\ln n)^8 n^{1/4}+3 (\ln n)^{10}$ for any $i \in \{1,...,\lfloor \varepsilon n \rfloor+1\}$. 
  
  We deduce that if $\mathcal{B}^c\cap\bigcap_{r=0}^6 \mathcal{B}_r^c$ occurs and $n$ is large enough, for any $i\leq\lfloor \varepsilon n \rfloor+1$ we have $|\breve E_{k,i}^{N,-}-\tilde E_{k,i}^{N,-}|\leq (\ln n)^8 n^{1/4}+3 (\ln n)^{10}$. In addition, Propositions \ref{prop_bound_m} and \ref{prop_bound_bad_evts} imply $\lim_{N \to +\infty}\mathds{P}(\mathcal{B}\cup\bigcup_{r=0}^6 \mathcal{B}_r)=0$. Furthermore, we proved that $(\Xi^N,(\frac{1}{\sqrt{n}}\breve E_{k,nt}^{N,-})_{t\in[-\varepsilon,\varepsilon]})$ converges in distribution to $(\Xi,(\tilde W_t^{k})_{t\in[-\varepsilon,\varepsilon]})$ when $N$ tends to $+\infty$. This allows us to conclude that $(\Xi^N,(\frac{1}{\sqrt{n}}\tilde E_{k,nt}^{N,-})_{t\in[-\varepsilon,\varepsilon]})$ converges in distribution to $(\Xi,(\tilde W_t^{k})_{t\in[-\varepsilon,\varepsilon]})$ when $N$ tends to $+\infty$, which ends the proof of the claim.
 \end{proof}
 
 We are now going to write $(\frac{1}{n}\sigma_{k,-}^N,(\frac{1}{\sqrt{n}}E_{k,nt}^{N,-})_{t\in[-\varepsilon,\varepsilon]})$ as a function of $((\frac{1}{\sqrt{n}}\tilde E_{k,nt}^{N,-})_{t\in[-\varepsilon,\varepsilon]},(\frac{1}{\sqrt{n}} E_{k,nt}^{N})_{t\in[-\varepsilon,\varepsilon]})$. We define a function $F$ so that for $f_1,f_2 : [-\varepsilon,\varepsilon] \mapsto \mathds{R}$ continuous functions, $F(f_1,f_2)=(s,f_3)$ with $s=\inf\{t\in[0,\varepsilon]\,|\,f_1(t)=f_2(t)\}$ (defined to be $+\infty$ if there is no such $t$) and $f_3$ is defined by $f_3(t)=f_1(t)$ if $t \leq s$ and $f_3(t)=f_2(t)$ if $t \geq s$. For $n$ large enough, we also define functions $F_N$ so that for $f_1,f_2 : [-\varepsilon,\varepsilon] \mapsto \mathds{R}$ continuous functions, $F_N(f_1,f_2)=(s,f_3)$ with $s=\inf\{t\in[\frac{1}{n},\varepsilon]\,|\,f_1(t)=f_2(t)\}$ and $f_3$ is defined by $f_3(t)=f_1(t)$ if $t \leq s$ and $f_3(t)=f_2(t)$ if $t \geq s$. We then have $(\frac{1}{n}\sigma_{k,-}^N,(\frac{1}{\sqrt{n}}E_{k,nt}^{N,-})_{t\in[-\varepsilon,\varepsilon]})=F_N((\frac{1}{\sqrt{n}}\tilde E_{k,nt}^{N,-})_{t\in[-\varepsilon,\varepsilon]},(\frac{1}{\sqrt{n}}E_{k,nt}^{N})_{t\in[-\varepsilon,\varepsilon]})$. 
 
 We now deduce the convergence of $(\Xi^N,\frac{1}{n}\sigma_{k,-}^N,(\frac{1}{\sqrt{n}}E_{k,nt}^{N,-})_{t\in[-\varepsilon,\varepsilon]})$. By Claim \ref{claim_conv_mixed_evts}, $(\Xi^N, (\frac{1}{\sqrt{n}}\tilde E_{k,nt}^{N,-})_{t\in[-\varepsilon,\varepsilon]})$ converges in distribution to $(\Xi, (\tilde W_t^k)_{t\in[-\varepsilon,\varepsilon]})$ when $N$ tends to $+\infty$, so by the Skorohod Representation Theorem (Theorem 1.8 of Chapter 3 of \cite{Ethier_Kurtz1986}), there exists a probability space containing random variables $(\hat \Xi^N,(\frac{1}{\sqrt{n}}\hat{\tilde E}_{k,nt}^{N,-})_{t\in[-\varepsilon,\varepsilon]})$ for any $N\in\mathds{N}^*$ and $(\hat \Xi,(\hat{\tilde W}_t^k)_{t\in[-\varepsilon,\varepsilon]})$ having the respective laws of $(\Xi^N,(\frac{1}{\sqrt{n}}{\tilde E}_{k,nt}^{N,-})_{t\in[-\varepsilon,\varepsilon]})$ and $(\Xi,({\tilde W}_t^k)_{t\in[-\varepsilon,\varepsilon]})$, and so that $(\hat \Xi^N,(\frac{1}{\sqrt{n}}\hat{\tilde E}_{k,nt}^{N,-})_{t\in[-\varepsilon,\varepsilon]})$ converges almost surely to $(\hat \Xi,(\hat{\tilde W}_t^k)_{t\in[-\varepsilon,\varepsilon]})$ when $N$ tends to $+\infty$. We denote by $(\frac{1}{\sqrt{n}}\hat E_{k,nt}^{N})_{t\in[-a-\varepsilon,a+\varepsilon]}$ the last coordinate of $\hat \Xi^N$ and by $(\hat W_t^k)_{t\in[-a-\varepsilon,a+\varepsilon]}$ the last coordinate of $\hat \Xi$. We then have the following. 
 
 \begin{claim}\label{claim_conv_va_Skorohod}
 $F_N((\frac{1}{\sqrt{n}}\hat{\tilde E}_{k,nt}^{N,-})_{t\in[-\varepsilon,\varepsilon]},(\frac{1}{\sqrt{n}}\hat E_{k,nt}^{N})_{t\in[-\varepsilon,\varepsilon]})$ converges in probability to the quantity $F((\hat{\tilde W}_t^k)_{t\in[-\varepsilon,\varepsilon]},(\hat W_t^k)_{t\in[-\varepsilon,\varepsilon]})$ when $N$ tends to $+\infty$.
 \end{claim}
 
  The proof of Claim \ref{claim_conv_va_Skorohod} is detailed in the appendix. It basically comes down to proving that $F$ is almost surely continuous at the limit point $((\hat{\tilde W}_t^k)_{t\in[-\varepsilon,\varepsilon]},(\hat W_t^k)_{t\in[-\varepsilon,\varepsilon]})$. This can be proven with the help of Lemma \ref{lem_conv_abs_times}, which we are able to use thanks to Proposition \ref{prop_envts_above_barrier}.  
 
 We can now prove the convergence in distribution of $(\Xi^N,\frac{1}{n}\sigma_{k,-}^N,(\frac{1}{\sqrt{n}}E_{k,nt}^{N,-})_{t\in[-\varepsilon,\varepsilon]})$. Indeed, if $\Phi$ is a continous real bounded function accepting $(\Xi^N,\frac{1}{n}\sigma_{k,-}^N,(\frac{1}{\sqrt{n}}E_{k,nt}^{N,-})_{t\in[-\varepsilon,\varepsilon]})$ as argument, then 
 \[
  \mathds{E}\left(\Phi\left(\Xi^N,\frac{1}{n}\sigma_{k,-}^N,\left(\frac{1}{\sqrt{n}}E_{k,nt}^{N,-}\right)_{t\in[-\varepsilon,\varepsilon]}\right)\right)
  =\mathds{E}\left(\Phi\left(\Xi^N,F_N\left(\left(\frac{1}{\sqrt{n}}{\tilde E}_{k,nt}^{N,-}\right)_{t\in[-\varepsilon,\varepsilon]},\left(\frac{1}{\sqrt{n}}E_{k,nt}^{N}\right)_{t\in[-\varepsilon,\varepsilon]}\right)\right)\right) 
  \]
  \[
  =\mathds{E}\left(\Phi\left(\hat \Xi^N,F_N\left(\left(\frac{1}{\sqrt{n}}\hat{\tilde E}_{k,nt}^{N,-}\right)_{t\in[-\varepsilon,\varepsilon]},\left(\frac{1}{\sqrt{n}}\hat E_{k,nt}^{N}\right)_{t\in[-\varepsilon,\varepsilon]}\right)\right)\right), 
\]
which by Claim \ref{claim_conv_va_Skorohod} converges to $\mathds{E}(\Phi(\hat \Xi,F((\hat{\tilde W}_t^k)_{t\in[-\varepsilon,\varepsilon]},(\hat W_t^k)_{t\in[-\varepsilon,\varepsilon]})))=\mathds{E}(\Phi(\Xi,F((\tilde W_t^k)_{t\in[-\varepsilon,\varepsilon]},(W_t^k)_{t\in[-\varepsilon,\varepsilon]})))$ when $N$ tends to $+\infty$. Hence $(\Xi^N,\frac{1}{n}\sigma_{k,-}^N,(\frac{1}{\sqrt{n}}E_{k,nt}^{N,-})_{t\in[-\varepsilon,\varepsilon]})$ converges in distribution to $(\Xi,F((\tilde W_t^k)_{t\in[-\varepsilon,\varepsilon]},(W_t^k)_{t\in[-\varepsilon,\varepsilon]}))$ when $N$ tends to $+\infty$. This random variable is $(\Xi,\sigma_{k,-},(V_t^{k,-})_{t\in[-\varepsilon,\varepsilon]})$ where $V^{k,-}$ is a Brownian motion with $V_{-\varepsilon}^{k,-}=W_{-\varepsilon}^k$ reflected above $W^k$ on $W^k$ on $[-\varepsilon,0]$ and absorbed by $W^k$ on $[0,\varepsilon]$, while $\sigma_{k,-}$ is the absorption time. 
 
 This ends the study of the ``environment at the first time after $T_k$ at which the process reaches $X_{T_k}-\lfloor \varepsilon n\rfloor$''. We can define a similar process for the ``environment at the first time after $T_k$ at which the process reaches $X_{T_k}+\lfloor \varepsilon n\rfloor$'': $(E_{k,i}^{N,+})_{i\in\mathds{Z}}$ is defined by $E_{k,i}^{N,+}=E_{k,\lfloor \varepsilon n \rfloor}^N+1+\sum_{j=X_{T_k}+i}^{X_{T_k}+\lfloor \varepsilon n \rfloor-1}\zeta_j^{T_k,+,E}$ for $i<\lfloor \varepsilon n \rfloor$ and $E_{k,i}^{N,+}=E_{k,\lfloor \varepsilon n \rfloor}^N+1$ for $i\geq \lfloor \varepsilon n \rfloor$. We also define $\sigma_{k,+}^N=\sup\{i\leq0\,|\, L_{X_{T_k}+i}^{T_k,+}=0\}$. By the same arguments as before, we can prove that $(\Xi^N,\frac{1}{n}\sigma_{k,+}^N,(\frac{1}{\sqrt{n}}E_{k,nt}^{N,+})_{t\in[-\varepsilon,\varepsilon]})$ converges in distribution to a random variable $(\Xi,\sigma_{k,+},(V_t^{k,+})_{t\in[-\varepsilon,\varepsilon]})$ when $N$ tends to $+\infty$, where $V^{k,+}$ is a Brownian motion with $V_{\varepsilon}^{k,+}=W_{\varepsilon}^k$ above $W^k$ reflected on $W^k$ on $[0,\varepsilon]$ and absorbed by $W^k$ on $[-\varepsilon,0]$, while $\sigma_{k,+}$ is the absorption time.
 
 By putting the results about $(E_{k,i}^{N,-})_{i\in\mathds{Z}}$ and $(E_{k,i}^{N,+})_{i\in\mathds{Z}}$ together, we will now be able to complete the proof of Proposition \ref{prop_conv_envts}. $\bar \Xi^N$ and $\bar \Xi$ will denote the same objects as $\Xi^N$ and $\Xi$, but with $[-a,a]$ replacing $[-a-\varepsilon,a+\varepsilon]$. Let $\Psi$ be a continuous bounded function of $(\bar \Xi^N,Z_{k+1}^N,(\frac{1}{\sqrt{n}}E_{k+1,nt}^N)_{[-a,a]})$. If $\sigma_{k,-}^N\leq\lfloor \varepsilon n \rfloor$, we have $Z_{k+1}^N=Z_{k}^N-1$ and $(\frac{1}{\sqrt{n}}E_{k+1,nt}^N)_{[-a,a]}$ can be obtained as a continuous function of a deterministic modification of $\Xi^N$, $(\frac{1}{\sqrt{n}}E_{k,nt}^{N,-})_{t\in[-\varepsilon,\varepsilon]}$ and some $\frac{1}{\sqrt{n}}E_{k,i}^N$, $\frac{1}{\sqrt{n}}E_{k,i}^{N,-}$ whose convergence in distribution is implied by that of $\Xi^N$ and $(\frac{1}{\sqrt{n}}E_{k,nt}^{N,-})_{t\in[-\varepsilon,\varepsilon]}$, so in this case by an abuse of notation we write $\Psi(\bar \Xi^N,Z_{k+1}^N,(\frac{1}{\sqrt{n}}E_{k+1,nt}^N)_{[-a,a]})=\Psi_-(\Xi^N,(\frac{1}{\sqrt{n}}E_{k,nt}^{N,-})_{t\in[-\varepsilon,\varepsilon]})$ with $\Psi_-$ continuous and bounded. Similarly, if $\sigma_{k,+}^N>-\lfloor \varepsilon n \rfloor$, we write $\Psi(\bar \Xi^N,Z_{k+1}^N,(\frac{1}{\sqrt{n}}E_{k+1,nt}^N)_{[-a,a]})=\Psi_+(\Xi^N,(\frac{1}{\sqrt{n}}E_{k,nt}^{N,+})_{t\in[-\varepsilon,\varepsilon]})$ with $\Psi_+$ continuous and bounded. We then have 
 \begin{equation}\label{eq_conv_envts}
 \begin{split} \textstyle
  \mathds{E}\left(\Psi\left(\bar \Xi^N,Z_{k+1}^N,\left(\frac{1}{\sqrt{n}}E_{k+1,nt}^N\right)_{[-a,a]}\right)\right)
   = & \mathds{E}\left(\Psi_-\left(\Xi^N,\left(\frac{1}{\sqrt{n}}E_{k,nt}^{N,-}\right)_{t\in[-\varepsilon,\varepsilon]}\right)\mathds{1}_{\{\sigma_{k,-}^N\leq\lfloor \varepsilon n \rfloor\}}\right) \\
  + &\mathds{E}\left(\Psi_+\left(\Xi^N,\left(\frac{1}{\sqrt{n}}E_{k,nt}^{N,+}\right)_{t\in[-\varepsilon,\varepsilon]}\right)\mathds{1}_{\{\sigma_{k,+}^N>-\lfloor \varepsilon n \rfloor\}}\right).
   \end{split}
 \end{equation}
 We can use again the Skorohod Representation Theorem to assume the convergence in distribution of the variables $(\Xi^N,\frac{1}{n}\sigma_{k,\pm}^N,(\frac{1}{\sqrt{n}}E_{k,nt}^{N,\pm})_{t\in[-\varepsilon,\varepsilon]})$ to $(\Xi,\sigma_{k,\pm},(V_t^{k,\pm})_{t\in[-\varepsilon,\varepsilon]})$ is almost sure. Furthermore, by the definition of $\sigma_{k,-}$, the probability that $\sigma_{k,-}=\varepsilon$ is smaller than the probability that a Brownian motion starting at $V_0^{k,-}$ at time 0 is exactly at $W_\varepsilon^k$ at time $\varepsilon$, which is 0, hence $\mathds{P}(\sigma_{k,-}=\varepsilon)=0$. Similarly, $\mathds{P}(\sigma_{k,+}=-\varepsilon)=0$. Consequently, the right-hand side of (\ref{eq_conv_envts}) converges to $\mathds{E}(\Psi_-(\Xi,(V_t^{k,-})_{t\in[-\varepsilon,\varepsilon]})\mathds{1}_{\{\sigma_{k,-}< \varepsilon\}})  + \mathds{E}(\Psi_+(\Xi,(V_t^{k,+})_{t\in[-\varepsilon,\varepsilon]})\mathds{1}_{\{\sigma_{k,+}>- \varepsilon\}})$. Now, we remember the quantity $p_{k,-}=\mathds{P}(\sigma_{k,-}< \varepsilon|W^k)$ introduced in Definition \ref{def_lim_envts}. We then have $p_{k,-}=\mathds{P}(\sigma_{k,-}< \varepsilon|\Xi)$, therefore 
 \[
 \mathds{E}\left(\Psi_-\left(\Xi,(V_t^{k,-})_{t\in[-\varepsilon,\varepsilon]}\right)\mathds{1}_{\{\sigma_{k,-}< \varepsilon\}}\right)
 =\mathds{E}\left(\mathds{E}\left(\left.\Psi_-\left(\Xi,(V_t^{k,-})_{t\in[-\varepsilon,\varepsilon]}\right)\mathds{1}_{\{\sigma_{k,-}< \varepsilon\}}\right|\Xi\right)\right)
 \]
 \[
 =\mathds{E}\left(p_{k,-}\mathds{E}\left(\left.\Psi_-\left(\Xi,(V_t^{k,-})_{t\in[-\varepsilon,\varepsilon]}\right)\frac{\mathds{1}_{\{\sigma_{k,-}< \varepsilon\}}}{\mathds{P}(\sigma_{k,-}< \varepsilon|\Xi)}\right|\Xi\right)\right)
 =\mathds{E}\left(p_{k,-}\mathds{E}\left(\left.\Psi_-\left(\Xi,(\bar W_t^{k,-})_{t\in[-\varepsilon,\varepsilon]}\right)\right|\Xi\right)\right).
\]
 In the same way, $\mathds{E}(\Psi_+(\Xi,(V_t^{k,+})_{t\in[-\varepsilon,\varepsilon]})\mathds{1}_{\{\sigma_{k,+}>- \varepsilon\}})=\mathds{E}(p_{k,+}\mathds{E}(\Psi_+(\Xi,(\bar W_t^{k,+})_{t\in[-\varepsilon,\varepsilon]})|\Xi))$, where $p_{k,+}=\mathds{P}(\sigma_{k,+}>- \varepsilon|W^k)$ was also introduced in Definition \ref{def_lim_envts}. In addition, by Proposition \ref{prop_envts_above_barrier} we have $\mathds{P}(V_0^{k,-}>W_0^k)=1$, so $\mathds{P}(V_0^{k,-}>W_0^k|W^k)=1$ almost surely, therefore by Proposition \ref{prop_better_than_Z} $p_{k,-}+p_{k,+}=1$ almost surely. We deduce that when $N$ tends to $+\infty$, $\mathds{E}(\Psi(\bar \Xi^N,Z_{k+1}^N,(\frac{1}{\sqrt{n}}E_{k+1,nt}^N)_{[-a,a]}))$ converges to 
 \[
 \mathds{E}(p_{k,-}\mathds{E}(\Psi_-(\Xi,(\bar W_t^{k,-})_{t\in[-\varepsilon,\varepsilon]})|\Xi))+\mathds{E}(p_{k,+}\mathds{E}(\Psi_+(\Xi,(\bar W_t^{k,+})_{t\in[-\varepsilon,\varepsilon]})|\Xi))
 =\mathds{E}(\Psi(\bar \Xi,\breve Z_{k+1}^N,(W_t^{k+1})_{[-a,a]}))
 \]
 when $N$ tends to $+\infty$. Consequently, $(\bar \Xi^N,Z_{k+1}^N,(\frac{1}{\sqrt{n}}E_{k+1,nt}^N)_{[-a,a]})$ converges in distribution to $(\bar \Xi,\breve Z_{k+1}^N,(W_t^{k+1})_{[-a,a]})$ when $N$ tends to $+\infty$. Proposition \ref{prop_conv_envts} is thus true for $k+1$, therefore by induction it is true for all $k\in\mathds{N}$.
 \end{proof}
 
 \begin{proof}[Proof of Lemma \ref{lem_continuity}.]
  The proof is the same as in Proposition \ref{prop_conv_ZT}, except for a difference in the equivalent of Proposition \ref{prop_law_zeta}. The definition of $\mathcal{B}_0^{\lfloor \psi(N)\theta\rfloor,0,-}$ must be modified by replacing $\mathcal{B}_{0,2}^{\lfloor \psi(N)\theta\rfloor,0,-}$ by $\{\sup_{y\in\mathds{R}}|\frac{1}{\psi(N)}\ell_{T_0',\lfloor \psi(N) y\rfloor}^+-(\theta-\frac{|y|}{2})_+| \geq \theta/4\}$ (and $n^{(\alpha-1)/4}\lfloor \varepsilon n\rfloor$ by $\lfloor \theta \psi(N)/2 \rfloor$ in $\mathcal{B}_{0,1}^{\lfloor \psi(N)\theta\rfloor,0,-}$). With such a definition, $\mathcal{B}_0^{\lfloor \psi(N)\theta\rfloor,0,-}$ will contain $\{$there exists $-\lfloor \theta \psi(N)/2 \rfloor-1 \leq i \leq \lfloor \theta \psi(N)/2 \rfloor+1, \bar\Delta_{T_0',i}\neq\Delta_{T_0',i}\}$. Moreover, Theorem 1 of \cite{Toth_et_al2008} yields that $\sup_{y\in\mathds{R}}|\frac{1}{\psi(N)}\ell_{T_0',\lfloor \psi(N) y\rfloor}^+-(\theta-\frac{|y|}{2})_+|$ converges in probability to 0 when $N$ tends to $+\infty$, so $\mathds{P}(\mathcal{B}_{0,2}^{\lfloor \psi(N)\theta\rfloor,0,-})$ tends to 0 when $N$ tends to $+\infty$, so $\mathds{P}(\mathcal{B}_{0}^{\lfloor \psi(N)\theta\rfloor,0,-})$ tends to 0 when $N$ tends to $+\infty$.
 \end{proof}

\section*{Acknowledgements}

Laure Marêché was partially supported by the University of Strasbourg Initiative of Excellence. Thomas Mountford was partially supported by the Swiss National Science Foundation, grant FNS 200021L 169691.

\section*{Appendix}

In this appendix we give the proofs that were not included in the main body of the article so not to slow down the reader. The first proof is needed in the proof of Lemma \ref{lem_zeta_E_and_I}. The definitions of the various quantities are given in the part of the proof in the main body of the article.

\begin{proof}[Exponential bounds on the terms in \eqref{eq_B3}.]
 Firstly, when $n$ is large enough we have $\mathds{P}((\mathcal{B}_0^{m,i,\iota})^c \cap \mathcal{A}_1) 
 \leq \mathds{P}(|\bar\Delta_{\bar m,j}|>\frac{1-2p_w}{2}\lfloor(\ln n)^2/4\rfloor)$, 
 which is smaller than $C_{3,1}e^{-c_{3,1}(\ln n)^2}$ with $C_{3,1}=C_{3,1}(w) < +\infty$ and $c_{3,1}=c_{3,1}(w) > 0$ 
 since $\bar\Delta_{\bar m,j}$ has law $\rho_-$ or $\rho_+$, which have exponential tails. 
 
 We now bound 
 \begin{align*}
 \mathds{P}(\mathcal{A}_1^c \cap \mathcal{A}_2) 
 \leq & \mathds{P}(\mathcal{A}_2||\xi(0)| < i_w)\mathds{P}(|\xi(0)| < i_w) \\
 & + \mathds{P}\left(\mathcal{A}_2\left|i_w < \xi(0) \leq \frac{1-2p_w}{2}\lfloor(\ln n)^2/4\rfloor\right.\right)
 \mathds{P}\left(i_w < \xi(0) \leq \frac{1-2p_w}{2}\lfloor(\ln n)^2/4\rfloor\right) \\
 & + \mathds{P}\left(\mathcal{A}_2\left|i_w < -\xi(0) \leq \frac{1-2p_w}{2}\lfloor(\ln n)^2/4\rfloor\right.\right)
 \mathds{P}\left(i_w < -\xi(0) \leq \frac{1-2p_w}{2}\lfloor(\ln n)^2/4\rfloor\right),
 \end{align*}
 assuming $n$ is large enough (if $i_w=0$, we simply remove the first term). For $i_w > 0$, we notice that if $|\xi(0)| < i_w$, for any 
 $\ell \in \mathds{N}$, $\mathds{P}(E_0 \leq (2 i_w-1)(\ell+1)|E_0 > (2 i_w-1)\ell) \geq (1/2)^{2 i_w-1}$, 
 as it is at least the probability that $\xi((2 i_w-1)\ell+1)=\xi((2 i_w-1)\ell)+1$, $\xi((2 i_w-1)\ell+2)=\xi((2 i_w-1)\ell)+2$,\dots 
 until the chain reaches $i_w$. This yields $\mathds{P}(\mathcal{A}_2||\xi(0)| < i_w) 
 \leq (1-(1/2)^{2 i_w-1})^{\lfloor(\ln n)^2/(4(2 i_w-1))\rfloor} \leq C_{3,2}e^{-c_{3,2}(\ln n)^2}$ 
 with $C_{3,2}=C_{3,2}(w) < +\infty$ and $c_{3,2}=c_{3,2}(w) > 0$. We now consider 
 $\mathds{P}(\mathcal{A}_2|i_w < \xi(0) \leq \frac{1-2p_w}{2}\lfloor(\ln n)^2/4\rfloor)$. 
 We notice that it is possible to construct 
 i.i.d. random variables $(\breve \xi_\ell)_{\ell \in \mathds{N}}$ such that $\mathds{P}(\breve\xi_0=1)=1-\mathds{P}(\breve\xi_0=-1)=p_w$ 
 (we denote this law $R(p_w)$ for short) and if 
 $i_w < \xi(0) \leq \frac{1-2p_w}{2}\lfloor(\ln n)^2/4\rfloor$, for all $0 \leq \ell < E_0$, $\xi(\ell+1)-\xi(\ell) \leq \breve \xi_\ell$. 
 We then have $\mathds{P}(\mathcal{A}_2|i_w < \xi(0) \leq \frac{1-2p_w}{2}\lfloor(\ln n)^2/4\rfloor) \leq 
 \mathds{P}(\sum_{\ell=0}^{\lfloor(\ln n)^2/4\rfloor-1}\breve\xi_\ell > -\frac{1-2p_w}{2}\lfloor(\ln n)^2/4\rfloor) 
 = \mathds{P}(\sum_{\ell=0}^{\lfloor(\ln n)^2/4\rfloor-1}(\breve\xi_\ell-\mathds{E}(\breve\xi_\ell))>
 \frac{1-2p_w}{2}\lfloor(\ln n)^2/4\rfloor) \leq \exp(-\frac{(1-2p_w)^2}{8}\lfloor(\ln n)^2/4\rfloor)$ by the Hoeffding inequality. 
 This yields the existence of $C_{3,2}'=C_{3,2}'(w) < +\infty$ and $c_{3,2}'=c_{3,2}'(w) > 0$ such that 
 $\mathds{P}(\mathcal{A}_2|i_w < \xi(0) \leq \frac{1-2p_w}{2}\lfloor(\ln n)^2/4\rfloor) \leq C_{3,2}'e^{-c_{3,2}'(\ln n)^2}$. 
 In addition, $\xi$ is symmetric, hence $\mathds{P}(\mathcal{A}_2|i_w < -\xi(0) \leq \frac{1-2p_w}{2}\lfloor(\ln n)^2/4\rfloor) = 
 \mathds{P}(\mathcal{A}_2|i_w < \xi(0) \leq \frac{1-2p_w}{2}\lfloor(\ln n)^2/4\rfloor)$. We deduce 
 $\mathds{P}(\mathcal{A}_1^c \cap \mathcal{A}_2) \leq (C_{3,2} \vee C_{3,2}')e^{-(c_{3,2} \wedge c_{3,2}')(\ln n)^2}$.
 
 We now deal with $\mathds{P}(\mathcal{A}_3)$. We will begin by finding a constant $c_{3,3}=c_{3,3}(w) > 0$ 
 so that for any $\ell \in \mathds{N}$, we have $\mathds{E}(e^{c_{3,3}(E_{\ell+1}-E_\ell)}) < +\infty$. 
 Since $\xi$ is symmetric, the $E_{\ell+1}-E_\ell$, $\ell \in \mathds{N}$ are i.i.d., so we study $E_1-E_0$. 
 By the argument that allowed us to bound $\mathds{P}(\mathcal{A}_2||\xi(0)| < i_w)$, we can see 
 that if $i_w > 0$, for any $\ell \in \mathds{N}^*$, $\mathds{P}(E_1-E_0 > \ell||\xi(E_0+1)|<i_w) 
 \leq (1-(1/2)^{2 i_w-1})^{\lfloor(\ell-1)/(2 i_w-1)\rfloor}$. Moreover, we also have 
 $\mathds{P}(E_1-E_0 > \ell|\xi(E_0+1)=i_w+1)=\mathds{P}(E_1-E_0 > \ell|\xi(E_0+1)=-i_w-1) \leq 
 \mathds{P}(\sum_{\ell'=1}^{\ell-1}\breve\xi_{\ell'} \geq 0)$ where the $\breve\xi_{\ell'}$, $\ell' \in \mathds{N}$ 
 are i.i.d. with law $R(p_w)$, and we can use the Hoeffding inequality to bound the last probability by 
 $e^{-(1-2p_w)^2(\ell-1)/2}$. We deduce that there exist constants $C_{3,3}=C_{3,3}(w) < +\infty$ and 
 $c_{3,3}=c_{3,3}(w) > 0$ such that $\mathds{P}(E_1-E_0 > \ell) \leq C_{3,3}e^{-2c_{3,3}\ell}$, 
 which implies $\mathds{E}(e^{c_{3,3}(E_1-E_0)}) < +\infty$. Consequently, we have $\mathds{P}(\mathcal{A}_3) 
 = \mathds{P}(\sum_{\ell=0}^{\lfloor \bar c_3 (\ln n)^2\rfloor-1}(E_{\ell+1}-E_\ell)> (\ln n)^2/4) 
 \leq e^{-c_{3,3}(\ln n)^2/4}\mathds{E}(\exp(\sum_{\ell=0}^{\lfloor \bar c_3 (\ln n)^2\rfloor-1}c_{3,3}(E_{\ell+1}-E_\ell))) 
 = e^{-c_{3,3}(\ln n)^2/4}\mathds{E}(e^{c_{3,3}(E_1-E_0)})^{\lfloor \bar c_3 (\ln n)^2\rfloor}$. 
 If we choose $\bar c_3 = \frac{c_{3,3}}{8 \ln(\mathds{E}(e^{c_{3,3}(E_1-E_0)}))}$, $\bar c_{3}$ is positive, 
 depends only on $w$, and satisfies $\mathds{P}(\mathcal{A}_3) \leq e^{-c_{3,3}(\ln n)^2/8}$. 
 
 Finally, we bound $\mathds{P}(\mathcal{A}_4)$. In order to do that, we notice that the arguments used for $\mathcal{A}_3$ 
 yield in particular that for any $\ell \in \mathds{N}$, $E_{\ell+1}-E_\ell$ is finite a.s. In addition, 
 similar arguments can be used to show that $E_0$ is finite a.s., so for any $\ell \in \mathds{N}$, $E_{\ell}$ is finite a.s.
 Moreover for any $\ell \in \mathds{N}^*$, 
 if $i_w > 0$, $\mathds{P}(T' > E_{2\ell}) = \mathds{P}(T' > E_{2\ell},\xi(E_{2(\ell-1)})=-i_w,T' > E_{2(\ell-1)}) 
 + \mathds{P}(T' > E_{2\ell},\xi(E_{2(\ell-1)})=i_w,T' > E_{2(\ell-1)})$. Furthermore, we have 
 $\mathds{P}(T' \leq E_{2\ell},\xi(E_{2(\ell-1)})=-i_w,T' > E_{2(\ell-1)}) \geq 
 \mathds{P}(T' > E_{2(\ell-1)},\xi(E_{2(\ell-1)})=-i_w,\xi(E_{2(\ell-1)}+1)=-i_w+1,\dots,\xi(E_{2(\ell-1)}+i_w)=0,
 \xi(E_{2(\ell-1)}+i_w+1)=1,\xi(E_{2(\ell-1)}+i_w+2)=0)=(1/2)^{i_w+2}\mathds{P}(\xi(E_{2(\ell-1)})=-i_w,T' > E_{2(\ell-1)})$. 
 Similarly, $\mathds{P}(T' \leq E_{2\ell},\xi(E_{2(\ell-1)})=i_w,T' > E_{2(\ell-1)}) 
 \geq (1/2)^{i_w}\mathds{P}(\xi(E_{2(\ell-1)})=i_w,T' > E_{2(\ell-1)})$. We deduce $\mathds{P}(T' > E_{2\ell}) \leq 
 (1-(1/2)^{i_w+2})\mathds{P}(T' > E_{2(\ell-1)})$. The same argument yields that if $i_w=0$, 
 $\mathds{P}(T' \leq E_{2\ell},T' > E_{2(\ell-1)}) \geq ((1-p_w)/2)\mathds{P}(T' > E_{2(\ell-1)})$, 
 so $\mathds{P}(T' > E_{2\ell}) \leq (1-(1-p_w)/2)\mathds{P}(T' > E_{2(\ell-1)})$. Therefore in both cases 
 there exists a constant $c_{3,4}=c_{3,4}(w)>0$ such that $\mathds{P}(T' > E_{2\ell}) 
 \leq e^{-c_{3,4}}\mathds{P}(T' > E_{2(\ell-1)})$. We deduce that 
 $\mathds{P}(\mathcal{A}_4) \leq e^{-c_{3,4}\lfloor\lfloor \bar c_3 (\ln n)^2\rfloor/2\rfloor}$, 
 so $\mathds{P}(\mathcal{A}_4) \leq e^{3c_{3,4}/2}e^{-c_{3,4} \bar c_3 (\ln n)^2/2}$, which ends the proof of the exponential bounds on the terms in \eqref{eq_B3}. 
\end{proof}

We now detail the proof of Proposition \ref{prop_excursions}, which states the existence of a constant $c_4=c_4(w,\varepsilon)>0$ such that when $n$ is large enough, $\mathds{P}(\mathcal{B}_0^c \cap \mathcal{B}_1^c \cap \mathcal{B}_3^c \cap \mathcal{B}_4) \leq e^{-c_4(\ln n)^2}$. 

\begin{proof}[Proof of Proposition \ref{prop_excursions}.]
Let $\lfloor N \theta\rfloor-2n^{(\alpha+4)/5} \leq m \leq \lfloor N \theta\rfloor+2n^{(\alpha+4)/5}$, $\lfloor Nx \rfloor-n^{(\alpha+4)/5} \leq i 
\leq \lfloor Nx \rfloor+n^{(\alpha+4)/5}$, $\iota \in \{+,-\}$. 
We denote $\bar m = \mathbf{T}_{m,i}^\iota$ and $\bar i = X_{\mathbf{T}_{m,i}^\iota}$. It is enough to find 
a constant $\tilde c_4=\tilde c_4(w,\varepsilon) > 0$ 
such that when $n$ is large enough, $\mathds{P}((\mathcal{B}_0^{m,i,\iota})^c \cap (\mathcal{B}_{\bar m,1}^-)^c
\cap (\mathcal{B}_{\bar m,3}^-)^c \cap \mathcal{B}_{\bar m,4}^-) \leq e^{-\tilde c_4 (\ln n)^2}$ and 
$\mathds{P}((\mathcal{B}_0^{m,i,\iota})^c \cap (\mathcal{B}_{\bar m,1}^+)^c
\cap (\mathcal{B}_{\bar m,3}^+)^c \cap \mathcal{B}_{\bar m,4}^+) \leq e^{-\tilde c_4 (\ln n)^2}$. 
We will write the proof only for the $-$ case, as the $+$ case can be dealt with in the same way. 
The idea will be to look at $L^{\bar m,-}$ as a random walk and 
to show that there are at most $(\ln n)^2\sqrt{n}$ ``excursions of $L^{\bar m,-}$ below $(\ln n)^3$'', 
each one having length at most $(\ln n)^8$ since $(\mathcal{B}_{\bar m,1}^-)^c$ occurs. We will achieve it 
by noticing that the ``excursions of $L^{\bar m,-}$ below $(\ln n)^3$'' are between ``excursions above $(\ln n)^3$''. To control those, we see that by Observation \ref{obs_rec_temps_local} we have 
$L_{j+1}^{\bar m,-}=L_j^{\bar m,-}+\zeta^{\bar m,-,E}_j-\zeta^{\bar m,-,B}_j$ and since 
$(\mathcal{B}_0^{m,i,\iota})^c \cap (\mathcal{B}_{\bar m,3}^-)^c$ occurs, 
when $L_j^{\bar m,-} \geq (\ln n)^3$, $\zeta^{\bar m,-,E}_j=\zeta^{\bar m,-,I}_j$. We deduce  $L_{j+1}^{\bar m,-}-L_j^{\bar m,-}=\zeta^{\bar m,-,I}_j-\zeta^{\bar m,-,B}_j$, hence $L_j^{\bar m,-}$ will roughly be a random walk with i.i.d. increments. Therefore each ``excursion of $L^{\bar m,-}$ above $(\ln n)^3$'' has probability roughly $\frac{1}{\sqrt{n}}$ to have length at least $n$, thus to be the last ``excursion'' we see as we only consider an interval of size $\varepsilon n$. Consequently, we will not see more than $(\ln n)^2\sqrt{n}$ ``excursions of $L^{\bar m,-}$ above $(\ln n)^3$'', which is enough.

We define $J(0)=\bar i - \lfloor \varepsilon n\rfloor$, and for all $\ell \in \mathds{N}$, 
$J(\ell+1)=\inf\{J(\ell)<j < \bar i \,|\,\bar L^{\bar m,-}_j \geq (\ln n)^3,\exists\, j' \in \{J(\ell)+1,\dots,j-1\},
\bar L^{\bar m,-}_{j'} < (\ln n)^3\}$. 
If $(\mathcal{B}_0^{m,i,\iota})^c \cap (\mathcal{B}_{\bar m,1}^-)^c$ occurs and $n$ is large enough, 
$\bar L_{j}^{\bar m,-} = L_{j}^{\bar m,-}$ for any $\bar i -\lfloor\varepsilon n\rfloor \leq j \leq \bar i$, 
so for any $\ell \in \mathds{N}$, 
$|\{J(\ell) < j < J(\ell+1) \wedge \bar i \,|\, \bar L^{\bar m,-}_{j} < (\ln n)^3\}| \leq \lfloor (\ln n)^8 \rfloor$. 
Therefore if $(\mathcal{B}_0^{m,i,\iota})^c \cap (\mathcal{B}_{\bar m,1}^-)^c \cap \mathcal{B}_{\bar m,4}^-$ 
occurs and $n$ is large enough, 
$J(\lfloor(\ln n)^2\sqrt{n}\rfloor) < \bar i$. For any $1 \leq \ell < \lfloor(\ln n)^2\sqrt{n}\rfloor$ such that $J(\ell+1) < \bar i$, 
we set $\mathcal{D}_\ell = \{\exists\, j \in \{J(\ell),\dots,J(\ell+1)\} \,|\, 
\sum_{j'=J(\ell)}^{j-1}(\bar\zeta_{j'}^{\bar m ,-,I}+\bar\Delta_{\bar m,j'}+1/2) < 0\}$. 

We are going to show that when $n$ is large enough, for any $1 \leq \ell < \lfloor(\ln n)^2\sqrt{n}\rfloor$, if 
$(\mathcal{B}_0^{m,i,\iota})^c \cap (\mathcal{B}_{\bar m,3}^-)^c$ 
occurs and $J(\ell),J(\ell+1) < \bar i$, 
$\mathcal{D}_\ell$ occurs. Let $1 \leq \ell < \lfloor(\ln n)^2\sqrt{n}\rfloor$. Let $j_0$ be the smallest 
$j \in \{J(\ell)+1,\dots,J(\ell+1)-1\}$ such that $\bar L_j^{\bar m,-} < (\ln n)^3$. 
One can prove by induction on $j \in \{J(\ell),\dots,j_0\}$ that 
$\bar L_j^{\bar m,-}=\bar L_{J(\ell)}^{\bar m,-}+\sum_{j'=J(\ell)}^{j-1}(\bar\zeta_{j'}^{\bar m ,-,I}+\bar\Delta_{\bar m,j'}+1/2)$. 
Indeed, it is true for $j=J(\ell)$; now suppose 
we have $\bar L_j^{\bar m,-}=\bar L_{J(\ell)}^{\bar m,-}+\sum_{j'=J(\ell)}^{j-1}(\bar\zeta_{j'}^{\bar m ,-,I}+\bar\Delta_{\bar m,j'}+1/2)$ 
for some $j \in \{J(\ell),\dots,j_0-1\}$. Since $(\mathcal{B}_0^{m,i,\iota})^c$ occurs and $n$ is large enough, 
$L_j^{\bar m,-} = \bar L_j^{\bar m,-} \geq (\ln n)^3$. Moreover, we notice that $L_{j+1}^{\bar m,-}=
L_j^{\bar m,-}+\zeta^{\bar m,-,E}_j-\zeta^{\bar m,-,B}_j=L_j^{\bar m,-}+\zeta^{\bar m,-,I}_j+\Delta_{\bar m,j}+1/2$ 
since $L_j^{\bar m,-} \geq (\ln n)^3$ and $(\mathcal{B}_{\bar m,3}^-)^c$ occurs. Since $(\mathcal{B}_0^{m,i,\iota})^c$ 
occurs and $n$ is large enough, we deduce $\bar L_{j+1}^{\bar m,-}=\bar L_{J(\ell)}^{\bar m,-}
+\sum_{j'=J(\ell)}^{j}(\bar\zeta_{j'}^{\bar m ,-,I}+\bar\Delta_{\bar m,j'}+1/2)$. 
Therefore $\bar L_{j_0}^{\bar m,-}=\bar L_{J(\ell)}^{\bar m,-}
+\sum_{j'=J(\ell)}^{j_0-1}(\bar\zeta_{j'}^{\bar m ,-,I}+\bar\Delta_{\bar m,j'}+1/2)$. 
In addition, $\bar L_{J(\ell)}^{\bar m,-} \geq (\ln n)^3$ and $\bar L_{j_0}^{\bar m,-} < (\ln n)^3$, so 
$\sum_{j'=J(\ell)}^{j_0-1}(\bar\zeta_{j'}^{\bar m ,-,I}+\bar\Delta_{\bar m,j'}+1/2) < 0$, hence $\mathcal{D}_\ell$ is satisfied. 

As a consequence, if $n$ is large enough and $(\mathcal{B}_0^{m,i,\iota})^c \cap (\mathcal{B}_{\bar m,1}^-)^c 
\cap (\mathcal{B}_{\bar m,3}^-)^c \cap \mathcal{B}_{\bar m,4}^-$ occurs, 
for all $1 \leq \ell < \lfloor(\ln n)^2\sqrt{n}\rfloor$ we have that $\mathcal{D}_\ell$ 
occurs. Therefore, when $n$ is large enough, $\mathds{P}((\mathcal{B}_0^{m,i,\iota})^c \cap (\mathcal{B}_{\bar m,1}^-)^c 
\cap (\mathcal{B}_{\bar m,3}^-)^c \cap \mathcal{B}_{\bar m,4}^-) \leq \mathds{P}(\{J(\lfloor(\ln n)^2\sqrt{n}\rfloor) < \bar i\}
\cap\bigcap_{\ell=1}^{\lfloor(\ln n)^2\sqrt{n}\rfloor-1}\mathcal{D}_\ell)$, so it is enough to bound the latter probability. 
To do that, for any $\bar i - \lfloor \varepsilon n\rfloor < j < \bar i$, 
we define $\mathcal{G}_j=\sigma(\bar\zeta_{j'}^{\bar m ,-,I},\bar\Delta_{\bar m,j'},j' < j;\bar L^{\bar m,-}_{j'},j' \leq j)$. 
We stress that these are not the same 
$\mathcal{G}_j$ as in the proof of Proposition \ref{prop_TCL}, though they are ``morally'' the same thing: 
the $\sigma$-algebra of what happens at the left of $j$. 
We also set $\hat\zeta_j = \bar\zeta_{j_2}^{\bar m ,-,I}+\bar\Delta_{\bar m,j_2}+1/2$. 
To have more convenient notation, we also introduce i.i.d. random variables $\hat\zeta_j$, $j \geq \bar i$, with law that 
of the sum of two independent random variables of law $\rho_0$ (we call this law $\rho_0^{*2}$), independent of everything else. 
The $\hat\zeta_j$, $j > \bar i - \lfloor \varepsilon n\rfloor$, are then i.i.d. random variables with law $\rho_0^{*2}$. 
Furthermore, for $\bar i - \lfloor \varepsilon n\rfloor < j < \bar i$, we define $\mathcal{U}_j = 
\{\forall\, j_1 \in \{j+1,\dots,j+\lfloor\varepsilon n\rfloor\} \,|\, \sum_{j_2=j}^{j_1-1}\hat\zeta_j \geq 0\}$. 
 We will prove that there exists a constant $\bar c_4=\bar c_4(w,\varepsilon) > 0$ such that 
 when $n$ is large enough for any $\bar i -\lfloor \varepsilon n \rfloor < j < \bar i$, $\mathds{P}(\mathcal{U}_{j}|\mathcal{G}_j) 
 \geq \frac{\bar c_4}{\sqrt{n}}$. Indeed, if this is true, we have the following when $n$ is large enough 
 (the third inequality is due to an induction, the fourth one to the fact that $1-x \leq e^{-x}$ for any $x \geq 0$):
 \[
  \mathds{P}\left(\{J(\lfloor(\ln n)^2\sqrt{n}\rfloor) < \bar i\}
\cap\bigcap_{\ell=1}^{\lfloor(\ln n)^2\sqrt{n}\rfloor-1}\mathcal{D}_\ell\right)
 \leq \mathds{E}\left( \mathds{1}_{\mathcal{U}_{J(\lfloor(\ln n)^2\sqrt{n}\rfloor-1)}^c}
 \mathds{1}_{\{J(\lfloor(\ln n)^2\sqrt{n}\rfloor-1) < \bar i\}}
 \prod_{\ell=1}^{\lfloor(\ln n)^2\sqrt{n}\rfloor-2}\mathds{1}_{\mathcal{D}_\ell}\right)
 \]
 \[
  \leq \left(1-\frac{\bar c_4}{\sqrt{n}}\right)\mathds{E}\left(\mathds{1}_{\{J(\lfloor(\ln n)^2\sqrt{n}\rfloor-1) < \bar i\}}
  \prod_{\ell=1}^{\lfloor(\ln n)^2\sqrt{n}\rfloor-2}\mathds{1}_{\mathcal{D}_\ell}\right)
  \leq \left(1-\frac{\bar c_4}{\sqrt{n}}\right)^{\lfloor(\ln n)^2\sqrt{n}\rfloor-1}
  \leq e^{-\frac{\bar c_4}{2\sqrt{n}}(\ln n)^2\sqrt{n}} = e^{-\frac{\bar c_4}{2}(\ln n)^2},
 \]
 which is enough. 
 
 Consequently, we only have to show that there exists a constant $\bar c_4=\bar c_4(w,\varepsilon) > 0$ such that 
 when $n$ is large enough, for any $\bar i -\lfloor \varepsilon n \rfloor < j < \bar i$, $\mathds{P}(\mathcal{U}_{j}|\mathcal{G}_j) 
 \geq \frac{\bar c_4}{\sqrt{n}}$. For any $\bar i -\lfloor \varepsilon n \rfloor < j < \bar i$, the $\hat \zeta_{j'}$, $j' \geq j$ 
 are independent from $\mathcal{G}_j$, so what we have to prove is that $\mathds{P}(\mathcal{U})\geq \frac{\bar c_4}{\sqrt{n}}$ where 
 $\mathcal{U}$ is defined as follows: $\breve\zeta_j$, $j \in \mathds{N}$ are i.i.d. random variables 
 with law $\rho_0^{*2}$ and $\mathcal{U} = 
 \{\forall\, j \in \{1,\dots,\lfloor\varepsilon n\rfloor\} \,|\, \sum_{j'=0}^{j-1}\breve\zeta_j \geq 0\}$. 

 In order to do that, we denote $\breve\tau=\inf\{j \geq 1 \,|\,\sum_{j'=0}^{j-1}\breve\zeta_j \geq \sqrt{\varepsilon n}\}$ 
 and $\breve\tau'=\inf\{j \geq 1 \,|\,\sum_{j'=0}^{j-1}\breve\zeta_{j'} \leq 0\}$. We notice that 
 $\mathds{P}(\mathcal{U}) \geq \mathds{P}(\{\breve\zeta_0 > 0\} \cap \{\breve\tau < \breve\tau'\} \cap \mathcal{U}')$, 
 where $\mathcal{U}'=\{\forall\, j \in \{0,\dots,\lfloor\varepsilon n\rfloor-1\},
 \sum_{j'=\breve\tau \wedge \breve\tau'}^{\breve\tau \wedge \breve\tau'+j}\breve\zeta_{j'} \geq -\sqrt{\varepsilon n}\}$. 
 For any $j \in \mathds{N}^*$, we denote $\mathcal{G}_j'=\sigma(\breve\zeta_{j'},j'<j)$. 
 We then have $\mathds{P}(\mathcal{U}) \geq \mathds{E}(\mathds{1}_{\{\breve\zeta_0 > 0\} \cap \{\breve\tau < \breve\tau'\}}
  \mathds{P}(\mathcal{U}'|\mathcal{G}_{\breve\tau \wedge \breve\tau'}'))$. In addition, if we denote 
  $\mathcal{U}''=\{\forall\, j \in \{0,\dots,\lfloor\varepsilon n\rfloor-1\}, \sum_{j'=0}^{j}\breve\zeta_{j'} \geq -\sqrt{\varepsilon n}\}$, 
  we have $\mathds{P}(\mathcal{U}'|\mathcal{G}_{\breve\tau \wedge \breve\tau'}')) = \mathds{P}(\mathcal{U}'')$, 
  thus $\mathds{P}(\mathcal{U}) \geq \mathds{P}(\{\breve\zeta_0 > 0\} \cap \{\breve\tau < \breve\tau'\})
  \mathds{P}(\mathcal{U}'')$. 
  
  We begin by dealing with $\mathds{P}(\{\breve\zeta_0 > 0\} \cap \{\breve\tau < \breve\tau'\})$. 
  We can write $\mathds{P}(\{\breve\zeta_0 > 0\} \cap \{\breve\tau < \breve\tau'\}) = \sum_{\ell=1}^{+\infty}\mathds{P}(\breve\tau < \breve\tau'|\breve\zeta_0 = \ell)\mathds{P}(\breve\zeta_0 = \ell)$. For $\ell \geq \sqrt{\varepsilon n}$, $\mathds{P}(\breve\tau < \breve\tau'|\breve\zeta_0 = \ell)=1$. 
 Set $1 \leq \ell < \sqrt{\varepsilon n}$. We notice that since the $\breve\zeta_j$ are i.i.d. with law $\rho_0^{*2}$ 
 and $\rho_0$ has exponential tails, the $\breve\zeta_j$ have expectation 0 and finite nonzero variance. 
 Therefore, by the classical gambler's ruin result for i.i.d. random walks (see for example Theorem 5.1.7 of 
 \cite{Lawler_Limic_RW}), there exists a constant $\breve c_4=\breve c_4(w,\varepsilon) > 0$ such that when $n$ is large enough, 
 $\mathds{P}(\breve\tau < \breve\tau'|\breve\zeta_0 = \ell) \geq \frac{\breve c_4}{\sqrt{n}}$. 
 We deduce that when $n$ is large enough, $\mathds{P}(\{\breve\zeta_0 > 0\} \cap \{\breve\tau < \breve\tau'\}) 
 \geq \sum_{\ell=1}^{+\infty}\frac{\breve c_4}{\sqrt{n}}\mathds{P}(\breve\zeta_0 = \ell) 
 = \frac{\breve c_4}{\sqrt{n}}\mathds{P}(\breve\zeta_0 > 0)$, where $\breve c_4\mathds{P}(\breve\zeta_0 > 0)$ 
 is positive and depends only on $w$ and $\varepsilon$. We now deal with $\mathds{P}(\mathcal{U}'')$. 
 Let $\breve\tau''=\inf\{j \geq 1 \,|\,\sum_{j'=0}^{j-1}\breve\zeta_j \leq -\sqrt{\varepsilon n}\}$, 
 we then have $\mathds{P}(\mathcal{U}'') \geq \mathds{P}(\breve\tau''>\lfloor\varepsilon n\rfloor)$. 
 Now, theorem 5.1.7 of \cite{Lawler_Limic_RW} yields that there exists a constant $\breve c_4'=\breve c_4'(w)>0$ 
 such that when $n$ is large enough, $\mathds{P}(\breve\tau''>\lfloor\varepsilon n\rfloor) \geq 
 \breve c_4'\frac{\sqrt{\varepsilon n}}{\sqrt{\lfloor\varepsilon n\rfloor}} \geq \breve c_4'$. 
 We conclude that when $n$ is large enough, 
 $\mathds{P}(\mathcal{U}) \geq \frac{\breve c_4}{\sqrt{n}}\mathds{P}(\breve\zeta_0 > 0)\breve c_4'$, 
 with $\breve c_4 \breve c_4' \mathds{P}(\breve\zeta_0 > 0)$ positive depending only on $w$ and $\varepsilon$, 
 which ends the proof. 
\end{proof}

We now prove Lemma \ref{lem_no_atoms}, which states that for all $k\in\mathds{N}^*$, the random variables $\breve T_k$ and $\sum_{k'=1}^k \breve T_{k'}$ (defined in Definition \ref{def_lim_envts}) have no atoms.
 
\begin{proof}[Proof of Lemma \ref{lem_no_atoms}.]
 Set $k\in\mathds{N}$ and let us prove that $\sum_{k'=1}^{k+1} \breve T_{k'}$ has no atoms (the argument for $\breve T_{k+1}$ is similar). By Definition \ref{def_lim_envts} it is enough to show that for any $a \geq 0$, we have $\mathds{P}(2\int_{-\varepsilon}^\varepsilon(\bar W_t^{k,-}-W_t^k)\mathrm{d}t+\sum_{k'=1}^{k} \breve T_{k'}=a)=0$ and $\mathds{P}(2\int_{-\varepsilon}^\varepsilon(\bar W_t^{k,+}-W_t^k)\mathrm{d}t+\sum_{k'=1}^{k} \breve T_{k'}=a)=0$ almost-surely. We write the proof for $\bar W_t^{k,-}$, as the proof for $\bar W_t^{k,+}$ is similar. Let $a \geq 0$. By Definition \ref{def_lim_envts}, it is enough to prove that $\mathds{P}(2\int_{-\varepsilon}^\varepsilon(V_t^{k,-}-W_t^k)\mathrm{d}t+\sum_{k'=1}^{k} \breve T_{k'}=a)=0$. We set $\delta>0$. It is enough to prove that $\mathds{P}(2\int_{-\varepsilon}^\varepsilon(V_t^{k,-}-W_t^k)\mathrm{d}t+\sum_{k'=1}^{k} \breve T_{k'}=a) \leq \delta$. 

In order to do that, we need to introduce a new process. We recall that $(V_t^{k,-})_{t\in[-\varepsilon,\varepsilon]}$ is a Brownian motion reflected on $W^k$ above $W^k$ on $[-\varepsilon,0]$ and absorbed by $W^k$ on $[0,\varepsilon]$ with $V_{-\varepsilon}^{k,-}=W_{-\varepsilon}^{k}$ (see Definition \ref{def_lim_envts}). We consider the process $(\tilde W_t^{k})_{t\in[-\varepsilon,\varepsilon]}$ that is ``the same Brownian motion reflected on $W^k$ on $[-\varepsilon,0]$, but free on $[0,\varepsilon]$'': if $(V_t^{k,-})_{t\in[-\varepsilon,\varepsilon]}$ was constructed as the reflection and absorption of a given Brownian motion, $(\tilde W_t^{k})_{t\in[-\varepsilon,\varepsilon]}$ is the latter Brownian motion reflected on $W^k$ above $W^k$ on $[-\varepsilon,0]$ and free on $[0,\varepsilon]$, with $\tilde W_{-\varepsilon}^{k}=W_{-\varepsilon}^{k}$. In Definition \ref{def_lim_envts} we also denoted $\sigma_{k,-}=\inf\{t \geq 0\,|\,V_t^{k,-}=W_t^k\}$ the absorption time of $(V_t^{k,-})_{t\in[-\varepsilon,\varepsilon]}$; for any $t \leq \sigma_{k,-}$ we have $V_t^{k,-}=\tilde W_t^{k}$. By Proposition \ref{prop_envts_above_barrier}, $\mathds{P}(V_0^{k,-}>W_0^k)=1$, and the processes $V^{k,-}$ and $W^k$ are continuous, so there exists $\bar\varepsilon>0$ so that $\mathds{P}(\sigma_{k,-}<\bar\varepsilon)\leq \delta$. This implies $\mathds{P}(2\int_{-\varepsilon}^\varepsilon(V_t^{k,-}-W_t^k)\mathrm{d}t+\sum_{k'=1}^{k} \breve T_{k'}=a) \leq \mathds{P}(\int_{-\varepsilon}^\varepsilon(V_t^{k,-}-W_t^k)\mathrm{d}t+\frac{1}{2}\sum_{k'=1}^{k} \breve T_{k'}=\frac{a}{2},\sigma_{k,-}\geq\bar\varepsilon)+\mathds{P}(\sigma_{k,-}<\bar\varepsilon)\leq \mathds{P}(\int_{0}^{\bar\varepsilon}(\tilde W_t^{k}-W_t^k)\mathrm{d}t=\frac{a}{2}-\int_{-\varepsilon}^0(\tilde W_t^{k}-W_t^k)\mathrm{d}t-\int_{\bar\varepsilon}^{\sigma_{k,-}}(\tilde W_t^{k}-W_t^k)\mathrm{d}t-\frac{1}{2}\sum_{k'=1}^{k} \breve T_{k'})+\delta$. Consequently, we only have to prove that $\mathds{P}(\int_{0}^{\bar\varepsilon}(\tilde W_t^{k}-W_t^k)\mathrm{d}t=\frac{a}{2}-\int_{-\varepsilon}^0(\tilde W_t^{k}-W_t^k)\mathrm{d}t-\int_{\bar\varepsilon}^{\sigma_{k,-}}(\tilde W_t^{k}-W_t^k)\mathrm{d}t-\frac{1}{2}\sum_{k'=1}^{k} \breve T_{k'})=0$. 

 Now, we may consider a Brownian motion $(\breve W_t^{k})_{t\in[-\varepsilon,\varepsilon]}$ so that $\breve W_{-\varepsilon}^k=W_{-\varepsilon}^k$ and $(\tilde W_t^{k})_{t\in[-\varepsilon,\varepsilon]}$ is the Brownian motion $(\breve W_t^{k})_{t\in[-\varepsilon,\varepsilon]}$ reflected on $W^k$ on $[-\varepsilon,0]$ and free on $[0,\varepsilon]$. Then the process $(\breve W_t^{k}-\breve W_0^k - \frac{t}{\bar\varepsilon}(\breve W_{\bar\varepsilon}^{k}-\breve W_0^k))_{t\in[0,\bar\varepsilon]}$ is independent from $(\breve W_t^{k}-\breve W_{-\varepsilon}^k)_{t\in[-\varepsilon,0]}$ and $(\breve W_t^{k}-\breve W_{0}^k)_{t\in[\bar\varepsilon,\varepsilon]}$, as these are Gaussian processes and the covariances are 0. Furthermore, these three processes are independent from $W^k,\breve T_1,...,\breve T_k$, so $(\breve W_t^{k}-\breve W_0^k - \frac{t}{\bar\varepsilon}(\breve W_{\bar\varepsilon}^{k}-\breve W_0^k))_{t\in[0,\bar\varepsilon]}$ is independent from $(\breve W_t^{k}-\breve W_{-\varepsilon}^k)_{t\in[-\varepsilon,0]},(\breve W_t^{k}-\breve W_{0}^k)_{t\in[\bar\varepsilon,\varepsilon]},W^k,\breve T_1,...,\breve T_k$. This implies $(\tilde W_t^{k}-\tilde W_0^k - \frac{t}{\bar\varepsilon}(\tilde W_{\bar\varepsilon}^{k}-\tilde W_0^k))_{t\in[0,\bar\varepsilon]}$ is independent from $(\tilde W_t^{k})_{t\in[-\varepsilon,0]},(\tilde W_t^{k})_{t\in[\bar\varepsilon,\varepsilon]},W^k,\breve T_1,...,\breve T_k$. In addition, we may write $\int_{0}^{\bar\varepsilon}(\tilde W_t^{k}-W_t^k)\mathrm{d}t=\int_{0}^{\bar\varepsilon}(\tilde W_t^{k}-\tilde W_0^k - \frac{t}{\bar\varepsilon}(\tilde W_{\bar\varepsilon}^{k}-\tilde W_0^k))\mathrm{d}t+\int_{0}^{\bar\varepsilon}(\tilde W_0^k + \frac{t}{\bar\varepsilon}(\tilde W_{\bar\varepsilon}^{k}-\tilde W_0^k)-W_t^k)\mathrm{d}t$, so we can write $\mathds{P}(\int_{0}^{\bar\varepsilon}(\tilde W_t^{k}-W_t^k)\mathrm{d}t=\frac{a}{2}-\int_{-\varepsilon}^0(\tilde W_t^{k}-W_t^k)\mathrm{d}t-\int_{\bar\varepsilon}^{\sigma_{k,-}}(\tilde W_t^{k}-W_t^k)\mathrm{d}t-\frac{1}{2}\sum_{k'=1}^{k} \breve T_{k'})=\mathds{P}(\int_{0}^{\bar\varepsilon}(\tilde W_t^{k}-\tilde W_0^k - \frac{t}{\bar\varepsilon}(\tilde W_{\bar\varepsilon}^{k}-\tilde W_0^k))\mathrm{d}t=\frac{a}{2}-I)$, where $I$ depends only on $(\tilde W_t^{k})_{t\in[-\varepsilon,0]},(\tilde W_t^{k})_{t\in[\bar\varepsilon,\varepsilon]},W^k,\breve T_1,...,\breve T_k$ hence is independent from $\int_{0}^{\bar\varepsilon}(\tilde W_t^{k}-\tilde W_0^k - \frac{t}{\bar\varepsilon}(\tilde W_{\bar\varepsilon}^{k}-\tilde W_0^k))\mathrm{d}t$. Consequently, it is enough to prove that $\int_{0}^{\bar\varepsilon}(\tilde W_t^{k}-\tilde W_0^k - \frac{t}{\bar\varepsilon}(\tilde W_{\bar\varepsilon}^{k}-\tilde W_0^k))\mathrm{d}t$ has no atoms. In addition, we may write $\int_{0}^{\bar\varepsilon}(\tilde W_t^{k}-\tilde W_0^k - \frac{t}{\bar\varepsilon}(\tilde W_{\bar\varepsilon}^{k}-\tilde W_0^k))\mathrm{d}t=\int_0^{\bar\varepsilon/2}(\tilde W_t^{k}-\tilde W_0^k - \frac{t}{\bar\varepsilon}(\tilde W_{\bar\varepsilon}^{k}-\tilde W_0^k)-\frac{2}{\bar\varepsilon}t (\tilde W_{\bar\varepsilon/2}^k-\frac{1}{2}\tilde W_{0}^k-\frac{1}{2}\tilde W_{\bar\varepsilon}^k))\mathrm{d}t+\int_{\bar\varepsilon/2}^{\bar\varepsilon}(\tilde W_t^{k}-\tilde W_0^k - \frac{t}{\bar\varepsilon}(\tilde W_{\bar\varepsilon}^{k}-\tilde W_0^k)-(2-\frac{2}{\bar\varepsilon}t) (\tilde W_{\bar\varepsilon/2}^k-\frac{1}{2}\tilde W_{0}^k-\frac{1}{2}\tilde W_{\bar\varepsilon}^k)))\mathrm{d}t+\frac{\bar\varepsilon}{2}(\tilde W_{\bar\varepsilon/2}^k-\frac{1}{2}\tilde W_{0}^k-\frac{1}{2}\tilde W_{\bar\varepsilon}^k)$. Moreover, $\frac{\bar\varepsilon}{2}(\tilde W_{\bar\varepsilon/2}^k-\frac{1}{2}\tilde W_{0}^k-\frac{1}{2}\tilde W_{\bar\varepsilon}^k)$ is independent from the two integrands (one can check the covariances are 0), hence from the sum of the integrals. Furthermore, $\frac{\bar\varepsilon}{2}(\tilde W_{\bar\varepsilon/2}^k-\frac{1}{2}\tilde W_{0}^k-\frac{1}{2}\tilde W_{\bar\varepsilon}^k)=\frac{\bar\varepsilon}{4}(\tilde W_{\bar\varepsilon/2}^k-\tilde W_{0}^k)+\frac{\bar\varepsilon}{4}(\tilde W_{\bar\varepsilon/2}^k-\tilde W_{\bar\varepsilon}^k)$ has no atoms, therefore $\int_{0}^{\bar\varepsilon}(\tilde W_t^{k}-\tilde W_0^k - \frac{t}{\bar\varepsilon}(\tilde W_{\bar\varepsilon}^{k}-\tilde W_0^k))\mathrm{d}t$ has no atoms, which ends the proof of Lemma \ref{lem_no_atoms}.
\end{proof}

Finally, we give the proof of Claim \ref{claim_conv_va_Skorohod}, which is needed in the proof of Proposition \ref{prop_conv_envts} (the notations are defined there).

 \begin{proof}[Proof of Claim \ref{claim_conv_va_Skorohod}.]
  We denote $F_N((\frac{1}{\sqrt{n}}\hat{\tilde E}_{k,nt}^{N,-})_{t\in[-\varepsilon,\varepsilon]},(\frac{1}{\sqrt{n}}\hat E_{k,nt}^{N})_{t\in[-\varepsilon,\varepsilon]})=(\frac{1}{n}\hat\sigma_{k,-}^N,(\frac{1}{\sqrt{n}}\hat E_{k,nt}^{N,-})_{t\in[-\varepsilon,\varepsilon]})$; we also denote $F((\hat{\tilde W}_t^k)_{t\in[-\varepsilon,\varepsilon]},(\hat W_t^k)_{t\in[-\varepsilon,\varepsilon]})=(\hat \sigma_{k,-},(\hat V_t^{k,-})_{t\in[-\varepsilon,\varepsilon]})$. We begin by proving that $\frac{1}{n}\hat\sigma_{k,-}^N$ converges in probability to $\hat \sigma_{k,-}$. Let $\delta_1,\delta_2 > 0$. 
  For any $\delta \in \mathds{R}$, we denote $\frac{1}{n}\hat\sigma_{k,-}^N(\delta)=\inf\{t\in[0,\varepsilon]\,|\,\frac{1}{\sqrt{n}} \hat{\tilde E}_{k,nt}^{N,-} \leq \frac{1}{\sqrt{n}}\hat E_{k,nt}^{N}+\delta\}$ and $\hat\sigma_{k,-}(\delta)=\inf\{t\in[0,\varepsilon]\,|\,\hat{\tilde W}_t^k \leq \hat W_t^k+\delta\}$. By Proposition \ref{prop_envts_above_barrier} we have $\mathds{P}(V_0^{k,-}>W_0^k)=1$, hence $\mathds{P}(\hat {\tilde W}_0^k>\hat W_0^k)=1$, hence Lemma \ref{lem_conv_abs_times} implies that $\mathds{P}(|\hat\sigma_{k,-}(0)-\hat\sigma_{k,-}(\delta)|>\delta_1)$ tends to 0 when $\delta$ tends to 0, thus there exists a $\delta_0 > 0$ so that $\mathds{P}(|\hat\sigma_{k,-}(0)-\hat\sigma_{k,-}(\pm\delta_0)|>\delta_1)\leq\delta_2/4$. We also denote $\mathcal{I}_1(\delta)=\{\|(\frac{1}{\sqrt{n}}\hat E_{k,nt}^{N})_{t\in[-\varepsilon,\varepsilon]}-(\hat W_t^k)_{t\in[-\varepsilon,\varepsilon]}\|_\infty \geq \delta\}$ and $\mathcal{I}_1'(\delta)=\{\|(\frac{1}{\sqrt{n}}\hat{\tilde E}_{k,nt}^{N,-})_{t\in[-\varepsilon,\varepsilon]}-(\hat{\tilde W}_t^k)_{t\in[-\varepsilon,\varepsilon]}\|_\infty \geq \delta\}$. Since $((\frac{1}{\sqrt{n}}\hat{\tilde E}_{k,nt}^{N,-})_{t\in[-\varepsilon,\varepsilon]},(\frac{1}{\sqrt{n}}\hat E_{k,nt}^{N})_{t\in[-\varepsilon,\varepsilon]})$ converges in probability to $((\hat{\tilde W}_t^k)_{t\in[-\varepsilon,\varepsilon]},(\hat W_t^k)_{t\in[-\varepsilon,\varepsilon]})$ when $N$ tends to $+\infty$, we deduce that $\mathds{P}(\mathcal{I}_1(\delta_0/3)),\mathds{P}(\mathcal{I}_1'(\delta_0/3)) \leq \delta_2/4$ when $N$ is large enough. Now, if $(\mathcal{I}_1(\delta_0/3))^c$ and $(\mathcal{I}_1'(\delta_0/3))^c$ occur, then for $0\leq t\leq\varepsilon\wedge\hat\sigma_{k,-}(\delta_0)$ we have $\frac{1}{\sqrt{n}}\hat{\tilde E}_{k,nt}^{N,-} \geq \hat {\tilde W}_t^k-\delta_0/3 \geq \hat W_t^k +2\delta_0/3 \geq \frac{1}{\sqrt{n}}\hat E_{k,nt}^{N}+\delta_0/3$, hence $\frac{1}{n}\hat\sigma_{k,-}^N \geq \hat\sigma_{k,-}(\delta_0)$. We now prove a symmetric bound. If $\hat\sigma_{k,-}(-\delta_0)$ is finite (we temporarily call it $s$), we have $\frac{1}{\sqrt{n}}\hat{\tilde E}_{k,ns}^{N,-} \leq \hat {\tilde W}_{s}^k+\delta_0/3 \leq \hat W_{s}^k -2\delta_0/3 \leq \frac{1}{\sqrt{n}}\hat E_{k,ns}^{N}-\delta_0/3$. Moreover, we already saw that $\mathds{P}(\hat{\tilde W}_0^{k} \geq \hat W_0^{k})=1$, so $\hat\sigma_{k,-}(-\delta_0) > 0$ almost-surely, which implies that $\frac{1}{n}<\hat\sigma_{k,-}(-\delta_0)$ when $N$ is large enough. In addition, $\mathds{P}(\frac{1}{\sqrt{n}}\tilde E_{k,1}^{N,-} \geq \frac{1}{\sqrt{n}}E_{k,1}^{N})=\mathds{P}(L_{X_{T_k}+1}^{T_k,-}\geq0)=1$ so $\mathds{P}(\frac{1}{\sqrt{n}}\hat{\tilde E}_{k,1}^{N,-} \geq \frac{1}{\sqrt{n}}\hat E_{k,1}^{N})=1$, so if $\hat\sigma_{k,-}(-\delta_0)$ is finite, when $N$ is large enough there exists $t\in[\frac{1}{n},\hat\sigma_{k,-}(-\delta_0)]$ so that $\frac{1}{\sqrt{n}}\hat{\tilde E}_{k,nt}^{N,-}=\frac{1}{\sqrt{n}}\hat E_{k,nt}^{N}$ therefore $\frac{1}{n}\hat\sigma_{k,-}^N \leq \hat\sigma_{k,-}(-\delta_0)$. We deduce 
 \[
 \textstyle
  \mathds{P}(|\frac{1}{n}\hat\sigma_{k,-}^N-\hat \sigma_{k,-}(0)|>\delta_1)
  \leq \mathds{P}(\{|\hat\sigma_{k,-}(0)-\hat\sigma_{k,-}(-\delta_0)|>\delta_1\}\cup\{|\hat\sigma_{k,-}(0)-\hat\sigma_{k,-}(\delta_0)|>\delta_1\}\cup\mathcal{I}_1(\delta_0/3)\cup\mathcal{I}_1'(\delta_0/3))
 \]
\[
 \leq \delta_2/4+\delta_2/4+\delta_2/4+\delta_2/4=\delta_2
\]
when $N$ is large enough. This yields that $\frac{1}{n}\hat\sigma_{k,-}^N$ converges in probability to $\hat \sigma_{k,-}(0)=\hat\sigma_{k,-}$ a.s.
 
 We now prove that $(\frac{1}{\sqrt{n}}\hat E_{k,nt}^{N,-})_{t\in[-\varepsilon,\varepsilon]}$ converges in probability to $(\hat V_t^{k,-})_{t\in[-\varepsilon,\varepsilon]}$. Let $\delta_1,\delta_2 > 0$. For any $\delta > 0$, we denote $\mathcal{I}_2(\delta)=\{\exists\, t,t'\in[-\varepsilon,\varepsilon],|t-t'|\leq\delta,|\hat W_t^k-\hat W_{t'}^k|\geq\delta_1/3\}$ and $\mathcal{I}_2'(\delta)=\{\exists\, t,t'\in[-\varepsilon,\varepsilon],|t-t'|\leq\delta,|\hat {\tilde W}_t^k-\hat {\tilde W}_{t'}^k|\geq\delta_1/3\}$. The functions $(\hat{\tilde W}_t^k)_{t\in[-\varepsilon,\varepsilon]}$ and $(\hat W_t^k)_{t\in[-\varepsilon,\varepsilon]}$ are continous on $[-\varepsilon,\varepsilon]$, hence uniformly continuous, so there exists $\delta_0 > 0$ so that for $\mathds{P}(\mathcal{I}_2(\delta_0))\leq\delta_2/6$ and $\mathds{P}(\mathcal{I}_2'(\delta_0))\leq\delta_2/6$. In addition, we can use the previous reasoning to prove that $\mathds{P}(|\frac{1}{n}\hat\sigma_{k,-}^N-\hat \sigma_{k,-}|>\delta_0) \leq \delta_2/3$ when $N$ is large enough. We now assume $|\frac{1}{n}\hat\sigma_{k,-}^N-\hat \sigma_{k,-}|\leq\delta_0$, $(\mathcal{I}_1(\delta_1/3))^c$, $(\mathcal{I}_1'(\delta_1/3))^c$, $(\mathcal{I}_2(\delta_0))^c$ and $(\mathcal{I}_2'(\delta_0))^c$ occur, and bound $\|(\frac{1}{\sqrt{n}}\hat E_{k,nt}^{N,-})_{t\in[-\varepsilon,\varepsilon]}-(\hat V_t^{k,-})_{t\in[-\varepsilon,\varepsilon]}\|_\infty$. Let $t\in[-\varepsilon,\varepsilon]$. If $t \leq \hat \sigma_{k,-}-\delta_0$, we have $\hat V_t^{k,-}=\hat{\tilde W}_t^k$, and we also have $t \leq \frac{1}{n}\hat\sigma_{k,-}^N$ which yields $\frac{1}{\sqrt{n}}\hat E_{k,nt}^{N,-}=\frac{1}{\sqrt{n}}\hat {\tilde E}_{k,nt}^{N,-}$, therefore $(\mathcal{I}_1'(\delta_1/3))^c$ implies $|\frac{1}{\sqrt{n}}\hat E_{k,nt}^{N,-}-\hat V_t^{k,-}| \leq \delta_1/3$. If $t \geq \hat \sigma_{k,-}+\delta_0$, we have $\hat V_t^{k,-}=\hat W_t^k$, and we also have $t \geq \frac{1}{n}\hat\sigma_{k,-}^N$ which yields $\frac{1}{\sqrt{n}}\hat E_{k,nt}^{N,-}=\frac{1}{\sqrt{n}}\hat E_{k,nt}^{N}$, therefore $(\mathcal{I}_1(\delta_1/3))^c$ implies $|\frac{1}{\sqrt{n}}\hat E_{k,nt}^{N,-}-\hat V_t^{k,-}| \leq \delta_1/3$. Now let $t \in [\hat \sigma_{k,-}-\delta_0,\hat \sigma_{k,-}+\delta_0]$ (which means $\hat \sigma_{k,-}$ is finite). Denoting temporarily $\hat \sigma_{k,-}=s$, by $(\mathcal{I}_2(\delta_0))^c$ and $(\mathcal{I}_2'(\delta_0))^c$ we have $|\hat W_t^k-\hat W_{s}^k|\leq\delta_1/3$ and $|\hat {\tilde W}_t^k-\hat {\tilde W}_{s}^k|\leq\delta_1/3$, thus since $\hat W_{s}^k=\hat {\tilde W}_{s}^k$ we get $|\hat W_t^k-\hat {\tilde W}_t^k|\leq2\delta_1/3$. If $t \leq \frac{1}{n}\hat\sigma_{k,-}^N$ then $\frac{1}{\sqrt{n}}\hat E_{k,nt}^{N,-}=\frac{1}{\sqrt{n}}\hat {\tilde E}_{k,nt}^{N,-}$, so by $(\mathcal{I}_1'(\delta_1/3))^c$ we have $|\frac{1}{\sqrt{n}}\hat E_{k,nt}^{N,-}-\hat{\tilde W}_t^k|\leq \delta_1/3$; if $t \geq \frac{1}{n}\hat\sigma_{k,-}^N$ then $\frac{1}{\sqrt{n}}\hat E_{k,nt}^{N,-}=\frac{1}{\sqrt{n}}\hat {E}_{k,nt}^{N}$, so by $(\mathcal{I}_1'(\delta_1/3))^c$ we have $|\frac{1}{\sqrt{n}}\hat E_{k,nt}^{N,-}-\hat W_t^k|\leq \delta_1/3$. Since $\hat V_t^{k,-}$ is either $\hat W_t^k$ or $\hat{\tilde W}_t^k$, we deduce $|\frac{1}{\sqrt{n}}\hat E_{k,nt}^{N,-}-\hat V_t^{k,-}|\leq\delta_1$. Consequently, for all $t\in[-\varepsilon,\varepsilon]$ we have $|\frac{1}{\sqrt{n}}\hat E_{k,nt}^{N,-}-\hat V_t^{k,-}|\leq\delta_1$, therefore $\|(\frac{1}{\sqrt{n}}\hat E_{k,nt}^{N,-})_{t\in[-\varepsilon,\varepsilon]}-(\hat V_t^{k,-})_{t\in[-\varepsilon,\varepsilon]}\|_\infty\leq\delta_1$. We deduce
 \[
  \textstyle
  \mathds{P}\left(\left\|(\frac{1}{\sqrt{n}}\hat E_{k,nt}^{N,-})_{t\in[-\varepsilon,\varepsilon]}-(\hat V_t^{k,-})_{t\in[-\varepsilon,\varepsilon]}\right\|_\infty>\delta_1\right)
  \leq \mathds{P}\left(\{|\frac{1}{n}\hat\sigma_{k,-}^N-\hat \sigma_{k,-}|>\delta_0\} \cup \mathcal{I}_1(\delta_1/3)\cup \mathcal{I}_1'(\delta_1/3) \cup \mathcal{I}_2(\delta_0) \cup \mathcal{I}_2'(\delta_0)\right)
 \]
 \[
  \leq \delta_2/3+\delta_2/6+\delta_2/6+\delta_2/6+\delta_2/6=\delta_2
 \]
 when $N$ is large enough. This yields that $(\frac{1}{\sqrt{n}}\hat E_{k,nt}^{N,-})_{t\in[-\varepsilon,\varepsilon]}$ converges in probability to $(\hat V_t^{k,-})_{t\in[-\varepsilon,\varepsilon]}$.
 \end{proof}

\end{document}